\documentclass[11pt, a4paper]{article}

\pdfoutput=1 

\usepackage[textwidth=6.5in, textheight=9.5in]{geometry}

\usepackage[nottoc]{tocbibind}

\usepackage[hidelinks, unicode=false]{hyperref}

\usepackage[mathletters]{ucs}
\usepackage[utf8x]{inputenc}
\usepackage[T1]{fontenc}

\usepackage[scr=rsfso]{mathalpha}

\usepackage[shortlabels]{enumitem}
\setlist{itemsep=0pt}

\usepackage{amssymb}
\usepackage{amsmath}
\usepackage{amsthm}

\newcommand{\bbN}{\mathbb{N}}
\newcommand{\bbZ}{\mathbb{Z}}

\newcommand{\bbR}{\mathbb{R}}
\newcommand{\bbC}{\mathbb{C}}

\newcommand{\bbT}{\mathbb{T}}
\newcommand{\pt}{\mathrm{pt}}
\newcommand{\id}{\mathrm{id}}
\newcommand{\vol}{\mathrm{vol}}
\newcommand{\Cl}{\mathscr{C}\ell}
\newcommand{\bfS}{\mathbf{S}}

\DeclareMathOperator{\Aut}{Aut}
\DeclareMathOperator{\End}{End}
\DeclareMathOperator{\Hom}{Hom}

\DeclareMathOperator{\ad}{ad}

\DeclareMathOperator{\Span}{span}

\DeclareMathOperator{\dom}{dom}

\DeclareMathOperator{\Iso}{Iso}
\DeclareMathOperator{\Conf}{Conf}

\DeclareMathOperator{\Lip}{Lip}
\DeclareMathOperator{\im}{im}
\DeclareMathOperator{\hodge}{{\star}}
\DeclareMathOperator{\supp}{supp}
\DeclareMathOperator{\Mtc}{Mtc}

\numberwithin{equation}{section}

\newtheorem{theorem}[equation]{Theorem}
\newtheorem*{theorem*}{Theorem}
\newtheorem{proposition}[equation]{Proposition}
\newtheorem{lemma}[equation]{Lemma}

\newtheorem{corollary}[equation]{Corollary}

\theoremstyle{definition}
\newtheorem{definition}[equation]{Definition}

\theoremstyle{remark}
\newtheorem{remark}[equation]{Remark}
\newtheorem{example}[equation]{Example}

\usepackage{thmtools}
\declaretheorem[
	style=remark,
	sibling=equation,
	postheadhook={\leavevmode \begin{enumerate}[topsep=\parsep]},
	prefoothook={\end{enumerate} \medbreak}
]{remarks}

\usepackage{slashed}

\usepackage{tikz}
\usetikzlibrary{cd}

\usepackage{microtype}

\title{Conformal transformations and equivariance \\ in unbounded KK-theory}
\author{Ada Masters$^{\dagger*}$, Adam Rennie$^{\ddagger}$\thanks{\textsc{email:} 
	\texttt{ada.masters@math.lth.se}, \texttt{renniea@uow.edu.au}
}
\smallskip \\
$^{\dagger}$Centre for Mathematical Sciences, Lund University \\
Box 118, 221 00 Lund, Sweden \smallskip \\
$^{\ddagger}$School of Mathematics and Applied Statistics, University of Wollongong \\
Wollongong, Australia
}

\begin{document}

\maketitle

\begin{abstract}
	We extend unbounded Kasparov theory to encompass conformal group and quantum group equivariance. This new framework allows us to treat conformal actions on both manifolds and noncommutative spaces. As examples, we present unbounded representatives of Kasparov's $\gamma$-element for the real and complex Lorentz groups and display the conformal $SL_q(2)$-equivariance of the standard spectral triple of the Podle\'{s} sphere.
	 In pursuing descent for conformally equivariant cycles, we are led to a new framework for representing Kasparov classes. Our new representatives are unbounded, possess a dynamical quality, and also include known twisted spectral triples. We define an equivalence relation on these new representatives whose classes form an abelian group surjecting onto KK.
	The technical innovation which underpins these results is a novel multiplicative perturbation theory. By these means, we obtain Kasparov classes from the bounded transform with minimal side conditions.
\end{abstract}

{\small \tableofcontents}

\section{Introduction}

In this paper we present a unifying framework for conformal transformations, group equivariance and quantum group equivariance for unbounded Kasparov theory. Our techniques lead to a new class of representatives of Kasparov classes which we call \emph{conformally generated cycles}. These new representatives include known examples of twisted spectral triples.

Kasparov's $KK$-theory \cite{Kasparov_1981,Kasparov_1988} encodes group equivariance in a conformal way, and this was made precise at the level of bounded cycles (Fredholm modules) by B\"{a}r \cite{Bar_2007}. The first attempts at encoding equivariance for the unbounded Kasparov cycles of \cite{Baaj_1983} was by Kucerovsky \cite{Kucerovsky_1994}, capturing isometric and nearly isometric actions, as we discuss below.

Amongst our results on conformal equivariance, we show that the descent maps applied to \emph{conformally equivariant} unbounded Kasparov modules are not unbounded Kasparov modules, but are conformally generated cycles. This applies to both group and quantum group equivariance, and demonstrates the need for a wider class of cycles in the unbounded theory. We also give unbounded representatives of Kasparov's $\gamma$-element for the real and complex Lorentz groups, a conformally equivariant (higher order) spectral triple for the Heisenberg group with dilation action and display the conformal $SL_q(2)$-equivariance of the standard spectral triple of the Podle\'{s} sphere. 

To implement and study conformal equivariance for unbounded cycles we introduce several novel techniques. The most important is a multiplicative perturbation theory for linear operators that complements the usual additive perturbation theory. To fully accommodate conformal equivariance requires a dynamical notion of operators locally bounded over the right-hand algebra of a Hilbert module, which we call matched operators. We now explain these results and our methods in more detail.

One aspect of the unbounded picture of KK-theory which has not been fully explored is group equivariance. One reason for this is that the definition of equivariance for unbounded Kasparov modules  made by Kucerovsky in \cite{Kucerovsky_1994} fails to capture all the degrees of freedom available in the bounded picture of equivariant KK-theory. In the following we refer to Kucerovsky's definition as \emph{uniform equivariance}, see Definition \ref{definition:ordinary-unbounded-equivariance}. Perhaps the easiest example  is the Dirac spectral triple on a Riemannian manifold, equipped with the action of a group. If the action is isometric, the Dirac operator is invariant. If the action is a conformal one, the Fredholm module defined by the bounded transform yields a bounded equivariant Fredholm module, but it unclear how to think of the spectral triple as being equivariant.

For the case of a compact manifold, conformally equivalent Dirac operators have been addressed in the context of noncommutative geometry by Bär \cite{Bar_2007}. A conformal change of metric has the effect \( \slashed{D} \leadsto k^{-1/2} \slashed{D} k^{-1/2} \) on the Atiyah–Singer Dirac operator. By considering principal symbols, the bounded transform \( \slashed{D} (1 + \slashed{D}^2)^{-1/2} \) changes only by a compact operator. In \S\ref{section:conf-tr-mult}, we give new tools to identify two self-adjoint regular operators as having `close' bounded transforms in much more general circumstances.

One interpretation of conformal actions and changes of metric is via Connes and Moscovici's \emph{twisted spectral triples} \cite{Connes_2008}. One of the two main examples \cite[\S 2.2]{Connes_2008} of twisted spectral triples given by Connes and Moscovici is built from a multiplicative perturbation \( D \leadsto k D k \). The other main example \cite[\S 2.3]{Connes_2008} \cite[§3.1]{Moscovici_2010} is built from a Dirac spectral triple \( (C_0(X), L^2(X, S), \slashed{D}) \) on a Riemannian manifold \( X \), equipped with the conformal action of a discrete group \( G \). One extends the algebra \( C_0(X) \) to the crossed product \( C_0(X) ⋊ G \) and
\[ (C_0(X) ⋊ G, L^2(X, S), \slashed{D}) \]
becomes a Lipschitz regular twisted spectral triple. In \S\ref{subsec:conf-descent}, we will interpret this as the dual Green–Julg map of a \emph{conformally equivariant} unbounded cycle and show that such examples possess well-defined bounded transforms without recourse to the Lipschitz regularity condition of \cite[Definition 3.1]{Connes_2008}.

In the framework of the spectral action principle, Chamseddine and Connes \cite{Chamseddine_2006} calculate the effect of rescaling the Spectral Standard Model Dirac operator \( D \leadsto e^{-ϕ/2} D e^{-ϕ/2} \), where the \emph{dilaton} \( ϕ \) is interpreted as a scalar field. Apart from the Higgs mass term, the entire Lagrangian of the Standard Model of particle physics is conformally invariant, which was a background motivation for this work.

The technical innovation which underpins our results is a multiplicative perturbation theory for self-adjoint regular operators on Hilbert modules. This perturbation theory relates the bounded transforms \( D(1+D^2)^{-1/2} \) and \( \mu D\mu^*(1+(\mu D\mu^*)^2)^{-1/2} \) of \( D \) and its multiplicative perturbation \( \mu D\mu^* \), for suitable $\mu$. Together with the well-known additive perturbation theory $D \leadsto D+A$ for (relatively) bounded $A$, Theorem \ref{thm:converse} says, roughly, that any perturbation  preserving the KK-class of the bounded transform  takes the form $\mu D \mu^*+A$.

We introduce several concepts making use of the multiplicative perturbation theory of \S\ref{section:conf-tr-mult}, among which are:
\begin{itemize}
	\item \emph{Conformal transformations} between unbounded Kasparov modules, Definition \ref{definition:conformal-transformation}, and a singular version, Definition \ref{definition:sing-conf-transformation};
	\item \emph{Conformal group equivariance} for unbounded Kasparov modules, Definition \ref{definition:conformal-equivariance};
	\item \emph{Conformal quantum group equivariance} for unbounded Kasparov modules, Definition \ref{defn:conf-ess-equi};
		\item \emph{Conformally generated cycles}, Definition \ref{definition:conformally-generated-cycle-general}, providing a new picture of KK-theory, generalising unbounded KK-theory; and
\item The equivalence relation of \emph{conformism} for unbounded Kasparov modules, Definition \ref{definition:conformism}. 
\end{itemize}
Conformally generated cycles have a dynamical aspect in addition to a geometrical one. We show that this framework is adapted to all known examples of twisted spectral triples with well-defined bounded transforms. Key features of our approach are the lack of a `twist', in the sense of an algebra automorphism, and a bounded transform which does not depend on any additional smoothness condition such as Lipschitz regularity. We show in \S\ref{subsec:conf-descent} that Kasparov's descent map (and the dual Green–Julg map) applied to group and quantum group conformally equivariant unbounded Kasparov modules give rise to conformally generated cycles whose bounded transforms define the same classes as the descent map (dual Green–Julg map) applied to the bounded transforms of the original modules.

\bigskip

We begin by considering conformal transformations between (higher order) unbounded Kasparov modules in \S\ref{section:conf-tr-mult}. The motivation for such a framework is conformal changes of metric of Riemannian manifolds and the noncommutative torus, of which we give some details in \S\ref{section:conformal-motivating-examples}. In the simplest instance for unbounded Kasparov modules $(A,E,D_1)$ and $(A,E',D_2)$, these transformations are a pair $(U,\mu)$ with $U:E\to E'$ unitary and $\mu$ a bounded invertible endomorphism (which is even if the module is graded) such that, for all $a$ in a dense subset of $A$,
\begin{equation}
U^* D_2 U a - a μ D_1 μ^*
\label{eq:con-basic}
\end{equation}
is bounded. The Leibniz rule shows that those $a$ for which \eqref{eq:con-basic} is bounded naturally form a (not norm-closed) ternary ring of operators, rather than a $*$-algebra.
The implicit presence of ternary rings of operators will be a feature of many of our definitions.
For the technical results in \S\ref{section:multiplicative-perturbations}, we require that the `conformal factor' $\mu$ be a bounded and invertible operator, although it need not have a globally bounded derivative. We prove the following as Theorem \ref{theorem:bdd-transform-conformal-transformation}.

\begin{theorem*}
	\label{theorem:bdd-transform-conformal-transformation-intro}
	Let \( (U, μ) \) be a conformal transformation from the order-\( \frac{1}{1 - α} \) cycle \( (A, E_B, D_1) \) to the order-\( \frac{1}{1 - α} \) cycle \( (A, E'_B, D_2) \). Then the bounded transforms \( (A, E_B, F_{D_1}) \) and \( (A, E'_B, F_{D_2}) \) are unitarily equivalent up to locally compact perturbation via the unitary \( U \); that is
	\[ (U^* F_{D_2} U - F_{D_1}) a ∈ \End^0(E) \]
	for all \( a ∈ A \). Hence \( [(A, E_B, F_{D_1})]=[(A, E'_B, F_{D_2})]\in KK(A,B). \)
\end{theorem*}

On a noncompact manifold, this is not sufficient to describe all conformal changes of metric. One technical issue which arises is that a complete Riemannian manifold, such as the hyperbolic plane, may be conformally equivalent to an incomplete manifold, such as the unit disc, and therefore the self-adjointness of a Dirac operator may not be preserved. With this caveat, we give in \S\ref{section:sing-conf-tr} a framework modelled abstractly on the idea of an open cover extending the idea in \eqref{eq:con-basic}.

We also show in \S\ref{section:log} that the \emph{logarithmic transform} $D\to L_D=F_D\log((1+D^2)^{1/2})$, due to Goffeng, Mesland, and the second named author \cite{Goffeng_2019a}, turns multiplicative perturbations into additive ones. In Theorem \ref{thm:conf-log-bdd} we prove

\begin{theorem*}
	Let \( (U, μ) \) be a conformal transformation from the order-\( \frac{1}{1 - α} \) cycle \( (A, E_B, D_1) \) to the order-\( \frac{1}{1 - α} \) cycle \( (A, E'_B, D_2) \). Then the logarithmic transforms \( (A, E_B, L_{D_1}) \) and \( (A, E'_B, L_{D_2}) \) are related by the unitary \( U \), up to locally bounded perturbation; in particular, \( A \) is contained in the closure of the set of \( a ∈ \End^*(E) \) such that
	\[ (U^* L_{D_2} U - L_{D_1}) a \qquad [L_{D_1}, a] \]
	are bounded.
\end{theorem*}

We then extend the existing definitions of uniform group equivariance, due to Kucerovsky \cite{Kucerovsky_1994}, to higher order 
unbounded Kasparov modules and incorporating conformal actions in \S\ref{section:group-equivariant}, based on the idea of conformal transformation in \eqref{eq:con-basic}. This is necessary to include the full range of equivariance encoded for bounded Kasparov modules, as indicated by the results of B\"{a}r \cite{Bar_2007} and explained using the example of the $ax+b$ group acting on $\mathbb{R}$. In Theorem \ref{theorem:conformal-equivariance-bdd-transform} we prove 

\begin{theorem*}
	\label{theorem:main-equi}
	The bounded transform of a conformally equivariant higher-order unbounded Kasparov module is an equivariant bounded Kasparov module.
\end{theorem*}

The logarithmic transform again changes multiplicative perturbations coming from conformal actions to additive perturbations. In Theorem \ref{thm:conf-log} we prove

\begin{theorem*}
	\label{theorem:main-equi-log}
	The logarithmic transform of a conformally equivariant higher-order unbounded Kasparov module is a uniformly equivariant unbounded Kasparov module.
\end{theorem*}

These results allow us to represent the $γ$-elements of Kasparov and Chen for the Lorentz groups and of Julg and Kasparov for the complex Lorentz groups, in \S\ref{section:lorentz-gamma}. In \S\ref{section:heisenberg}, we give a genuinely noncommutative example, a second order spectral triple for the C*-algebra of the Heisenberg group which is equivariant for the dilation action.

In \S\ref{section:quantum-group-equivariant} we study C*-bialgebra equivariance for higher-order unbounded Kasparov modules, following the treatment in the bounded picture by Baaj and Skandalis \cite{Baaj_1989}. We give a definition for uniform (non-conformal) equivariance of unbounded Kasparov modules which, to our knowledge, has not previously appeared in the literature (except in the \emph{isometric} case \cite{Goswami_2016}). We show how the descent and dual Green–Julg maps work in the setting of uniform  equivariance.

With the definition of uniform equivariance in hand, we define conformal quantum group equivariance for higher-order unbounded Kasparov modules in \S\ref{section:conformal-quantum}. The main example to which we apply this framework is the action of $SL_q(2)$ on the Podleś sphere. In Theorems \ref{theorem:bbd-transform-quantum-conformal} and \ref{theorem:bbd-transform-quantum-conformal-log} we prove

\begin{theorem*}
	\label{theorem:main-equi-q}
	The bounded transform of a conformally quantum group equivariant higher-order unbounded Kasparov module is a quantum group equivariant bounded Kasparov module.
\end{theorem*}

\begin{theorem*}
	\label{theorem:main-equi-q-log}
	The logarithmic transform of a conformally quantum group equivariant higher-order unbounded Kasparov module is a uniformly quantum group equivariant unbounded Kasparov module.
\end{theorem*}

All of the generalisations we have considered so far are brought together in \S \ref{sec:cgc} wherein we introduce conformally generated cycles. These unbounded representatives of Kasparov classes are general enough to include known examples of twisted spectral triples, as we outline at the beginning of \S \ref{sec:cgc}, as well as the result of applying descent and dual Green–Julg maps to group and quantum group conformally equivariant Kasparov modules as we see in \S \ref{subsec:conf-descent}, generalising the constructions for uniform equivariance given in \S \ref{section:descent-group}, in the group case, and \S \ref{section:descent-quantum-group}, in the quantum group case.

In \S\ref{sec:equiv-unbdd}, we generalise cobordism of bounded Kasparov modules, as defined by Cuntz and Skandalis \cite{Cuntz_1986}, to unbounded Kasparov modules. We show in Theorem \ref{theorem:cobordism-gp} that cobordism classes of unbounded Kasparov modules form a \( \bbZ/2\bbZ \)-graded abelian group which surjects onto the usual KK-group.
Finally, in \S \ref{section:equivalence-conf-gen}, we show that cobordism extends to an equivalence relation on conformally generated cycles, and the cobordism classes of such cycles form an abelian group which surjects onto the usual KK-group. 
As a special case, we define \emph{conformism} of unbounded Kasparov modules, using the framework of cobordism to turn the conformal transformations of \S \ref{section:conf-tr-mult} and singular conformal transformations of \S \ref{section:sing-conf-tr} into an equivalence relation.
We show also that conformism classes of unbounded Kasparov modules are an abelian group which surjects onto the usual KK-group.

For the multiplicative perturbation theory of \S\ref{section:multiplicative-perturbations}, we require certain bounds and domain relationships involving fractional powers of positive regular operators on Hilbert modules. Although these are well known in the Hilbert space case, we provide a complete proof in the Hilbert module case in Appendix A.1. For group equivariance, we require certain identifications of Hilbert modules over locally compact Hausdorff spaces and their operators, which we cover in Appendix A.2, based on the approach of Kucerovsky \cite{Kucerovsky_1994}. 

For quantum group equivariance and conformally generated cycles, we use the ideas of \emph{matched operators} and \emph{compactly supported states}. These generalise the multipliers of the Pedersen ideal of a C*-algebra and their positive continuous dual. Given a C*-algebra \( C \) acting on the right of a Hilbert \( B \)-module via a nondegenerate homomorphism \( C \to M(B) \), the \( C \)-matched operators on \( E \) are a subset of the regular operators which form a $*$-algebra (in fact, a pro-C*-algebra), as we show in Appendix A.3. In Appendix A.4, we characterise compactly supported states \cite{Harris_2023} on a C*-algebra in terms of the Pedersen ideal and show that they are weak-$*$-dense in all states.

\medskip\smallskip\noindent {\bf Acknowledgements.} We thank D. Kucerovsky for providing a copy of his thesis, F. Arici and B. Mesland for hospitality at the University of Leiden, and A. Carey for enlightening conversations. AM thanks M. Goffeng and M. Fries for hospitality at the University of Lund and subsequent discussions, and 
R. Yuncken and C. Voigt for helpful conversations. AM also acknowledges the support of an Australian Government RTP scholarship.

\section{Conformal transformations}
\label{section:conf-tr-mult}

For us, Kasparov cycles and their generalisations will be over ungraded \( \bbC \)-algebras. We never need the grading of the module, so our results apply to both even and odd cycles with trivially graded algebras. When we consider Kasparov classes, we will write $KK$ generically for classes of even or odd cycles, and unless mentioned all C*-algebras will be trivially graded and conformal factors are even if the module is graded.

\begin{definition}
	\cite[Definition 2.2]{Kasparov_1988}
	A bounded Kasparov \( A \)-\( B \)-module consists of an \( A \)-\( B \)-correspondence \( E \) and a bounded operator \( F \) on \( E \) such that, for all \( a ∈ A \), the operators
	\[ (F^* - F) a \qquad (1 - F^2) a \qquad [F, a] \]
	are compact. If \( E \) is a \( \bbZ/2\bbZ \)-graded \( A \)-\( B \)-correspondence (that is, with \( A \) acting by even operators), we require that \( F \) be an odd operator and call \( (A, E_B, F) \) an \emph{even} bounded Kasparov module. If \( E \) is ungraded, \( (A, E_B, F) \) is \emph{odd}. If $B=\mathbb{C}$, so that \( E \) is a Hilbert space, \( (A, E, F) \) is a \emph{Fredholm module}.
\end{definition}

We will mostly work in the generality of higher order unbounded Kasparov modules, due to Wahl \cite{Wahl_2007}. We refer to \cite{Woronowicz_1991, Lance_1995} for the theory of regular operators on Hilbert C*-modules.
Throughout we use the notations $\langle D\rangle =(1+D^2)^{1/2}$ and \( F_D = D \langle D\rangle^{-1} = D (1+D^2)^{-1/2} \) for a self-adjoint regular operator $D$ on a Hilbert module.

\begin{definition}
	cf. \cite[Definition A.1]{Goffeng_2015}
	Let \( D \) be a self-adjoint regular operator on a right Hilbert \( B \)-module \( E \). For \( 0 ≤ α ≤ 1 \), let
	\[ \Lip_α^*(D) ⊆ \End_B^*(E) \]
	be the subspace consisting of elements \( a \in \End_B^*(E) \) for which \( a \dom D ⊆ \dom D \) and 
	\[
	 [D, a] ⟨D⟩^{-α} \quad\mbox{and}\quad ⟨D⟩^{-α} [D, a] 
	 \] 
	 extend to bounded adjointable operators. By \cite[Proposition A.5]{Goffeng_2015}, \( \Lip_α^*(D) \) is a $*$-algebra.
\end{definition}

It is shown in \cite[\S A.4]{Masters_2025a} that \( \Lip_α^*(D) \) is a Banach $*$-algebra under an appropriate norm and is closed under the holomorphic functional calculus, but we do not use this here. We will also weaken our definition of unbounded cycles along the lines of \cite[Definition 1.1]{Dungen_2020a} since morphisms between cycles may not naturally preserve a given smooth subalgebra.

\begin{definition}
	cf. \cite[Definition 2.4]{Wahl_2007} \cite[Definition A.2]{Goffeng_2015} \cite[Definition 1.1]{Dungen_2020a}
	Let $0\leq \alpha<1$.
	An order-\( \frac{1}{1 - α} \) \( A \)-\( B \)-cycle consists of an \( A \)-\( B \)-correspondence \( E \) and a regular operator \( D \) on \( E \) such that:
	\begin{enumerate}
		\item \( D \) is self-adjoint;
		\item \( (1 + D^2)^{-1} a \) is compact for all \( a \in A \); and
		\item \( A \) is contained in the operator norm closure of \( \Lip_α^*(D) \).
	\end{enumerate}
	If \( E \) is a \( \bbZ/2\bbZ \)-graded \( A \)-\( B \)-correspondence (that is, with \( A \) acting by even operators), we require that \( D \) be an odd operator and call \( (A, E_B, D) \) an \emph{even} cycle. If \( E \) is ungraded, \( (A, E_B, D) \) is \emph{odd}.
	
	If we have a dense subalgebra \( \mathscr{A} \) of \( A \) which is contained in \( \Lip_α^*(D) \), we will call the cycle an order-\( \frac{1}{1 - α} \) \( \mathscr{A} \)-\( B \)-cycle. If $\alpha=0$ then we refer to order-$1$ cycles as unbounded Kasparov modules, and if $B=\mathbb{C}$, so that \( E \) is a Hilbert space, we call these cycles spectral triples.
\end{definition}

\begin{example}
	\cite[Remark A.0.3]{Goffeng_2015}
	Let \( X \) be a complete Riemannian manifold and \( V \) a vector bundle over \( X \). If \( D \) is a self-adjoint elliptic pseudodifferential operator of order \( m > 0 \) acting on sections of \( V \) then
	\( (C_0(X), L^2(X, V), D) \)
	is an order-\( m \) spectral triple.
\end{example}

The generalisation to `higher order operators' does not interfere with the main topological result for unbounded Kasparov modules.

\begin{theorem}
	cf. \cite[Definition 2.4]{Wahl_2007} \cite[Theorem A.6]{Goffeng_2015} \cite[Proposition 1.7]{Dungen_2020a}
	Let $(A,E_B,D)$ be an order $\frac{1}{1-\alpha}$ $A$-$B$-cycle. Then the bounded transform $D ↦ F_D:=D(1+D^2)^{-1/2}$ gives a Kasparov module $(A,E_B,F_D)$ of the same parity.
\end{theorem}

\begin{remark}
	\label{remark:w45yne57u357ue56ne56uw46hv4ec}
	In this paper, we have chosen to work only with ungraded C*-algebras. We therefore work with even and odd Kasparov modules, about which a small remark is in order.
	Let \( A \) and \( B \) be ungraded C*-algebras. By definition \cite[\S 2.22]{Kasparov_1988}, \( KK_1(A, B) = KK_0(A, B \otimes \Cl_1) \), where the Clifford algebra \( \Cl_1 \) is treated as a graded C*-algebra. The following is well-known; variations can be found in \cite[Proposition IV.A.13(b)]{Connes_1994} and \cite[(8.1.10)]{Higson_2000}. A more sophisticated discussion could involve \emph{multigradings} \cite[Definition 8.1.11, \S A.3]{Higson_2000}.
	
	If \( (A, E_B, F) \) is an odd bounded Kasparov module, we can build a bounded Kasparov \( A \)-\( B \otimes \Cl_1 \)-module
	\[ \bigg( A, (E \oplus E)_{B \otimes \Cl_1}, \begin{pmatrix} & F \\ F & \end{pmatrix} \bigg) \]
	where \( B \otimes \Cl_1 \) acts on the right of \( E \oplus E \) by
	\begin{equation}
		\label{eq:e56nu7irkbj3v56bdnurjbe5vh65}
		\begin{pmatrix} \xi & \eta \end{pmatrix} (b + c \gamma_1) = \begin{pmatrix} \xi & \eta \end{pmatrix} \begin{pmatrix} b & c \\ c & b \end{pmatrix} = \begin{pmatrix} \xi b + \eta c & \xi c + \eta b \end{pmatrix} \qquad (\xi, \eta \in E, b + c \gamma_1 \in B \otimes \Cl_1) ,
	\end{equation}
	and the grading on \( E \oplus E \) is given by \( \big(\begin{smallmatrix} 1 & \\ & -1 \end{smallmatrix}\big) \); cf. \cite[(8.1.10)]{Higson_2000}.

	This process is completely reversible. Given a bounded Kasparov module \( (A, E'_{B \otimes \Cl_1}, F') \), the action of \( \gamma_1 \in \Cl_1 \) on the right of \( E' \) identifies the even and odd parts of \( E' \). Writing, therefore, \( E' = E \oplus E \) for \( E = E'^{\mathrm{ev}} = E'^{\mathrm{odd}} \), we write \( (A, E'_{B \otimes \Cl_1}, F') \) as
	\[ \bigg( A, (E \oplus E)_{B \otimes \Cl_1}, \begin{pmatrix} & V \\ U & \end{pmatrix} \bigg) , \]
	where the action of \( B \otimes \Cl_1 \) on the right of \( E \oplus E \) is given by \eqref{eq:e56nu7irkbj3v56bdnurjbe5vh65}. The off-diagonal form of the operator is a consequence of the requirement that \( F' \) is odd. Since \( F' \) must also be linear in the action of \( \Cl_1 \), we must have \( U = V \). We thus obtain an odd bounded Kasparov \( A \)-\( B \)-module \( (A, E_B, U) \).
	
	The same discussion applies equally to odd unbounded cycles, though the reversibility for unbounded cycles does not hold directly, but only up to locally bounded perturbations \cite{Dungen_2018}.
\end{remark}

We recall here a few facts about ternary rings of operators. Ring- or algebra-like objects with ternary product operations are known also as triple systems, and come in Lie, Jordan, and associative varieties, the latter in two kinds. In the context of abstract operator algebras there are C*- and W*-ternary rings, due to \cite{Zettl_1983}.

\begin{definition}
	A \emph{ternary ring of operators} on a Hilbert \( B \)-module \( E \) is a collection \( \mathscr{X} ⊆ \End^*(E) \) which is closed under the operation
	\[ (x, y, z) ↦ x y^* z . \]
	We will not by default assume that a ternary ring of operators is norm-closed.
	
	In the sense of \cite[Lemma 2.16]{Raeburn_1998}, \( \mathscr{X} \) is a right pre-Hilbert module over the algebra \( \Span(\mathscr{X}^* \mathscr{X}) \). Its completion \( \bar{\mathscr{X}} \) is then a right Hilbert \( \overline{\Span}(\mathscr{X}^* \mathscr{X}) \)-module. By similar considerations on the left, \( \bar{\mathscr{X}} \) is a Morita equivalence \( \overline{\Span}(\mathscr{X} \mathscr{X}^*) \)-\( \overline{\Span}(\mathscr{X}^* \mathscr{X}) \)-bimodule. We remark that, for instance, \( \overline{\Span}(\mathscr{X} \mathscr{X}^* \mathscr{X}) = \bar{\mathscr{X}} \).
	
	In particular, every norm-closed ternary ring of operators is a Morita equivalence bimodule in a natural way. By \cite[Theorem 2.6]{Zettl_1983}, any Hilbert C*-module can be represented as a norm-closed ternary ring of operators on some Hilbert space \( H \).
\end{definition}

The implicit presence of ternary rings of operators will be a feature of many of our definitions. This occurs because, just as the Leibniz rule makes the domain of a commutator with a self-adjoint operator \( D \) a $*$-algebra, the domain of a mixed commutator \( a \mapsto D_1 a - a D_2 \) is naturally closed under the ternary product. Indeed, if, for \( a, b, c \in \End^*(E) \),
\[ D_1 a - a D_2 \qquad D_1 b - b D_2 \qquad D_1 c - c D_2 \]
are bounded, then \( [D_1, a b^*] \), \( [D_2, a^* b] \), and \( D_1 a b^* c - a b^* c D_2 \) (and all other like permutations) are bounded. This can also be seen by writing \( D_1 \) and \( D_2 \) as diagonal entries of a two-by-two matrix and placing \( a, b, c \) in the upper-right corner.

We will formulate our definition of conformal transformation for higher order cycles.

\begin{definition}
	\label{definition:conformal-transformation}
	A \emph{conformal transformation} \( (U, μ) \) from one order-\( \frac{1}{1 - α} \) cycle, \( (A, E_B, D_1) \), to another, \( (A, E'_B, D_2) \), is a unitary map \( U: E \to E' \), intertwining the representations of \( A \), and an (even) invertible operator \( μ ∈ \End^*(E) \) which is even if the module is graded, satisfying the following. We require that \( A ⊆ \overline{\Span}(A \mathscr{M}) ∩ \overline{\Span}(\mathscr{M} A) \), where \( \mathscr{M} \) is the set of \( a ∈ \End^*(E) \) such that the operators
	\[ (U^* D_2 U a - a μ D_1 μ^*) μ^{-1 *} ⟨D_1⟩^{-α} \qquad ⟨D_2⟩^{-α} U (U^* D_2 U a - a μ D_1 μ^*) \]
	are bounded and \( a, a μ, a μ^{-1 *} ∈ \Lip_α^*(D_1) \).
\end{definition}

\begin{remarks}
	\item The easiest way for the closure condition to be satisfied is if \( 1 ∈ \mathscr{M} \); for nonunital \( A \) an approximate unit might be found to lie in \( \mathscr{M} \).
	\item We have \( \mathscr{M} \mathscr{M}^* \mathscr{M} ⊆ \mathscr{M} \) and so \( \mathscr{M} \) is a ternary ring of operators, in general not norm-closed.
	\item Conformal transformations are generally neither reversible nor composable. This latter occurs very easily for two noncommuting conformal factors \( μ \) and \( ν \). We ultimately address this issue with the \emph{conformisms} of §\ref{section:equivalence-conf-gen}.
\end{remarks}

In the next section, on page \pageref{proof2.2}, we will prove the following Theorem.

\begin{theorem}
	\label{theorem:bdd-transform-conformal-transformation}
	Let \( (U, μ) \) be a conformal transformation from the order-\( \frac{1}{1 - α} \) cycle \( (A, E_B, D_1) \) to the order-\( \frac{1}{1 - α} \) cycle \( (A, E'_B, D_2) \). Then the bounded transforms \( (A, E_B, F_{D_1}) \) and \( (A, E'_B, F_{D_2}) \) are unitarily equivalent up to locally compact perturbation via the unitary \( U \). That is
	\[ (U^* F_{D_2} U - F_{D_1}) a ∈ \End^0(E) \]
	for all \( a ∈ A \). Hence \( [(A, E_B, F_{D_1})]=[(A, E'_B, F_{D_2})]\in KK(A,B) \) (and also \( [(A, E_B, D_1)]=[(A, E'_B, D_2)] \)).
\end{theorem}

\subsection{Motivating examples}
\label{section:conformal-motivating-examples}

\begin{example}
	cf. \cite[Lemma 2.8]{Dungen_2020}
	The simplest nontrivial example of a conformal transformation between unbounded cycles can be contructed from an unbounded cycle \( (A, E_B, D) \) and a positive number \( κ \). The pair \( (\id, κ^{1/2}) \) is a conformal transformation from \( (A, E_B, D) \) to \( (A, E_B, κ D) \).
\end{example}

On a geodesically complete Riemannian manifold \( X \), there are two standard spectral triples. One relies on a spin\(^c\) structure and takes the form \( (C_0(X), L^2(X, S), \slashed{D}) \), where \( S \) is a spinor bundle and \( \slashed{D} \) is the Atiyah–Singer Dirac operator. The other depends on only the orientation and Riemannian metric, taking the form
\( (C_0(X), L^2(Ω^* X), d + δ) \)
where \( d \) is the exterior derivative on differential forms $Ω^* X$ and \( δ \) is its adjoint, the codifferential, their sum being the Hodge–de Rham Dirac operator. We consider the effect of a conformal change of metric on both these spectral triples.

\begin{example}
	\label{example:atiyah-singer-conformal-transformation}
	The behaviour of the Atiyah–Singer Dirac operator under conformal transformations was first recorded in \cite[Proof of Proposition 1.3]{Hitchin_1974}. In the context of noncommutative geometry, see also \cite[Proof of Theorem 3.1]{Bar_2007}. Let \( (X, \mathbf{g}) \) and \( (X, \mathbf{h}) \) be Riemannian spin\(^c\) manifolds such that \( \mathbf{h} = k^2 \mathbf{g} \). Let \( \slashed{S}_{\mathbf{g}} \) and \( \slashed{S}_{\mathbf{h}} \) be their associated spinor bundles. There is a canonical fibrewise isometry
	\[ ψ : \slashed{S}_{\mathbf{g}} \to \slashed{S}_{\mathbf{h}}. \]
	Let \( \slashed{D}_{\mathbf{g}} : Γ^∞(\slashed{S}_{\mathbf{g}}) \to Γ^∞(\slashed{S}_{\mathbf{g}}) \) and \( \slashed{D}_{\mathbf{h}} : Γ^∞(\slashed{S}_{\mathbf{h}}) \to Γ^∞(\slashed{S}_{\mathbf{h}}) \) be the corresponding Dirac operators. Then, by e.g. \cite[Proposition 4.3.1]{Hijazi_1986},
	\[ \slashed{D}_{\mathbf{h}} = k^{(-n-1)/2} \circ ψ \circ \slashed{D}_{\mathbf{g}} \circ ψ^{-1} \circ k^{(n-1)/2}. \]
	Although \( ψ \) is a fibrewise isometry, the induced map \( V : L^2(X, \slashed{S}_{\mathbf{g}}) \to L^2(X, \slashed{S}_{\mathbf{h}}) \) is not unitary, as the volume form changes. With the relation \( \vol_{\mathbf{h}} = k^n \vol_{\mathbf{g}} \), we find that
	\( V^* = k^n V^{-1} \). The polar decomposition is
	\[ U = V (V^* V)^{-1/2} = k^{-n/2} V \]
	and we find that
	\[ \slashed{D}_{\mathbf{h}} = k^{-1/2} U \slashed{D}_{\mathbf{g}} U^* k^{-1/2} \]
	or, in other words, 
	\[ U^* \slashed{D}_{\mathbf{h}} U = k^{-1/2} \slashed{D}_{\mathbf{g}} k^{-1/2}. \]
	In terms of Definition \ref{definition:conformal-transformation}, if \( (X, \mathbf{g}) \) is complete and the conformal factor \( k \) and its inverse are bounded (which is automatic if \( X \) is compact), then \( (U, k^{-1/2}) \) is a conformal transformation from \( (C_0(X), L^2(X, \slashed{S}_{\mathbf{g}}), \slashed{D}_{\mathbf{g}}) \) to \( (C_0(X), L^2(X, \slashed{S}_{\mathbf{h}}), \slashed{D}_{\mathbf{h}}) \). (Note that there is no need for the derivative of the conformal factor to be globally bounded.)
\end{example}

\begin{example}
	\label{example:hodge-de-rham-conformal-transformation}
	Next, we consider the Hodge–de Rham Dirac operator. As before, let \( (X, \mathbf{g}) \) and \( (X, \mathbf{h}) \) be Riemannian manifolds such that \( \mathbf{h} = k^2 \mathbf{g} \). Consider the two inner products on \( Ω^* X \) given by \( \mathbf{g} \) and \( \mathbf{h} \), which we will label \( ⟨ \cdot,\cdot ⟩_{\mathbf{g}} \) and \( ⟨\cdot,\cdot ⟩_{\mathbf{h}} \). We will call the resulting Hilbert spaces \( L^2(Ω^* X, {\mathbf{g}}) \) and \( L^2(Ω^* X, {\mathbf{h}}) \). There is an obvious map
	\[ V : L^2(Ω^* X, {\mathbf{g}}) \to L^2(Ω^* X, {\mathbf{h}}) \]
	given by the identity on \( Ω^* X \), in other words, for \( ω ∈ Ω^* X ⊆ L^2(Ω^* X, {\mathbf{g}}) \),
	$ V : ω ↦ ω $.
	Its adjoint is given on homogenous forms $\omega$ by \( V^* : ω ↦ k^{n-2|ω|} ω \). Observe that if \( n \) is even the restriction of \( V \) to the middle degree forms is unitary. We make the (rather trivial) observation that
	\begin{equation}
	 V V^* : ω ↦ k^{n-2|ω|} ω \qquad V^* V : ω ↦ k^{n-2|ω|} ω. 
	 \label{eq:middle}
	 \end{equation}
	The unitary in the polar decomposition \( U = V (V^* V)^{-1/2} = (V V^*)^{-1/2} V \) is given by
	\[ U : ω ↦ k^{(-n+2|ω|)/2} ω \qquad U^* : ω ↦ k^{(n-2|ω|)/2} ω. \]
	The exterior derivative \( d \) does not depend on the metric, but its adjoint the codifferential does, so we use the notation \( δ_{\mathbf{g}} \) and \( δ_{\mathbf{h}} \) to distinguish the two codifferentials acting on \( Ω^* X \). The invariance of the exterior derivative means that \( d V = V d \). With care over which inner product is being used,
	\( (V d)^* = δ_{\mathbf{g}} V^*\) and \((d V)^* = V^* δ_{\mathbf{h}}. \)
	So, \( δ_{\mathbf{g}} V^* = V^* δ_{\mathbf{h}} \) and we obtain the relations
	\[ V (d + δ_{\mathbf{g}}) V^* = d (V V^*) + (V V^*) δ_{\mathbf{h}} \]
	and
	\[ U (d + δ_{\mathbf{g}}) U^* = (V V^*)^{-1/2} V (d + δ_{\mathbf{g}}) V^* (V V^*)^{-1/2} = (V V^*)^{-1/2} d (V V^*)^{1/2} + (V V^*)^{1/2} δ_{\mathbf{h}} (V V^*)^{-1/2}. \]
	On a differential form \( ω \) of degree $|\omega|$,
	\begin{align*}
		U (d + δ_{\mathbf{g}}) U^* ω
		& = k^{-(n - 2 (|ω| + 1))/2} d (k^{(n - 2 |ω|)/2} ω) + k^{(n - 2 (|ω| - 1))/2} δ_{\mathbf{h}} (k^{-(n - 2 |ω|)/2} ω) \\
		& = k \left( k^{-(n - 2 |ω|)/2} d (k^{(n - 2 |ω|)/2} ω) + k^{(n - 2 |ω|)/2} δ_{\mathbf{h}} (k^{-(n - 2 |ω|)/2} ω) \right).
	\end{align*}
	For any function \( f ∈ C^∞(X) \),
	\begin{align*}
		f^{-1} d f ω + f δ_{\mathbf{h}} f^{-1} ω
		& = (d + δ_{\mathbf{h}}) ω + f^{-1} [d, f] ω + [f, δ_{\mathbf{h}}] f^{-1} ω \\
		& = (d + δ_{\mathbf{h}}) ω + f^{-1} [d, f] ω - [δ_{\mathbf{h}}, f] f^{-1} ω \\
		& = (d + δ_{\mathbf{h}}) ω + f^{-1} [d - δ_{\mathbf{h}}, f] ω.
	\end{align*}
	Hence
	\begin{align*}
		& \left( U (d + δ_{\mathbf{g}}) U^* - k^{1/2} (d + δ_{\mathbf{h}}) k^{1/2} \right) ω \\
		& \qquad = \left( k (d + δ_{\mathbf{h}}) + k^{-(n - 2 |ω| - 2)/2} [d - δ_{\mathbf{h}}, k^{(n - 2 |ω|)/2}] - k^{1/2} (d + δ_{\mathbf{h}}) k^{1/2} \right) ω \\
		& \qquad = \left( - k^{1/2} [d + δ_{\mathbf{h}}, k^{1/2}] + k^{-(n - 2 |ω| - 2)/2} [d - δ_{\mathbf{h}}, k^{(n - 2 |ω|)/2}] \right) ω.
	\end{align*}
	In terms of Definition \ref{definition:conformal-transformation}, if \( (X, \mathbf{g}) \) is complete and the conformal factor \( k \) and its inverse are bounded (which is automatic if \( X \) is compact), the data \( (U, k^{-1/2}) \) define a conformal transformation from \( (C_0(X), L^2(Ω^* X, {\mathbf{g}}), d + δ_{\mathbf{g}}) \) to \( (C_0(X), L^2(Ω^* X, {\mathbf{h}}), d + δ_{\mathbf{h}}) \). 
\end{example}

\begin{remark}
	\label{rmk:cliff}
	The extension of the Hodge-de Rham spectral triple to a spectral triple for the $\bbZ_2$-graded Clifford algebra bundle is important for Poincar\'{e} duality \cite[\S 4]{Kasparov_1988}. In the case of a manifold, where the functions and conformal factors are in the centre of the Clifford algebra, it is not difficult to show that our definition of conformal transformation can be modified to handle the graded commutators. We leave a discussion of the general $\bbZ_2$-graded case to another place.
\end{remark}

\begin{example}
\label{eg:families}
Suppose that we have the data of a continuous family of compact Riemannian spin\(^c\) manifolds $(M_x,\mathbf{g}_x)_{x \in X}$ parameterised by a locally compact Hausdorff space $X$, as in the families index theorem \cite[\S III.15]{Lawson_1989}. Integration over the fibres of the total space $\mathcal{M}\to X$ along with the Dirac operators $D_x$ on the fibre spinor bundles $S_x$ yields an unbounded Kasparov module
\begin{equation}
	\label{eq:families-1}
	\left( C_0(\mathcal{M}),L^2(\mathcal{M},S_\bullet, \mathbf{g}_\bullet)_{C_0(X)}, D_\bullet \right) .
\end{equation}
Let $k:\mathcal{M}\to[0,\infty)$ be a family of conformal factors parameterised by $X$. The commutation of the conformal factors with the algebra means we obtain a new unbounded Kasparov module
\[ \left( C_0(\mathcal{M}),L^2(\mathcal{M},S_\bullet,k^2 \mathbf{g}_\bullet)_{C_0(X)},k^{-1/2}D_\bullet k^{-1/2} \right) . \]
We observe that the integration over the fibres changes, but the compactness of the fibres means we get equivalent measures. That we obtain a new unbounded Kasparov module is straightforward but of more consequence is that the classes defined by $F_D$ and $F_{k^{-1/2}Dk^{-1/2}}$ in $KK(C_0(\mathcal{M}),C_0(X))$ coincide. 

Suppose that we have another family of metrics \( \mathbf{h}_\bullet \), for the same family of manifolds, giving an unbounded Kasparov module
\begin{equation}
	\label{eq:families-2}
	\left( C_0(\mathcal{M}),L^2(\mathcal{M},S_\bullet, \mathbf{h}_\bullet)_{C_0(X)},D_\bullet \right) .
\end{equation}
Suppose that \( \mathbf{h}_x = k_x^2 \mathbf{g} \) for a (pointwise) continuous family \( k_\bullet ∈ C^∞(M_\bullet) \) of smooth functions and that \( \sup_{x ∈ X} \{ \| k_x \|_∞, \| k_x^{-1} \|_∞ \} < ∞ \). Then \( (\id, k_\bullet^{-1/2}) \) is a conformal transformation from \eqref{eq:families-1} to \eqref{eq:families-2}.
\end{example}

The first appearance of conformal transformations in noncommutative geometry was with the preprint \cite{Cohen_1992} on the noncommutative torus, followed up by the same authors in \cite{Connes_2011}; see also \cite{Connes_2014}. This is not to be confused with the twisted spectral triples of \cite{Connes_2008}, which will be examined in \S\ref{sec:cgc}.

\begin{example}
\label{eg:torus-ct}
	Fix a real number \( α \). Let \( C(\bbT^2_α) \) be the universal C*-algebra generated by unitaries \( U \) and \( V \) subject to the relation
	\[ V U = e^{2 π i α} U V . \]
	There are two self-adjoint (unbounded) derivations \( δ_1 \) and \( δ_2 \) on \( C(\bbT^2_α) \), given on generators by
	 \[ δ_1(U) = U \qquad δ_1(V) = 0 \qquad δ_2(U) = 0 \qquad δ_2(V) = V . \]
	 When \( α = 0 \), these are the derivatives \( -i ∂_{θ_1} \) and \( -i ∂_{θ_2} \) on the classical torus. There is a trace on \( C(\bbT^2_α) \) given by
	 \[ φ(U^m V^n) = δ_{m, 0} δ_{n, 0} . \]
	 The completion of \( C(\bbT^2_α) \) in the inner product given by \( φ \) is \( L^2(\bbT^2_α) \). Fix a complex number \( τ \) with \( \Im(τ) > 0 \). Then
	\[ \left( C(\bbT^2_α), L^2(\bbT^2_α) ⊗ \bbC^2, D := \begin{pmatrix} & δ_1 + τ δ_2 \\ δ_1 + \bar{τ} δ_2 & \end{pmatrix} \right) \]
	 is a spectral triple. Now choose a positive invertible element \( k ∈ C(\bbT^2_α) \) in the domains of \( δ_1 \) and \( δ_2 \). Let \( k^{\circ} ∈ B(L^2(\bbT^2_α)) \) be the operator of right multiplication. Then
	 \[ \left( C(\bbT^2_α), L^2(\bbT^2_α) ⊗ \bbC^2, D_{k^2} := \begin{pmatrix} & (k^{\circ})^2 (δ_1 + τ δ_2) \\ (δ_1 + \bar{τ} δ_2) (k^{\circ})^2 & \end{pmatrix} \right) \]
	 is still a spectral triple. We have that
	 \[ D_{k^2} - k^{\circ} D k^{\circ} = \begin{pmatrix} & -k^{\circ} \left[δ_1 + τ δ_2, k^{\circ} \right] \\ \left[ δ_1 + \bar{τ} δ_2, k^{\circ} \right] k^{\circ} & \end{pmatrix} \]
	 is bounded. Hence $1\in\mathscr{M}$ and \( (\id, k^{\circ}) \) is a conformal transformation from the spectral triple \( (C(\bbT^2_α), L^2(\bbT^2_α) ⊗ \bbC^2, D) \) to \( (C(\bbT^2_α), L^2(\bbT^2_α) ⊗ \bbC^2, D_{k^2}) \).
	 
	 Let $\Phi: C(\bbT^2_α) \to C(\bbT)$ be the expectation coming from averaging over the circle action $U\mapsto zU$, $z\in \bbT$. Then $(C(\bbT^2_α), L^2(C(\bbT^2_α),\Phi)_{C(\bbT)}, \delta_2)$ is an unbounded  Kasparov module by \cite[Proposition 2.9]{Bourne_2015}. Now choose a positive invertible element \( k ∈ C(\bbT^2_α) \) in the domain of \( δ_2 \). Then \( (\id, k^{\circ}) \) is a conformal transformation from \( (C(\bbT^2_α),L^2(C(\bbT^2_α),\Phi)_{C(\bbT)},\delta_2)\) to the spectral triple
	 \[ \left( C(\bbT^2_α), L^2(C(\bbT^2_α), \Phi)_{C(\bbT)}, k^{\circ} \delta_2 k^{\circ} \right) . \]
\end{example}

Example \ref{eg:torus-ct} can be generalised along the lines of \cite{Sitarz_2015}, using a real spectral triple satisfying the  order zero condition. Theorem \ref{theorem:bdd-transform-conformal-transformation} gives a refinement of \cite[Lemma 14]{Sitarz_2015} which shows that the class in KK-theory of the conformally perturbed spectral triple is unchanged.

\subsection{Technical preliminaries and additive perturbation theory}
\label{section:tech}

Throughout this section we fix a countably generated right Hilbert $B$-module $E$ for some C*-algebra $B$. The main tool in our proofs is the integral formula
\begin{equation}
(1 + D^2)^{-α} = \frac{\sin(α π)}{π} \int_0^∞ λ^{-α} (λ + 1 + D^2)^{-1} dλ \label{M} 
\end{equation}
whose use in noncommutative geometry is due to Baaj and Julg \cite{Baaj_1983}; for more details we refer to \cite[Lemma A.4]{Carey_1998}.
We quote the following refinement of Baaj and Julg's bounded transform result which follows easily from the results of \cite[\S2.1]{Wahl_2007}, \cite[\S7]{Grensing_2012}, \cite[Appendix A]{Goffeng_2015}.

\begin{theorem}
	\label{theorem:relatively_bounded_commutator_to_compact_commutator}
	Let \( D \) be a self-adjoint regular operator on a right Hilbert \( B \)-module \( E \). Let \( S \) be an adjointable operator such that \( S \dom D \subseteq \dom D \) and \( [D, S] \langle D\rangle^{-\alpha} \) extends to a bounded operator for some \( 0 ≤ α < 1 \). Then
	\[ [F_D, S] \langle D\rangle^\beta\]
	is bounded for \( β < 1-α \).
\end{theorem}

Theorem \ref{theorem:relatively_bounded_commutator_to_compact_commutator} allows us to study
additive perturbations in a more-or-less optimal way, and the following two results can be compared to \cite[Lemmas B.6--7]{Carey_1998}.

\begin{proposition}
	\label{proposition:bounded-transform-of-commutator-additive-perturbation}
	Let \( D_0 \) and \( D_1 \) be self-adjoint regular operators on right Hilbert \( B \)-modules \( E_0 \) and \( E_1 \). Suppose that there is an operator \( a ∈ \Hom_B^*(E_0, E_1) \) such that \( a \dom D_0 ⊆ \dom D_1 \) and
	\[ (D_1 a - a D_0) ⟨D_0⟩^{-α} \]
	extends to an adjointable operator for some \( 0 ≤ α < 1 \). Then, fixing \( β < 1-α \),
	\[ (F_{D_1} a - a F_{D_0}) ⟨D_0⟩^β \]
	is bounded.
\end{proposition}
\begin{proof}
	Consider the operators
	\[ D = \begin{pmatrix} D_0 & \\ & D_1 \end{pmatrix} \qquad S = \begin{pmatrix} & 0 \\ a & \end{pmatrix} \]
	on \( E_0 \oplus E_1 \). Then
	\[ S \dom D = \begin{pmatrix} 0 \\ a \dom D_0 \end{pmatrix} ⊆ \begin{pmatrix} \dom D_0 \\ \dom D_1 \end{pmatrix} = \dom D \]
	and
	\[ [D, S] ⟨D⟩^{-α} = \begin{pmatrix} & 0 \\ (D_1 a - a D_0) ⟨D_0⟩^{-α} & \end{pmatrix} . \]
	By Theorem \ref{theorem:relatively_bounded_commutator_to_compact_commutator},
	\[ [F_D, S] ⟨D⟩^β = \begin{pmatrix} & 0 \\ (F_{D_1} a - a F_{D_0}) ⟨D_0⟩^β & \end{pmatrix} \]
	is bounded for \( β < 1-α \), as required.
\end{proof}

\begin{corollary}
	\label{proposition:bounded-transform-of-additive-perturbation-local}
	Let \( D_0 \) and \( D_1 \) be self-adjoint regular operators on a right Hilbert \( B \)-module \( E \) with densely intersecting domains. Suppose that there is a bounded operator \( a \) such that \( a \dom D_0 ⊆ \dom D_0 ∩ \dom D_1 \) and
	\[ (D_1 - D_0) a ⟨D_0⟩^{-α} \qquad [D_0, a] ⟨D_0⟩^{-α} \]
	extend to bounded operators for some \( 0 ≤ α < 1 \). Then, fixing \( β < 1-α \),
	\[ (F_{D_1} - F_{D_0}) a ⟨D_0⟩^β \]
	is bounded.
\end{corollary}
\begin{proof}
	We have
	\[ (D_1 a - a D_0) ⟨D_0⟩^{-α} = (D_1 - D_0) a ⟨D_0⟩^{-α} + [D_0, a] ⟨D_0⟩^{-α} \]
	and
	\[ (F_{D_1} a - a F_{D_0}) ⟨D_0⟩^β = (F_{D_1} - F_{D_0}) a ⟨D_0⟩^β + [F_{D_0}, a] ⟨D_0⟩^β . \]
	By Theorem \ref{theorem:relatively_bounded_commutator_to_compact_commutator}, \( [F_{D_0}, a] ⟨D_0⟩^β \) is bounded, so
	$ (F_{D_1} - F_{D_0}) a ⟨D_0⟩^β $
	is also, as required.
\end{proof}

The chief subtlety in using the integral formula \eqref{M} to study the bounded transform for an unbounded Kasparov module $(A,E_B,D)$ is the commutator $(λ + 1 + D^2)^{-1}a-a(λ + 1 + D^2)^{-1}$ for $a\in A$, \cite[Lemma 2.3]{Carey_1998}. For us, the analogous computation is still the heart of the matter, see Lemma \ref{lem:beating-heart}, but our techniques are different and described next.

\begin{lemma}
\label{lemma:sym-perturb}
	Let \( A \) and \( B \) be regular operators on \( E \). If \( B \) is a symmetric operator, then so is \( A B A^* \), provided that the domain
	\[ \dom(A B A^*) = \{ x ∈ \dom A^* \mid A^* x ∈ \dom B,\ B A^* x ∈ \dom A \} \]
	is dense. If \( A \) is bounded and invertible then \( A B A^* \) is regular. If moreover $B$ is self-adjoint then $ABA^*$ is self-adjoint.
\end{lemma}
\begin{proof}
	Given \( x, y ∈ \dom(A B A^*) \), $x,y\in \dom(A^*)$ and $A^*y\in\dom(B)$, the symmetry of $B$ gives
	\[ ⟨A B A^* x | y⟩ = ⟨B A^* x | A^* y⟩ = ⟨A^* x | B A^* y⟩ = ⟨x | A B A^* y⟩ \]
	so \( A B A^* \) is symmetric. If $A$ is bounded and invertible, \cite[\S2, Example 2]{Woronowicz_1991} shows that \( A B \) is regular and, by \cite[\S2, Example 3]{Woronowicz_1991}, \( A B A^* \) is regular. Applying the definition of the domain of the adjoint, one readily sees that \( \dom((ABA^*)^*) = \dom(ABA^*) = A^{-1 *} \dom(B) \).
\end{proof}
In the second statement of Lemma \ref{lemma:sym-perturb}, the invertibility of \( A \) can be relaxed given additional assumptions \cite[\S 6]{Kaad_2017}. 
We will consider perturbations of the form \( D \leadsto μ D μ^* \) for a self-adjoint regular operator \( D \) and an invertible, adjointable operator \( μ \). 
The following bound is the result of a relation between the domains of fractional powers of $\langle D\rangle$ and $\langle \mu D\mu^*\rangle$, using Theorem \ref{theorem:fractional-power-bounds} of Appendix \ref{section:fractional-powers}. 

\begin{lemma}
	\label{lemma:fractional-power-conformal-perturbation}
	Let \( D \) be a self-adjoint regular operator and \( μ \) an invertible adjointable operator. For all \( 0 < α \leq 1 \) we have
	\[
		\dom(\mu\langle D\rangle^\alpha\mu^*)
		=\dom((\mu\langle D\rangle\mu^*)^\alpha)
		=\dom\langle \mu D\mu^*\rangle^\alpha
	\]
	and the inequalities
	\[ \left\| ⟨D⟩^α μ^* (μ ⟨D⟩ μ^*)^{-α} \right\| \leq \| μ^{-1} \|^α \| μ \|^{1-α} \qquad \left\| (μ ⟨D⟩ μ^*)^α μ^{-1 *} ⟨D⟩^{-α} \right\| \leq \| μ \|^α \| μ^{-1} \|^{1-α} . \]
\end{lemma}
\begin{proof}
	The domain statement follows from Theorem \ref{theorem:fractional-power-bounds}.
	For the first inequality, in the context of Theorem \ref{theorem:fractional-power-bounds}, let \( A = ⟨D⟩ \) and \( B = μ ⟨D⟩ μ^* \) so that \( μ^* \dom B = \dom A \). We have
	\begin{align*}
		\left\| ⟨D⟩^α μ^* (μ ⟨D⟩ μ^*)^{-α} \right\|
		& = \| A^α μ^* B^{-α} \| 
		 \leq \| A μ^* B^{-1} \|^α \| μ^* \|^{1-α} \\
		& = \left\| ⟨D⟩ μ^* (μ ⟨D⟩ μ^*)^{-1} \right\|^α \| μ^* \|^{1-α} \\
		& = \| μ^{-1} \|^α \| μ \|^{1-α} .
	\end{align*}
	For the second, in the context of Theorem \ref{theorem:fractional-power-bounds}, let \( A = μ ⟨D⟩ μ^* \) and \( B = ⟨D⟩ \), so that \( μ^{-1 *} \dom B = \dom A \). We obtain that
	\begin{align*}
		\left\| (μ ⟨D⟩ μ^*)^α μ^{-1 *} ⟨D⟩^{-α} \right\|
		& = \| A^α μ^{-1 *} B^{-α} \| 
		 \leq \| A μ^{-1 *} B^{-1} \|^α \| μ^{-1 *} \|^{1-α} \\
		& = \left\| (μ ⟨D⟩ μ^*) μ^{-1 *} ⟨D⟩^{-1} \right\|^α \| μ^{-1} \|^{1-α} \\
		& = \| μ \|^α \| μ^{-1} \|^{1-α}
	\end{align*}
	as required.
\end{proof}

We recall tools ensuring convergence of regular self-adjoint operators on a Hilbert module $E_B$.

\begin{theorem}
	\label{theorem:strong-convergence-from-compact-convergence}
	\cite[§1]{Woronowicz_1992}
	Let \( T \) be a normal regular operator on \( E \) and \( f ∈ C_b(σ(T)) \). Let \( (f_n)_{n ∈ \bbN} ⊆ C_b(σ(T)) \) be a sequence of functions with common bound which converge to \( f \) uniformly on compact subsets. Then \( f_n(T) \) converges to \( f(T) \) as \( n \to ∞ \) in the strict topology on \( M(\End^0(E)) \), and hence in the $*$-strong topology on \( \End^*(E) \).
\end{theorem}

For the final statement, recall that the strict topology on \( M(\End^0(E)) = \End^*(E) \) agrees with the $*$-strong topology on norm-bounded subsets \cite[Proposition C.7]{Raeburn_1998}.

The proofs of the following two Theorems are essentially unchanged from the Hilbert space case.

\begin{theorem}
	\label{theorem:strong-resolvent-convergence-from-pointwise}
	cf.~\cite[Theorem VIII.25(a)]{Reed_1980}, \cite[Proposition 10.1.18]{de_Oliveira_2009}
	Let \( \mathscr{C} ⊆ E \) be a core for a self-adjoint regular operator \( T \) on \( E \). Let \( (T_n)_{n∈\bbN} \) be a sequence of self-adjoint regular operators such that, for all \( n ∈ \bbN \), \( \mathscr{C} ⊆ \dom T_n \) and, for all \( ξ ∈ \mathscr{C} \), \( T_n ξ \) converges to \( T ξ \) as \( n \to ∞ \). Then \( T_n \) converges to \( T \) in the strong resolvent sense as \( n \to ∞ \).
\end{theorem}

\begin{theorem}
	\label{theorem:strong-resolvent-convergence-functional-calculus}
	cf.~\cite[Theorem VIII.20(b)]{Reed_1980}, \cite[Proposition 10.1.9]{de_Oliveira_2009}
	A sequence \( (T_n)_{n∈\bbN} \) of self-adjoint regular operators on \( E \) converges to a self-adjoint regular operator \( T \) in the strong resolvent sense if and only if, for all \( f ∈ C_b(\bbR) \), \( f(T_n) \) converges strongly to \( f(T) \) as \( n \to ∞ \).
\end{theorem}

Let \( (φ_n)_{n ∈ \bbN} ⊂ C_c(\bbR) \) be a sequence of positive functions, bounded by 1 and converging uniformly on compact subsets to the constant function 1. Let \( D \) be a self-adjoint regular operator. By Theorem \ref{theorem:strong-convergence-from-compact-convergence}, the bounded operators \( (φ_n(D))_{n ∈ \bbN} \) converge $*$-strongly to 1. We will consider the bounded operators \( d_n = D φ_n(D) \). On an element \( ξ ∈ \dom D \),
\[ d_n ξ = D φ_n(D) ξ = φ_n(D) (D ξ) \to D ξ . \]
In particular, by Theorem \ref{theorem:strong-resolvent-convergence-from-pointwise}, \( d_n \to D \) in the strong resolvent sense. By Theorem \ref{theorem:strong-resolvent-convergence-functional-calculus}, \( F_{d_n} \) converges strongly to \( F_D \) as \( n \to ∞ \).

\begin{proposition}
	\label{proposition:bee-7}
	Let \( D \) be a self-adjoint regular operator and \( μ \) an invertible adjointable operator. Then \( μ d_n μ^* \) converges to \( μ D μ^* \) in the strong resolvent sense as \( n \to ∞ \). Furthermore, \( μ ⟨d_n⟩ μ^* \) converges to \( μ ⟨D⟩ μ^* \) in the strong resolvent sense.
	
	Let \( a \) be a bounded operator such that \( a \dom D ⊆ \dom D \). With \( a_n = φ_n(D) a φ_n(D) \), we find that \( d_n a_n ⟨d_n⟩^{-1} \) converges strongly to \( D a ⟨D⟩^{-1} \) as \( n \to ∞ \). In consequence, \( [d_n, a_n] ⟨d_n⟩^{-1} \) converges strongly to \( [D, a] ⟨D⟩^{-1} \).
\end{proposition}
\begin{proof}
	First, apply Theorem \ref{theorem:strong-resolvent-convergence-from-pointwise} to the self-adjoint regular operator \( μ D μ^* \) and the sequence \( (μ d_n μ^*)_{n ∈ \bbN} \) of bounded operators. Noting that \( \dom(μ D μ^*) = μ^{-1 *} \dom D \), on an element \( μ^{-1 *} ξ ∈ \dom(μ D μ^*) \),
	\[ (μ d_n μ^*) μ^{-1 *} ξ = μ d_n ξ \to μ D ξ \]
	as \( n \to ∞ \). Hence, \( μ d_n μ^* \) converges to \( μ D μ^* \) in the strong resolvent sense.
	On an element \( ξ ∈ \dom D \),
	\[ ⟨d_n⟩ ξ = (1 + (D φ_n(D))^2)^{1/2} ξ = (1 + (D φ_n(D))^2)^{1/2} ⟨D⟩^{-1} (⟨D⟩ ξ) . \]
	The function
	\[ x ↦ \frac{(1 + (x φ_n(x))^2)^{1/2}}{(1 + x^2)^{1/2}} = \left( 1 - \frac{1 - φ_n(x)^2}{1 + x^{-2}} \right)^{1/2} \]
	is bounded above by 1 and below by \( φ_n \) and so converges to 1 on compact subsets. Applying Theorem \ref{theorem:strong-convergence-from-compact-convergence},
	\[ ⟨d_n⟩ ξ = (1 + (D φ_n(D))^2)^{1/2} ⟨D⟩^{-1} (⟨D⟩ ξ) \to ⟨D⟩ ξ \]
	and we proceed as before.	
	For the second statement we have
	\[ 
	d_n a_n ⟨d_n⟩^{-1} = D φ_n(D)^2 a φ_n(D) ⟨D φ_n(D)⟩^{-1} = φ_n(D)^2 \left( D a ⟨D⟩^{-1} \right) ⟨D⟩ φ_n(D) ⟨D φ_n(D)⟩^{-1}. 
	\]
	The function
	\[ 
	x ↦ \frac{(1 + x^2)^{1/2} φ_n(x)}{(1 + (x φ_n(x))^2)^{1/2}} = \left( 1 - \frac{1 - φ_n(x)^2}{1 + x^2 φ_n(x)^2} \right)^{1/2} 
	\]
	 is bounded above by 1 and below by \( φ_n \) and so converges to 1 on compact subsets. Applying Theorem \ref{theorem:strong-convergence-from-compact-convergence},
	\[ 
	d_n a_n ⟨d_n⟩^{-1} = φ_n(D)^2 \left( D a ⟨D⟩^{-1} \right) ⟨D⟩ φ_n(D) ⟨D φ_n(D)⟩^{-1} \to D a ⟨D⟩^{-1} 
	\]
	strongly, as \( n \to ∞ \). For the second part,
	\[ [d_n, a_n] ⟨d_n⟩^{-1} = d_n a_n ⟨d_n⟩^{-1} - a_n F_{d_n} \to D a ⟨D⟩^{-1} - a F_D \]
	strongly, as required.
\end{proof}

As an application, we prove an operator inequality needed for applications involving summability.

\begin{proposition}
	Let \( D \) be a self-adjoint regular operator on a Hilbert \( B \)-module \( E \) and \( μ \) an invertible adjointable operator on \( E \). Then
	\[ C^{-1} μ^{-1 *} (1 + D^2)^{-1} μ^{-1} ≤ (1 + (μ D μ^*)^2)^{-1} ≤ C μ^{-1 *} (1 + D^2)^{-1} μ^{-1} \]
	where \( C = \max\{ \| μ \|^2, \| μ^{-1} \|^2 \} \). 
		
	Hence if \( J \) is a hereditary ideal of \( \End^*(B) \), not necessarily closed, then \( (1 + (μ D μ^*)^2)^{-1} ∈ J \) if and only if \( (1 + D^2)^{-1} ∈ J \). In particular, this applies if \( B = \bbC \), so that \( E \) is a Hilbert space and \( J \) is any two-sided ideal of \( B(E) \), not necessarily closed \cite[§II.5.2]{Blackadar_2006}, such as Schatten ideals.
\end{proposition}
\begin{proof}
	If \( μ^* μ \dom D ⊆ \dom D \), we could proceed more straightforwardly. As we do not assume this, we will use the (bounded) operators \( d_n = D φ_n(D) \) and 
	 Proposition \ref{proposition:bee-7} to write
	\begin{align*}
		1 + (μ d_n μ^*)^2
		& = 1 + μ d_n μ^* μ d_n μ^* 
		 ≤ 1 + \| μ \|^2 μ d_n^2 μ^* 
		 = μ (μ^{-1} μ^{-1 *} + \| μ \|^2 d_n^2) μ^* \\
		 &≤ μ (\| μ^{-1} \|^2 + \| μ \|^2 d_n^2) μ^* 
		 = \| μ \|^2 μ (\| μ \|^{-2} \| μ^{-1} \|^2 + d_n^2) μ^* \\
		 &≤ \| μ \|^2 \max\{ 1, \| μ \|^{-2} \| μ^{-1} \|^2 \} μ (1 + d_n^2) μ^* 
		 = \max\{ \| μ \|^2, \| μ^{-1} \|^2 \} μ (1 + d_n^2) μ^*.
	\end{align*}
	Hence, \( (1 + (μ d_n μ^*))^{-1} ≥ C^{-1} μ^{-1 *} (1 + d_n^2)^{-1} μ^{-1} \), and by Theorem \ref{theorem:strong-resolvent-convergence-functional-calculus} and Proposition \ref{proposition:bee-7}, \( (1 + (μ d_n μ^*)^2)^{-1} \) converges strongly to \( (1 + (μ D μ^*)^2)^{-1} \) and \( (1 + d_n^2)^{-1} \) converges strongly to \( (1 + D^2)^{-1} \) as \( n \to ∞ \). Thus,
	\[ C^{-1} μ^{-1 *} (1 + D^2)^{-1} μ^{-1} ≤ (1 + (μ D μ^*)^2)^{-1}, \]
	and similarly,
	\[ 1 + (μ d_n μ^*)^2 ≥ \min\{ \| μ \|^{-2}, \| μ^{-1} \|^{-2} \} μ (1 + d_n^2) μ^* \]
	and
	\[ (1 + (μ D μ^*)^2)^{-1} ≤ C μ^{-1 *} (1 + D^2)^{-1} μ^{-1} \]
	as required.
\end{proof}

We use the notation $\mathfrak{T}_{a, b}(x)= a x - x b$ for $a,b,x\in \End^*(E)$. 
The following inequality controlling $\mathfrak{T}_{a, b}(x)$ is based on Stampfli \cite[Theorem 8]{Stampfli_1970}; see also Archbold \cite{Archbold_1978}.

\begin{lemma}
\label{lemma:Tee}
	Let \( a \) and \( b \) be elements of a C*-algebra \( A \). Define the bounded linear operator
	\[ \mathfrak{T}_{a, b} : A \to A \qquad x ↦ a x - x b . \]
	If \( a \) and \( b \) are positive, then
	\[ \| \mathfrak{T}_{a, b} \| ≤ \max\{ \| a \| - \| b^{-1} \|^{-1}, \| b \| - \| a^{-1} \|^{-1} \} \]
	where \( \| a^{-1} \|^{-1} \) is considered to be zero if \( a \) is not invertible, and likewise for \( b \).
\end{lemma}
\begin{proof}
	Firstly, \( \| \mathfrak{T}_{a, b} \| ≤ \| a \| + \| b \| \). For any \( λ ∈ \bbC \), \( \mathfrak{T}_{a-λ, b-λ} = \mathfrak{T}_{a, b} \), so
	\( \| \mathfrak{T}_{a, b} \| ≤ \| a - λ \| + \| b - λ \| \).
	For any \( λ_1, λ_2 ∈ \bbC \),
	\[ \| \mathfrak{T}_{a, b} \| ≤ \| a - λ_1 \| + \| b - λ_2 \| + | λ_1 - λ_2 |. \]
	To obtain the required bound, let
	\[ λ_1 = \frac{1}{2} (\| a \| + \| a^{-1} \|^{-1}) \qquad λ_2 = \frac{1}{2} (\| b \| + \| b^{-1} \|^{-1}) \]
	so that, because \( a \) and \( b \) are positive,
	\[ \| a - λ_1 \| = \frac{1}{2} (\| a \| - \| a^{-1} \|^{-1}) \qquad \| b - λ_2 \| = \frac{1}{2} (\| b \| - \| b^{-1} \|^{-1}) . \]
	Then
	\begin{align*}
		\| \mathfrak{T}_{a, b} \|
		& ≤ \frac{1}{2} (\| a \| - \| a^{-1} \|^{-1}) + \frac{1}{2} (\| b \| - \| b^{-1} \|^{-1}) + \Big| \frac{1}{2} (\| a \| + \| a^{-1} \|^{-1}) - \frac{1}{2} (\| b \| + \| b^{-1} \|^{-1}) \Big| \\
		& = \frac{1}{2} \left( (\| a \| - \| b^{-1} \|^{-1}) + (\| b \| - \| a^{-1} \|^{-1}) + \left| (\| a \| - \| b^{-1} \|^{-1}) - (\| b \| - \| a^{-1} \|^{-1}) \right| \right) \\
		& = \max\{ \| a \| - \| b^{-1} \|^{-1}, \| b \| - \| a^{-1} \|^{-1} \}
	\end{align*}
	as required.
\end{proof}

It is proved in \cite[Theorem 8, Corollary 2]{Stampfli_1970} that, if \( A \)  has a faithful irreducible representation, then there is an equality
\[ \| \mathfrak{T}_{a, b} \| = \inf_{λ ∈ \bbC} \left( \| a - λ \| + \| b - λ \| \right) \]
for any \( a, b ∈ A \).

\subsection{A multiplicative perturbation theory}
\label{section:multiplicative-perturbations}

The technical tool which allows us to extend the definitions of conformality and equivariance to unbounded Kasparov cycles is a multiplicative perturbation theory. This perturbation theory allows us to relate properties of an unbounded self-adjoint regular operator $D$ and its bounded transform $F_D:=D(1+D^2)^{-1/2}=D\langle D\rangle^{-1}$ to conformally rescaled versions $D_1=\mu D\mu^*$ and $F_{D_1}$.

\begin{lemma}
\label{lem:beating-heart}
	Let \( D \) be a self-adjoint regular operator and \( μ \) an invertible adjointable operator on $E$. Let \( a \) be an adjointable operator such that \( a μ^{-1 *} \dom D ⊆ μ^{-1 *} \dom D \). Then, with \( D_1 = μ D μ^* \) and \( D_2 = μ ⟨D⟩ μ^* \), and for all \( λ ≥ 0 \)
	\begin{align*}
		 - (λ + ⟨D_1⟩^2)^{-1} a &+ a (λ + D_2^2)^{-1} 
		 = (λ + ⟨D_1⟩^2)^{-1} a μ \mathfrak{T}_{(μ^* μ)^{-1}, μ^* μ}(⟨D⟩^{-1}) μ^{-1} D_2 (λ + D_2^2)^{-1} \\
		& + D_1 (λ + ⟨D_1⟩^2)^{-1} \left( [D_1, a] D_2^{-1} - μ^{-1 *} [F_D, μ^* a μ] μ^{-1} \right) D_2 (λ + D_2^2)^{-1} \\
		& + (λ + ⟨D_1⟩^2)^{-1} \left( μ F_D μ^{-1} [D_1, a] D_2^{-1} + μ [F_D, μ^{-1} a μ] F_D μ^{-1} \right) D_2^2 (λ + D_2^2)^{-1}
	\end{align*}
	as everywhere-defined operators.
\end{lemma}
\begin{proof}
	If \( μ^* μ \dom D ⊆ \dom D \), we could proceed more straightforwardly. As we do not make this assumption, we use the approximation arguments of \S\ref{section:tech}. Let \( (φ_n)_{n ∈ \bbN} ⊂ C_c(\bbR) \) be a sequence of positive functions, bounded by 1 and converging uniformly on compact subsets to the constant function 1. Let \( d_n = D φ_n(D) \) and set
	\[ a_n = μ^{-1 *} φ_n(D) μ^* a μ^{-1 *} φ_n(D) μ^* . \]
	Note for future reference that we may use the bounded transform $F_{d_n}=d_n\langle d_n\rangle^{-1}$ to write
	\begin{align*} 
	[μ d_n μ^*,& a_n] (μ ⟨d_n⟩ μ^*)^{-1} \\
	&=μ d_n φ_n(D) μ^* a μ^{-1 *} φ_n(D) ⟨d_n⟩^{-1} μ^{-1} - μ^{-1 *} φ_n(D) μ^* a μ^{-1 *} φ_n(D) μ^* μ F_{d_n} μ^{-1}\\
	&=\mu[d_n,\varphi_n(D)\mu^*a\mu^{-1*}\varphi_n(D)]\langle d_n\rangle^{-1}\mu^{-1}+\mu\varphi_n(D)\mu^*a\mu^{-1*}\varphi_n(D)F_{d_n}\mu^{-1}\\
	&-\mu^{-1*}\varphi_n(D)\mu^*a\mu^{-1*}\varphi_n(D)\mu^*\mu F_{d_n}\mu^{-1}
	\end{align*}
	so that we will be in a position to apply Proposition \ref{proposition:bee-7} to the first term, while the other two are uniformly bounded in $n$. Because \( d_n \) is bounded, we may write
	\begin{align}
		\label{align:mu-resolvent}
		& -(λ + ⟨μ d_n μ^*⟩^2)^{-1} a_n + a_n (λ + (μ ⟨d_n⟩ μ^*)^2)^{-1}\nonumber \\
		& \qquad = (λ + ⟨μ d_n μ^*⟩^2)^{-1} \left( - a_n (λ + (μ ⟨d_n⟩ μ^*)^2) + (λ + ⟨μ d_n μ^*⟩^2) a_n \right) (λ + (μ ⟨d_n⟩ μ^*)^2)^{-1} \nonumber\\
		& \qquad = (λ + ⟨μ d_n μ^*⟩^2)^{-1} \left( - a_n (μ ⟨d_n⟩ μ^*)^2 + ⟨μ d_n μ^*⟩^2 a_n \right) (λ + (μ ⟨d_n⟩ μ^*)^2)^{-1} .
	\end{align}
	Expanding the middle factor and using the identity $F_{d_n}d_n-\langle d_n\rangle=-\langle d_n\rangle^{-1}$ yields
	\begin{align*}
		& ⟨μ d_n μ^*⟩^2 a_n - a_n (μ ⟨d_n⟩ μ^*)^2\\ 
		 & \qquad = a_n + μ d_n μ^* μ d_n μ^* a_n - a_n μ ⟨d_n⟩ μ^* μ ⟨d_n⟩ μ^* \\
		& \qquad = a_n + μ d_n μ^* [μ d_n μ^*, a_n] + μ d_n μ^* a_n μ d_n μ^* - a_n μ ⟨d_n⟩ μ^* μ ⟨d_n⟩ μ^* \\
		& \qquad = a_n + μ d_n μ^* [μ d_n μ^*, a_n] - μ d_n [F_{d_n}, μ^* a_n μ] ⟨d_n⟩ μ^* + μ d_n F_{d_n} μ^* a_n μ ⟨d_n⟩ μ^* \\
		& \qquad\qquad - a_n μ ⟨d_n⟩ μ^* μ ⟨d_n⟩ μ^* \\
		& \qquad = a_n + μ d_n μ^* [μ d_n μ^*, a_n] - μ d_n [F_{d_n}, μ^* a_n μ] ⟨d_n⟩ μ^* + μ F_{d_n} μ^{-1} μ d_n μ^* a_n μ ⟨d_n⟩ μ^* \\
		& \qquad\qquad - a_n μ ⟨d_n⟩ μ^* μ ⟨d_n⟩ μ^* \\
		& \qquad = a_n + μ d_n μ^* [μ d_n μ^*, a_n] - μ d_n [F_{d_n}, μ^* a_n μ] ⟨d_n⟩ μ^* + μ F_{d_n} μ^{-1} [μ d_n μ^*, a_n] μ ⟨d_n⟩ μ^* \\
		& \qquad\qquad + μ F_{d_n} μ^{-1} a_n μ d_n μ^* μ ⟨d_n⟩ μ^* - a_n μ ⟨d_n⟩ μ^* μ ⟨d_n⟩ μ^* \\
		& \qquad = a_n + μ d_n μ^* [μ d_n μ^*, a_n] - μ d_n [F_{d_n}, μ^* a_n μ] ⟨d_n⟩ μ^* + μ F_{d_n} μ^{-1} [μ d_n μ^*, a_n] μ ⟨d_n⟩ μ^* \\
		& \qquad\qquad + μ [F_{d_n}, μ^{-1} a_n μ] d_n μ^* μ ⟨d_n⟩ μ^* + a_n μ (F_{d_n} d_n - ⟨d_n⟩) μ^* μ ⟨d_n⟩ μ^* \\
		& \qquad = a_n + μ d_n μ^* [μ d_n μ^*, a_n] - μ d_n [F_{d_n}, μ^* a_n μ] ⟨d_n⟩ μ^* + μ F_{d_n} μ^{-1} [μ d_n μ^*, a_n] μ ⟨d_n⟩ μ^* \\
		& \qquad\qquad + μ [F_{d_n}, μ^{-1} a_n μ] d_n μ^* μ ⟨d_n⟩ μ^* - a_n μ ⟨d_n⟩^{-1} μ^* μ ⟨d_n⟩ μ^* \\
		& \qquad = a_n μ \mathfrak{T}_{(μ^* μ)^{-1}, μ^* μ}(⟨d_n⟩^{-1}) μ^{-1} (μ ⟨d_n⟩ μ^*) \\
		& \qquad\qquad + (μ d_n μ^*) \left( [μ d_n μ^*, a_n] (μ ⟨d_n⟩ μ^*)^{-1} - μ^{-1 *} [F_{d_n}, μ^* a_n μ] μ^{-1} \right) (μ ⟨d_n⟩ μ^*) \\
		& \qquad\qquad + \left( μ F_{d_n} μ^{-1} [μ d_n μ^*, a_n] (μ ⟨d_n⟩ μ^*)^{-1} + μ [F_{d_n}, μ^{-1} a_n μ] F_{d_n} μ^{-1} \right) (μ ⟨d_n⟩ μ^*)^2 
	\end{align*}
	since $\mathfrak{T}_{\mu^{-1}\mu^{-1*},\mu^*\mu}(\langle d_n\rangle^{-1})=\mu^{-1}\mu^{-1*}\langle d_n\rangle^{-1}-\langle d_n\rangle^{-1}\mu^*\mu$.
	Substituting into \eqref{align:mu-resolvent} yields
		\begin{align}
		& -(λ + ⟨μ d_n μ^*⟩^2)^{-1} a_n + a_n (λ + (μ ⟨d_n⟩ μ^*)^2)^{-1} \nonumber\\
		& \qquad = (λ + ⟨μ d_n μ^*⟩^2)^{-1} a_n μ \mathfrak{T}_{(μ^* μ)^{-1}, μ^* μ}(⟨d_n⟩^{-1}) μ^{-1} (μ ⟨d_n⟩ μ^*) (λ + (μ ⟨d_n⟩ μ^*)^2)^{-1}\nonumber \\
		& \qquad\qquad + (μ d_n μ^*) (λ + ⟨μ d_n μ^*⟩^2)^{-1} \left( [μ d_n μ^*, a_n] (μ ⟨d_n⟩ μ^*)^{-1} - μ^{-1 *} [F_{d_n}, μ^* a_n μ] μ^{-1} \right) \nonumber\\
		& \qquad\qquad\qquad × (μ ⟨d_n⟩ μ^*) (λ + (μ ⟨d_n⟩ μ^*)^2)^{-1} \nonumber\\
		& \qquad\qquad + (λ + ⟨μ d_n μ^*⟩^2)^{-1} \left( μ F_{d_n} μ^{-1} [μ d_n μ^*, a_n] (μ ⟨d_n⟩ μ^*)^{-1} + μ [F_{d_n}, μ^{-1} a_n μ] F_{d_n} μ^{-1} \right) \nonumber\\
		& \qquad\qquad\qquad × (μ ⟨d_n⟩ μ^*)^2 (λ + (μ ⟨d_n⟩ μ^*)^2)^{-1}.
		\label{eq:X-formerly-twitter}
	\end{align}
	By Proposition \ref{proposition:bee-7}, the right-hand side of \eqref{eq:X-formerly-twitter} converges strongly to
	\begin{align*}
		& (λ + ⟨μ D μ^*⟩^2)^{-1} a μ \mathfrak{T}_{(μ^* μ)^{-1}, μ^* μ}(⟨D⟩^{-1}) μ^{-1} (μ ⟨D⟩ μ^*) (λ + (μ ⟨D⟩ μ^*)^2)^{-1} \\
		& \qquad + (μ D μ^*) (λ + ⟨μ D μ^*⟩^2)^{-1} \left( [μ D μ^*, a] (μ ⟨D⟩ μ^*)^{-1} - μ^{-1 *} [F_D, μ^* a μ] μ^{-1} \right) \\
		& \qquad\qquad\qquad × (μ ⟨D⟩ μ^*) (λ + (μ ⟨D⟩ μ^*)^2)^{-1} \\
		& \qquad + (λ + ⟨μ D μ^*⟩^2)^{-1} \left( μ F_D μ^{-1} [μ D μ^*, a] (μ ⟨D⟩ μ^*)^{-1} + μ [F_D, μ^{-1} a μ] F_D μ^{-1} \right) \\
		&\qquad\qquad\qquad ×(μ ⟨D⟩ μ^*)^2 (λ + (μ ⟨D⟩ μ^*)^2)^{-1}
	\end{align*}
	and we obtain the required equality of everywhere-defined operators.
\end{proof}

\begin{lemma}
	\label{proposition:conformal-expression-bound}
	Let \( D \) be a self-adjoint regular operator and \( μ \) an invertible,adjointable operator on $E$. Let \( a \) be an adjointable operator such that \( a μ^{-1 *} \dom D ⊆ μ^{-1 *} \dom D \). Suppose further that, for some \( 0 ≤ α < 1 \),
	\[ [F_D, μ^* a μ] ⟨D⟩^{1-α} \qquad [F_D, μ^{-1} a μ] ⟨D⟩^{1-α} \qquad [μ D μ^*, a] μ^{-1 *} ⟨D⟩^{-α} \]
	are bounded. Then, with \( D_1 = μ D μ^* \) and \( D_2 = μ ⟨D⟩ μ^* \), for \( λ ≥ 0 \) and \( β ≤ 1-α \)
	\[ \left\| D_1 \left( (λ + ⟨D_1⟩^2)^{-1} a - a (λ + D_2^2)^{-1} \right) μ ⟨D⟩^β \right\| ≤ c_1 (λ + c_0)^{-1+(α+β)/2} \]
	where \( c_0 = \min\{1, \| μ^{-1} \|^{-4}\} \) and \( c_1 ≥ 0 \) is independent of \( λ \).
\end{lemma}
\begin{proof}
	First, by Lemma \ref{lemma:fractional-power-conformal-perturbation},
	\( \| D_2^{-β} μ ⟨D⟩^β \| = \| ⟨D⟩^β μ^* D_2^{-β} \| ≤ \| μ^{-1} \|^β \| μ \|^{1-β} \) so
	\begin{align*}
		& \left\| D_1 \left( (λ + ⟨D_1⟩^2)^{-1} a - a (λ + D_2^2)^{-1} \right) μ ⟨D⟩^β \right\| \\
		& \qquad ≤ \left\| D_1 \left( (λ + ⟨D_1⟩^2)^{-1} a - a (λ + D_2^2)^{-1} \right) D_2^β \right\| \| μ^{-1} \|^β \| μ \|^{1-β} ,
	\end{align*}
	By Lemma \ref{lemma:Tee},
	\( \| \mathfrak{T}_{(μ^* μ)^{-1}, (μ^* μ)} \| ≤ \max\{ \| μ^{-1} \|^2 - \| μ^{-1} \|^{-2}, \| μ \|^2 - \| μ \|^{-2} \} \). 
	We compute that
	\begin{align*}
		& \left\| D_1 \left( (λ + ⟨D_1⟩^2)^{-1} a - a (λ + D_2^2)^{-1} \right) D_2^β \right\| \\
		& \qquad ≤ \left\| D_1 (λ + ⟨D_1⟩^2)^{-1} a μ \mathfrak{T}_{(μ^* μ)^{-1}, μ^* μ}(⟨D⟩^{-1}) μ^{-1} D_2^{1+β} (λ + D_2^2)^{-1} \right\| \\
		& \qquad\qquad + \left\| D_1^2 (λ + ⟨D_1⟩^2)^{-1} \left( [D_1, a] μ^{-1 *} ⟨D⟩^{-α} - μ^{-1 *} [F_D, μ^* a μ] ⟨D⟩^{1-α} \right)\times\right.\\
		&\qquad\qquad\qquad × \left. ⟨D⟩^α μ^* D_2^{-α} D_2^{α+β} (λ + D_2^2)^{-1} \right\| \\
		& \qquad\qquad + \left\| D_1 (λ + ⟨D_1⟩^2)^{-1} \left( μ F_D μ^{-1} [D_1, a] μ^{-1 *} ⟨D⟩^{-α} + μ [F_D, μ^{-1} a μ] ⟨D⟩^{1-α} F_D^2 \right) \right. \\
		& \qquad\qquad\qquad × \left. ⟨D⟩^α μ^* D_2^{-α} D_2^{1+α+β} (λ + D_2^2)^{-1} \right\| \displaybreak[0] \\
		& \qquad ≤ \left\| D_1 (λ + ⟨D_1⟩^2)^{-1} \right\| \| a \| \| μ \| \left\| \mathfrak{T}_{(μ^* μ)^{-1}, μ^* μ}(⟨D⟩^{-1}) \right\| \| μ^{-1} \| \left\| D_2^{1+β} (λ + D_2^2)^{-1} \right\| \\
		& \qquad\qquad + \left\| D_1^2 (λ + ⟨D_1⟩^2)^{-1} \right\| \left( \left\| [D_1, a] μ^{-1 *} ⟨D⟩^{-α} \right\| - \left\| μ^{-1 *} [F_D, μ^* a μ] ⟨D⟩^{1-α} \right\| \right) \\
		& \qquad\qquad\qquad × \left\| ⟨D⟩^α μ^* D_2^{-α} \right\| \left\| D_2^{α+β} (λ + D_2^2)^{-1} \right\| \\
		& \qquad\qquad + \left\| D_1 (λ + ⟨D_1⟩^2)^{-1} \right\| \left( \| μ \| \| μ^{-1} \| \left\| [D_1, a] μ^{-1 *} ⟨D⟩^{-α} \right\| + \| μ \| \left\| [F_D, μ^{-1} a μ] ⟨D⟩^{1-α} \right\| \right) \\
		& \qquad\qquad\qquad × \left\| ⟨D⟩^α μ^* D_2^{-α} \right\| \left\| D_2^{1+α+β} (λ + D_2^2)^{-1} \right\| \displaybreak[0] \\
		& \qquad ≤ (λ + 1)^{-1/2} \| a \| \| μ \| \max\{ \| μ^{-1} \|^2 - \| μ^{-1} \|^{-2}, \| μ \|^2 - \| μ \|^{-2} \} \| μ^{-1} \| (λ + \| μ^{-1} \|^{-4})^{(-1+β)/2} \\
		& \qquad\qquad + \left( \left\| [D_1, a] μ^{-1 *} ⟨D⟩^{-α} \right\| - \| μ^{-1} \| \left\| [F_D, μ^* a μ] ⟨D⟩^{1-α} \right\| \right) \\
		& \qquad\qquad\qquad × \| μ^{-1} \|^α \| μ \|^{1-α} (λ + \| μ^{-1} \|^{-4})^{(-2+α+β)/2} \\
		& \qquad\qquad + (λ + 1)^{-1/2} \left( \| μ \| \| μ^{-1} \| \left\| [D_1, a] μ^{-1 *} ⟨D⟩^{-α} \right\| + \| μ \| \left\| [F_D, μ^{-1} a μ] ⟨D⟩^{1-α} \right\| \right) \\
		& \qquad\qquad\qquad × \| μ^{-1} \|^α \| μ \|^{1-α} (λ + \| μ^{-1} \|^{-4})^{(-1+α+β)/2} \\
		& \qquad ≤ c_1' (λ + c_0)^{-1+(α+β)/2}
	\end{align*}
	where \( c_0 = \min\{1, \| μ^{-1} \|^{-4}\} \) and \( c_1'\geq 0 \) is a constant independent of \( λ \). Hence,
	\[ \left\| D_1 \left( (λ + ⟨D_1⟩^2)^{-1} a - a (λ + D_2^2)^{-1} \right) μ ⟨D⟩^β \right\| ≤ c_1 (λ + c_0)^{-1+(α+β)/2} \]
	for \( c_1 = c_1' \| μ^{-1} \|^β \| μ \|^{1-β} \).
\end{proof}

\begin{lemma}
	\label{proposition:conformal-expression-integral}
	Let \( D \) be a self-adjoint regular operator and \( μ \) an invertible adjointable operator on $E$. Let \( a \) be an adjointable operator such that \( a μ^{-1 *} \dom D ⊆ μ^{-1 *} \dom D \). Suppose further that, for some \( 0 ≤ α < 1 \),
	\[ [F_D, μ^* a μ] ⟨D⟩^{1-α} \qquad [F_D, μ^{-1} a μ] ⟨D⟩^{1-α} \qquad [μ D μ^*, a] μ^{-1 *} ⟨D⟩^{-α} \]
	are bounded. Then, with \( D_1 = μ D μ^* \) and \( D_2 = μ ⟨D⟩ μ^* \),
	\[ D_1 \left( ⟨D_1⟩^{-1} a - a D_2^{-1} \right) μ ⟨D⟩^β \]
	is bounded for \( β < 1-α \).
\end{lemma}
\begin{proof}
	Using the integral formula \eqref{M},
	\begin{align*}
		& D_1 \left( ⟨D_1⟩^{-1} a - a D_2^{-1} \right) μ ⟨D⟩^β 
		 = \frac{1}{π} \int_0^∞ λ^{-1/2} D_1 \left( (λ + ⟨D_1⟩^2)^{-1} a - a (λ + D_2^2)^{-1} \right) μ ⟨D⟩^β d λ.
	\end{align*}
	By Proposition \ref{proposition:conformal-expression-bound}, the integrand is bounded and the integral is norm convergent when
	\[ \int_0^∞ λ^{-1/2} (λ + c_0)^{-1+(α+β)/2} dλ \]
	is convergent, that is, when \( β < 1-α \).
\end{proof}

\begin{theorem}
	\label{theorem:conformal-perturbation}
	Let \( D_0 \) be a self-adjoint regular operator and \( μ \) an invertible adjointable operator on $E$. Let \( a \) be an adjointable operator such that \( a μ^{-1 *} \dom D_0 ⊆ μ^{-1 *} \dom D_0 \). Suppose further that, for some \( 0 ≤ α < 1 \),
	\[ [F_{D_0}, μ^* a μ] ⟨D_0⟩^{1-α} \qquad [F_{D_0}, μ^{-1} a μ] ⟨D_0⟩^{1-α} \qquad [F_{D_0}, a μ] ⟨D_0⟩^{1-α} \qquad [μ D_0 μ^*, a] μ^{-1 *} ⟨D_0⟩^{-α} \]
	are bounded. Then, with \( D_1 = μ D_0 μ^* \), the operator
	\[ (F_{D_1} - F_{D_0}) a μ ⟨D_0⟩^β \]
	is bounded for \( β < 1-α \).
\end{theorem}
\begin{proof}
	We have
	\begin{align*}
		(F_{D_1} - F_{D_0}) a μ
		& = F_{D_1} a μ - a μ F_{D_0} - [F_{D_0}, a μ] \\
		& = F_{D_1} a μ - a D_1 μ^{-1 *} ⟨D_0⟩^{-1} - [F_{D_0}, a μ] \\
		& = F_{D_1} a μ - D_1 a μ^{-1 *} ⟨D_0⟩^{-1} + [D_1, a] μ^{-1 *} ⟨D_0⟩^{-1} - [F_{D_0}, a μ] \\
		& = D_1 \left( ⟨D_1⟩^{-1} a - a (μ ⟨D_0⟩ μ^*)^{-1} \right) μ + [D_1, a] μ^{-1 *} ⟨D_0⟩^{-1} - [F_{D_0}, a μ].
	\end{align*}
	Multiplying on the right by \( ⟨D⟩^β \), the first term remains bounded by Lemma \ref{proposition:conformal-expression-integral}. The remaining two terms are bounded owing to the last two of our displayed assumptions.
\end{proof}

\begin{theorem}
	\label{theorem:conformal-perturbation-ordinary-differentiability}
	Let \( D \) be a self-adjoint regular operator and \( μ \) an invertible adjointable operator on $E$. Let \( a \) be an adjointable operator such that \( \{ μ^* a μ, μ^{-1} a μ, a μ, μ^* a μ^{-1*} \} \dom D ⊆ μ^{-1*} \dom D \). Suppose further that, for some \( 0 ≤ α < 1 \),
	\[ [D, μ^* a μ] ⟨D⟩^{-α} \qquad [D, μ^{-1} a μ] ⟨D⟩^{-α} \qquad [D, a μ] ⟨D⟩^{-α} \qquad [μ D μ^*, a] μ^{-1 *} ⟨D⟩^{-α} \]
	are bounded. Then, with \( D_1 = μ D μ^* \),
	\[ (F_{D_1} - F_D) a μ ⟨D⟩^β \]
	is bounded for \( β < 1-α \). If \( b \) is an adjointable operator such that \( b^* μ^{-1 *} \dom D ⊆ \dom D \), then \( (F_{D_1} - F_D) a b ⟨D⟩^β \) is bounded. If \( c \) is a bounded operator such that \( (1 + D^2)^{-1} c \) is compact, then \( (F_{D_1} - F_D) a b c \) is compact.
\end{theorem}
\begin{proof}
	Applying Theorem \ref{theorem:relatively_bounded_commutator_to_compact_commutator}, we find that
	\[ [F_D, μ^* a μ] ⟨D⟩^{1-\gamma} \qquad [F_D, μ^{-1} a μ] ⟨D⟩^{1-\gamma} \qquad [F_{D_0}, a μ] ⟨D⟩^{1-\gamma} \]
	are bounded for \( γ > α \). Then, by Theorem \ref{theorem:conformal-perturbation},
	$ (F_{D_1} - F_D) a μ ⟨D⟩^β $
	is bounded for all \( β < 1-γ \), and so for all \( β < 1-α \). The remaining statements follow immediately.
\end{proof}

\begin{remark}
	In Theorem \ref{theorem:conformal-perturbation-ordinary-differentiability}, 
	that \( [μ D μ^*, a] μ^{-1 *} ⟨D⟩^{-α} \) is bounded is equivalent to
	\begin{multline*}
		D μ^{-1} [μ μ^*, a] μ^{-1 *} ⟨D⟩^{-α} = D (μ^* a μ^{-1 *} - μ^{-1} a μ) ⟨D⟩^{-α} \\
		= μ^{-1} [μ D μ^*, a] μ^{-1 *} ⟨D⟩^{-α} - [D, μ^{-1} a μ] ⟨D⟩^{-α}
	\end{multline*}
	being bounded, using the assumption that \( [D, μ^{-1} a μ] ⟨D⟩^{-α} \) is bounded. In other words, that \( μ μ^* \) and \( a \) almost commute.
\end{remark}

\begin{corollary}
	\label{corollary:conformal-perturbation-global}
	Let \( D \) be a self-adjoint regular operator and \( μ \) an invertible adjointable operator on $E$. Suppose that, for some \( 0 ≤ α < 1 \),
	\[ [F_D, μ] ⟨D⟩^{1-α} \qquad [F_D, μ^* μ] ⟨D⟩^{1-α} \]
	are bounded. Then, with \( D_1 = μ D μ^* \),
	\[ (F_{D_1} - F_D) μ ⟨D⟩^β \]
	is bounded for \( β < 1-α \). If \( μ^* \dom D ⊆ \dom D \), then \( (F_{D_1} - F_D) ⟨D⟩^β \) is bounded.
\end{corollary}

\begin{corollary}
	\label{corollary:conformal-perturbation-global-ordinary-differentiability}
	Let \( D \) be a self-adjoint regular operator and \( μ \) an invertible adjointable operator on $E$. Suppose that \( μ \dom D ⊆ \dom D \) and, for some \( 0 ≤ α < 1 \),
	\[ [D, μ] ⟨D⟩^{-α} \qquad ⟨D⟩^{-α} [D, μ] \]
	are bounded. Then, with \( D_1 = μ D μ^* \), the operator
	\[ (F_{D_1} - F_D) ⟨D⟩^β \]
	is bounded for \( β < 1-α \).
\end{corollary}

\begin{corollary}
	Let \( D_0 \) and \(D_1 \) be self-adjoint regular operators and \( μ \) an invertible adjointable operator on $E$. Suppose that \( μ \dom D_0 ⊆ \dom D_0 \) and, for some \( 0 ≤ α < 1 \),
	\[ (μ^{-1} D_1 μ^{-1 *} - D_0) ⟨D_0⟩^{-α} \qquad [D_0, μ] ⟨D_0⟩^{-α} \qquad ⟨D_0⟩^{-α} [D_0, μ] \]
	are bounded. Then the operator
	\[ (F_{D_1} - F_{D_0}) ⟨D_0⟩^β \]
	is bounded for \( β < 1-α \).
\end{corollary}

\begin{theorem}
	\label{theorem:conformal-result-nonunital}
	Let \( D \) be a self-adjoint regular operator and \( μ \) an invertible adjointable operator on $E$. Let \( a \) and \( b \) be adjointable operators such that \( \{ μ^* a, μ^{-1} a, a, b μ, b μ^{-1*} \} \dom D ⊆ \dom D \). Suppose further that, for some \( 0 ≤ α < 1 \),
	\[ ⟨D⟩^{-α} [D, a μ] \qquad ⟨D⟩^{-α} [D, a μ^{-1 *}] \qquad ⟨D⟩^{-α} [D, a] \qquad [D, b μ] ⟨D⟩^{-α} \qquad [μ D μ^*, a^* b] μ^{-1 *} ⟨D⟩^{-α} \]
	are bounded. Then, with \( D_1 = μ D μ^* \), the operator
	\[ (F_{D_1} - F_D) a^* b μ ⟨D⟩^β \]
	is bounded for \( β < 1-α \). If \( c \) is an adjointable operator such that \( c μ^{-1 *} \dom D ⊆ \dom D \), then \( (F_{D_1} - F_D) a^* b c^* ⟨D⟩^β \) is bounded. If \( d \) is an adjointable operator such that \( (1 + D^2)^{-1} d \) is compact, then \( (F_{D_1} - F_D) a^* b c^* d \) is compact.
\end{theorem}
\begin{proof}
	This follows from Theorem \ref{theorem:conformal-perturbation-ordinary-differentiability}, using \cite[Proposition A.5]{Goffeng_2015} for the appropriate Leibniz rule to relate the differing commutator conditions.
\end{proof}

Now, returning to the concept of conformal transformation, we have:

\begin{proof}[Proof of Theorem \ref{theorem:bdd-transform-conformal-transformation}]
\label{proof2.2}
	Let \( (U, μ) \) be a conformal transformation from \( (A, E_B, D_1) \) to \( (A, E'_B, D_2) \). By Proposition \ref{proposition:bounded-transform-of-commutator-additive-perturbation} and Lemma \ref{lemma:fractional-power-conformal-perturbation},
	\[ (U^* F_{D_2} U a - a F_{μ D_1 μ^*}) μ ⟨D_0⟩^β \]
	is bounded for \( a ∈ \mathscr{M} \). Let \( b, c ∈ \mathscr{M} \) and consider the operators
	\[ D = \begin{pmatrix} U^* D_2 U & \\ & μ D_1 μ^* \end{pmatrix} \qquad B = \begin{pmatrix} & b \\ 0 & \end{pmatrix} \qquad C = \begin{pmatrix} & c \\ 0 & \end{pmatrix} \]
	on \( E ⊕ E' \). By assumption and using Lemma \ref{lemma:fractional-power-conformal-perturbation},
	\[ [D, B] ⟨D⟩^{-α} = \begin{pmatrix} & (U^* D_2 U b - b μ D_1 μ^*) ⟨μ D_1 μ^*⟩^{-α} \\ 0 & \end{pmatrix} \quad\mbox{and}\quad [D, C] ⟨D⟩^{-α} \]
	are bounded. By \cite[Proposition A.5]{Goffeng_2015},
	\[ [D, B^* C] ⟨D⟩^{-α} = \begin{pmatrix} 0 & \\ & [μ D_1 μ^*, b^* c] ⟨μ D_1 μ^*⟩^{-α} \end{pmatrix} \]
	extends to an adjointable operator. Again using Lemma \ref{lemma:fractional-power-conformal-perturbation},
	$ [μ D_1 μ^*, b^* c] μ^{-1 *} ⟨D_1⟩^{-α} $
	is bounded and we may apply Theorem \ref{theorem:conformal-result-nonunital} to obtain that
	\[ (F_{μ D_1 μ^*} - F_{D_1}) b^* c μ ⟨D_1⟩^β \]
	is bounded for \( β < 1-α \).  Then
	\[ (U^* F_{D_2} U - F_{D_1}) a b^* c = (U^* F_{D_2} U a - a F_{μ D_1 μ^*}) b^* c - [F_{D_1}, a] b^* c + a (F_{μ D_1 μ^*} - F_{D_1}) b^* c \]
	so that \( (U^* F_{D_2} U - F_{D_1}) a b^* c μ ⟨D_0⟩^β \) is bounded. For \( d ∈ \mathscr{M} \) and \( e ∈ A \) we find
	\[ (U^* F_{D_2} U - F_{D_1}) a b^* c d^* e = (U^* F_{D_2} U - F_{D_1}) a^* b c μ ⟨D_1⟩^β (⟨D_1⟩^{-β} μ^{-1} d^* ⟨D_1⟩^β) ⟨D_1⟩^{-β} e \]
	is compact. By the inclusion \( A ⊆ \overline{\Span}((\mathscr{M}^* \mathscr{M})^2 A) \), we are done.
\end{proof}

\subsubsection{A partial converse}

A partial converse result is possible, in the sense that these kinds of estimates on bounded transforms always arise from an additive and a multiplicative perturbation of the unbounded operator. This is not quite precise due to differences in the differentiability assumptions. The following is nearly a converse to Corollary \ref{corollary:conformal-perturbation-global}.

\begin{theorem}
\label{thm:converse}
	Let \( D_1 \) and \( D_2 \) be self-adjoint regular operators with equal domains such that, for some \( 0 < α ≤ 1 \),
	\[ (F_{D_1} - F_{D_2}) ⟨D_1⟩^α \]
	is bounded on \( \dom ⟨D_1⟩^α \). Then there exist a bounded invertible operator \( μ \) and a self-adjoint regular operator \( T \) such that
	\[ D_2 = μ D_1 μ^* + T \]
	and both
	\[ ⟨D_1⟩^{-1/2} T ⟨D_1⟩^{-1/2+α} \qquad \left( [F_{D_1}, μ] - T ⟨D_2⟩^{-1} \right) ⟨D_1⟩^α \]
	are bounded. Furthermore, if \( 1/2 ≤ α \),
	\[ T ⟨D_1⟩^{-1+α} \qquad [F_{D_1}, μ] ⟨D_1⟩^α \]
	are bounded.
\end{theorem}
\begin{proof}
	Let \( μ = ⟨D_2⟩^{1/2} ⟨D_1⟩^{-1/2} \) and \( T = ⟨D_2⟩^{1/2} (F_{D_2} - F_{D_1}) ⟨D_2⟩^{1/2} \), defined on \( \dom D_1 \), so that
	\begin{align*}
		μ D_1 μ^* + T
		& = ⟨D_2⟩^{1/2} ⟨D_1⟩^{-1/2} D_1 ⟨D_1⟩^{-1/2} ⟨D_2⟩^{1/2} 
		 + ⟨D_2⟩^{1/2} (F_{D_2} - F_{D_1}) ⟨D_2⟩^{1/2} \\
		& = ⟨D_2⟩^{1/2} \left( F_{D_1} + (F_{D_2} - F_{D_1}) \right) ⟨D_2⟩^{1/2} \\
		& = D_2.
	\end{align*}
	We have
	\begin{align*}
		[F_{D_1}, μ]
		& = \left( F_{D_1} ⟨D_2⟩^{1/2} - ⟨D_2⟩^{1/2} F_{D_1} \right) ⟨D_1⟩^{-1/2} \\
		& = \left( ⟨D_2⟩^{1/2} (F_{D_2} - F_{D_1}) + (F_{D_2} - F_{D_1}) ⟨D_2⟩^{1/2} \right) ⟨D_1⟩^{-1/2} \\
		& = \left( T ⟨D_2⟩^{-1/2} + ⟨D_2⟩^{-1/2} T \right) ⟨D_1⟩^{-1/2} \\
		& = T ⟨D_2⟩^{-1} + (F_{D_2} - F_{D_1}).
	\end{align*}
	Because the domains of \( D_1 \) and \( D_2 \) are equal, \( (F_{D_2} - F_{D_1}) ⟨D_2⟩^α \) is bounded and the statement follows from the boundedness of
	\[ ⟨D_2⟩^{-1/2} T ⟨D_2⟩^{-1/2+α} = (F_{D_2} - F_{D_1}) ⟨D_2⟩^α \qquad \left( [F_{D_1}, μ] - T ⟨D_2⟩^{-1} \right) ⟨D_2⟩^α = (F_{D_2} - F_{D_1}) ⟨D_2⟩^α . \]
	Suppose that \( 1/2 ≤ α \). It is sufficient to prove that
	\[ T ⟨D_2⟩^{-1+α} = ⟨D_2⟩^{1/2} (F_{D_2} - F_{D_1}) ⟨D_2⟩^{-1/2+α} \]
	is bounded. If \( α = 1/2 \),
	\[ T ⟨D_2⟩^{-1/2} = ⟨D_2⟩^{1/2} (F_{D_2} - F_{D_1}) \]
	and we are done. If \( 1/2 < α ≤ 1 \), both \( 1/2 \) and \( -1/2 + α \) are positive, and we can interpolate between
	\[ (F_{D_1} - F_{D_2}) ⟨D_2⟩^α \quad\mbox{and}\quad ⟨D_2⟩^α (F_{D_1} - F_{D_2}) \]
	as in \cite[Proposition A.1]{Lesch_2005}, adjusted for Hilbert modules in \cite[Lemma 7.7]{Lesch_2019} (see also Appendix \ref{section:fractional-powers}).
\end{proof}

\subsection{The logarithmic transform: multiplicative to additive}
\label{section:log}

Conformal transformations of unbounded Kasparov modules are not preserved by the exterior product. This is exemplified by the fact that the Cartesian product of two conformally perturbed Riemannian manifolds \( (X_1, k_1^2 \mathbf{g}_1) \) and \( (X_2, k_2^2 \mathbf{g}_2) \) is not a conformal perturbation of the Cartesian product \( (X_1 × X_2, \mathbf{g}_1 ⊕ \mathbf{g}_2) \), unless \( k_1(x) = k_2(y) \) for all \( x ∈ X_1 \) and \( y ∈ X_2 \), i.e. \( k_1 = k_2 \) is a constant. The \emph{logarithmic dampening} of \cite{Goffeng_2019a} provides a way of turning conformal transformations into locally bounded perturbations, at the expense of much of the geometrical information encoded by the Dirac operator.

\begin{proposition}
	\label{proposition:logarithmic-transform}
	Let \( D \) be a self-adjoint regular operator on a right Hilbert \( B \)-module \( E \) and let \( a ∈ \End_B^*(E) \) preserve \( \dom D \). Suppose also that \( [F_D, a] \log⟨D⟩ \) is bounded. Then, with
	\[ L_D = F_D \log ⟨D⟩ = D \log((1 + D^2)^{1/2}) (1 + D^2)^{-1/2} , \]
	the commutator \( [L_D, a] \) is bounded.
\end{proposition}
\begin{proof}
	By \cite[Lemma 1.15]{Goffeng_2019a}, the condition \( a \dom D ⊆ \dom D \) implies that \( a \dom \log ⟨D⟩ ⊆ \dom \log ⟨D⟩ \) and that \( [\log ⟨D⟩, a] \) is bounded. Using also the condition on \( [F_D, a] \),
	\[ [L_D, a] = F_D [\log ⟨D⟩, a] + [F_D, a] \log ⟨D⟩ \]
	is bounded.
\end{proof}

\begin{corollary}
	Let \( D_0 \) and \( D_1 \) be self-adjoint regular operators on right Hilbert \( B \)-modules \( E_0 \) and \( E_1 \). Suppose that there is an operator \( a ∈ \Hom_B^*(E_0, E_1) \) such that \( a \dom D_0 ⊆ \dom D_1 \) and
	\[ (F_{D_1} a - a F_{D_0}) \log⟨D_0⟩ \]
	extends to an adjointable operator. Then \( L_{D_1} a - a L_{D_0} \) is bounded.
\end{corollary}

\begin{theorem}
\label{thm:conf-log-bdd}
	Let \( (U, μ) \) be a conformal transformation from the order-\( \frac{1}{1 - α} \) cycle \( (A, E_B, D_1) \) to the order-\( \frac{1}{1 - α} \) cycle  \( (A, E'_B, D_2) \). Then the logarithmic transforms \( (A, E_B, L_{D_1}) \) and \( (A, E'_B, L_{D_2}) \) are related by the unitary \( U \), up to locally bounded perturbation; in particular, \( A \) is contained in the closure of the set of \( a ∈ \End^*(E) \) such that
	\[ (U^* L_{D_2} U - L_{D_1}) a \qquad [L_{D_1}, a] \qquad [L_{D_2}, U a U^*] \]
	are bounded.
\end{theorem}
\begin{proof}
	Let \( a, b, c ∈ \mathscr{M} \) so that \( (U^* F_{D_2} U - F_{D_1}) a b^* c μ ⟨D_0⟩^β \)
	\begin{align*}
		(U^* L_{D_2} U - L_{D_1}) a b^* c μ
		& = U^* L_{D_2} U a b^* c μ - a b^* c μ L_{D_1} - [L_{D_1}, a b^* c μ] \\
		& = U^* F_{D_2} U (U^* \log ⟨D_2⟩ U a b^* c μ - a b^* c μ \log ⟨D_1⟩) \\
		& \qquad + (U^* F_{D_2} U - F_{D_1}) a b^* c μ \log ⟨D_1⟩ - F_{D_1} [\log ⟨D_1⟩, a b^* c μ]
	\end{align*}
	is bounded, by the proof of Theorem \ref{theorem:bdd-transform-conformal-transformation}. Let \( d ∈ \Lip_\alpha^*(D) \) and multiply on the right by \( μ^{-1}d \). Then \( (U^* L_{D_2} U - L_{D_1}) a^* b c d \) is bounded and, by the inclusions \( A ⊆ \overline{\Span}(\mathscr{M} A) ⊆ \overline{\Span}(\mathscr{M} \mathscr{M}^* \mathscr{M} \Lip_\alpha^*(D)) \), we are done. 
\end{proof}

\subsection{The singular case}
\label{section:sing-conf-tr}

Conformal factors on noncompact manifolds need not be bounded nor have bounded inverse. In that setting, we can take a suitable open cover and assemble local estimates. This idea motivates the next definition.
In the following we stress that $\overline{\Span}$ means the norm completion of finite linear combinations.

\begin{definition}
	\label{definition:sing-conf-transformation}
	A \emph{singular conformal transformation} \( (U, (μ_i)_{i ∈ I}) \) from one order-\( \frac{1}{1 - α} \) cycle, \( (A, E_B, D_1) \), to another, \( (A, E'_B, D_2) \), is a unitary map \( U: E \to E' \), intertwining the representations of \( A \), and a family \( (μ_i)_{i ∈ I} ⊆ \End^*(E) \) of (even) invertible operators such that
	\[ A ⊆ \overline{\Span}_{i ∈ I} A \mathscr{M}_i ∩ \overline{\Span}_{i ∈ I} \mathscr{M}_i A \]
	where \( \mathscr{M}_i \) is the set of \( a ∈ \End^*(E) \) such that
	\[ (U^* D_2 U a - a μ_i D_1 μ_i^*) μ_i^{-1 *} ⟨D_1⟩^{-α} \qquad ⟨D_2⟩^{-α} U (U^* D_2 U a - a μ_i D_1 μ_i^*) \]
	are bounded, \( a, a μ_i, a μ_i^{-1 *} ∈ \Lip_α^*(D_1) \), and \( U a U^* ∈ \Lip_α^*(D_2) \).
\end{definition}

\begin{remark}
	As in the non-singular case, \( \mathscr{M}_i \) is a ternary ring of operators, generally not closed. In particular, \( \overline{\Span}(\mathscr{M}_i \mathscr{M}_i^* \mathscr{M}_i) = \overline{\mathscr{M}_i} \).
\end{remark}

\begin{theorem}
	Let \( (U, (μ_n)_{n ∈ \bbN}) \) be a singular conformal transformation from \( (A, E_B, D_1) \) to \( (A, E'_B, D_2) \). Then the bounded transforms \( (A, E_B, F_{D_1}) \) and \( (A, E'_B, F_{D_2}) \) are related by the unitary \( U \), up to locally compact perturbation, i.e.
	\[ (U^* F_{D_2} U - F_{D_1}) a ∈ \End^0(E) \]
	for all \( a ∈ A \).
\end{theorem}
\begin{proof}
	As in the Proof of Theorem \ref{theorem:bdd-transform-conformal-transformation}, \( (U^* F_{D_2} U - F_{D_1}) a b^* c μ_i ⟨D_0⟩^β \) is bounded for all \( a, b, c ∈ \mathscr{M}_i \). For \( d, e ∈ \mathscr{M}_i \) and \( f ∈ A \) we find
	\[ (U^* F_{D_2} U - F_{D_1}) a b^* c d^* e f = (U^* F_{D_2} U - F_{D_1}) a^* b c μ_i ⟨D_1⟩^β (⟨D_1⟩^{-β} μ^{-1}_i d^* e ⟨D_1⟩^β) ⟨D_1⟩^{-β} f \]
	is compact. The inclusion of \( A ⊆ \overline{\Span}_{i ∈ I}(\mathscr{M}_i A) = \overline{\Span}_{i ∈ I}((\mathscr{M}_i \mathscr{M}_i^*)^2 \mathscr{M}_i A) \) proves the statement.
\end{proof}

\begin{example}
Let us reprise Example \ref{example:atiyah-singer-conformal-transformation}, in which we considered Riemannian spin\(^c\) manifolds \( (X, \mathbf{g}) \) and \( (X, \mathbf{h}) \) such that \( \mathbf{h} = k^2 \mathbf{g} \). Suppose that \( (X, \mathbf{g}) \) is geodesically complete, so that \( \slashed{D}_{\mathbf{g}} \) is self-adjoint. It may or may not be the case that \( (X, \mathbf{h}) \) is complete and \( \slashed{D}_{\mathbf{h}} \) is self-adjoint, depending on the properties of \( k \), although that is guaranteed if \( k \) is bounded with bounded inverse. Let \( (O_i)_{i ∈ I} \) be an open cover of \( X \) such that \( k \) is bounded and invertible when restricted to any \( O_i \). (This can be ensured by choosing a relatively compact cover.) Choose a family \( (k_i)_{i ∈ I} \) of positive smooth functions which are bounded and invertible and agree with \( k \) on the corresponding \( O_i \). Let \( f ∈ C_c^∞(O_i) \), so that
\begin{align*}
	U^* \slashed{D}_{\mathbf{h}} U f - f k_i^{-1/2} \slashed{D}_{\mathbf{g}} k_i^{-1/2}
	& = k^{-1/2} \slashed{D}_{\mathbf{g}} k^{-1/2} f - f k_i^{-1/2} \slashed{D}_{\mathbf{g}} k_i^{-1/2} \\
	& = k^{-1/2} [\slashed{D}_{\mathbf{g}}, f] k_i^{-1/2} \\
	& = k_i^{-1/2} [\slashed{D}_{\mathbf{g}}, f] k_i^{-1/2}
\end{align*}
is bounded.
Then \( (U, (k_i^{-1/2})_{i ∈ I}) \) is a singular conformal transformation from the spectral triple \( (C_0(X), L^2(X, S_{\mathbf{g}}), \slashed{D}_{\mathbf{g}}) \) to \( (C_0(X), L^2(X, S_{\mathbf{h}}), \slashed{D}_{\mathbf{h}}) \), provided that \( (X, \mathbf{h}) \) is complete so that the latter is a spectral triple. In the context of Example \ref{example:hodge-de-rham-conformal-transformation}, \( (U, (k_i^{-1/2})_{i ∈ I}) \) is a singular conformal transformation from \( (C_0(X), L^2(Ω^* X, {\mathbf{g}}), d + δ_{\mathbf{g}}) \) to \( (C_0(X), L^2(Ω^* X, {\mathbf{h}}), d + δ_{\mathbf{h}}) \).

If either or both of \( (X, \mathbf{g}) \) and \( (X, \mathbf{h}) \) fails to be complete, the failure of self-adjointness of the Dirac operator(s) means that one requires the technology of half-closed chains and relative spectral triples. We do not pursue this here; for more details, see \cite{Hilsum_2010, Deeley_2018, Forsyth_2019}.
\end{example}

An abstract treatment of open covers, for the purposes of unbounded KK-theory, can be found in \cite{Dungen_2022}; see, in particular, \cite[Lemma 4.3]{Dungen_2022}.

In the following example, inspired by the modular cycles of \cite{Kaad_2021}, one should think of \( Δ_- Δ_+^{-1} \) as the conformal factor, which can be both unbounded and noninvertible.
Later, in Proposition \ref{proposition:conformally-generated-delta}, we directly generalise the results of \cite{Kaad_2021}.

\begin{proposition}
	\label{proposition:singular-conformal-transformation-delta}
	Let \( (A, E_B, D_1) \) and \( (A, E_B, D_2) \) be unbounded Kasparov modules. Let \( Δ_+ \) and \( Δ_- \) be commuting positive adjointable operators such that
	\begin{itemize}
		\item For all \( a ∈ A \), \( (a (Δ_+ + Δ_-) (Δ_+ + Δ_- + \frac{1}{n})^{-1})_{n=1}^∞ \) converges in operator norm to \( a \).
		\item \( Δ_+, Δ_- \in \Lip_{\alpha}^*(D_1) \); and
		\item \( A ⊆ \overline{\Span}(A \mathscr{N}) ∩ \overline{\Span}(\mathscr{N} A) \), where \( \mathscr{N} \) is the set of \( a ∈ \Lip^*_{\alpha}(D_1) ∩ \Lip^*_{\alpha}(D_2) \) such that \( a \dom D_1 \subseteq \dom D_2 \), \( a^* \dom D_2 \subseteq \dom D_1 \), and
			\[ (D_2 a Δ_+ - a D_1 Δ_-) \langle D_1 \rangle^{-\alpha} \qquad \langle D_2 \rangle^{-\alpha} (D_2 a Δ_+ - a D_1 Δ_-)  \]
			extend to adjointable operators.
	\end{itemize}
	Let \( (h_n)_{n ∈ \bbN_{≥1}} ⊆ C_b^∞(\bbR^×_+) \) be any sequence of positive functions with bounded reciprocals which agree with the function \( x ↦ x^{-1/2} \) on the interval \( [\frac{1}{n}, n] \). Then \( (1, (h_n(Δ_+) h_n(Δ_-)^{-1})_{n ∈ \bbN_{≥1}}) \) is a singular conformal transformation from \( (A, E_B, D_1) \) to \( (A, E_B, D_2) \).
\end{proposition}

For the proof, we recall a statement of the relevant aspects of the smooth functional calculus. For \( \alpha = 0 \), the following is due to Powers \cite[Theorem 3]{Powers_1975} (although corrected in \cite[\S 2]{Bratteli_1976}). For a discussion of the functional calculus for higher order Kasparov modules, when \( \alpha \in (0, 1] \), we refer to the first named authors PhD thesis \cite[\S A.4]{Masters_2025a} (but see also \cite[Lemma 3.2]{Bratteli_1984} and \cite[Proposition 6.4]{Blackadar_1991}).

\begin{theorem}
	\label{theorem:smooth-fc}
	\cite[Theorem A.4.18]{Masters_2025a}
	Let \( D \) be a self-adjoint regular operator on a Hilbert \( B \)-module \( E \). If \( S \in \Lip_{\alpha}^*(D) \) is self-adjoint then, for any function \( f ∈ C_c^∞(\bbR) \), \( f(S) \in \Lip_{\alpha}^*(D) \).
\end{theorem}

\begin{lemma}
	\label{lemma:delta-approximate-unit-convergence}
	Let \( A \) be a C*-algebra represented by \( π \) on a Hilbert module \( E \). Let \( h\in C ⊆ \End^*(E) \) be a strictly positive element of a C*-algebra \( C \) such that, for a dense subset of \( a ∈ A \), the sequence
	\[ (π(a) h (h + 1/n)^{-1})_{n=1}^∞ \]
	converges to \( π(a) \). Then \( π(A) \) is contained in the closure of \( π(A) C \).
\end{lemma}
\begin{proof}	
	First, note that \( (h (h + 1/n)^{-1})_{n=1}^∞ \) is an approximate unit for \( C \). For every \( a ∈ A \) such that the sequence
	\( (π(a) h (h + 1/n)^{-1})_{n=1}^∞ ⊆ π(a) C \)
	converges in norm to \( π(a) \), \( π(a) ∈ \overline{π(a) C} \).
\end{proof}

\begin{proof}[Proof of Proposition \ref{proposition:singular-conformal-transformation-delta}]
	First, the smooth functional calculus of Theorem \ref{theorem:smooth-fc} shows that the \( h_n(Δ_+) h_n(Δ_-)^{-1} \in \Lip^*_{\alpha}(D_1) \) is bounded. Second, let $f_1, f_2 \in C_c^\infty((\tfrac{1}{n}, n))$ and \( a ∈ \mathscr{N} \), and define $b∈ \End^*(E)$ to be the product 
	\[ a f_1(Δ_+) f_2(Δ_-) ∈ \mathscr{N} C_0((\tfrac{1}{n}, n))(Δ_+) C_0((\tfrac{1}{n}, n))(Δ_-) . \]
	Then $b h_n(Δ_+) h_n(Δ_-)^{-1} = b Δ_+^{-1/2} Δ_-^{1/2}$. Again using the smooth functional calculus,
	\begin{align*}
		& \big( D_2 b - b h_n(Δ_+) h_n(Δ_-)^{-1} D_1 h_n(Δ_+) h_n(Δ_-)^{-1} \big) (h_n(Δ_+) h_n(Δ_-)^{-1})^{-1} \langle D_1 \rangle^{-\alpha} \\
		& \qquad = (D_2 a Δ_+ - a D_1 Δ_-) \langle D_1 \rangle^{-\alpha} \big( \langle D_1 \rangle^{\alpha} Δ_+^{-1/2} Δ_-^{-1/2} f_1(Δ_+) f_2(Δ_-) \langle D_1 \rangle^{-\alpha} \big) \\
		& \qquad\qquad + a \left[ D_1, Δ_-^{1/2} Δ_+^{-1/2} f_1(Δ_+) f_2(Δ_-) \right] \langle D_1 \rangle^{-\alpha}
	\end{align*}
	and
	\begin{align*}
		& \langle D_2 \rangle^{-\alpha} \big( D_2 b - b h_n(Δ_+) h_n(Δ_-)^{-1} D_1 h_n(Δ_+) h_n(Δ_-)^{-1} \big) \\
		& \qquad = \langle D_2 \rangle^{-\alpha} (D_2 a Δ_+ - a D_1 Δ_-) Δ_+^{-1} f_1(Δ_+) f_2(Δ_-) \\
		& \qquad\qquad + \big( \langle D_1 \rangle^{\alpha} a^* \langle D_2 \rangle^{-\alpha} \big)^* \langle D_1 \rangle^{-\alpha} \left[ D_1, Δ_-^{1/2} Δ_+^{-1/2} f_1(Δ_+) f_2(Δ_-) \right] h_n(Δ_+) h_n(Δ_-)^{-1}
	\end{align*}
	extend to adjointable operators. 
	Hence \( b \in \mathscr{M}_n \). The closure of \( C_0((\tfrac{1}{n}, n))(Δ_+) C_0((\tfrac{1}{n}, n))(Δ_-) \) is \( C^*(Δ_+, Δ_-) \). By Lemma \ref{lemma:delta-approximate-unit-convergence}, we have \( A ⊆ \overline{A C^*(Δ_+, Δ_-)} \) and
	\[ \overline{\Span}_{i ∈ I} A \mathscr{M}_i ∩ \overline{\Span}_{i ∈ I} \mathscr{M}_i A ⊇ \overline{\Span}(A \mathscr{N} C^*(Δ_+, Δ_-)) ∩ \overline{\Span}(\mathscr{N} C^*(Δ_+, Δ_-) A) ⊇ A , \]
	as required.
\end{proof}

\section{Group-equivariant KK-theory}
\label{section:group-equivariant}

In this section we begin by recalling the definitions of equivariant KK-theory and the descent map, due to Kasparov \cite{Kasparov_1988}. The first attempt to generalise equivariance to unbounded KK-theory is \cite[\S 1]{Julg_1987}, for the case of \( KK^G(\bbC, \bbC) \). The first detailed treatment is by Kucerovsky \cite[§8]{Kucerovsky_1994}, which we mildly generalise in \S\ref{subsec:uni-equi} to apply to the higher-order case and allow for local boundedness in the definition. In \S\ref{subsec:con-equi-KK}, we provide a generalisation to conformal equivariance for unbounded cycles that provides greater flexibility. 

The case of compact groups is much easier to handle in both the bounded and unbounded settings. This is because, given the action of a compact group on a Kasparov module, one can integrate using the Haar measure to produce a module for which the operator is actually invariant under the action of the group. This fact has led to the definition of unbounded equivariant KK-theory in the case of a compact group as unbounded Kasparov modules with group actions for which the operator is invariant under the action. Alas, this does not represent the full range of geometrical situations available under equivariant KK-theory.

The following definition introduces notation for tracking the action of operators implementing equivariance. Throughout this section, \( G \) is a locally compact group.

\begin{definition}
	Let \( E \) be a right Hilbert \( B \)-module and \( τ ∈ \Aut A \). We define \( \End^{*, τ}_B(E) \) to be the set of \( \bbC \)-linear maps \( T : E \to E \) for which there exists a map \( T^* : E \to E \) such that
	\[ (T(x), y)_B = τ((x, T^*(y))_B). \]
	These maps are not \( B \)-linear; however they satisfy
	$ T(x b) = T(x) τ(b) $
	since
	\[ (T(x b), y)_B = τ((x b, T^*(y))_B) = τ(b^*) τ((x, T^*(y))_B) = τ(b^*) (T(x), y)_B = (T(x) τ(b), y)_B. \]
	This gives an identification of \( \End^{*, τ}_B(E) \) with \( \Hom^*_B(E, E ⊗_τ B) \), where \( E ⊗_τ B \) is the internal tensor product of \( E \) with \( {}_τ B \). The adjoint \( T^* ∈ \End^{*, τ^{-1}}_B(E) \), since
	\[ (T^*(x), y)_B = (y, T^*(x))_B^* = τ^{-1}((T(y), x)_B^*) = τ^{-1}((x, T(y))_B). \]
	The composition of \( S ∈ \End^{*, σ}_B(E) \) and \( T ∈ \End^{*, τ}_B(E) \) is \( S T ∈ \End^{*, σ \circ τ}_B(E) \). In particular, if $\tau=\sigma^{-1}$ then $ST$ is an adjointable operator.
\end{definition}

\begin{definition}
	e.g.~\cite[\S 1.2]{Kasparov_1988}
	Let \( β : G \to \Aut B \) be an action of a group \( G \) on a C*-algebra \( B \). A \emph{\( G \)-equivariant} Hilbert \( B \)-module \( E \) is a Hilbert \( B \)-module equipped with a continuous \( \bbC \)-linear map \( U : G × E \to E \) such that
	\[ U_{g h} = U_g U_h \qquad U_g(x b) = U_g(x) β_g(b) \qquad β_g((x, y)_B) = (U_g(x), U_g(y))_B \]
	for \( g, h ∈ G \), \( x, y ∈ E \), and \( b ∈ B \). We may equivalently say that \( U_g ∈ \End^{*, β_g}_B(E) \) with the conditions
	\[ U_{g h} = U_g U_h \qquad U_{g^{-1}} = U_g^{-1} = U_g^* \]
	for all \( g, h ∈ G \).
\end{definition}

\begin{definition}
	Let \( α : G \to \Aut A \) be an action of a group \( G \) on a C*-algebra \( A \). A \emph{\( G \)-equivariant} \( A \)-\( B \)-correspondence \( E \) is an \( A \)-\( B \)-correspondence \( E \) which is also a \( G \)-equivariant Hilbert \( B \)-module, such that
	\[ U_g(a x) = α_g(a) U_g(x) \]
	for \( g ∈ G \), \( a ∈ A \) and \( x ∈ E \).
\end{definition}

\begin{definition}
	\cite[Definition 2.2]{Kasparov_1988} cf.~\cite[Definition 8.5, Remark]{Kucerovsky_1994}
	A bounded Kasparov \( A \)-\( B \)-module \( (A, E_B, F) \) is \emph{\( G \)-equivariant} if \( E \) is a \( G \)-equivariant \( A \)-\( B \)-correspondence and, for all \( a ∈ A \), the map \( g \mapsto (U_g F U_g^* - F) a \) is norm-continuous from \( G \) into \( \End^0(E) \).
\end{definition}

\begin{remark}
	cf.~\cite[Definition 8.5, Remark]{Kucerovsky_1994}
	By Lemma \ref{lemma:norm-continuous_on_compact_subsets_to_compact_endomorphism}, the norm continuity of the map \( g \mapsto (U_g F U_g^* - F) a \) into \( \End^0(E) \) is equivalent to the condition that, when restricted to any compact subset \( K \subseteq G \), the function \( g \mapsto (U_g F U_g^* - F) a \) is in \( \End^0(C(K, E)) \).
\end{remark}

\subsection{Uniformly equivariant unbounded KK-theory}
\label{subsec:uni-equi}
Again, throughout this section, \( G \) is a locally compact group. The following definition slightly generalises that of Kucerovsky.

\begin{definition}
	\label{definition:ordinary-unbounded-equivariance}
	cf. \cite[Definition 8.7]{Kucerovsky_1994}
	An order-\( \frac{1}{1 - α} \) \( A \)-\( B \)-cycle \( (A, E_B, D) \) is \emph{uniformly \( G \)-equivariant} if \( E \) is a \( G \)-equivariant \( A \)-\( B \)-correspondence and \( A \) is contained in the closure of \( \mathscr{Q} \), the set of \( a ∈ \End^*(E) \) such that \( a \dom D ⊆ U_g \dom D \) for all $g\in G$ and the maps
	\[ g ↦ (U_g D U_g^* a - a D) ⟨D⟩^{-α} \qquad g ↦ ⟨D⟩^{-α} U_g^* (U_g D U_g^* a - a D) \]
	are $*$-strongly continuous as a map from \( G \) into bounded operators (on \( \dom D \)).
	If \( U_g D U_g^* = D \) for all \( g ∈ G \), we say that the cycle is \emph{isometrically equivariant}. If \( \mathscr{A} \) is a dense $*$-subalgebra of \( A \) contained in \( \mathscr{Q} \), we say that \( (\mathscr{A}, E_B, D) \) is a uniformly \( G \)-equivariant order-\( \frac{1}{1 - α} \) \( \mathscr{A} \)-\( B \)-cycle.
\end{definition}

\begin{remarks}
	\item We remark that \( \mathscr{Q} ⊆ \Lip^*_α(D) \) by considering the conditions at \( g = e \), the identity of the group. Indeed, \( \mathscr{Q} \) is a right ideal of \( \Lip_α^*(D) \).
	\item By Lemma \ref{lemma:strongly-continuous_on_compact_subsets_to_bounded_endomorphism}, the conditions on \( a ∈ \mathscr{Q} \) are equivalent to the condition that \( a \dom D ⊆ U_g \dom D \) and, when restricted to any compact subset \( K \subseteq G \), the functions 
		\[ g ↦ (U_g D U_g^* a - a D) ⟨D⟩^{-α} \qquad g ↦ ⟨D⟩^{-α} U_g^* (U_g D U_g^* a - a D) \]
		be in \( \End^*(C(K, E)) \).
	\item When $\alpha=0$, the conditions on \( a ∈ \mathscr{Q} \) are equivalent to requiring that \( [D, a] \) extend to an adjointable operator and
		\[ g ↦ (U_g D U_g^* - D) a \]
		be $*$-strongly continuous as a map from \( G \) into bounded operators. The higher order generalisation allows for higher order differential operators on manifolds, for example.
\end{remarks}

To prove that the bounded transform is well-defined, we use the results of Appendix \ref{appendix:hilbert-modules-top}, based on the approach of Kucerovsky \cite[Chapter 8, Appendix A]{Kucerovsky_1994}; see also \cite[Appendix A]{Abdolmaleki_2023}.

\begin{theorem}
	\cite[Proposition 8.11]{Kucerovsky_1994}
	Let \( (A, E_B, D) \) be a uniformly \( G \)-equivariant order-\( \frac{1}{1-α} \) cycle. Then \( (A, E_B, F_D) \) is a \( G \)-equivariant bounded Kasparov module.
\end{theorem}
\begin{proof}
	The only difference from the non-equivariant case is the need to show that, for every \( a ∈ A \), \( g \mapsto (F_D - U_g F_D U_g^*) a \) is norm-continuous as a map from \( G \) into \( \End^0(E) \).

	Fix \( b ∈ \mathscr{Q} \), where $\mathscr{Q}$ is as in Definition \ref{definition:ordinary-unbounded-equivariance}. By definition, the map \( f : g ↦ (U_g D U_g^* b - b D) ⟨D⟩^{-α} \) is $*$-strongly continuous as a map from \( G \) into \( \End^*(E) \). By Lemma \ref{lemma:strongly-continuous_on_compact_subsets_to_bounded_endomorphism}, this is equivalent to \( f|_K \) residing in \( \End^*(C(K, E)) \) for every compact subset \( K ⊆ G \).

	Fix a compact subset \( K ⊆ G \) and let \( \tilde{E} = C(K, E) \). Define \( \tilde{D} \) to be the self-adjoint regular operator on \( \tilde{E} \) given by \( D \) at each point of \( K \). Similarly, let \( \tilde{b} ∈ \End^*(\tilde{E}) \) be given by \( b \) at each point of \( K \). Let \( U \) denote the \( \bbC \)-linear map from \( \tilde{E} \) to itself given by \( g \mapsto U_g \). Then
	\[ (U \tilde{D} U^* \tilde{b} - \tilde{b} \tilde{D}) ⟨\tilde{D}⟩^{-α} \]
	is bounded. Applying Proposition \ref{proposition:bounded-transform-of-commutator-additive-perturbation}, the operator
	$ (F_{U \tilde{D} U^*} - F_{\tilde{D}}) \tilde{b} ⟨\tilde{D}⟩^β $
	is bounded for all \( β < 1 - α \). By the functional calculus,
	$ F_{U \tilde{D} U^*} = U F_{\tilde{D}} U^*. $
	Fixing an element \( c ∈ A \), let \( \tilde{c} \) denote the operator on \( \tilde{E} \) given by \( c ∈ \End^*(E) \) at every point of \( K \). Since \( ⟨D⟩^{-β} c ∈ \End^0(E) \),
	\[ ⟨\tilde{D}⟩^{-β} \tilde{c} ∈ C(K, \End^0(E)) = \End^0(\tilde{E}). \]
	Hence
	\[ (U F_{\tilde{D}} U^* - F_{\tilde{D}}) \tilde{b} \tilde{c} = (F_{U \tilde{D} U^*} - F_{\tilde{D}}) \tilde{b} ⟨\tilde{D}⟩^β ⟨\tilde{D}⟩^{-β} \tilde{c} \]
	is in \( \End^0(\tilde{E}) = \End^0(C(K, E)) \).

	Define the map \( f' : g \mapsto (F_D - U_g F_D U_g^*) b c \) from \( G \) into bounded operators on \( E \). By Lemma \ref{lemma:norm-continuous_on_compact_subsets_to_compact_endomorphism}, the norm-continuity of \( f' \) is equivalent to the condition that \( f'|_K \) be in \( \End^0(C(K, E)) \) for every compact subset \( K ⊆ G \). By the inclusion of \( A ⊆ \overline{\mathscr{Q} A} \), we are done.
\end{proof}

\subsection{Descent and the dual Green–Julg map for uniform equivariance}
\label{section:descent-group}

An important feature of equivariant KK-theory is Kasparov's descent map
\[ j^G_t : KK^G(A, B) \to KK(A ⋊_t G, B ⋊_t G) \]
for either topology \( t ∈ \{ u, r \} \), universal or reduced \cite[Theorem 3.11]{Kasparov_1988}. There can be other, exotic, topologies \( t \) for which there is a descent map \cite[§6]{Buss_2015} but we will not pursue this.

\begin{definition}
	\label{definition:hilbert-module-crossed-products}
	\cite[Definition 3.8]{Kasparov_1988}, \cite[Definition 20.6.1]{Blackadar_1998}
	Let \( E \) be a \( G \)-equivariant \( A \)-\( B \)-correspondence. The algebra \( C_c(G, B) \) acts on the right of \( C_c(G, E) \) by
	\[ (ξ f)(g) = \int_G ξ(h) β_h(f(h^{-1} g)) dμ(h) \qquad (ξ ∈ C_c(G, E), f ∈ C_c(G, B)) \]
	where \( β \) is the action of \( G \) on \( B \). We define a right \( C_c(G, B) \)-valued inner product on \( C_c(G, E) \) by
	\[ ⟨ξ|η⟩_{C_c(G, B)}(g) = \int_G β_{h^{-1}}(⟨ξ(h)|η(h g)⟩_B) dμ(h) \qquad (ξ, η ∈ C_c(G, E)) . \]
	The algebra \( C_c(G, A) \) acts on the left of \( C_c(G, E) \) by
	\[ (f ξ)(g) = \int_G f(h) U_h ξ(h^{-1} g) dμ(h) \qquad (f ∈ C_c(G, A), ξ ∈ C_c(G, E)) \]
	where \( U \) is the representation of \( G \) on \( E \). For \( t ∈ \{ u, r \} \), we denote by \( E ⋊_t G \) the \( A ⋊_t G \)-\( B ⋊_t G \)-correspondence obtained by completing \( C_c(G, E) \) in the \( C_c(G, B) \)-valued inner product. We may also realise \( E ⋊_t G \) as the internal tensor product \( E ⊗_B (B ⋊_t G) \), but the left action of \( A ⋊_t G \) is difficult to see in this picture.
\end{definition}

\begin{proposition}
	\cite[Theorem 3.11]{Kasparov_1988}
	Let \( (A, E_B, F) \) be a \( G \)-equivariant bounded Kasparov module. Then, for \( t ∈ \{ u, r \} \), \( (A ⋊_t G, (E ⋊_t G)_{B ⋊_t G}, \tilde{F}) \) is a bounded Kasparov module, where \( \tilde{F} \) is the operator given on \( ξ ∈ C_c(G, E) ⊆ E ⋊_t G \) by \( (\tilde{F} ξ)(g) = F(ξ(g)) \).
\end{proposition}

If \( G \) is a compact group and acts trivially on \( A \) there is the \emph{Green–Julg} isomorphism
\[ Φ_G : KK^G(A, B) \to KK(A, B ⋊ G) ; \]
see \cite[20.2.7(b)]{Blackadar_1998}. On the other hand, when \( G \) acts trivially on \( B \), there is the \emph{dual Green–Julg} map
\[ Ψ^G : KK^G(A, B) \to KK(A ⋊_u G, B) \]
which is an isomorphism when \( G \) is discrete \cite[20.2.7(b)]{Blackadar_1998}. The existence of $\Psi^G$ is proved in the next proposition, and then we present the isomorphism for discrete groups. The universal crossed product is needed because it is universal for covariant representations.

\begin{proposition}
	Let \( (A, E_B, F) \) be a \( G \)-equivariant bounded Kasparov module, with \( G \) acting trivially on \( B \). Then \( (A ⋊_u G, E_B, F) \) is a bounded Kasparov module, with the integrated representation of \( A ⋊_u G \). 
\end{proposition}
\begin{proof}
	With \( α \) the action of \( G \) on \( A \), \( π \) the representation of \( A \) on \( E \), and \( U \) the representation of \( G \) on \( E \), the pair \( (π, U) \) is a covariant representation of the C*-dynamical system \( (A, G, α) \). We obtain by \cite[§A.2]{Echterhoff_2006} the integrated representation \( π ⋊ U \) of \( A ⋊_u G \) on \( E \), and it is here that the universal crossed product is needed. We will consider the dense subalgebra \( C_c(G, A) ⊆ A ⋊_u G \). For an element \( f ∈ C_c(G, A) \),
	\[ (F^* - F) (π ⋊ U)(f) = \int_G (F^* - F) π(f(g)) U_g dμ(g). \]
	Because \( f \) is compactly supported and the integrand norm continuous, the integral converges. The integrand being valued in compact operators, the result is also compact. In the same way,
	\[ (F^2 - 1) (π ⋊ U)(f) = \int_G (F^2 - 1) π(f(g)) U_g dμ(g) \]
	and
	\[ [F, (π ⋊ U)(f)] = \int_G [F, π(f(g)) U_g] dμ(g) = \int_G \left( [F, π(f(g))] U_g + π(f(g)) (F - U_g F U_g^*) U_g \right) dμ(g) \]
	are compact. By the density of \( C_c(G, A) ⊆ A ⋊_u G \) we are done.
\end{proof}

\begin{proposition}
	Let \( (A ⋊_u G, E_B, F) \) be a bounded Kasparov module, with \( G \) a discrete group and \( A ⋊_u G \) represented nondegenerately on \( E \). Then \( (A, E_B, F) \) is a \( G \)-equivariant bounded Kasparov module, with the group action given by \( (U_g)_{g ∈ G} ⊆ C^*_u(G) ⊆ M(A ⋊_u G) \), acting trivially on \( B \).
\end{proposition}
\begin{proof}
	Because \( G \) is discrete, \( A \) is included in \( A ⋊_u G \). Hence,
	\[ (F^* - F) a \qquad (F^2 - 1) a \qquad [F, a] \]
	are compact for all \( a ∈ A \). Inside \( M(A ⋊_u G) \) are unitary elements \( (U_g)_{g ∈ G} \) representing \( G \), such that \( a U_g ∈ A ⋊_u G \) for all \( a ∈ A \) and \( g ∈ G \). Then
	\[ (F - U_g F U_g^*) a = [F, U_g] U_g^* a = [F, a] - [F, U_g^* a] = [F, a] + [F, a U_g]^* \]
	is compact, as required.
\end{proof}

Before we define the descent map for uniformly equivariant cycles, let us introduce some notation.

\begin{definition}
	\label{definition:crossed-product-dense-subalg}
	Let \( \mathscr{A} \) be a dense $*$-subalgebra of a C*-algebra \( A \). Let \( \alpha \) be an action of a locally compact group \( G \) on \( A \) which preserves \( \mathscr{A} \). If \( G \) is discrete, we write \( \mathscr{A} \rtimes G \) for the algebraic crossed product, which is dense in \( A \rtimes_t G \). For a non-discrete group, we will generalise this by defining \( \mathscr{A} \rtimes G \subseteq A \rtimes_t G \) as the (dense) \(*\)-subalgebra generated by \( \mathscr{A} \) and \( C_c(G) \) under the canonical inclusions \( \mathscr{A} \subseteq A \subseteq M(A \rtimes_t G ) \) and \( C_c(G) \subseteq C^*(G) \subseteq M(A \rtimes_t G ) \).
\end{definition}

\begin{proposition}
	\label{proposition:descent-ordinary-unbounded}
	Let \( (A, E_B, D) \) be a uniformly \( G \)-equivariant order-\( \frac{1}{1 - α} \) cycle. Then for either topology \( t ∈ \{ u, r \} \), \( (A ⋊_t G, (E ⋊_t G)_{B ⋊_t G}, \tilde{D}) \) is an order-\( \frac{1}{1 - α} \) cycle, where \( \tilde{D} \) is the regular operator given on \( ξ ∈ C_c(G, E) ⊆ E ⋊_t G \) by \( (\tilde{D} ξ)(g) = D(ξ(g)) \).

	If, for a dense $*$-subalgebra \( \mathscr{A} ⊆ A \), \( (\mathscr{A}, E_B, D) \) is a uniformly \( G \)-equivariant order-\( \frac{1}{1 - α} \) cycle, \( (\mathscr{A} \rtimes G, (E ⋊_t G)_{B ⋊_t G}, \tilde{D}) \) is an order-\( \frac{1}{1 - α} \) cycle.
	\end{proposition}
\begin{proof}
	We have, for \( f ∈ C_c(G, A) \) and \( ξ ∈ C_c(G, E) \)
	\[ ((1 + \tilde{D}^2)^{-1} f ξ)(g) = \int_G (1 + D^2)^{-1} f(h) U_h ξ(h^{-1} g) dμ(h). \]
	As \( f \) is compactly supported and the integrand continuous, the integral converges. Observe that \( (1 + \tilde{D}^2)^{-1} f \) is an element of \( C_c(G, \End^0(E)) \), given by \( g ↦ (1 + D^2)^{-1} f(g) \). By \cite[Proof of Theorem 3.11]{Kasparov_1988}, \( C_c(G, \End^0(E)) ⊆ \End^0(E ⋊_t G) \), so \( (1 + \tilde{D}^2)^{-1} f \) is compact.
		
	Next, note that the closure of \( \mathscr{Q} C_c(G) \) includes \( A ⋊_t G \). Let \( a ∈ \mathscr{Q} \), \( f \in C_c(G) \), and \( ξ ∈ \Span(C_c(G) \dom D) ⊆ C_c(G, \dom D) \). Then we find that
	\begin{align*} ([\tilde{D}, a f] ⟨\tilde{D}⟩^{-α} ξ)(g) &= \int_G [D, a U_h] ⟨\tilde{D}⟩^{-α} f(h) ξ(h^{-1} g) dμ(h)\\
	& = \int_G (D a - a U_h D U_h^*) U_h ⟨\tilde{D}⟩^{-α} U_h^* U_h f(h) ξ(h^{-1} g) dμ(h).
	\end{align*}
	As \( f \) is compactly supported and the integrand is continuous, the integral converges. Observe that the closure of \( [\tilde{D}, a f] ⟨\tilde{D}⟩^{-α} \) is an element of \( C_c(G, \End^*(E)_{*-s}) \) given by
	\[ g ↦ f(g) \overline{(D a - a U_g D U_g^*) U_g ⟨D⟩^{-α} U_g^*} . \]
	As \( C_c(G, \End^*(E)_{*-s}) ⊆ \End^*(E ⋊_t G) \) (see  \cite[Lemma 7]{Raeburn_1988}), \( [\tilde{D}, a f] ⟨\tilde{D}⟩^{-α} \) is bounded. Similarly, \( ⟨\tilde{D}⟩^{-α} [\tilde{D}, a f] \) is bounded.
Hence for \( b ∈ \mathscr{A} \rtimes G \), \( [\tilde{D}, b] ⟨\tilde{D}⟩^{-α} \) and \( ⟨\tilde{D}⟩^{-α} [\tilde{D}, b] \) are bounded, proving the second statement.
\end{proof}

For uniformly equivariant cycles, we have a  dual Green–Julg map for the universal crossed product.

\begin{proposition}
	\label{proposition:dual-green-julg-ordinary-unbounded}
	Let \( (A, E_B, D) \) be a uniformly \( G \)-equivariant order-\( \frac{1}{1 - α} \) cycle, with \( G \) acting trivially on \( B \). Then \( (A ⋊_u G, E_B, D) \) is an order-\( \frac{1}{1 - α} \) cycle, with the integrated representation of \( A ⋊_u G \).

	If, for a dense $*$-subalgebra \( \mathscr{A} ⊆ A \), \( (\mathscr{A}, E_B, D) \) is a uniformly \( G \)-equivariant order-\( \frac{1}{1 - α} \) cycle, with \( G \) acting trivially on \( B \), \( (\mathscr{A} \rtimes G, E_B, D) \) is an order-\( \frac{1}{1 - α} \) cycle.
\end{proposition}
\begin{proof}
	With \( α \) the action of \( G \) on \( A \), \( π \) the representation of \( A \) on \( E \), and \( U \) the representation of \( G \) on \( E \), the pair \( (π, U) \) is a covariant representation of the C*-dynamical system \( (A, G, α) \) and we obtain the integrated representation \( π ⋊ U \) of \( A ⋊_u G \) on \( E \). For an element \( f ∈ C_c(G, A) \),
	\[ (1 + D^2)^{-1} (π ⋊ U)(f) = \int_G (1 + D^2)^{-1} π(f(g)) U_g dμ(g) . \]
	As \( f \) is compactly supported and the integrand norm-continuous, the integral converges, and as the integrand is valued in compact operators, the integral is also compact.
	As in the proof of Proposition \ref{proposition:descent-ordinary-unbounded}, we note that the closure of \( \mathscr{Q} C_c(G) \) includes \( A ⋊_u G \). Let \( a ∈ \mathscr{Q} \), \( f \in C_c(G) \), and \( ξ ∈ \dom D \); then
	\begin{align*}
		[D, (π ⋊ U)(a f)] ⟨D⟩^{-α} ξ
		& = \int_G f(g) [D, π(a) U_g] ⟨D⟩^{-α} ξ dμ(g) \\
		& = \int_G f(g) (D π(a) - π(a) U_g D U_g^*) U_g ⟨D⟩^{-α} ξ dμ(g).
	 \end{align*}
	As \( f \) is compactly supported and the integrand is continuous, the integral converges. By Corollary \ref{corollary:strongly_continuous_to_norm_bounded}, \( [D, (π ⋊ U)(a f)] ⟨D⟩^{-α} \) extends to an adjointable operator, as does \( ⟨D⟩^{-α} [D, (π ⋊ U)(a f)] \).
\end{proof}

In order to display the inverse of the dual Green-Julg map for discrete groups,
 a dense subalgebra \( \mathscr{A} \) of \( A \) is required.

\begin{proposition}
	Let \( (\mathscr{A} ⋊ G, E_B, D) \) be an order-\( \frac{1}{1 - α} \) cycle, with \( G \) a discrete group and the representation of \( \mathscr{A} ⋊ G \) on \( E \) nondegenerate. Then \( (\mathscr{A}, E_B, D) \) is a uniformly \( G \)-equivariant order-\( \frac{1}{1 - α} \) cycle, with group action given by \( (U_g)_{g ∈ G} ⊆ C^*_u(G) ⊆ M(A ⋊_u G) \), acting trivially on \( B \).
\end{proposition}
\begin{proof}
	Because \( G \) is discrete, \( \mathscr{A} \) is included in \( \mathscr{A} ⋊ G \). Hence, \( (1 + D^2)^{-1} a \) is compact and \( [D, a] \) is bounded for all \( a ∈ \mathscr{A} \). Inside \( M(A ⋊_u G) \) are unitary elements \( (U_g)_{g ∈ G} \) representing \( G \), such that \( a U_g ∈ \mathscr{A} ⋊ G \) for all \( a ∈ \mathscr{A} \) and \( g ∈ G \). Then
	\[ U_g D U_g^* a - a D = U_g [D, U_g^* a] \]
	so that \( (U_g D U_g^* a - a D) ⟨D⟩^{-α} \) and \( ⟨D⟩^{-α} U_g^* (U_g D U_g^* a - a D) \) are bounded, as required.
\end{proof}

\begin{remark}
	It is immediate that the bounded transform \( (A ⋊_t G, (E ⋊_t G)_{B ⋊_t G}, F_{\tilde{D}} = \tilde{F_D}) \) of the descent \( (A ⋊_t G, (E ⋊_t G)_{B ⋊_t G}, \tilde{D}) \) of a uniformly \( G \)-equivariant cycle \( (A, E_B, D) \) is exactly the descent of the bounded transform \( (A, E_B, F_D) \). The same is true for the dual Green–Julg map.	 
\end{remark}

\subsection{Conformally equivariant unbounded KK-theory}
\label{subsec:con-equi-KK}

It is not clear that Definition \ref{definition:ordinary-unbounded-equivariance} is the correct generalisation of equivariance to unbounded KK-theory. Definition \ref{definition:ordinary-unbounded-equivariance} is natural in the sense that the exterior product and descent map are well-defined and Kucerovsky's conditions \cite[Theorem 13]{Kucerovsky_1997} for the Kasparov product still suffice \cite[Theorem 8.12]{Kucerovsky_1994}. On the other hand, let us examine `patient zero' of noncommutative geometry: a complete Riemannian spin\(^c\) manifold \( (X, \mathbf{g}) \) with spinor bundle \( S \) and Dirac operator \( \slashed{D} \) , forming the spectral triple
\( \left( C(X), L^2(X, S), \slashed{D} \right) \).
The largest group for which this is uniformly equivariant, in the sense of Definition \ref{definition:ordinary-unbounded-equivariance}, is the isometry group \( \Iso(X, \mathbf{g}) \). What is the largest group for which the Fredholm module 
\[ \left( C(X), L^2(X, S), F_{\slashed{D}} \right) \]
given by the bounded transform is equivariant, and can a geometric interpretation be put upon it? The answer to this question is that the Fredholm module above is equivariant under the conformal group \( \Conf(X, \mathbf{g}) \) of \( X \). That this is maximal is confirmed by \cite[Theorem 3.1]{Bar_2007}.

\begin{example}
The simplest example exhibiting this discrepancy is the real line and its Dirac spectral triple \( (C_0(\bbR), L^2(\bbR), i ∂_x) \). We will compare two group actions on \( \bbR \): translations by \( \bbR \) and dilation by \( \bbR^×_+ \), i.e. addition and multiplication, respectively. The affine group \( \bbR ⋊ \bbR^×_+ \) acts on \( \bbR \) by \( φ_{(a, b)} : x ↦ a x + b \), for \( (a, b) ∈ \bbR ⋊ \bbR^×_+ \). Let \( V_{(a, b)} \) be the pullback by \( φ_{(a, b)}^{-1} = φ_{(a^{-1}, - a^{-1} b)} \) on \( L^2(\bbR) \). For \( ξ, η ∈ L^2(\bbR) \), we have
\[ \int_0^∞ \overline{(V_{(a, b)} ξ)(x)} η(x) dx = \int_0^∞ \overline{ξ(a^{-1}(x - b))} η(x) dx = \int_0^∞ \overline{ξ(y)} η(a y + b) a dy \]
so \( V_{(a, b)}^* = a V_{(a, b)}^{-1} = a V_{(a^{-1}, -a^{-1} b)} \). The unitary part of the polar decomposition of \( V_{(a, b)} \) is, therefore, \( U_{(a, b)} = a^{-1/2} V_{(a, b)} \). By the chain rule, for \( ξ ∈ C_c^∞(\bbR) \),
\[ (U_{(a, b)} ∂_x U_{(a, b)}^* ξ)(x) = a^{-1/2} (∂_x U_{(a, b)}^* ξ)(a^{-1} (x - b)) = a^{-3/2} (U_{(a, b)}^* ξ)'(a^{-1} (x - b)) = a^{-1} ξ'(x) \]
so that \( U_{(a, b)} i ∂_x U_{(a, b)}^* = a^{-1} i ∂_x \). For the subgroup \( \bbR \) (\( a = 1 \)), the spectral triple
\( (C_0(\bbR), L^2(\bbR), i ∂_x) \)
is isometrically equivariant in the sense of Definition \ref{definition:ordinary-unbounded-equivariance}. On the other hand, when \( a ≠ 1 \), for \( f ∈ C_c^∞(\bbR) \),
\[ U_{(a, b)} i ∂_x U_{(a, b)}^* f - f i ∂_x = (a^{-1} - 1) i ∂_x f + [i ∂_x, f] \]
is as unbounded as \( i ∂_x \), so condition 4 of Definition \ref{definition:ordinary-unbounded-equivariance} is not satisfied. On the other hand,
\[ (U_{(a, b)} F_{i ∂_x} U_{(a, b)}^* - F_{i ∂_x}) f = (F_{a^{-1} i ∂_x} - F_{i ∂_x}) f = i ∂_x \left( (a^2 + (i ∂_x)^2)^{-1/2} - (1 + (i ∂_x)^2)^{-1/2} \right) f \]
is compact, as \( y ↦ y \left( (a^2 + y^2)^{-1/2} - (1 + y^2)^{-1/2} \right) \) is in \( C_0(\bbR) \). Hence
\( (C_0(\bbR), L^2(\bbR), F_{i ∂_x}) \)
is equivariant for all of \( \bbR ⋊ \bbR^×_+ \). In this section, we will make a definition of equivariance in unbounded KK-theory which can cope with this and similar examples. (We remark that multiplication by \( -1 \), although an isometry, is not orientation-preserving and has the effect of multiplying by \( -1 \) in \( KK_1(C_0(\bbR), \bbC) \), rather than preserving the class.)
\end{example}

\begin{definition}
	\label{definition:conformal-equivariance}
	An order-\( \frac{1}{1 - α} \) \( A \)-\( B \)-cycle \( (A, E_B, D) \) is \emph{conformally equivariant} if \( E \) is a \( G \)-equivariant \( A \)-\( B \)-correspondence and there exists a $*$-strongly continuous family \( (μ_g)_{g ∈ G} ⊆ \End^*(E) \) of (even) invertible operators satisfying the following. We require that \( A ⊆ \overline{\Span}(A \mathscr{Q}) ∩ \overline{\Span}(\mathscr{Q} A) \), where \( \mathscr{Q} \) is the set of \( a ∈ \Lip_α^*(E) \) such that for all $g\in G$ we have \( \{ a μ_g, a μ_g^{-1 *} \} \dom D ⊆ \dom D ∩ U_g \dom D \), and the maps
	\begin{align*}
		g & ↦ (U_g D U_g^* a - a μ_g D μ_g^*) μ_g^{-1*} ⟨D⟩^{-α} &
		g & ↦ [D, a μ_g] ⟨D⟩^{-α} &
		g & ↦ [D, a μ_g^{-1 *}] ⟨D⟩^{-α} \\
		g & ↦ U_g ⟨D⟩^{-α} U_g^* (U_g D U_g^* a - a μ_g D μ_g^*) &
		g & ↦ ⟨D⟩^{-α} [D, a μ_g] &
		g & ↦ ⟨D⟩^{-α} [D, a μ_g^{-1*}]
	\end{align*}
	are $*$-strongly continuous from \( G \) into bounded operators (but need not be globally bounded). We call $ \mu = (μ_g)_{g ∈ G} $ the conformal factor. 
\end{definition}

\begin{remarks}
	\item When \( μ_g = 1 \) for all \( g ∈ G \), this Definition reduces to Definition \ref{definition:ordinary-unbounded-equivariance} of uniformly equivariant \( G \)-cycles. 
	\item Also, if \( μ_e = 1 \), for elements \( a ∈ \End^*(E) \) satisfying that
		\[ [D, a μ_g] ⟨D⟩^{-α} \]
		is bounded, \( a \) is automatically in \( \Lip_α^*(D) \).
	\item Note also that it is sufficient that \( 1 ∈ \mathscr{Q} \) for the closure conditions to be satisfied; in the nonunital case, an approximate unit might be used.
\end{remarks}

\begin{theorem}
	\label{theorem:conformal-equivariance-bdd-transform}
	Let \( (A, E_B, D) \) be a conformally \( G \)-equivariant order-\( \frac{1}{1-α} \) cycle. Then \( (A, E_B, F_D) \) is a \( G \)-equivariant bounded Kasparov module.
\end{theorem}
\begin{proof}
	The only difference from the non-equivariant case is the need to show that, for every \( a ∈ A \), \( g \mapsto (F_D - U_g F_D U_g^*) a \) is norm-continuous as a map from \( G \) into \( \End^0(E) \).

	By definition, for every \( a ∈ \mathscr{Q} \), the maps \( f_0 : g ↦ μ_g^{-1} \) and
		\[ f_{1, a} : g ↦ (U_g D U_g^* a - a μ_g D μ_g^*) μ_g^{-1*} ⟨D⟩^{-α} \qquad f_{2, a} : g ↦ ⟨D⟩^{-α} U_g^* (U_g D U_g^* a - a μ_g D μ_g^*) \]
	\[ f_{3, a} : g ↦ [D, a μ_g] ⟨D⟩^{-α} \qquad f_{4, a} : g ↦ ⟨D⟩^{-α} [D, a μ_g] \]
	\[ f_{5, a} : g ↦ [D, a μ_g^{-1 *}] ⟨D⟩^{-α} \qquad f_{6, a} : g ↦ ⟨D⟩^{-α} [D, a μ_g^{-1 *}] \]
	are $*$-strongly continuous as a map from \( G \) into \( \End^*(E) \). By Lemma \ref{lemma:strongly-continuous_on_compact_subsets_to_bounded_endomorphism}, this is equivalent to \( f_{i, a}|_K \) residing in \( \End^*(C(K, E)) \) for every compact subset \( K ⊆ G \).

	Fix a compact subset \( K ⊆ G \) and let \( \tilde{E} = C(K, E) \). Define \( \tilde{D} \) to be the self-adjoint regular operator on \( \tilde{E} \) given by \( D \) at each point of \( K \). Let \( U \) denote the \( \bbC \)-linear map from \( \tilde{E} \) to itself given by \( g \mapsto U_g \). Let \( \tilde{μ} ∈ \End^*(\tilde{E}) \) be given by \( g \mapsto μ_g \). For every \( a ∈ \End^*(E) \), let \( \tilde{a} \) be given by \( a \) at each point of \( G \). Then, for every \( a ∈ \mathscr{Q} \),
	\begin{gather*}
		(U \tilde{D} U^* \tilde{a} - \tilde{a} \tilde{μ} \tilde{D} \tilde{μ}^*) \tilde{μ}^{-1*} ⟨\tilde{D}⟩^{-α} \qquad ⟨\tilde{D}⟩^{-α} U^* (U \tilde{D} U_g^* \tilde{a} - \tilde{a} \tilde{μ} \tilde{D} \tilde{μ}^*) \\
		[\tilde{D}, \tilde{a} \tilde{μ}] ⟨\tilde{D}⟩^{-α} \qquad ⟨\tilde{D}⟩^{-α} [\tilde{D}, \tilde{a} \tilde{μ}] \qquad [\tilde{D}, \tilde{a} \tilde{μ}^{-1 *}] ⟨\tilde{D}⟩^{-α} \qquad ⟨\tilde{D}⟩^{-α} [\tilde{D}, \tilde{a} \tilde{μ}^{-1 *}] \qquad [\tilde{D}, \tilde{a}]
	\end{gather*}
	are adjointable endomorphisms of \( \tilde{E} \).
	Let \( a, b, c, d ∈ \mathscr{Q} \). As in the Proof of Theorem \ref{theorem:bdd-transform-conformal-transformation},
	\[ [\tilde{μ} \tilde{D} \tilde{μ}^*, \tilde{b}^* \tilde{c}] \tilde{μ}^{-1*} ⟨\tilde{D}⟩^{-α} \]
	is bounded. We apply Theorem \ref{theorem:conformal-result-nonunital} to obtain that
	$ (F_{\tilde{μ} \tilde{D} \tilde{μ}^*} - F_{\tilde{D}}) \tilde{b}^* \tilde{c} \tilde{d}^* ⟨\tilde{D}⟩^β $
	is bounded for \( β < 1-α \). Furthermore, as
	\[ (U\tilde{D} U^* \tilde{a} - \tilde{a} \tilde{μ} \tilde{D} \tilde{μ}^*) \tilde{μ}^{-1*} ⟨\tilde{D}⟩^{-α} \]
	is bounded, Proposition \ref{proposition:bounded-transform-of-commutator-additive-perturbation}, shows that
	\[ (UF_{\tilde{D}} U^* \tilde{a} -  \tilde{a} F_{\tilde{μ} \tilde{D} \tilde{μ}^*}) \tilde{μ} ⟨\tilde{D}⟩^β \]
	is too. Taking care because \( U \) is only \( \bbC \)-linear, we have
	\begin{align*}
		(U F_{\tilde{D}} U^* - F_{\tilde{D}}) \tilde{a} \tilde{b}^* \tilde{c} \tilde{d}^*
		& = U [F_{\tilde{D}}, U^*] \tilde{a} \tilde{b}^* \tilde{c} \tilde{d}^* 
		 = U [F_{\tilde{D}}, U^* \tilde{a} \tilde{b}^* \tilde{c}] \tilde{d}^* - [F_{\tilde{D}}, \tilde{a} \tilde{b}^* \tilde{c}] \tilde{d}^* \\
		& = U (F_{\tilde{D}} U^* \tilde{a} - U^* \tilde{a} F_{\tilde{μ} \tilde{D} \tilde{μ}^*}) \tilde{b}^* \tilde{c} \tilde{d}^* + \tilde{a} (F_{\tilde{μ} \tilde{D} \tilde{μ}^*} \tilde{b}^* \tilde{c} \tilde{d}^* - \tilde{b}^* \tilde{c} F_{\tilde{D}}) - [F_{\tilde{D}}, \tilde{a} \tilde{b}^* \tilde{c}] \tilde{d}^* \\
		& = U (F_{\tilde{D}} U^* \tilde{a} - U^* \tilde{a} F_{\tilde{μ} \tilde{D} \tilde{μ}^*}) \tilde{b}^* \tilde{c} \tilde{d}^* + \tilde{a} (F_{\tilde{μ} \tilde{D} \tilde{μ}^*} - F_{\tilde{D}}) \tilde{b}^* \tilde{c} \tilde{d}^* - [F_{\tilde{D}}, \tilde{a}] \tilde{b}^* \tilde{c} \tilde{d}^*
	\end{align*}
	so that \( (U F_{\tilde{D}} U^* - F_{\tilde{D}}) \tilde{a} \tilde{b}^* \tilde{c} \tilde{d}^* ⟨\tilde{D}⟩^β \) is bounded. Letting \( e ∈ A \) we have
	\begin{equation} 
	(U F_{\tilde{D}} U^* - F_{\tilde{D}}) \tilde{a} \tilde{b}^* \tilde{c} \tilde{d}^* \tilde{e} 
	\label{eq:3.19}
	\end{equation}
	is in \( \End^0(\tilde{E}) = \End^0(C(K, E)) \).

	Define the map \( f' : g \mapsto (F_D - U_g F_D U_g^*) a b^* c d^* e \) from \( G \) into bounded operators on \( E \). By Lemma \ref{lemma:norm-continuous_on_compact_subsets_to_compact_endomorphism}, the norm-continuity of \( f' \) is equivalent to the condition that \( f'|_K \) be in \( \End^0(C(K, E)) \) for every compact subset \( K ⊆ G \). By the inclusion of \( A ⊆ \overline{\mathscr{Q} \mathscr{Q}^* \mathscr{Q} \mathscr{Q}^* A} \), we are done.
\end{proof}

\begin{example}
	Let \( (X, \mathbf{g}) \) be a complete Riemannian spin\(^c\) manifold with spinor bundle \( S \) and Dirac operator \( \slashed{D} \). Let \( G \) be a locally compact group with a conformal action \( φ \) on \( X \), so that \( φ_g^*(\mathbf{g}) = k_g^2 \mathbf{g} \) for \( g ∈ G \). If the conformal factors \( (k_g)_{g ∈ G} \) are each bounded and invertible (for instance, if \( X \) is compact), then
	\( (C_0(X), L^2(X, \slashed{S}_{\mathbf{g}}), \slashed{D}) \)
	is a conformally \( G \)-equivariant spectral triple with conformal factors \( (k_{g^{-1}}^{-1/2})_{g ∈ G} \).
\end{example}

\begin{example}
	\label{eg:HdR-equi}
	Let \( (X, \mathbf{g}) \) be a complete oriented Riemannian manifold with Hodge–de Rham operator \( d + δ \). Let \( G \) be a locally compact group with a conformal action \( φ \) on \( X \), so that \( φ_g^*(\mathbf{g}) = k_g^2 \mathbf{g} \) for \( g ∈ G \). If the conformal factors \( (k_g)_{g ∈ G} \) are each bounded and invertible (for instance, if \( X \) is compact), then
	\( (C_0(X), L^2(Ω^* X), d + δ) \)
	is a conformally \( G \)-equivariant spectral triple with conformal factors \( (k_{g^{-1}}^{-1/2})_{g ∈ G} \).
\end{example}

\begin{example}
	\label{example:circle-bundle}
	Let \( P \) be a principal circle bundle over a compact Hausdorff space \( X \). Let \( Φ : C(P) \to C(X) \) be the conditional expectation given by averaging over the circle action. By \cite[Proposition 2.9]{Carey_2011},
	\begin{equation}
		\label{eq:circle-bundle}
		(C(P), L^2(P, Φ)_{C(X)}, N = -i ∂_θ)
	\end{equation}
	is an unbounded Kasparov module, where \( N \) is the number operator on the spectral subspaces, equivalent to the vertical Dirac operator \( -i ∂_θ \) acting on each fibre. Let \( G \) be a group acting on \( P \) and \( X \), compatibly with the surjection \( P \to X \). Suppose that \( φ \) acts differentiably between the fibres. Since the circle is one-dimensional, \( φ_g^*(dθ^2) = k_g^2 dθ^2 \) for a family of functions \( (k_g)_{g ∈ G} ∈ C(P) \). We obtain that \eqref{eq:circle-bundle} is conformally \( G \)-equivariant with conformal factors \( (k_{g^{-1}}^{-1/2})_{g ∈ G} \).
\end{example}

One limitation of conformal equivariance is that the exterior product becomes ill-defined. This is exemplified by the fact that the conformal group of the Cartesian product of Riemannian manifolds is generically smaller than the product of the conformal groups. Example \ref{example:circle-bundle} also demonstrates that the internal Kasparov product is generally not constructive for conformally equivariant cycles. However, at the bounded level of KK-theory, the exterior product is known to exist by Kasparov's technical theorem. Recall the logarithmic transform of §\ref{section:log}, which will provide a way of turning conformal equivariance into uniform equivariance, making the exterior product constructive, at the expense of much of the geometric information encoded by the Dirac operator.

\begin{theorem}
\label{thm:conf-log}
	Let \( (A, E_B, D) \) be a conformally \( G \)-equivariant order-\( \frac{1}{1-α} \) cycle with conformal factor \( μ \). Then \( (A, E_B, L_D) \) is a uniformly \( G \)-equivariant unbounded Kasparov module.
\end{theorem}
\begin{proof}
	The only difference from the non-equivariant case is the need to show that \( A \) is contained in the closure of the set of \( a ∈ \End^*(E) \) such that \( [L_D, a] \) extends to an adjointable operator and \( g \mapsto (L_D - U_g L_D U_g^*) a \) is $*$-strongly continuous as a map from \( G \) into \( \End^*(E) \).

	Fix a compact subset \( K ⊆ G \) and let \( \tilde{E} = C(K, E) \). As in the Proof of Theorem \ref{theorem:conformal-equivariance-bdd-transform}, define \( \tilde{D} \) to be the self-adjoint regular operator on \( \tilde{E} \) given by \( D \) at each point of \( K \). Let \( U \) denote the \( \bbC \)-linear map from \( \tilde{E} \) to itself given by \( g \mapsto U_g \). Let \( \tilde{μ} ∈ \End^*(\tilde{E}) \) be given by \( g \mapsto μ_g \). For every \( a ∈ \End^*(E) \), let \( \tilde{a} \) be given by \( a \) at each point of \( G \). Let \( a, b, c ∈ \mathscr{Q} \); then as in \eqref{eq:3.19}
	\[ (U F_{\tilde{D}} U^* - F_{\tilde{D}}) \tilde{a} \tilde{b}^* \tilde{c} \tilde{μ} ⟨\tilde{D}⟩^β \]
	is bounded for \( β < 1-α \). Hence,
	\begin{align*}
		(U L_{\tilde{D}} U^* - L_{\tilde{D}}) \tilde{a} \tilde{b}^* \tilde{c} \tilde{μ}
		& = U L_{\tilde{D}} U^* \tilde{a} \tilde{b}^* \tilde{c} \tilde{μ} - \tilde{a}^* b \tilde{μ} L_{\tilde{D}} - [L_{\tilde{D}}, \tilde{a} \tilde{b}^* \tilde{c} \tilde{μ}] \\
		& = U F_{\tilde{D}} U^* (U \log ⟨\tilde{D}⟩ U^* \tilde{a} \tilde{b}^* \tilde{c} \tilde{μ} - \tilde{a} \tilde{b}^* \tilde{c} \tilde{μ} \log ⟨\tilde{D}⟩) \\
		& \qquad + (U F_{\tilde{D}} U^* - F_{\tilde{D}}) \tilde{a} \tilde{b}^* \tilde{c} \tilde{μ} \log ⟨\tilde{D}⟩ - F_D [\log ⟨\tilde{D}⟩, \tilde{a} \tilde{b}^* \tilde{c} \tilde{μ}]
	\end{align*}
	is bounded. By the invertibility of \( \tilde{μ} \), \( (U L_{\tilde{D}} U^* - L_{\tilde{D}}) \tilde{a} \tilde{b}^* \tilde{c} ∈ \End^*(C(K, E)) \).
	
	Let \( d ∈ A \) and define the map \( f' : g \mapsto (L_D - U_g L_D U_g^*) a b^* c d^* \) from \( G \) into bounded operators on \( E \). By Lemma \ref{lemma:strongly-continuous_on_compact_subsets_to_bounded_endomorphism}, the $*$-strong-continuity of \( f' \) is equivalent to the condition that \( f'|_K \) be in \( \End^*(C(K, E)) \) for every compact subset \( K ⊆ G \), which it is. By the inclusion of \( A ∈ \overline{\Span}(\mathscr{Q} \mathscr{Q}^* \mathscr{Q} A) \), we are done.
\end{proof}

\subsection{The γ-element for the real and complex Lorentz groups}
\label{section:lorentz-gamma}

In this section, we lift to unbounded KK-theory the γ-elements constructed for \( SO(2n+1, 1) \), \( SO(2n, 1) \), and \( SU(n, 1) \) by Kasparov \cite{Kasparov_1984}, Chen \cite{Chen_1996}, and Julg and Kasparov \cite{Julg_1995}, respectively. We have opted to present them with notation close to the original sources, in the interests of space. For a unified treatment, see \cite[§5.3]{Aparicio_2019}.

In each case, the Bernstein–Gelfand–Gelfand (BGG) complex \cite{Cap_2009} for a sphere, considered as a symmetric space, is cleft in twain. For the real Lorentz groups, the BGG complex is the de Rham complex and, for the complex Lorentz groups, it is the Rumin complex \cite{Rumin_1994}. In the case of \( SO(2n+1, 1) \), the symmetric space is \( \bfS^{2n} \). The sphere being even dimensional, the middle-degree forms are split into the two eigenspaces of the Hodge star operator, which division is conformally invariant and, indeed, appears in the BGG complex. In the cases of \( SO(2n, 1) \) and \( SU(n, 1) \), the symmetric space is \( \bfS^{2n-1} \). The sphere being odd-dimensional necessitates the addition of the \( L^2 \) harmonic forms on a real or complex hyperbolic space to be added to the half-complex, along with an operator related to the Poisson transform. The sphere \( \bfS^{2n-1} \) is considered as the boundary of \( \bbR H^{2n} = SO(2n, 1)/S(O(n) × O(1)) \) or \( \bbC H^n = SU(n, 1)/S(U(n) × U(1)) \).

It is possible that the framework of conformally equivariant unbounded KK-theory could be used to treat the other rank-one groups, \( Sp(n, 1) \) and the real form \( F_{4(-20)} \), lifting the construction in \cite{Julg_2019}; however, there, the resulting complex contains differential operators of different orders.
In rank two, there is a construction by Yuncken \cite{Yuncken_2011} of the γ-element in bounded KK-theory of \( SL(3, \bbC) \), using the BGG complex of the flag manifold. A similar construction is proposed for the other rank-two complex semisimple groups \cite{Yuncken_2018}.
The BGG complex, in full generality, has been put on a sound analytical footing in \cite{Dave_2022} and subsequently fitted into bounded KK-theory in \cite{Goffeng_2024}, although with limitations on equivariance.
The lifting of these constructions to the unbounded picture remains a difficult task, likely to require a substantial renovation of the axioms of an unbounded Kasparov module, beyond what is done here.
A step in this direction is the treatment of `mixed-order' situations within noncommutative geometry which appears in \cite{Fries_2025a}, using the new concept of the \emph{tangled spectral triple}.

\subsubsection{The case of \( SO(2n+1, 1) \)}

Following \cite[§4]{Kasparov_1984}, we begin with the sphere \( \bfS^{2n} \) on which \( SO(2n+1, 1) \) acts conformally and its Hodge–de Rham Dirac operator. As we have seen, we can build a conformally \( SO(2n+1, 1) \)-equivariant spectral triple
\[ (C(\bfS^{2n}), L^2(Ω^* \bfS^{2n}), d + δ) . \]
In order to obtain the KK-class of the $\gamma$-element, we split the complexified exterior algebra into two subspaces, each preserved by the Dirac operator. On a \( 2n \)-dimensional manifold, the codifferential is equal to \( δ = d^* = -\hodge d \hodge \) and the Hodge star satisfies that
\[ \hodge^2 : α ↦ (-1)^{|α|} α \qquad \hodge^* : α ↦ (-1)^{|α|} \hodge α \]
for homogeneous $\alpha\in Ω^* \bfS^{2n}$.
The Hodge star and Hodge-de Rham operator are related by
\[ (d + δ) \hodge α = (d \hodge - (-1)^{|α|} \hodge d) α = \hodge ((-1)^{|α|+1} \hodge d \hodge - (-1)^{|α|} d) α = (-1)^{|α|+1} \hodge (d - δ) α . \]
Define the map \( ϵ : α ↦ i^{|α| (|α|+1) - n} α = (-1)^{|α| (|α|+1)/2} i^{-n} α \), so that
\[ (\hodge ϵ)^2 α = i^{|α| (|α|+1) - n} \hodge ϵ \hodge α = i^{|α| (|α|+1) - n} i^{(2n-|α|) ((2n-|α|)+1) - n} (-1)^{|α|} α = α \]
and
\begin{align*}
	(\hodge ϵ)^* α
	& = (-1)^{(2n-|α|) (2n-|α|+1)/2} i^n (-1)^{|α|} \hodge α 
	= (-1)^{|α| (|α|+1)/2} i^{-n} \hodge α 
	= \hodge ϵ α ,
\end{align*}
meaning that \( \hodge ϵ \) is a self-adjoint unitary. We have
\[ (d + δ) \hodge ϵ α = i^{|α| (|α|+1) - n} (d + δ) \hodge α = i^{2|α|+2 + |α| (|α|+1) - n} \hodge (d - δ) α \]
and
\begin{align*}
	ϵ dα
	& = i^{2|α|+2 + |α| (|α|+1) - n} dα &
	ϵ δ α
	& = -i^{2|α|+2 + |α| (|α|+1) - n} δ α .
\end{align*}
Hence \( \hodge ϵ \) commutes with \( d + δ \) and we can decompose the exterior algebra into
\[ Ω^* \bfS^{2n} = Ω_1^* ⊕ Ω_2^* := \im\left( \frac{1}{2} (1 + \hodge ϵ) \right) ⊕ \im\left( \frac{1}{2} (1 - \hodge ϵ) \right). \]
We thus have a spectral triple
\[ (C(\bfS^{2n}), L^2(Ω_1^*), d + δ) \]
which is still conformally \( SO(2n+1, 1) \)-equivariant and isometrically \( SO(2n+1) \)-equivariant. By forgetting the action of the algebra, we obtain a representative \( (\bbC, L^2(Ω_1^*), d + δ) \)
of a class \( γ ∈ KK^{SO(2n+1, 1)}(\bbC, \bbC) \).
The only harmonic forms on \( \bfS^{2n} \) are scalar multiples of \( 1 ∈ Ω^0 \bfS^{2n} \) and the volume form \( \vol ∈ Ω^{2n} \bfS^{2n} \). One can check that
\[ \hodge ϵ 1 = i^{-n} \vol \qquad \hodge \epsilon \vol = i^n 1 \qquad \frac{1}{2}(1 + \hodge ϵ) (1 + i^{-n} \vol) = 1 + i^{-n} \vol . \]
Hence the only harmonic forms in \( Ω_1^* \) are scalar multiples of \( (1 + i^{-n} \vol) \). The form \( (1 + i^{-n} \vol) \) being \( SO(2n+1) \)-invariant, the restriction \( r^{SO(2n+1, 1), SO(2n+1)}(γ) \) represents \( 1 ∈ KK^{SO(2n+1)}(\bbC, \bbC) \). By \cite[Proposition 5.9]{Aparicio_2019}, because \( γ \) is the image of an element of \( KK^{SO(2n+1, 1)}(C(\bfS^{2n}), \bbC) \) and restricts to \( 1 ∈ KK^{SO(2n+1)}(\bbC, \bbC) \), \( γ \) is really the γ-element of \( SO(2n+1, 1) \).

\subsubsection{The case of \( SO(2n, 1) \)}

Following \cite[\S3.1]{Chen_1996}, we begin with the sphere \( \bfS^{2n-1} \), on which \( SO(2n, 1) \) acts conformally, and its Hodge–de Rham operator. As in the even-dimensional case, we can build a conformally \( SO(2n, 1) \)-equivariant spectral triple
\[ (C(\bfS^{2n-1}), L^2(Ω^* \bfS^{2n-1}), d + δ) . \]
To obtain the correct class in \( KK^{SO(2n, 1)}_0(\bbC, \bbC) \) for the γ-element, we will cut the differential forms in two, as we did for \( SO(2n+1, 1) \), and add an additional operator.

Let \( D^{2n} \) be the open unit ball with Euclidean metric. The Poincaré disc model is a conformal identification of the hyperbolic space \( \bbR H^{2n} \) with \(  D^{2n} \).
As we saw in Example \ref{example:hodge-de-rham-conformal-transformation} (in particular \eqref{eq:middle}) the pullback map \( L^2(Ω^n \bbR H^{2n}) \to L^2(Ω^n D^{2n}) \) is automatically unitary because the forms are of middle degree.
Let \( I : \dom(I) ⊂ L^2(Ω^n \bbR H^{2n}) \to L^2(Ω^n \bfS^{2n-1}) \) be the restriction to the boundary \( \bfS^{2n-1} \) of the ball.
Let \( \mathscr{H} ⊆ L^2(Ω^n \bbR H^{2n}) \) be the \( L^2 \) harmonic forms on the real hyperbolic \( 2n \)-space and let \( \mathscr{H}_∞ ⊂ \mathscr{H} \) be those forms in the domain of \( I \).
We have a complex
\[ \begin{tikzcd}[column sep=2.2em]
	0 \arrow[r] & \mathscr{H}_∞ \arrow[r, "I"] & Ω^n \bfS^{2n-1} \arrow[r, "d"] & Ω^{n+1} \bfS^{2n-1} \arrow[r, "d"] & \cdots \arrow[r, "d"] & Ω^{2n-1} \bfS^{2n-1} \arrow[r] & 0
\end{tikzcd}  \]
which is invariant under the pullback by the action \( φ \) of \( SO(2n, 1) \). When we complete the spaces of the complex to Hilbert spaces, pullback by the action of \( SO(2n, 1) \) is not unitary. On \( L^2(Ω^n \bfS^{2n-1}) \) the unitaries \( (U_g)_{g ∈ G} \) implementing the group action \( φ \) act by
\[ U_g : ξ ↦ k_{g^{-1}}^{-(-(2n-1) + 2n)/2} φ_{g^{-1}}^*(ξ) = k_{g^{-1}}^{-1/2} φ_{g^{-1}}^*(ξ). \]
As in Example \ref{example:hodge-de-rham-conformal-transformation},
\[ U_g d U_g^* - k_{g^{-1}}^{-1/2} d k_{g^{-1}}^{-1/2} \]
is bounded. However, on the hyperbolic space \( \bbR H^{2n} \), the group \( SO(2n,1) \) acts by isometries. Because the map \( I \) commutes with pullback by the group action, \( U_g I U_g^* = k_{g^{-1}}^{-1/2} I \),
which is not the same behaviour as the rest of the complex displays, the overall exponent of the conformal factor being \( -1/2 \) rather than \( -1 \). On all of \( L^2(Ω^* \bfS^{2n-1}) \) the Laplacian \( Δ = d δ + δ d \) transforms so that
\[ U_g Δ^{1/4} U_g^* - k_{g^{-1}}^{-1/2} Δ^{1/4} \]
is of order \( -1/2 \). We will replace the operator \( I \) in the complex with \( Δ^{1/4} I \), in the hope of obtaining the right conformal scaling.

We need also an operator on \( \mathscr{H} \) to act as the conformal factor, because neither functions on \( \bfS^{2n-1} \) nor on \( \bar{D}^{2n} \) are represented naturally on \( \mathscr{H} \). By \cite[Proposition 3.2]{Chen_1996}, there is a polar decomposition \( I = Δ^{1/4} B \), where \( B : \mathscr{H} \to L^2(Ω^n \bfS^{2n-1}) \) is an isometry with range \( Ω^n \bfS^{2n-1} ∩ \ker d \). The operator \( B^* k_{g^{-1}}^{-1/2} B \) is positive and invertible on \( \mathscr{H} \) because
\[ B^* k_{g^{-1}}^{-1/2} B ≥ B^* \| k_{g^{-1}}^{1/2} \|^{-1} B = \| k_{g^{-1}}^{1/2} \|^{-1} 1_{\mathscr{H}} . \]
We compute that both
\[ Δ^{1/4} I (B^* k_{g^{-1}}^{-1/2} B) - k_{g^{-1}}^{-1/2} Δ^{1/4} I = [Δ^{1/2} P_{\ker d}, k_{g^{-1}}^{-1/2}] B \]
and
\begin{multline*}
	U_g Δ^{1/4} I U_g^* - k_{g^{-1}}^{-1/2} Δ^{1/4} I (B^* k_{g^{-1}}^{-1/2} B) \\
	= \left( U_g Δ^{1/4} U_g^* - k_{g^{-1}}^{-1/2} Δ^{1/4} \right) k_{g^{-1}}^{-1/2} Δ^{1/4} B - k_{g^{-1}}^{-1/2} Δ^{1/4} \left[ Δ^{1/4} P_{\ker d}, k_{g^{-1}}^{-1/2} \right] B
\end{multline*}
are bounded. With $D=Δ^{1/4} I + I^* Δ^{1/4} + d + δ$, the Hodge decomposition theorem $Ω^n \bfS^{2n-1}=\ker(\Delta)\oplus{\rm Im}(d)\oplus{\rm Im}(\delta)$ shows that $D^2|_{\mathscr{H}_\infty} =B^*d\delta  B$ has at most a finite dimensional kernel, while $D^2|_{Ω^n \bfS^{2n-1}}=\Delta^{1/2}P_{\ker d}\Delta^{1/2}+\delta d = d \delta+\delta d = \Delta$. On the rest of the complex, $D^2$ agrees with $\Delta$ and so $D$ has compact resolvent.
Therefore,
\[ (\bbC, \mathscr{H} ⊕ L^2(Ω^{≥n} \bfS^{2n}), Δ^{1/4} I + I^* Δ^{1/4} + d + δ) \]
is a conformally \( SO(2n, 1) \)-equivariant spectral triple with conformal factors \( μ_g = B^* k_{g^{-1}}^{-1/2} B ⊕ k_{g^{-1}}^{-1/2} \). Its bounded transform (more exactly its phase) is the γ-element constructed by Chen \cite[\S3.1]{Chen_1996}.

To show that we have obtained the γ-element, independent of the bounded transform, we would need a representation of \( C(\bar{D}^{2n}) \) so as to apply \cite[Proposition 5.10]{Aparicio_2019}. For this purpose, Chen shows that the phase of the larger complex
\[ \begin{tikzcd}[column sep=1.6em, row sep=0em]
	& & 0 \arrow[r] & \mathscr{H}_∞ \arrow[r, "I"] & Ω^n \bfS^{2n-1} \arrow[r, "d"] & \cdots \arrow[r, "d"] & Ω^{2n-1} \bfS^{2n-1} \arrow[r] & 0 \\
	& & & ⊕ & ⊕ & & ⊕ & ⊕ \\
	0 \arrow[r] & Ω^0 \bbR H^{2n} \arrow[r, "d"] & \cdots \arrow[r, "d"] & Ω^n \bbR H^{2n}/\mathscr{H}_∞ \arrow[r, "d"] & Ω^{n+1} \bbR H^{2n} \arrow[r, "d"] & \cdots \arrow[r, "d"] & Ω^{2n} \bbR H^{2n} \arrow[r] & 0
\end{tikzcd} \]
gives a Fredholm module for \( C(\bar{D}^{2n}) \). Unfortunately, at the level of unbounded Kasparov modules, the construction cannot be carried through because the Hodge–de Rham operator on \( \bbR H^{2n} \) does not have compact resolvent. Although we do not pursue it here, this defect can be remedied by appealing to the framework of relative spectral triples \cite{Forsyth_2019, Fries_2025}. The larger complex can be assembled into a relative spectral triple for \( C_0(\bbR H^{2n}) ⊲ C(\bar{D}^{2n}) \) in the sense of \cite[Definition 2.8]{Fries_2025} cf. \cite[Example 2.15]{Fries_2025}. We can show that the $K$-homology class of the relative spectral triple extends to a class for $C(\bar{D}^{2n})$ by showing that the boundary map applied to the class of the relative spectral triple is zero. To compute the boundary map as in \cite[§8.5]{Higson_2000}, one uses the phase rather than the bounded transform. Since the phase already gives a Fredholm module for all of \( C(\bar{D}^{2n}) \) the boundary map is zero and we conclude that we do obtain a $K$-homology class for $C(\bar{D}^{2n})$.

\subsubsection{The case of \( SU(n, 1) \)}

Following \cite{Julg_1995}, we consider the sphere \( \bfS^{2n-1} \), on which \( SU(n, 1) \) acts by CR-automorphisms. This is not a conformal group action. We replace the de Rham complex with the Rumin complex \cite{Rumin_1994}, a refinement depending on a contact structure. A detailed discussion of the Rumin complex would be beyond the scope of this paper, especially as the construction of the γ-element does not use the whole complex. A treatment of the Rumin complex in the context of spectral noncommutative geometry and unbounded KK-theory can be found in \cite{Fries_2025a}. The analytical underpinnings of the Rumin complex, and the much more general class of Rockland complexes, have recently been examined in \cite{Dave_2022}. We limit ourselves to outlining those points which we require.

Let $X$ be a $(2n-1)$-dimensional contact manifold with contact structure $H\subseteq TX$. By this, it is meant that there exists a one-form \(\theta\) such that \( H = \ker \theta \) and \( d\theta|_H \) is nondegenerate. The nondegeneracy of \( d\theta|_H \) is equivalent to $\theta\wedge (d \theta)^{n-1}$ being a volume form. Such a one-form \( θ \) is a \emph{contact form} and is not unique. However, if \( τ \) is another contact form, then the equality \( \ker τ = \ker θ \) implies that \( τ = f θ \) for a nonvanishing smooth function \( f \) on \( X \). Conversely, \( f θ \) will be a contact form for any nonvanishing smooth function \( f \) on \( X \).

The \emph{Rumin complex} associated to a contact manifold \( X \) is a refinement of the de Rham complex of \( X \), depending only on the contact structure (and not on the choice of contact form). For the construction of the Rumin complex on \( X \), we do require a choice of \( θ \), to define two differential ideals of \( Ω^* X \),
\begin{itemize}
	\item \( \mathcal{I} \), the ideal generated by \( θ \) and \( d θ \), and
	\item \( \mathcal{J} \), the ideal of forms \( ω ∈ Ω^* X \) such that \( θ ∧ ω \) and \( dθ ∧ ω \) are zero.
\end{itemize}
The Rumin complex is built by combining the quotient complex \(\Omega^* X /\mathcal{I}^*\) and the subcomplex \(\mathcal{J}^*\). These complexes are spliced together using a map \( D_H : \Omega^{n-1} X /\mathcal{I}^{n-1} \to \mathcal{J}^n \). The \emph{Rumin differential} \( D_H \) is given by \( ω \mapsto d \tilde{ω} \) where \( \tilde{ω} \) is the unique lift of \( ω \) such that \( θ ∧ d \tilde{ω} = 0 \). Surprisingly, \( D_H \) is well-defined, is a second-order differential operator, and completes the Rumin complex
\[ \begin{tikzcd}[column sep=1.9em]
	0 \arrow[r] & \Omega^0 X \arrow[r, "d_H"] & \Omega^1 X / \mathcal{I}^1 \arrow[r, "d_H"] & \cdots \arrow[r, "d_H"] & \Omega^{n-1} X / \mathcal{I}^{n-1} \arrow[r, "D_H"] & \mathcal{J}^n \arrow[r, "d_H"] & \cdots \arrow[r, "d_H"] & \mathcal{J}^{2n-1} \arrow[r] & 0
\end{tikzcd} , \]
whose cohomology coincides with the de Rham cohomology. Here, we have denoted the exterior differential on the quotient complex and subcomplex by \( d_H \). The mixture of first- and second-order operators means that the construction of a spectral triple from the Rumin complex requires careful thought; see \cite{Fries_2025a}. For the construction of the γ-element of \( SU(n, 1) \), however, this issue will not arise, as we shall see.

Let us fix a contact form \( θ \) and choose a Riemannian metric \( \mathbf{g} \) on \( X \). We require that these be compatible, in the sense that \( H \) is orthogonal to the \emph{Reeb field}, the (unique) vector field \( Z \) such that \( θ(Z) = 1 \) and \( ι_Z(dθ) = 0 \). Using the metric on \( Ω^k X \) induced by \( \mathbf{g} \), we obtain a version
\[ \star_H : Ω^k X / \mathcal{I}^k \to \mathcal{J}^{2n-1-k} \qquad \star_H : \mathcal{J}^k \to Ω^{2n-1-k} X / \mathcal{I}^{2n-1-k} \]
of the Hodge star operator by the relation \( \bar{α} ∧ \star_H β = (α, β) θ ∧ (dθ)^{n-1} \). We thereby obtain formal adjoints of the operators in the Rumin complex, viz. \( d_H^* = (-1)^k \star_H d_H \star_H \) and \( D^* = (-1)^n \star_H D_H \star_H \). We also obtain the \emph{Rumin Laplacian}, given by
\[ Δ_H = \begin{cases}
	(n - 1 - k) d_H d_H^* + (n - k) d_H^* d_H & \text{on } Ω^k X / \mathcal{I}^k, 0 ≤ k ≤ n - 2 \\
	(d_H d_H^*)^2 + D_H^* D_H & \text{on } Ω^{n-1} X / \mathcal{I}^{n-1} \\
	D_H D_H^* + (d_H^* d_H)^2 & \text{on } \mathcal{J}^n \\
	(n - k) d_H d_H^* + (n - 1 - k) d_H^* d_H & \text{on } \mathcal{J}^k, n + 1 ≤ k ≤ 2n - 1
\end{cases} . \]
The Rumin Laplacian is hypoelliptic, fourth-order on \( Ω^{n-1} X / \mathcal{I}^{n-1} \) and \( \mathcal{J}^n \) and second-order elsewhere.

The contact form \( θ \) determines a symplectic form \( dθ \) on \( H \). A CR-structure on \( X \) is the additional datum of a complex structure \( J \) on \( H \) such that \( dθ(X, J Y) = \mathbf{g}(X, Y) \) for all \( X, Y ∈ H \). A CR-automorphism of \( X \) is a diffeomorphism \( φ \) such that the Jacobian \( φ' \) preserves and acts complex-linearly on \( H ⊆ T X \). Because the Rumin complex depends only on the contact structure, the operators \( d_H \) and \( D_H \) are unchanged. Again, because the contact structure is preserved, the pullback \( φ^*(θ) \) of the contact form must be \( f θ \) for some nonvanishing smooth  function on \( X \). Hence
\[ φ^*(\mathbf{g})(X, Y) = (f dθ + df ∧ θ)(X, JY) = f dθ(X, J Y) = f \mathbf{g}(X, Y) \]
for all \( X, Y ∈ H \). The induced metric on \( TX/H \) is multiplied by \( f^2 \). One can check that the induced metric on the Rumin complex is multiplied by \( f^{-k} \) on \( Ω^k X / \mathcal{I}^k \) and \( f^{-k-1} \) on \( \mathcal{J}^k \). In this sense, CR-automorphisms behave in a similar way to conformal diffeomorphisms.

To construct the γ-element for \( SU(n, 1) \), following \cite[\S 6(b)]{Julg_1995}, we begin with the Rumin complex on the sphere \( \bfS^{2n-1} \), on which the group acts by CR-automorphisms. To obtain the correct class in \( KK^{SU(n, 1)}(\bbC, \bbC) \) for the γ-element, we will cut the Rumin complex in two, as we did for \( SO(2n+1, 1) \) and \( SO(2n, 1) \), and add an additional operator, as we did for the latter. The extra map is the Szegö map \( S \) constructed in \cite[Theorem 2.12]{Julg_1995} from \( \Omega^{n-1} \bfS^{2n-1} / \mathcal{I}^{n-1} \) to the \( L^2 \) harmonic \( n \)-forms \( \mathscr{H}^n ⊆ Ω^n \bbC H^{2n} \) on the complex hyperbolic space. The sphere \( \bfS^{2n-1} \) can be attached to \( \bbC H^{2n} \) as its boundary, forming the closed disc \( \bar{D}^{2n} \). The Szegö map takes \( ω ∈ \Omega^{n-1} \bfS^{2n-1} / \mathcal{I}^{n-1} \), lifts it uniquely to \( \tilde{ω} \) such that \( θ ∧ d \tilde{ω} = 0 \) (as in the construction of \( D_H \)), extends \( \tilde{ω} \) to \( η ∈ Ω^{n-1} \bbC H^{2n} \) so that \( d η ∈ L^2(Ω^n \bbC H^{2n}) \), and then projects $\eta$ down to \( S ω ∈ \mathscr{H}^n \). It turns out that such a process gives a well-defined map, whose kernel is \( \ker D_H \). We dissect the Rumin complex and graft in the Szegö map \( S \), obtaining
\[ \begin{tikzcd}[column sep=2.2em]
	0 \arrow[r] & \Omega^0 \bfS^{2n-1} \arrow[r, "d_H"] & \Omega^1 \bfS^{2n-1}/ \mathcal{I}^1 \arrow[r, "d_H"] & \cdots \arrow[r, "d_H"] & \Omega^{n-1} \bfS^{2n-1}/ \mathcal{I}^{n-1} \arrow[r, "S"] & \mathscr{H}^n_{\infty} \arrow[r] & 0
\end{tikzcd} \]
where \( \mathscr{H}^n_{\infty} \) is the image of \( S \), dense in \( \mathscr{H}^n \). This complex is invariant under pullback by the action \( φ \) of \( SU(n, 1) \). When we complete the spaces of the complex to Hilbert spaces, pullback by the action of \( SU(n, 1) \) is not unitary. The unitary action is, for \( ω ∈ L^2(Ω^k / \mathcal{I}^k) \) and \( ξ ∈ \mathscr{H}^n \),
\[ U_g ω = f_{g^{-1}}^{\frac{n-k}{2}} φ_{g^{-1}}^* ω \qquad U_g ξ = φ_{g^{-1}}^* ξ , \]
where \( (f_g)_{g ∈ SU(n, 1)} \) is a family of nonvanishing, positive, smooth functions on \( \bfS^{2n-1} \). By similar computations to those for Example \ref{example:hodge-de-rham-conformal-transformation}, for the unitary implementors \( U_g \) we have that
\[  U_g d_H U_g^* ω = f_{g^{-1}}^{\frac{n-(k+1)}{2}} d_H f_{g^{-1}}^{-\frac{n-k}{2}} ω = f_{g^{-1}}^{-\frac{1}{2}} d_H ω + f_{g^{-1}}^{\frac{n-(k+1)}{2}} \big[d_H, f_{g^{-1}}^{-\frac{n-k}{2}}\big] ω \]
so that \( U_g d_H U_g^* - f_{g^{-1}}^{-1/4} d_H f_{g^{-1}}^{-1/4} \) is bounded. On the hyperbolic space \( \bbC H^n \), the group \( SU(n, 1) \) acts by isometries. Because the map \( S \) commutes with pullback by the group action, \( U_g S U_g^* = S f_{g^{-1}}^{-1/2} \).
Unlike in the case of \( SO(2n, 1) \), there is no discrepancy between the conformal behaviours of \( d_H \) and \( S \). It remains to construct a conformal factor on \( \mathscr{H}^n \). By \cite[Proof of Theorem 6.6(ii)]{Julg_1995}, there is a polar decomposition \( S = Φ(S) Δ_H^{1/4} \), where \( Φ(S) : L^2(\Omega^{n-1} \bfS^{2n-1} / \mathcal{I}^{n-1}) \to \mathscr{H}^n \) is a coisometry with kernel \( \ker D_H \). The operator \( Φ(S) f_{g^{-1}}^{-1/4} Φ(S)^* \) is positive and invertible on \( \mathscr{H}^n \) because
\[ Φ(S) f_{g^{-1}}^{-1/4} Φ(S)^* ≥ Φ(S) \| f_{g^{-1}}^{1/4} \|^{-1} Φ(S)^* = \| f_{g^{-1}}^{1/4} \|^{-1} . \]
We compute that both
\[ S f_{g^{-1}}^{-1/4} - \left( Φ(S) f_{g^{-1}}^{-1/4} Φ(S)^* \right) S = Φ(S) \left[ (1 - \ker D) Δ_H^{-1/4}, f_{g^{-1}}^{-1/4} \right] \]
and
\[ U_g S U_g^* - \left( Φ(S) f_{g^{-1}}^{-1/4} Φ(S)^* \right) S f_{g^{-1}}^{-1/4} = Φ(S) \left[ (1 - \ker D) Δ_H^{-1/4}, f_{g^{-1}}^{-1/4} \right] f_{g^{-1}}^{-1/4} \]
are bounded. The operator $d_H + d_H^* + S + S^*$ has compact resolvent by an argument
very similar to the case of $SO(2n,1)$, using this time the compactness of the resolvent of the Rumin Laplacian \cite[Corollary 5.20]{Julg_1995}. For example, on \( \Omega^{n-1} \bfS^{2n-1} / \mathcal{I}^{n-1} \) one can check that
\begin{align*}
	(d_H + d_H^* + S + S^*)^2|_{\Omega^{n-1} \bfS^{2n-1} / \mathcal{I}^{n-1}}
	& = Δ_H^{1/4} (1 - \ker D) Δ_H^{1/4} + d_H d_H^* \\
	& = (D_H^* D_H)^{1/2} + d_H d_H^* \\
	& = Δ_H^{1/2} ,
\end{align*}
and the other cases are similar.
In summary, we have constructed a conformally \( SU(n, 1) \)-equivariant spectral triple
\[ (\bbC, L^2(\Omega^{\leq n-1} \bfS^{2n-1} / \mathcal{I}^{\leq n-1}) \oplus \mathscr{H}^n, d_H + d_H^* + S + S^*) \]
with conformal factors \( μ_g = f_{g^{-1}}^{-1/4} ⊕ Φ(S) f_{g^{-1}}^{-1/4} Φ(S)^* \). The phase of this spectral triple is exactly the Fredholm module of \cite[Corollary 6.10]{Julg_1995} whose class is \( γ ∈ KK^{SU(n, 1)}(\bbC, \bbC) \).

To show that we have obtained the γ-element without directly using the result of Julg and Kasparov, it would be necessary, as in the case of \( SO(2n, 1) \) to expand the complex to accommodate a representation of \( \bar{D}^{2n} \). However, as before, the resolvent would not be compact. Furthermore, it is unclear whether sufficient analytical tools are available to obtain bounded commutators.

\subsection{C*-algebra of the Heisenberg group}
\label{section:heisenberg}

In this section we give a truly noncommutative example of conformal equivariance, building a conformally equivariant higher-order spectral triple for the C*-algebra of the Heisenberg group. An element of the 3-dimensional Heisenberg group \( H^3 \) can be written as
\[ \begin{pmatrix} 1 & a & c \\ & 1 & b \\ & & 1 \end{pmatrix} \]
for \( a, b, c ∈ \bbR \). There is an action of \( \bbR^×_+ \) on \( H^3 \) by automorphisms, given for $t\in\bbR^\times_+$ by
\[ \begin{pmatrix} 1 & a & c \\ & 1 & b \\ & & 1 \end{pmatrix} ↦ \begin{pmatrix} 1 & t a & t^2 c \\ & 1 & t b \\ & & 1 \end{pmatrix}. \]
We will construct a conformally equivariant higher-order spectral triple for \( C^*(H^3) \). Define a Clifford algebra–valued function \( \ell : H^3 \to \Cl_3 \) by
\[ \ell : \begin{pmatrix} 1 & a & c \\ & 1 & b \\ & & 1 \end{pmatrix} ↦ (a γ_1 + b γ_2) (a^2 + b^2)^{1/2} + c γ_3\,. \]
With
\[ g = \begin{pmatrix} 1 & a & c \\ & 1 & b \\ & & 1 \end{pmatrix} \qquad h = \begin{pmatrix} 1 & a' & c' \\ & 1 & b' \\ & & 1 \end{pmatrix} \qquad g h = \begin{pmatrix} 1 & a + a' & c + c' + a b' \\ & 1 & b + b' \\ & & 1 \end{pmatrix} \]
we can check that
\begin{align*}
	\ell(g h) - \ell(h)
	& = \big((a + a') γ_1 + (b + b') γ_2\big) \big((a + a')^2 + (b + b')^2\big)^{1/2} + (c + c' + a b') γ_3 \\
	& \qquad - (a' γ_1 + b' γ_2) ({a'}^2 + {b'}^2)^{1/2} + c' γ_3 \\
	& = (a' γ_1 + b' γ_2) \big( ((a + a')^2 + (b + b')^2)^{1/2} - ({a'}^2 + {b'}^2)^{1/2} \big) \\
	& \qquad +(a γ_1 + b γ_2) ((a + a')^2 + (b + b')^2)^{1/2} + (c + a b') γ_3
\end{align*}
and
\[ (1 + \ell(h)^2)^{1/2} = \left( 1 + ({a'}^2 + {b'}^2)^{2} + {c'}^2 \right)^{1/2}\,. \]
Hence \( \left( \ell(g h) - \ell(h) \right) (1 + \ell(h)^2)^{-1/4} \) is uniformly bounded in \( h ∈ G \). A computation then shows that, for $f\in C_c(H^3)$, the operator
$[M_\ell,f](1 + M_\ell^2)^{-1/4}=[M_\ell,f]\langle M_\ell\rangle^{-1/2}$ is bounded where $M_\ell$ is multiplication by $\ell$. We arrive at the order-2  spectral triple
\( (C^*(H^3), L^2(H^3, \bbC^2), M_{\ell}) \).
The local compactness of the resolvent is a consequence of \( (1 + \ell^2)^{-1} ∈ C_0(H^3, \Cl_3) \) and the isomorphism \( C_0(H^3) ⋊ H^3 ≅ K(L^2(H^3)) \).
Let \( V_t ∈ B(L^2(H^3)) \) be given by the pullback
\[ V_t ξ(a, b, c) = ξ(t^{-1} a, t^{-1} b, t^{-2} c) \]
on \( ξ ∈ L^2(H^3) \). Then
\[ ⟨ V_t^* ξ | η ⟩ = \int ξ(t^{-1} a, t^{-1} b, t^{-2} c) η(a, b, c) da db dc = \int ξ(x, y, z) η(t x, t y, t^2 z) t^4 dx dy dz = t^4 ⟨ ξ | V_{t^{-1}} η ⟩ \]
so that \( V_t^* = t^4 V_{t^{-1}} \). The unitary in the polar decomposition is given by \( U_t = t^{-2} V_t \). Noting that 
\[ \ell(t a, t b, t^2 c) = t^2 \ell(a, b, c) \]
we see that the operator \( M_{\ell} \) transforms as
\begin{align*}
	(U_t M_{\ell} U_t^* ξ)(a, b, c)
	& = t^{-2} (M_{\ell} U_t^* ξ)(t^{-1} a, t^{-1} b, t^{-2} c) \\
	& = t^{-2} \ell(t^{-1} a, t^{-1} b, t^{-2} c) (U_t^* ξ)(t^{-1} a, t^{-1} b, t^{-2} c) \\
	& = t^{-2} \ell(a, b, c) ξ(a, b, c) \\
	& = t^{-2} (M_{\ell} ξ)(a, b, c)
\end{align*}
on a vector \( ξ ∈ L^2(H^3, \bbC^2) \). In summary, we have:

\begin{proposition}
	The data \( (C^*(H^3), L^2(H^3, \bbC^2), M_{\ell}) \), together with the action \( (U_t)_{t ∈ \bbR} \) of the group \( \bbR^\times_+ \) and conformal factors given by \( μ_t = t^{-1} \), constitute a conformally \( \bbR^\times_+ \)-equivariant second order spectral triple.
\end{proposition}

The C*-algebra of the Heisenberg group can be identified with a continuous field of Moyal planes (with one classical plane) over \( \bbR \) \cite[§4]{Elliott_1993}. In this picture, the group action is dilation on \( \bbR \) and a corresponding scaling of the parameters of the Moyal planes.

A generalisation of the construction in this section to all Carnot groups and their dilations can be found in \cite{Fries_2025a}. For a broader discussion of the building of spectral triples for group C*-algebras, we refer to the first named author's PhD thesis \cite[Chapter II]{Masters_2025a}.

\section{Quantum-group-equivariant KK theory}
\label{section:quantum-group-equivariant}

Conformal group actions of a nontrivial kind are already rare in the classical setup of Riemannian manifolds, as the Ferrand–Obata theorem \cite[Theorem A]{Ferrand_1996} shows. The conformal group of a Riemannian metric must be the isometry group of a conformally equivalent metric, unless the manifold is conformally equivalent to a round sphere \( \bfS^n \) or Euclidean space \( \bbR^n \). It seems that the rarity of large conformal groups carries over to the noncommutative setting. A possible example of a noncommutative geometry with interesting conformal group is the Podleś sphere. As we shall see in §\ref{section:podles}, this hope is realised; however the conformal geometry of the Podleś sphere is not governed by a group but rather by a quantum group.

Quantum-group-equivariant KK-theory, in the bounded picture, is due to Baaj and Skandalis \cite{Baaj_1989}. A detailed account can be found in \cite{Vergnioux_2002}. We first recall the notions of a \emph{C*-bialgebra} and a \emph{locally compact quantum group}.

\begin{definition}
	e.g. \cite[Definitions 4.1.{1,3}]{Timmermann_2008}
	A \emph{C*-bialgebra} is a C*-algebra \( S \) equipped with a \emph{comultiplication} map, a coassociative, nondegenerate $*$-homomorphism \( Δ : S \to M(S ⊗ S) \) such that \( Δ(S) (S ⊗ 1) \) and \( (1 ⊗ S) Δ(S) \) are contained in \( S ⊗ S \). A C*-bialgebra \( S \) is \emph{simplifiable} if
	\[ \overline{\Span}(Δ(S) (S ⊗ 1)) = S ⊗ S = \overline{\Span}((1 ⊗ S) Δ(S)) . \]
	A \emph{von Neumann bialgebra} is a von Neumann algebra \( M \) with a \emph{comultiplication} map, a coassociative, unital, normal $*$-homomorphism \( Δ : M \to M ⊗ M \),  the von Neumann tensor product.
\end{definition}

Commutative C*-bialgebras are in duality with certain locally compact semigroups; see \cite[§3]{Vallin_1985} for precise statements.

\begin{definition}
	e.g. \cite[Chapter 8]{Timmermann_2008}
	A \emph{locally compact quantum group} \( \mathbb{G} \) is given by the equivalent data of either:
	\begin{itemize}\itemsep-2pt
		\item A simplifiable C*-bialgebra \( C_0^r(\mathbb{G}) \) with left- and right-invariant, KMS, faithful weights; or
		\item A von Neumann bialgebra \( L^∞(\mathbb{G}) \) with left- and right-invariant, normal, semifinite, faithful weights.
	\end{itemize}
	For the precise meaning of the adjectives on the weights, see e.g. \cite[§8.1.1-2]{Timmermann_2008}, but we will not use these details. From such data, one obtains:
	\begin{itemize}\itemsep-2pt
		\item The Hilbert space \( L^2(\mathbb{G}) \), on which \( L^∞(\mathbb{G}) \) and \( C_0^r(\mathbb{G}) \) are represented, obtained by the GNS construction from the left Haar weight (of either algebra);
		\item The \emph{universal} function algebra \( C_0^u(\mathbb{G}) \), which surjects onto \( C_0^r(\mathbb{G}) \);
		\item The \emph{dual} locally compact quantum group \( \hat{\mathbb{G}} \), for which \( L^2(\hat{\mathbb{G}}) ≅ L^2(\mathbb{G}) \), and the C*-algebras \( C^*_r(\mathbb{G}) := C_0^r(\hat{\mathbb{G}}) \) and \( C^*_u(\mathbb{G}) := C_0^u(\hat{\mathbb{G}}) \);
		\item The \emph{multiplicative unitary} \( W ∈ M(C_0^r(\mathbb{G}) ⊗ C_0^r(\hat{\mathbb{G}})) ⊆ B(L^2(\mathbb{G}) ⊗ L^2(\mathbb{G})) \) satisfying the equation \( W_{12} W_{13} W_{23} = W_{23} W_{12} \) and, for \( a ∈ C_0^r(\mathbb{G}) \), \( Δ(a) = W^* (1 ⊗ a) W \) on \( L^2(\mathbb{G}) ⊗ L^2(\mathbb{G}) \); and
		\item A Banach algebra \( L^1(\mathbb{G}) := L^∞(\mathbb{G})_* \), the predual of \( L^∞(\mathbb{G}) \).
	\end{itemize}
\end{definition}

We next recall the details of C*-bialgebra-coactions on C*-algebras and Hilbert modules.

\begin{definition}
	\cite[Definitions 1.39, A.3]{Echterhoff_2006}
	Let \( B \) and \( C \) be C*-algebras. The \emph{\( C \)-multiplier algebra} of \( B ⊗ C \) is
	\[ M_C(B ⊗ C) = \big\{ m ∈ M(B ⊗ C) | \,m (1 ⊗ C) ∪ (1 ⊗ C) m ∈ B ⊗ C \big\} . \]
	If \( E \) is a Hilbert \( B \)-module, the \emph{\( C \)-multiplier module} of \( E ⊗ S_{B ⊗ S} \) is the Hilbert \( M_C(B ⊗ C) \)-module
	\[ M_C(E ⊗ C) = \big\{ m ∈ \Hom_{B ⊗ C}^*(B ⊗ C, E ⊗ C) |\, m (1 ⊗ C) ∪ (1 ⊗ C) m ∈ E ⊗ C \big\} . \]
\end{definition}

\begin{definition}
	\cite[§2]{Baaj_1989}, \cite[§3.1]{Vergnioux_2002}
	A coaction of a C*-bialgebra \( S \) on a C*-algebra \( B \) is a coassociative nondegenerate $*$-homomorphism \( δ_B : B \to M_S(B ⊗ S) \).
	A coaction of \( S \) on a Hilbert \( B \)-module \( E \) is a coassociative \( \bbC \)-linear map \( δ_E : E \to M_S(E ⊗ S) \) such that
	\begin{itemize}
		\item \( δ_E(ξ) δ_B(b) = δ_E(ξ b) \) and \( ⟨δ_E(ξ) | δ_E(η)⟩_{M_S(B ⊗ S)} = δ_B(⟨ξ | η⟩_B) \) for all \( ξ, η ∈ E \) and \( b ∈ B \); and
		\item \( δ_E(E) (B ⊗ S) \) is dense in \( E ⊗ S \).
	\end{itemize}
	Let \( E ⊗_{δ_B}\!\!(B ⊗ S) \) be the internal tensor product of Hilbert modules where the left action of \( B \) on \( B ⊗ S \) is given by \( δ_B \). For an element \( ξ ∈ E \), denote by \( T_ξ ∈ \Hom_{B ⊗ S}^*(B ⊗ S, E ⊗_{δ_B}\!\!(B ⊗ S)) \) the map \( b ⊗ s ↦ ξ ⊗_{δ_B} (b ⊗ s) \).
	A unitary \( V_E ∈ \Hom_{B ⊗ S}^*(E ⊗_{δ_B}\!\!(B ⊗ S), E ⊗ S) \) is \emph{admissible} if
	\begin{itemize}
		\item \( V_E T_ξ ∈ M_S(E ⊗ S) \) for all \( ξ ∈ E \); and
		\item \( (V_E ⊗_{\bbC} 1) (V_E ⊗_{δ_B ⊗ \id_S}\!\!1) = (V_E ⊗_{\id_B ⊗ Δ_S}\!\!1) ∈ \Hom_{B ⊗ S ⊗ S}^*(E ⊗_{δ_B^2}\!\!(B ⊗ S ⊗ S), E ⊗ S ⊗ S) \), where \( δ_B^2 = (δ_B ⊗ \id_S) δ_B = (\id_B ⊗ Δ_S) δ_B \).
	\end{itemize}
	A coaction on \( E \) can equivalently be described by an admissible unitary \( V_E \) using the identity \( V_E T_ξ = δ_E(ξ) \) for \( ξ ∈ E \).
	
	If \( A \) is a C*-algebra with an \( S \)-coaction \( δ_A \), an \( A \)-\( B \)-correspondence \( E \) is \emph{\( S \)-equivariant} if it possesses a Hilbert \( B \)-module coaction \( δ_E \) such that
	\[ δ_A(a) δ_E(ξ) = δ_E(a ξ) \]
	for all \( a ∈ A \) and \( ξ ∈ E \). In terms of the admissible unitary, this is equivalent to \( V_E (a ⊗ 1) V_E^* = δ_A(a) \).
\end{definition}

\begin{definition}
	cf. \cite[Definition 1.4(b)]{Podles_1995}, \cite[§5.2]{Baaj_2003}
	Let \( S \) be a C*-bialgebra. An \( S \)-coaction \( δ_B \) on a C*-algebra \( B \) satisfies the \emph{Podleś condition} (sometimes called simply \emph{continuity}) if \( \overline{\Span}(δ_B(B) (1 ⊗ S)) = B ⊗ S \). An \( S \)-coaction \( δ_E \) on a Hilbert \( B \)-module \( E \) then automatically satisfies
	\[ \overline{\Span}(δ_E(E) (1 ⊗ S)) = \overline{\Span}(δ_E(E) δ_B(B) (1 ⊗ S)) = \overline{\Span}(δ_E(E) (B ⊗ S)) = E ⊗ S \]
	 and \( \overline{V_E (E ⊗_{δ_B}\!\!(1 ⊗ S))} \) is dense in \( E ⊗ S \).
\end{definition}

\begin{definition}
	An \emph{action} of a locally compact quantum group \( \mathbb{G} \) on a C*-algebra \( B \) is a \( C_0^r(\mathbb{G}) \)-coaction on \( B \) satisfying the Podleś condition. A \( \mathbb{G} \)-action on a Hilbert \( B \)-module \( E \) is a \( C_0^r(\mathbb{G}) \)-coaction on \( E \).
\end{definition}

In the above Definition, the reduced C*-algebra is used following \cite{Vergnioux_2002} and \cite[§4]{Nest_2010}. One could perhaps define the action of a quantum group \( \mathbb{G} \) as a \( C_0^u(\mathbb{G}) \)-coaction instead, as is done in \cite{Echterhoff_2006}, although it is unclear what the consequences of this would be, particularly for the descent map.

\begin{definition}
	\cite[Définition 3.1]{Baaj_1989} cf. \cite[§4]{Nest_2010}
	Let \( A \) and \( B \) be C*-algebras equipped with coactions of a C*-bialgebra \( S \). A bounded Kasparov \( A \)-\( B \)-module \( (A, E_B, F) \) is \emph{\( S \)-equivariant} if \( E \) is an \( S \)-equivariant \( A \)-\( B \)-correspondence and for all \( a ∈ A \) and \( s ∈ S \)
	\[ (V_E (F ⊗_{δ_B}\!\!1) V_E^* - F ⊗ 1) a ⊗ s \]
	is compact.
	If \( A \) and \( B \) are C*-algebras with \( \mathbb{G} \)-actions, a bounded Kasparov module \( (A, E_B, F) \) is \emph{\( \mathbb{G} \)-equivariant} if it is \( C_0^r(\mathbb{G}) \)-equivariant.
\end{definition}

\subsection{Uniform quantum group equivariance}

We make the following definition in the unbounded setting. To our knowledge, except in the case of the isometric coaction of a compact quantum group (see e.g. \cite[Definition 2.3.1]{Goswami_2016}), such a definiton has not appeared in the published literature (but see \cite[Definition 3.3.1]{Goffeng_2009}).

\begin{definition}
	\label{definition:ordinary-unbounded-equivariance-quantum}
	Let \( A \) and \( B \) be C*-algebras equipped with coactions of a C*-bialgebra \( S \). Fix \( 0 \leq \alpha < 1 \). For \( a ∈ \Lip_{\alpha}^*(D) \) let \( \mathscr{S}_a \) be the set of \( s \in S \) such that \( a ⊗ s \dom(D ⊗ 1) ⊆ V_E \dom(D ⊗_{δ_B} 1) \) and
	\begin{multline*}
		\big( V_E (D ⊗_{δ_B}\!\!1) V_E^* (a ⊗ s) - (a ⊗ s) (D ⊗ 1) \big) \langle D ⊗ 1 \rangle^{-\alpha} \\
		\text{and } V_E \langle D ⊗_{δ_B}\!\!1 \rangle^{-\alpha} V_E^* \big( V_E (D ⊗_{δ_B}\!\!1) V_E^* (a ⊗ s) - (a ⊗ s) (D ⊗ 1) \big)
	\end{multline*}
	extend to adjointable operators on \( E\otimes S \).
	An order-\( \frac{1}{1-\alpha} \) \( A \)-\( B \)-cycle \( (A, E_B, D) \) is \emph{uniformly \( S \)-equivariant} if \( E \) is an \( S \)-equivariant \( A \)-\( B \)-correspondence and \( A \)  is contained in the closure of 
	\[
	\mathscr{Q}=\left\{a ∈ \Lip_{\alpha}^*(D) \middle|\, \overline{\mathscr{S}_a} = S \right\}\,.
	\]
	If \( V_E (D ⊗_{δ_B}\!\!1) V_E^* = D ⊗ 1 \), we say that the cycle is \emph{isometrically equivariant}.
	
	If \( A \) and \( B \) are C*-algebras with \( \mathbb{G} \)-actions, a cycle \( (A, E_B, D) \) is \emph{uniformly \( \mathbb{G} \)-equivariant} if it is uniformly \( C_0^r(\mathbb{G}) \)-equivariant.
	
	If \( \mathscr{A} \) is a dense $*$-subalgebra of \( A \) such that \( \mathscr{A} ⊆ \mathscr{Q} \), we say that \( (\mathscr{A}, E_B, D) \) is \( S \)-equivariant (or \( \mathbb{G} \)-equivariant, as the case may be).
\end{definition}

\begin{remark}
	The dense subset \( \mathscr{S}_a ⊆ S \) need not be the same for different \( a ∈ \mathscr{Q} \).
	For many locally compact quantum groups, there may be a natural choice, fixed for all \( a \). For a discrete quantum group \( \mathbb{G} \), i.e. when \( C_0(\mathbb{G}) \) is isomorphic as an algebra to the C*-algebraic direct sum
	\[ \bigoplus_{λ ∈ Λ} M_{n_λ}(\bbC) \]
	of finite-dimensional matrix algebras, \( \mathscr{S}_a \) would contain all elements of the algebraic direct sum. In this case, the admissible unitary would be labelled by the index set \( λ ∈ Λ \), so that
	\[ V_E^λ ∈ \Hom_B^*(E ⊗_{δ_B}\!\! (B ⊗ \bbC^{n_λ}), E ⊗ \bbC^{n_λ}) \]
	and the equivariance condition becomes that
	\begin{multline*}
		\big( V_E^λ (D ⊗_{δ_B}\!\! 1) V_E^{λ *} (a ⊗ 1_{n_λ}) - (a ⊗ 1_{n_λ}) (D ⊗ 1_{n_λ}) \big) \langle D ⊗ 1_{n_\lambda} \rangle^{-\alpha} \\
		\text{and } V_E^λ \langle D ⊗_{δ_B}\!\!1 \rangle^{-\alpha} V_E^{λ *} \big( V_E^λ (D ⊗_{δ_B}\!\! 1) V_E^{λ *} (a ⊗ 1_{n_λ}) - (a ⊗ 1_{n_λ}) (D ⊗ 1_{n_λ}) \big)
	\end{multline*}
	be bounded for all \( λ ∈ Λ \). (Note that there need not be any bound uniform in \( λ ∈ Λ \).) For the dual \( \hat{G} \) of a  group \( G \), we suspect it always makes sense to assume that \( \mathscr{S}_a \)  contains the right ideal \( C^*_r(G)^∞ \) of smooth elements \cite[§§2–3]{Woronowicz_1992}, as  in Example \ref{example:dual-group-invariant-metric}.
\end{remark}

\begin{theorem}
	A uniformly \( S \)-equivariant order-\( \frac{1}{1-\alpha} \) cycle \( (A, E_B, D) \) gives rise to an \( S \)-equivariant bounded Kasparov module \( (A, E_B, F_D) \).
\end{theorem}
\begin{proof}
	The only difference from the non-equivariant case is the need to show that, for every \( a ∈ A \) and \( s ∈ S \), \( (F_D ⊗ 1 - V_E (F_D ⊗_{δ_B}\!\! 1) V_E^* ) a ⊗ s \) is compact.
	Let \( b ∈ \mathscr{Q} \) and \( s ∈ \mathscr{S}_a \) so that
	\[ (V_E (D ⊗_{δ_B}\!\! 1) V_E^* - D ⊗ 1) (b ⊗ s) ⟨D ⊗ 1⟩^β \]
	extends to an adjointable operator. By Corollary \ref{proposition:bounded-transform-of-additive-perturbation-local},
	\[ (V_E (F_D ⊗_{δ_B}\!\! 1) V_E^* - F_D ⊗ 1) (b ⊗ s) ⟨D⟩^β ⊗ 1 \]
	is bounded for all \( β < 1 - \alpha \). With \( c ∈ A \),
	\[ (V_E (F_D ⊗_{δ_B}\!\! 1) V_E^* - F_D ⊗ 1) b c ⊗ s = (V_E (F_D ⊗_{δ_B}\!\! 1) V_E^* - F_D ⊗ 1) (b ⊗ s) (⟨D⟩^β ⊗ 1) ⟨D⟩^{-β} c ⊗ 1 \]
	is compact and, by the density of \( \mathscr{S}_a ⊆ S \) and the inclusion of \( A ⊆ \overline{\mathscr{Q} A} \), we are done.
\end{proof}

\begin{example}
	\label{example:dual-group-invariant-metric}
	Let \( G \) be a connected Lie group with a left-invariant Riemannian metric \( \mathbf{g} \), such as the affine group \( \bbR ⋊ \bbR^×_+ \) of the real line as the real hyperbolic plane. The left-invariant Riemannian metric on \( G \) is exactly determined by the inner product \( \mathbf{g}_e \) on the tangent space \( T_e G = \mathfrak{g} \) at the identity \( e ∈ G \). The left-invariant differential operators and differential forms on \( G \) can be identified with \( U(\mathfrak{g}) \) and \( Λ^*(\mathfrak{g}) \), respectively. The Clifford algebra \( \Cl(\mathfrak{g}) \) acts on  the left of \( Λ^*(\mathfrak{g}) \). The Hodge–de Rham Dirac operator \( d + δ \) on \( (G, \mathbf{g}) \) can be written as
	\[ d + δ = \sum_{i = 1}^{\dim \mathfrak{g}} X_i ⊗ γ_i , \]
	where \( X_i ∈ \mathfrak{g} ⊆ U(\mathfrak{g}) \) and \( γ_i ∈ \mathfrak{g} ⊆ \Cl(\mathfrak{g}) \). We have an isometrically \( G \)-equivariant spectral triple
	\[ (C_0(G), L^2(G, Λ^*(\mathfrak{g})), d + δ) . \]
	Elements of \( \mathfrak{g} ⊆ U(\mathfrak{g}) \) act as affiliated operators on \( C^*_r(G) \); see \cite[§§2–3]{Woronowicz_1992}. By abuse of notation, we also write \( d + δ ∈ U(\mathfrak{g}) ⊗ \Cl(\mathfrak{g}) \) for the corresponding regular operator on \( (C^*_r(G) ⊗ Λ^*(\mathfrak{g}))_{C^*_r(G)} \). By Baaj–Skandalis duality \cite[§6]{Baaj_1989}, it is reasonable to expect that
	\[ (\bbC, (C^*_r(G) ⊗ Λ^*(\mathfrak{g}))_{C^*_r(G)}, d + δ) \]
	is a uniformly \( \hat{G} \)-equivariant \( \bbC \)-\( C^*_r(G) \)-unbounded Kasparov module. To see that it is, first consider the coaction on the module \( (C^*_r(G) ⊗ Λ^*(\mathfrak{g}))_{C^*_r(G)} \). The admissible unitary is a map from
	\[ (C^*_r(G) ⊗ Λ^*(\mathfrak{g})) ⊗_{δ_{C^*_r(G)}} (C^*_r(G) ⊗ C^*_r(G)) = C^*_r(G) ⊗ Λ^*(\mathfrak{g}) ⊗ C^*_r(G) \]
	to
	\[ (C^*_r(G) ⊗ Λ^*(\mathfrak{g})) ⊗_{\bbC} C^*_r(G) = C^*_r(G) ⊗ Λ^*(\mathfrak{g}) ⊗ C^*_r(G). \]
	Under these identifications,
	\[ T_{x ⊗ ψ} : C^*_r(G) ⊗ C^*_r(G) \to C^*_r(G) ⊗ Λ^*(\mathfrak{g}) ⊗ C^*_r(G) \qquad y ⊗ z ↦ x_{(1)} y ⊗ ψ ⊗ x_{(2)} z \]
	\[ x_{(1)} ⊗ ψ ⊗ x_{(2)} = δ(x ⊗ ψ) = V T_x = V (x_{(1)} ⊗ ψ ⊗ x_{(2)}) \]
	so \( V \) is just the identity in \( \End_{C^*_r(G)}^*(C^*_r(G) ⊗ Λ^*(\mathfrak{g}) ⊗ C^*_r(G)) \). Because \( X_i ∈ \mathfrak{g} \), in the universal enveloping algebra \( U(\mathfrak{g}) \), \( Δ X_i = X_i ⊗ 1 + 1 ⊗ X_i \) and
	\[ (d + δ) ⊗_{δ_{C^*_r(G)}} 1 = \sum_i (X_i ⊗ γ_i) ⊗_{Δ_{U(\mathfrak{g})}} 1 = \sum_i (X_i ⊗ γ_i ⊗ 1 + 1 ⊗ γ_i ⊗ X_i) . \]
	Therefore,
	\[ V ((d + δ) ⊗_{δ_{C^*_r(G)}} 1) V^* - (d + δ) ⊗ 1 = 1 ⊗ γ_i ⊗ X_i . \]
	For \( (\bbC, (C^*_r(G) ⊗ Λ^*(\mathfrak{g}))_{C^*_r(G)}, d + δ) \) to be \( C^*_r(G) \)-equivariant, we require a dense subalgebra of \( C^*_r(G) \) in the common domain of the derivations \( \mathfrak{g} \). There is in fact such a subalgebra, the right ideal \( C^*_r(G)^∞ \) of smooth elements for the \( G \)-action on \( C^*_r(G) \) by unitary multipliers \cite[§§2–3]{Woronowicz_1992}. 
\end{example}

\subsection{Descent and the dual Green–Julg map for uniform equivariance}
\label{section:descent-quantum-group}

Crossed products are not defined in the generality of Hopf C*-algebra–coactions. One needs a well-defined notion of duality and, for that, we restrict to locally compact quantum groups. 
(It is possible to work in the greater generality of a weak Kac system \cite[§2.2]{Vergnioux_2002}, but we forgo this in the interests of readability.)

We use the symbol \( Σ \) for the flip map on a tensor product. Recall the multiplicative unitary \( W ∈ M(C_0^r(\mathbb{G}) ⊗ C_0^r(\hat{\mathbb{G}})) ⊆ B(L^2(\mathbb{G}) ⊗ L^2(\mathbb{G})) \) of a locally compact quantum group \( \mathbb{G} \).

\begin{definition}
	\cite[Definition 7.3.1]{Timmermann_2008} cf. \cite[Proposition 3.2, Définition 3.3]{Baaj_1993}
	A locally compact quantum group \( \mathbb{G} \) is \emph{regular} if
	\[ \overline{\Span}\{ (ω ⊗ 1)(W Σ) \mid \,ω ∈ B(L^2(\mathbb{G}))_* \} = K(L^2(\mathbb{G})) . \]
	Equivalently, \( \mathbb{G} \) is regular if the reduced crossed product \( C_0^r(\mathbb{G}) ⋊_r \mathbb{G} ≅ K(L^2(\mathbb{G})) \); see Definition \ref{definition:quantum-group-crossed-products} below.
\end{definition}

For example, every locally compact group \( G \) and its dual \( \hat{G} \) are regular.

In the following proof, one might expect
	\[ (A ⊗ 1) U (1 ⊗ B) ⊆ A ⊗ B \]
	to hold automatically for a unitary \( U ∈ M(A ⊗ B) \) but this is not the case, as \cite[Remark after Lemma 1.2]{Landstad_1987} shows.

\begin{lemma}
	\label{lem:reg-needed?}
	Let \( E \) be a Hilbert \( B \)-module with a \( \mathbb{G} \) action, \( \mathbb{G} \) acting trivially on \( B \). Then \( C^*_u(\mathbb{G}) \) is represented on \( E \). Conversely, if \( \mathbb{G} \) is a regular quantum group, a (nondegenerate) representation of \( C^*_u(\mathbb{G}) \) on a Hilbert \( B \)-module gives rise to a \( \mathbb{G} \) action on \( E \) which is trivial on \( B \).
\end{lemma}
\begin{proof}
	Let \( E \) be a Hilbert \( B \)-module with a \( \mathbb{G} \) action, \( \mathbb{G} \) acting trivially on \( B \). The fundamental unitary \( V_E \) is then an element of \( \End^*(E ⊗ C_0^r(\mathbb{G})) \) and can be thought of as an element of \( \End^*(E ⊗ L^2(\mathbb{G})) \) by the left regular representation of \( C_0^r(\mathbb{G}) \). By \cite[Proposition 5.2]{Kustermans_2001}, there is a nondegenerate representation of \( C^*_u(\mathbb{G}) \) on \( E \).
	
	On the other hand, suppose that \( C^*_u(\mathbb{G}) \) is represented nondegenerately by \( π \) on a Hilbert \( B \)-module \( E \). Let \( \hat{\mathcal{V}} ∈ M(C_0^r(\mathbb{G}) ⊗ C^*_u(\mathbb{G})) \) be the unitary of \cite[Proposition 4.2]{Kustermans_2001}. By \cite[Corollary 4.3]{Kustermans_2001}, we obtain an element \( X = (π ⊗ \id)(Σ \hat{\mathcal{V}} Σ) ∈ \End^*(E ⊗ S) \) such that \( (1 ⊗ Δ)(X) = X_{1 2} X_{1 3} \). The only thing stopping \( X \) from being the admissible unitary of an action of \( \mathbb{G} \) on \( E \) (with trivial action on \( B \)) is the possible failure of \( (1 ⊗ C_0^r(\mathbb{G})) X (E ⊗ 1) \) to be contained in \( E ⊗ C_0^r(\mathbb{G}) \). If we assume \( \mathbb{G} \) to be regular, by \cite[Proposition A.3(d)]{Baaj_1993},
	\[ \overline{\Span} (1 ⊗ C_0^r(\mathbb{G})) X (π(C^*_u(\mathbb{G})) ⊗ 1) = π(C^*_u(\mathbb{G})) ⊗ C_0^r(\mathbb{G}) \]
	and therefore
	\[ (1 ⊗ C_0^r(\mathbb{G})) X (E ⊗ 1) = (1 ⊗ C_0^r(\mathbb{G})) X (π(C^*_u(\mathbb{G})) E ⊗ 1) ⊆ E ⊗ C_0^r(\mathbb{G}) , \]
	as required.
\end{proof}

It is unclear if the converse statement of Lemma \ref{lem:reg-needed?} is true without the assumption of regularity.

\begin{definition}
	\label{definition:quantum-group-crossed-products}
	cf. \cite[Définitions 4.2, 5.1, Lemmes 4.1, 5.2]{Vergnioux_2002}
	Let \( A \) be a C*-algebra with a \( \mathbb{G} \)-action. The \emph{reduced} crossed product \( A ⋊_r \mathbb{G} \) is given by
	\[ \overline{\Span}(δ_A(A) (1 ⊗ C^*_r(\mathbb{G}))) ⊆ M(A ⊗ K(L^2(\mathbb{G}))) . \]
	There is also a \emph{universal} crossed product \( A ⋊_u \mathbb{G} \); for a definition we refer to \cite[§2.3]{Vaes_2005}. For our purposes, the following details will suffice.
	Let \( {}_π E \) be a \( \mathbb{G} \)-equivariant \( A \)-\( B \)-correspondence, with \( \mathbb{G} \) acting trivially on \( B \). There is an \emph{integrated} representation of \( A ⋊_u \mathbb{G} \) on \( E \) whose image is
	\[ \overline{\Span}(π(A) C^*_u(\mathbb{G})) ⊆ \End^*(E) . \]
	If \( \mathbb{G} \) is regular, the algebra \( A ⋊_u \mathbb{G} \) is universal for such integrated representations; if \( \mathbb{G} \) is not regular \( A ⋊_u \mathbb{G} \) is universal for a slightly larger class of representations; see \cite[Définition 4.2]{Vergnioux_2002} and \cite[§2.3]{Vaes_2005}. There is a canonical surjection \( A ⋊_u \mathbb{G} \to A ⋊_r \mathbb{G} \).

	Let \( E \) be a right Hilbert \( B \)-module with an action of \( \mathbb{G} \). For either topology \( t ∈ \{ u, r \} \), the crossed product Hilbert module \( E ⋊_t \mathbb{G} \) is given by the internal tensor product \( E ⊗_B (B ⋊_t \mathbb{G}) \). By \cite[Lemme 5.2]{Vergnioux_2002}, \( \End_B^0(E) ⋊_t \mathbb{G} \) is naturally identified with \( \End_{B ⋊_t \mathbb{G}}^0(E ⋊_t \mathbb{G}) \).
\end{definition}

In the locally compact quantum group setting, there is a descent map
\[ j^{\mathbb{G}}_t : KK^{\mathbb{G}}(A, B) \to KK(A ⋊_t \mathbb{G}, B ⋊_t \mathbb{G}) \]
for either topology \( t ∈ \{ u, r \} \), universal or reduced, generalising Kasparov's descent map for classical groups. If \( \mathbb{G} \) is the dual of a classical group, descent is due to Baaj and Skandalis \cite[Théorème 6.19]{Baaj_1989}, and in general due to Vergnioux \cite[Proposition 5.3]{Vergnioux_2002}.
In the locally compact quantum group setting, a refinement of the reduced descent is possible, to a map
\[ J^{\mathbb{G}} : KK^{\mathbb{G}}(A, B) \to KK^{\hat{\mathbb{G}}}(A ⋊_r \mathbb{G}, B ⋊_r \mathbb{G}) \]
whose composition with the forgetful functor \( KK^{\hat{\mathbb{G}}} \to KK \) is \( j^{\mathbb{G}}_r \). If \( \mathbb{G} \) is regular, \( C_0^r(\mathbb{G}) ⋊_r \mathbb{G} ≅ K(L^2(\mathbb{G})) ≅ C^*_r(\mathbb{G}) ⋊_r \hat{\mathbb{G}} \) and the maps \( J^{\mathbb{G}} \) and \( J^{\hat{\mathbb{G}}} \) are mutually inverse isomorphisms \cite[Remarque 7.7(b)]{Baaj_1993}. We refer to these isomorphisms as \emph{Baaj–Skandalis duality}.

\begin{proposition}
	\cite[Proposition 5.3]{Vergnioux_2002}
	Let \( (A, E_B, F) \) be a \( \mathbb{G} \)-equivariant bounded Kasparov module. For \( t ∈ \{ u, r \} \), let \( ι \) be the inclusion \( \End^0(E) \to M(\End^0(E) ⋊_t \mathbb{G}) ≅ \End_{B ⋊_t \mathbb{G}}^*(E ⋊_t \mathbb{G}) \). Then \( (A ⋊_t \mathbb{G}, (E ⋊_t \mathbb{G})_{B ⋊_t \mathbb{G}}, ι(F)) \) is a bounded Kasparov module.
\end{proposition}

If \( \mathbb{G} \) is compact and acts trivially on \( A \), we have the \emph{Green–Julg} isomorphism
\[ Φ_{\mathbb{G}} : KK^{\mathbb{G}}(A, B) \to KK(A, B ⋊ \mathbb{G}) ; \]
see \cite[Théorème 5.10]{Vergnioux_2002}.
On the other hand, when \( \mathbb{G} \) acts trivially on \( B \), there is a dual Green–Julg map for the universal crossed product
\[ Ψ^{\mathbb{G}} : KK^{\mathbb{G}}(A, B) \to KK(A ⋊_u \mathbb{G}, B) \]
which is an isomorphism when \( \mathbb{G} \) is discrete \cite[Proposition 5.11]{Vergnioux_2002}.

\begin{proposition}
	\cite[Proposition 5.11]{Vergnioux_2002}
	Let \( (A, E_B, F) \) be a \( \mathbb{G} \)-equivariant bounded Kasparov module, with \( \mathbb{G} \) acting trivially on \( B \). Then \( (A ⋊_u \mathbb{G}, E_B, F) \) is a bounded Kasparov module, with the integrated representation of \( A ⋊_u \mathbb{G} \). 
\end{proposition}

\begin{proposition}
	\cite[Proposition 5.11]{Vergnioux_2002}
	Let \( (A ⋊_u \mathbb{G}, E_B, F) \) be a bounded Kasparov module, with \( \mathbb{G} \) a discrete quantum group and \( A ⋊_u \mathbb{G} \) represented nondegenerately on \( E \). Then \( (A, E_B, F) \) is a \( \mathbb{G} \)-equivariant bounded Kasparov module, with the coaction of \( C_0^r(\mathbb{G}) \) on \( E \) given by the action of \( C^*_u(\mathbb{G}) ⊆ M(A ⋊_u \mathbb{G}) \) on \( E \), acting trivially on \( B \).
\end{proposition}

In the unbounded setting, we have the following picture of descent.

\begin{proposition}
	\label{proposition:descent-uniform-unbounded-quantum}
	Let \( (A, E_B, D) \) be a uniformly \( \mathbb{G} \)-equivariant order-\( \frac{1}{1-\alpha} \) cycle. For \( t ∈ \{ u, r \} \), let \( ι \) be the inclusion \( \End^0(E) \to M(\End^0(E) ⋊_t \mathbb{G}) ≅ \End_{B ⋊_t \mathbb{G}}^*(E ⋊_t \mathbb{G}) \). Then \( (A ⋊_t \mathbb{G}, (E ⋊_t \mathbb{G})_{B ⋊_t \mathbb{G}}, ι(D)) \) is an order-\( \frac{1}{1-\alpha} \) cycle.
	
	If, for a dense $*$-subalgebra \( \mathscr{A} ⊆ A \), \( (\mathscr{A}, E_B, D) \) is a uniformly \( \mathbb{G} \)-equivariant order-\( \frac{1}{1-\alpha} \) cycle, the data
	\[ \left( \Span\{ (1 ⊗ ω)((ι(a)^* ⊗ s^*) X) |\, a ∈ \mathscr{A}, s ∈ \mathscr{S}_a, ω ∈ L^1(\mathbb{G}) \}, (E ⋊_t \mathbb{G})_{B ⋊_t \mathbb{G}}, ι(D) \right) \]
	defines an order-\( \frac{1}{1-\alpha} \) cycle, where \( X \) is a unitary on \( (E ⋊_t \mathbb{G}) ⊗ C_0^r(\mathbb{G}) \) described in the proof.
\end{proposition}
\begin{proof}
	Note that the image of the representation of \( A ⋊_t \mathbb{G} \) is \( \overline{\Span}(ι(A) C^*_t(\mathbb{G})) ⊆ \End^*(E ⋊_t \mathbb{G}) \). Using the identification \( \End_B^0(E) ⋊_t \mathbb{G} ≅ \End_{B ⋊_t \mathbb{G}}^0(E ⋊_t \mathbb{G}) \), we see that, for \( a ∈ A \) and \( f ∈ C^*_t(\mathbb{G}) \),
	\[ (1 + ι(D)^2)^{-1/2} (ι(a) f) = ι((1 + D^2)^{-1/2} a) f \]
	is compact, cf. \cite[Démonstration du Proposition 5.3]{Vergnioux_2002}. By the universality of the crossed product \cite[§4.1]{Vergnioux_2002} \cite[§2.3]{Vaes_2005}, the morphism \( \End^0(E) ⋊_u \mathbb{G} \to \End^0(E) ⋊_t \mathbb{G} \) gives rise to the morphism \( ι : \End^0(E) \to M(\End^0(E) ⋊_t \mathbb{G}) ≅ \End^*(E ⋊_t \mathbb{G}) \) and a unitary \( X ∈ M((\End^0(E) ⋊_t \mathbb{G}) ⊗ C_0^r(\mathbb{G})) ≅ \End^*((\End^0(E) ⋊_t \mathbb{G}) ⊗ C_0^r(\mathbb{G})) \) such that
	\[ X (ι(T) ⊗ 1) X^* = (ι ⊗ \id) (V_E (T ⊗_{δ_B} 1) V_E^*) \]
	for \( T ∈ \End^0(E) \). Let \( a ∈ \mathscr{Q} \) and \( s ∈ \mathscr{S}_a \). For \( X^* (ι(a) ⊗ s_1) ∈ \End^*((E ⋊_t \mathbb{G}) ⊗ C_0^r(\mathbb{G}) \),
	\begin{align*}
		& [ι(D) ⊗ 1, X^* (ι(a) ⊗ s)] \langle ι(D) \rangle^{-\alpha} \\
		& \qquad = X^* \left( X (ι(D) ⊗ 1) X^* (ι(a) ⊗ s) - (ι ⊗ \id) \left( (a ⊗ s) (D ⊗ 1) \right) \right) \langle ι(D) \rangle^{-\alpha} \\
		& \qquad = X^* (ι ⊗ \id) \left( \big( V_E (D ⊗_{δ_B} 1) V_E^* (a ⊗ s) - (a ⊗ s) (D ⊗ 1) \big) (\langle D \otimes \rangle^{-\alpha} \right)
	\end{align*}
	and
	\begin{align*}
		& \langle ι(D) \rangle^{-\alpha} [ι(D) ⊗ 1, X^* (ι(a) ⊗ s)] \\
		& \qquad = X^* X \langle ι(D) \rangle^{-\alpha} X^* \left( X (ι(D) ⊗ 1) X^* (ι(a) ⊗ s) - (ι ⊗ \id) \left( (a ⊗ s) (D ⊗ 1) \right) \langle ι(D) \rangle^{-\alpha} \right) \\
		& \qquad = X^* (ι ⊗ \id) \left( V_E \langle D ⊗_{δ_B}\!\!1 \rangle^{-\alpha} V_E^* \big( V_E (D ⊗_{δ_B} 1) V_E^* (a ⊗ s) - (a ⊗ s) (D ⊗ 1) \big) \right)
	\end{align*}
	are adjointable.
	The representation of \( A ⋊_t \mathbb{G} \) on \( E ⋊_t \mathbb{G} \) consists of
	\begin{align*}
		\overline{\Span}(ι(A) C^*_t(\mathbb{G}))
		& = \overline{\Span}\left\{ ι(a) (1 ⊗ ω) (X) \middle|\, a ∈ A, ω ∈ L^1(\mathbb{G}) \right\} \\
		& = \overline{\Span}\left\{ ι(a) (1 ⊗ η_1^*) X (1 ⊗ η_2^*) \middle|\, a ∈ A, η_1, η_2 ∈ L^2(\mathbb{G}) \right\} \\
		& = \overline{\Span}\left\{ (1 ⊗ η_1^*) (ι(a)^* ⊗ s^*) X (1 ⊗ η_2^*) \middle|\, a ∈ A, s ∈ C_0^r(\mathbb{G}), η_1, η_2 ∈ L^2(\mathbb{G}) \right\} \\
		& ⊆ \overline{\Span}\left\{ (1 ⊗ η_1^*) (ι(a)^* ⊗ s^*) X (1 ⊗ η_2^*) \middle|\, a ∈ \mathscr{Q}, s ∈ \mathscr{S}_a, η_1, η_2 ∈ L^2(\mathbb{G}) \right\}
	\end{align*}		
	by the density of \( \mathscr{S}_a^* ⊆ C_0^r(\mathbb{G}) \) and the inclusion \( A ⊆ \overline{\mathscr{Q}} \).
\end{proof}

We also have a realisation of the dual Green–Julg map on uniformly equivariant unbounded Kasparov modules.

\begin{proposition}
	Let \( (A, E_B, D) \) be a uniformly \( \mathbb{G} \)-equivariant order-\( \frac{1}{1-\alpha} \) cycle, with \( \mathbb{G} \) acting trivially on \( B \). Then \( (A ⋊_u \mathbb{G}, E_B, D) \) is an order-\( \frac{1}{1-\alpha} \) cycle, with the integrated representation of \( A ⋊_u \mathbb{G} \).
	
	If, for a dense \( * \)-subalgebra \( \mathscr{A} ⊆ A \), \( (\mathscr{A}, E_B, D) \) is a uniformly \( \mathbb{G} \)-equivariant unbounded Kasparov module, with \( G \) acting trivially on \( B \), then
	\[ 
	\left( \Span\{ (1 ⊗ ω)((a^* ⊗ s^*) V_E) |\, a ∈ \mathscr{A}, s ∈ \mathscr{S}_a, ω ∈ L^1(\mathbb{G}) \}, E_B, D \right) 
	\]
	is an order-\( \frac{1}{1-\alpha} \) cycle.
\end{proposition}
\begin{proof}
	The only point which is not immediate is the boundedness of commutators with \( D \). Let \( a ∈ \mathscr{Q} \) and \( s ∈ \mathscr{S}_a \) and let \( ω ∈ L^1(\mathbb{G}) \), so that
	\[ (1 ⊗ ω)((a^* ⊗ s^*) V_E) \]
	is in the integrated representation of \( A ⋊_u \mathbb{G} \) on \( E \). By the uniform equivariance condition,
	\[ [D, (1 ⊗ ω)((a^* ⊗ s^*) V_E)] = (1 ⊗ ω) \left( \left( V_E (D ⊗ 1) V_E^*  (a ⊗ s) - (a ⊗ s) (D ⊗ 1) \right)^* V_E \right) \]
	and so \( (1 ⊗ ω)((a^* ⊗ s^*) V_E) \in \Lip_{\alpha}^*(D) \). The representation of \( A ⋊_t \mathbb{G} \) on \( E ⋊_t \mathbb{G} \) consists of
	\begin{align*}
		\overline{\Span}(A C^*_u(\mathbb{G}))
		& = \overline{\Span}\left\{ a (1 ⊗ ω) (V_E) \middle|\, a ∈ A, ω ∈ L^1(\mathbb{G}) \right\} \\
		& = \overline{\Span}\left\{ a (1 ⊗ η_1^*) V_E (1 ⊗ η_2^*) \middle|\, a ∈ A, η_1, η_2 ∈ L^2(\mathbb{G}) \right\} \\
		& = \overline{\Span}\left\{ (1 ⊗ η_1^*) (a^* ⊗ s^*) V_E (1 ⊗ η_2^*) \middle|\, a ∈ A, s ∈ C_0^r(\mathbb{G}), η_1, η_2 ∈ L^2(\mathbb{G}) \right\} \\
		& ⊆ \overline{\Span}\left\{ (1 ⊗ η_1^*) (a^* ⊗ s^*) V_E (1 ⊗ η_2^*) \middle|\, a ∈ \mathscr{Q}, s ∈ \mathscr{S}_a, η_1, η_2 ∈ L^2(\mathbb{G}) \right\}
	\end{align*}		
	by the density of \( \mathscr{S}_a L^2(\mathbb{G}) ⊆ C_0^r(\mathbb{G}) L^2(\mathbb{G}) ⊆ L^2(\mathbb{G}) \) and the inclusion \( A ⊆ \overline{\mathscr{Q}} \).
\end{proof}

For the inverse map, more structure is required, including the presence of a dense subalgebra \( \mathscr{A} \) of \( A \). A discrete quantum group \( \mathbb{G} \) has a compact dual, whose polynomial algebra we denote by \( \mathcal{O}(\hat{\mathbb{G}}) \). We write \( \mathscr{A} ⋊ \mathbb{G} \) for the subalgebra of \( A ⋊_u \mathbb{G} \) generated by \( \mathscr{A} \) and \( \mathcal{O}(\hat{\mathbb{G}}) \).

\begin{proposition}
	Let \( (\mathscr{A} ⋊ \mathbb{G}, E_B, D) \) be an order-\( \frac{1}{1-\alpha} \) cycle, with \( \mathbb{G} \) a discrete quantum group and the representation of \( \mathscr{A} ⋊ \mathbb{G} \) on \( E \) nondegenerate. Then \( (\mathscr{A}, E_B, D) \) is a uniformly \( \mathbb{G} \)-equivariant order-\( \frac{1}{1-\alpha} \) cycle, with the \( \mathbb{G} \)-action on \( E \) given by Lemma \ref{lem:reg-needed?} and trivial on \( B \). 
\end{proposition}
\begin{proof}
	Because \( \mathbb{G} \) is discrete, \( \mathscr{A} \) is included in \( \mathscr{A} ⋊ \mathbb{G} \). Hence \( (1 + D^2)^{-1} a \) is compact and \( [D, a] \) is bounded for all \( a ∈ \mathscr{A} \). The inclusion \( C^*_u(\mathbb{G}) ⊆ M(A ⋊_u \mathbb{G}) \) gives a (nondegenerate) representation \( π \) of \( C^*_u(\mathbb{G}) \) on \( E \). Because \( \mathbb{G} \) is discrete, it is regular. Applying Lemma \ref{lem:reg-needed?}, we obtain an action of \( \mathbb{G} \) on \( E \), acting trivially on \( B \). Let \( V_E \) be the admissible unitary. Discreteness means that \( C_0(\mathbb{G}) \) is isomorphic as an algebra to the C*-algebraic direct sum
\[ \bigoplus_{λ ∈ Λ} M_{n_λ}(\bbC) \]
of finite-dimensional matrix algebras. The admissible unitary is the direct sum over the index set \( λ ∈ Λ \) of
\[ V_E^λ ∈ π(\mathcal{O}(\hat{\mathbb{G}})) ⊗ M_{n_λ}(\bbC) ⊆ π(C^*_u(\mathbb{G})) ⊗ M_{n_λ}(\bbC) ⊆ \End_B^*(E ⊗ \bbC^{n_λ}) , \]
cf. \cite[§4.2.3]{Voigt_2020} for the inclusion in the polynomial subalgebra. Then, for \( a ∈ \mathscr{A} \),
	\[ \big( V_E^λ (D ⊗ 1) V_E^{λ *} (a ⊗ 1) - (a ⊗ 1) (D ⊗ 1) \big) \langle D ⊗ 1 \rangle^{-\alpha} = V_E^λ \big[ D ⊗ 1, V_E^{λ *} (a ⊗ 1) \big] \langle D ⊗ 1 \rangle^{-\alpha} \]
	and
	\[ V_E^λ \langle D ⊗ 1 \rangle^{-\alpha} V_E^{λ *} \big( V_E^λ (D ⊗ 1) V_E^{λ *} (a ⊗ 1) - (a ⊗ 1) (D ⊗ 1) \big) = V_E^λ \langle D ⊗ 1 \rangle^{-\alpha} \big[ D ⊗ 1, V_E^{λ *} (a ⊗ 1) \big] \]
	are bounded for all \( λ ∈ Λ \), because \( V_E^{λ *} (a ⊗ 1) ∈ π(\mathcal{O}(\hat{\mathbb{G}})) \mathscr{A} ⊗ M_{n_λ}(\bbC) ⊆ (\mathscr{A} ⋊ \mathbb{G}) ⊗ M_{n_λ}(\bbC) \).
\end{proof}

\begin{remark}
	It is clear that the bounded transform \( (A ⋊_t \mathbb{G}, (E ⋊_t \mathbb{G})_{B ⋊_t \mathbb{G}}, F_{ι(D)} = ι(F_D)) \) of the descent \( (A ⋊_t \mathbb{G}, (E ⋊_t \mathbb{G})_{B ⋊_t \mathbb{G}}, ι(D)) \) of a uniformly \( \mathbb{G} \)-equivariant cycle \( (A, E_B, D) \) is exactly the descent of the bounded transform \( (A, E_B, F_D) \). The same is true for the dual Green–Julg map.
\end{remark}

\subsection{Conformal quantum group equivariance}
\label{section:conformal-quantum}

To generalise Definition \ref{definition:ordinary-unbounded-equivariance-quantum} to conformal (co)actions, we will consider a conformal factor \( μ \) which is an unbounded operator on \( E ⊗ S \), where \( E \) is a Hilbert \( B \)-module and \( S \) is a C*-bialgebra. It is necessary to allow \( μ \) to be unbounded in the `\( S \) direction', as can be seen from classical group equivariance. To apply the multiplicative perturbation theory of §\ref{section:multiplicative-perturbations}, we will require \( μ \) to be \emph{\( S \)-matched}, in the sense of \S\ref{appendix:temp-bdd-operators}, meaning roughly that $\mu$ is locally bounded in the $S$-direction. We denote by \( K_S \) the Pedersen ideal of \( S \).

\begin{definition}
	\label{defn:conf-ess-equi}
	Let \( A \) and \( B \) be C*-algebras equipped with coactions of a C*-bialgebra \( S \). An order-\( \frac{1}{1-\alpha} \) \( A \)-\( B \)-cycle \( (A, E_B, D) \) is \emph{conformally \( S \)-equivariant} if \( E \) is an \( S \)-equivariant \( A \)-\( B \)-correspondence and there exists an (even)  \( S \)-matched operator \( μ \) on \( (E ⊗ S)_{B ⊗ S} \) whose inverse is also \( S \)-matched, satisfying the following. For \( a ∈ \Lip_{\alpha}^*(D) \), let \( \mathscr{S}_a \) be the set of \( s ∈ M(S) \) such that
	\[ \{ (a ⊗ s) μ, (a ⊗ s) μ^{-1*} \} \dom(D ⊗ 1) (1 ⊗ K_S) ⊆ \dom(D ⊗ 1) ∩ V_E \dom(D ⊗_{δ_B} 1) \]
	and
	\begin{gather*}
	\begin{aligned}
		& \big( V_E (D ⊗_{δ_B}\!\! 1) V_E^* (a ⊗ s) - (a ⊗ s) μ (D ⊗ 1) μ^* \big) \mu^{-1 *} \langle D ⊗ 1 \rangle^{-\alpha} , \\
		V_E \langle D ⊗_{δ_B}\!\!1 \rangle^{-\alpha} V_E^* & \big( V_E (D ⊗_{δ_B}\!\! 1) V_E^* (a ⊗ s) - (a ⊗ s) μ (D ⊗ 1) μ^* \big) ,
	\end{aligned} \\
	\begin{aligned}
		& [D ⊗ 1, (a ⊗ s) μ] \langle D ⊗ 1 \rangle^{-\alpha} , & & [D ⊗ 1, (a ⊗ s) μ^{-1 *}] \langle D ⊗ 1 \rangle^{-\alpha} , \\
		\langle D ⊗ 1 \rangle^{-\alpha} & [D ⊗ 1, (a ⊗ s) μ] , & \text{and } \langle D ⊗ 1 \rangle^{-\alpha} & [D ⊗ 1, (a ⊗ s) μ^{-1 *}]
	\end{aligned}
	\end{gather*}
	extend to \( S \)-matched operators. Let \( \mathscr{Q} \) be the set of \( a ∈ \Lip_{\alpha}^*(D) \) such that \( S ⊆ \overline{\Span}(S \mathscr{S}_a) ∩ \overline{\Span}(\mathscr{S}_a S) \). Then we require that  \( A ⊆ \overline{\Span}(A \mathscr{Q}) ∩ \overline{\Span}(\mathscr{Q} A) \).

	If \( A \) and \( B \) are C*-algebras with \( \mathbb{G} \)-actions, an order-\( \frac{1}{1-\alpha} \) cycle \( (A, E_B, D) \) is \emph{conformally \( \mathbb{G} \)-equivariant} if it is conformally \( C_0^r(\mathbb{G}) \)-equivariant.
\end{definition}

\begin{remarks}
	\item When \( μ = 1 \), Definition \ref{defn:conf-ess-equi} reduces to Definition \ref{definition:ordinary-unbounded-equivariance-quantum} of uniformly \( S \)-equivariant cycles.
	\item For a discrete quantum group \( \mathbb{G} \), when \( C_0(\mathbb{G}) \) is isomorphic as an algebra to the C*-algebraic direct sum
	\[ \bigoplus_{λ ∈ Λ} M_{n_λ}(\bbC) \]
	of finite-dimensional matrix algebras, the Pedersen ideal \( K_{C_0(\mathbb{G})} \) is the algebraic direct sum. In this case, the conformal factor and the admissible unitary would be labelled by the index set \( λ ∈ Λ \), so that
	\[ V_E^λ ∈ \Hom_B^*(E ⊗_{δ_B} (B ⊗ \bbC^{n_λ}), E ⊗ \bbC^{n_λ}) \qquad μ^λ ∈ \End^*_B(E ⊗ \bbC^{n_λ}) \]
	and the equivariance conditions on \( a ∈ \mathscr{Q} \) become that
	\begin{gather*}
	\begin{aligned}
		& \big( V_E^λ (D ⊗_{δ_B}\!\! 1) V_E^{λ *} (a ⊗ s) - (a ⊗ s) μ^λ (D ⊗ 1) μ^{λ *} \big) (μ^λ)^{-1 *} \langle D ⊗ 1 \rangle^{-\alpha} , \\
		V_E^λ \langle D ⊗_{δ_B}\!\!1 \rangle^{-\alpha} V_E^{λ *} & \big( V_E^λ (D ⊗_{δ_B}\!\! 1) V_E^{λ *} (a ⊗ s) - (a ⊗ s) μ^λ (D ⊗ 1) μ^{λ *} \big) ,
	\end{aligned} \\
	\begin{aligned}
		& [D ⊗ 1, (a ⊗ s) μ^λ] \langle D ⊗ 1 \rangle^{-\alpha} , & & [D ⊗ 1, (a ⊗ s) (μ^λ)^{-1 *}] \langle D ⊗ 1 \rangle^{-\alpha} , \\
		\langle D ⊗ 1 \rangle^{-\alpha} & [D ⊗ 1, (a ⊗ s) μ^λ] , & \text{and } \langle D ⊗ 1 \rangle^{-\alpha} & [D ⊗ 1, (a ⊗ s) (μ^λ)^{-1 *}]
	\end{aligned}
	\end{gather*}
	be bounded for all \( λ ∈ Λ \).
\end{remarks}

\begin{theorem}
	\label{theorem:bbd-transform-quantum-conformal}
	A conformally \( S \)-equivariant order-\( \frac{1}{1-\alpha} \) cycle \( (A, E_B, D) \), with conformal factor \( μ \), gives rise to an \( S \)-equivariant bounded Kasparov module \( (A, E_B, F_D) \).
\end{theorem}
\begin{proof}
	The only point of difference from the non-equivariant case is the need to prove that, for every \( a ∈ A \) and \( s ∈ S \), \( (F_D ⊗ 1 - V_E (F_D ⊗_{δ_B}\!\! 1) V_E^* ) a ⊗ s \) is compact. Let \( c \) be a positive element of \( K_S \), so that, by Proposition \ref{proposition:temp-bdd-on-ideals}, the restriction of \( μ \) to the \( B ⊗ \overline{\Span}(S c S) \)-module \( E ⊗ \overline{\Span}(S c S) \) is bounded. For the time being, we work on the module \( E ⊗ \overline{\Span}(S c S) \). Let \( a_1, a_2, a_3, a_4 ∈ \mathscr{Q} \) and \( s_1, s_2, s_3, s_4 ∈ \mathscr{S}_{a_1}, \mathscr{S}_{a_2}, \mathscr{S}_{a_3}, \mathscr{S}_{a_4} \). As in the Proof of Theorem \ref{theorem:bdd-transform-conformal-transformation},
	\[ [μ (D ⊗ 1) μ^*, a_2^* a_3 ⊗ s_2^* s_3] μ^{-1*} ⟨D⟩^{-α} \]
	is bounded. We apply Theorem \ref{theorem:conformal-result-nonunital} to obtain that
	\[ (F_{μ (D ⊗ 1) μ^*} - F_D ⊗ 1) a_2^* a_3 a_4^* ⟨D⟩^β ⊗ s_2^* s_3 s_4^* \]
	is bounded for \( β < 1-α \). Furthermore, 
	\[ ((D ⊗_{δ_B}\!\!1) V_E^* (a_1 ⊗ s_1) - V_E^* (a_1 ⊗ s_1) μ (D ⊗ 1) μ^*) μ^{-1*} (⟨D⟩^{-α} ⊗ 1) \]
	is bounded and, by Proposition \ref{proposition:bounded-transform-of-commutator-additive-perturbation},
	\[ ((F_D ⊗_{δ_B}\!\! 1) V_E^* (a_1 ⊗ s_1) - V_E^* (a_1 ⊗ s_1) (F_{μ D μ^*} ⊗ 1)) μ (⟨D⟩^β ⊗ 1) \]
	is too. Now we have
	\begin{align*}
		& (V_E (F_D ⊗_{δ_B}\!\! 1) V_E^* - F_D ⊗ 1) a_1 a_2^* a_3 a_4^* ⊗ s_1 s_2^* s_3 s_4^* \\
		& \qquad = V_E ((F_D ⊗_{δ_B}\!\! 1) V_E^* - V_E^*(F_D ⊗ 1)) a_1 a_2^* a_3 a_4^* ⊗ s_1 s_2^* s_3 s_4^* \\
		& \qquad = V_E \left( (F_D ⊗_{δ_B}\!\! 1) V_E^* (a_1 a_2^* a_3 ⊗ s_1 s_2^* s_3) - V_E^* (a_1 a_2^* a_3 ⊗ s_1 s_2^* s_3) (F_D ⊗ 1) \right) (a_4^* ⊗ s_4^*) \\
		& \qquad\qquad - [F_D, a_1 a_2^* a_3] a_4^* ⊗ s_1 s_2^* s_3 s_4^* \\
		& \qquad = V_E \left( (F_D ⊗_{δ_B}\!\! 1) V_E^* (a_1 ⊗ s_1) - V_E^* (a_1 ⊗ s_1) F_{μ (D ⊗ 1) μ^*} \right) (a_2^* a_3 a_4^* ⊗ s_2^* s_3 s_4^*) \\
		& \qquad\qquad + (a_1 ⊗ s_1) \left( F_{μ (D ⊗ 1) μ^*} (a_2^* a_3 ⊗ s_2^* s_3) - (a_2^* a_3 ⊗ s_2^* s_3) (F_D ⊗ 1) \right) (a_4^* ⊗ s_4^*) \\
		& \qquad\qquad - [F_D, a_1 a_2^* a_3] a_4^* ⊗ s_1 s_2^* s_3 s_4^* \\
		& \qquad = V_E \left( (F_D ⊗_{δ_B}\!\! 1) V_E^* (a_1 ⊗ s_1) - V_E^* (a_1 ⊗ s_1) F_{μ (D ⊗ 1) μ^*} \right) (a_2^* a_3 a_4^* ⊗ s_2^* s_3 s_4^*) \\
		& \qquad\qquad + (a_1 ⊗ s_1) \left( F_{μ (D ⊗ 1) μ^*} - F_D ⊗ 1 \right) (a_2^* a_3 a_4^* ⊗ s_2^* s_3 s_4^*) \\
		& \qquad\qquad - [F_D, a_1] a_2^* a_3 a_4^* ⊗ s_1 s_2^* s_3 s_4^*
	\end{align*}
	so that \( (V_E (F_D ⊗_{δ_B}\!\! 1) V_E^* - F_D ⊗ 1) a_1 a_2^* a_3 a_4^* ⟨D⟩^β ⊗ s_1 s_2^* s_3 s_4^* \) is bounded. 
	Let \( a_5 ∈ A \) and note that \( c ∈ \overline{\Span}(S c S) ⊴ S \). Then
	\begin{equation}
	 (V_E (F_D ⊗_{δ_B}\!\! 1) V_E^* - F_D ⊗ 1) a_1 a_2^* a_3 a_4^* a_5 ⊗ s_1 s_2^* s_3 s_4^* c 
	 \label{eq:compact?}
	 \end{equation}
	is an element of \( \End^0(E) ⊗ \overline{\Span}(S c S) \).
	As the compacts on \( E ⊗ \overline{\Span}(S c S)_{B ⊗ \overline{\Span}(S c S)} \) are
	\[ \overline{\Span}(E E^* ⊗ S c S S c S) = \End^0(E) ⊗ \overline{\Span}(S c S) ⊴ \End^0(E) ⊗ S = \End^0(E ⊗ S) \] 
	for each \( c ∈ K_S \), we see that \eqref{eq:compact?} defines a compact endomorphism on \( E ⊗ S \). Because \( S ⊆ \overline{\mathscr{S}_{a_1} \mathscr{S}_{a_2}^* \mathscr{S}_{a_3} \mathscr{S}_{a_4}^* K_S} \) and \( A ⊆ \overline{\mathscr{Q}^* \mathscr{Q} \mathscr{Q}^* \mathscr{Q} A} \),
	\[ (V_E (F_D ⊗_{δ_B}\!\! 1) V_E^* - F_D ⊗ 1) a ⊗ s \]
	is compact for all \( a ∈ A \) and \( s ∈ S \).
\end{proof}

\begin{theorem}
\label{theorem:bbd-transform-quantum-conformal-log}
	A conformally \( S \)-equivariant order-\( \frac{1}{1-\alpha} \) cycle \( (A, E_B, D) \) gives rise to a uniformly \( S \)-equivariant order-\( \frac{1}{1-\alpha} \) cycle \( (A, E_B, L_D) \) via the logarithmic transform.
\end{theorem}
\begin{proof}
	By the Proof of Theorem \ref{theorem:bbd-transform-quantum-conformal}, \( (V_E (F_D ⊗_{δ_B}\!\! 1) V_E^* - F_D ⊗ 1) a_1 a_2^* a_3 a_4^* ⟨D⟩^β ⊗ s_1 s_2^* s_3 s_4^* c \) is bounded on \( E ⊗ S \) for \( a_1, a_2, a_3, a_4 ∈ \mathscr{Q} \), \( s_1, s_2, s_3, s_4 ∈ \mathscr{S}_{a_1}, \mathscr{S}_{a_2}, \mathscr{S}_{a_3}, \mathscr{S}_{a_4} \), \( c ∈ K_S \), and \( β < 1 - \alpha \). Then
	\[ \left[ \begin{pmatrix} V_E (F_D ⊗_{δ_B}\!\! 1) V_E^* & \\ & F_D ⊗ 1 \end{pmatrix}, \begin{pmatrix} & a_1 a_2^* a_3 a_4^* ⊗ s_1 s_2^* s_3 s_4^* c \\ 0 & \end{pmatrix} \right] \left⟨ \begin{pmatrix} V_E (F_D ⊗_{δ_B}\!\! 1) V_E^* & \\ & F_D ⊗ 1 \end{pmatrix} \right⟩^β \]
	is bounded and
	\[ \begin{pmatrix} & a_1 a_2^* a_3 a_4^* ⊗ s_1 s_2^* s_3 s_4^* c \\ 0 & \end{pmatrix} \dom \begin{pmatrix} V_E (D ⊗_{δ_B}\!\! 1) V_E^* & \\ & D ⊗ 1 \end{pmatrix} ⊆ \dom \begin{pmatrix} V_E (D ⊗_{δ_B}\!\! 1) V_E^* & \\ & D ⊗ 1 \end{pmatrix} . \]
	Applying Proposition \ref{proposition:logarithmic-transform},
	\[ \left[ \begin{pmatrix} V_E (L_D ⊗_{δ_B}\!\! 1) V_E^* & \\ & L_D ⊗ 1 \end{pmatrix}, \begin{pmatrix} & a_1 a_2^* a_3 a_4^* ⊗ s_1 s_2^* s_3 s_4^* c \\ 0 & \end{pmatrix} \right] \]
	is bounded and therefore so is \( (V_E (L_D ⊗_{δ_B}\!\! 1) V_E^* - L_D ⊗ 1) a_1 a_2^* a_3 a_4^* ⊗ s_1 s_2^* s_3 s_4^* c \). For any \( a_5 ∈ A \),
	\[ (V_E (L_D ⊗_{δ_B}\!\! 1) V_E^* - L_D ⊗ 1) a_1 a_2^* a_3 a_4^* a_5 ⊗ s_1 s_2^* s_3 s_4^* c \]
	is bounded. We have \( S ⊆ \overline{\mathscr{S}_{a_1} \mathscr{S}_{a_2}^* \mathscr{S}_{a_3} \mathscr{S}_{a_4}^* K_S} \) and \( A ⊆ \overline{\mathscr{Q}^* \mathscr{Q} \mathscr{Q}^* \mathscr{Q} A} \), as required.
\end{proof}

\begin{proposition}
	Let $G$ be a locally compact group. An order-\( \frac{1}{1-\alpha} \) cycle is conformally \( C_0(G) \)-equivariant if and only if it is conformally \( G \)-equivariant.
\end{proposition}
\begin{proof}
	Use Proposition \ref{proposition:temp-bbd-operators-topological}. Because \( C_0(G) \) is abelian, for \( a ∈ \mathscr{Q} \), \( \mathscr{S}_a \) will always contain the Pedersen ideal \( K_{C_0(G)} = C_c(G) \).
\end{proof}

\subsection{The Podleś sphere}
\label{section:podles}

The compact quantum group \( SU_q(2) \) has polynomial algebra \( \mathcal{O}(SU_q(2)) \) generated by \( a, b, c, d \) subject to the relations
\[ a b=q b a \qquad a c=q c a \qquad b d=q d b \qquad c d=q d c \qquad b c=c b \qquad a d=1+q b c \qquad d a=1+q^{-1} b c \]
and with adjoints
$ a^* = d$, $b^* = -q c$, $c^* = -q^{-1} b$, $d^* = a$.
The polynomial algebra \( \mathcal{O}(SU_q(2)) \) is spanned by the Peter–Weyl elements \( t^l_{i j} \) with \( l ∈ \frac{1}{2} \bbN \) and \( i, j ∈ \{-l, -l+1, \dots, l-1, l\} \). The generators form the fundamental representation \( l = \frac{1}{2} \), that is
\[ \begin{pmatrix} a & b \\ c & d \end{pmatrix} = \begin{pmatrix} t^{\frac{1}{2}}_{-\frac{1}{2}, -\frac{1}{2}} & t^{\frac{1}{2}}_{-\frac{1}{2}, \frac{1}{2}} \\ t^{\frac{1}{2}}_{\frac{1}{2}, -\frac{1}{2}} & t^{\frac{1}{2}}_{\frac{1}{2}, \frac{1}{2}} \end{pmatrix} . \]
In terms of this basis, the coproduct and counit are
\[ Δ(t^l_{i, j}) = \sum_k t^l_{i, k} ⊗ t^l_{k, j} \qquad ε(t^l_{i, j}) = δ_{i, j} \]
and the adjoint is related to the antipode by \( {t^l_{i, j}}^* = S(t^l_{j, i}) \).

Dual to \( SU_q(2) \) is the discrete quantum group \( \widehat{SU_q(2)} \) \cite[§4.2.3]{Voigt_2020}, whose function algebra \( C_0(\widehat{SU_q(2)}) = C^*(SU_q(2)) \) is the closed span of matrix elements \( τ^l_{i j} \) with \( l ∈ \frac{1}{2} \bbN \) and \( i, j ∈ \{-l, -l+1, \dots, l-1, l\} \), subject to
\[ 
τ^l_{i, j} τ^{l'}_{i', j'} = δ_{l, l'} δ_{j, i'} τ^l_{i, j'} \qquad {τ^l_{i, j}}^* = τ^l_{j, i}\,. 
\]
In particular, as C*-algebras,
\[ 
	C_0(\widehat{SU_q(2)}) = C^*(SU_q(2)) ≅ \bigoplus_{l ∈ \frac{1}{2} \bbN} M_{2l}(\bbC)\,. 
\]
We may choose \( τ^l_{i j} \) so that the pairing between \( C^*(SU_q(2)) \) and \( C(SU_q(2)) \) is given by
\[ 
	(τ^l_{i j}, t^{l'}_{i' j'}) = δ_{l, l'} δ_{i, i'} δ_{j, j'} 
\]
and the multiplicative unitary \( W ∈ M(C(SU_q(2)) ⊗ C^*(SU_q(2))) \) is \( W = \sum_{l, i, j} t^l_{i, j} ⊗ τ^l_{i, j} \).

The quantum universal enveloping algebra \( \breve{U}_q(\mathfrak{sl}(2)) \) is generated by \( K, K^{-1}, E, F \) subject to
\[ K K^{-1} = K^{-1} K = 1 \qquad K E K^{-1} = q E \qquad K F K^{-1} = q^{-1} F \qquad [E, F] = \frac{K^2 - K^{-2}}{q - q^{-1}} \]
with coproduct
\[ Δ(K) = K ⊗ K \qquad Δ(E) = E ⊗ K + K^{-1} ⊗ E \qquad Δ(F) = F ⊗ K + K^{-1} ⊗ F \]
and counit and antipode
\[ ε(K) = 1 \qquad ε(E) = ε(F) = 0 \qquad S(K) = K^{-1} \qquad S(E) = -q E \qquad S(F) = -q^{-1} F . \]
Note that this is not the same as \( U_q(\mathfrak{sl}(2)) \), although the latter is a Hopf subalgebra of \( \breve{U}_q(\mathfrak{sl}(2)) \) \cite[§3.1.2]{Klimyk_1997}. There is a nondegenerate pairing \( (\cdot, \cdot) \) between \( \breve{U}_q(\mathfrak{sl}(2)) \) and \( \mathcal{O}(SU_q(2)) \) \cite[Theorem 4.21]{Klimyk_1997}. By this pairing, \( \breve{U}_q(\mathfrak{sl}(2)) \) is an algebra of unbounded operators affiliated to \( C^*(SU_q(2)) \). We may define left and right actions of \( \breve{U}_q(\mathfrak{sl}(2)) \) on \( \mathcal{O}(SU_q(2)) \) by
\[ X ⇀ α = α_{(1)} (X, α_{(2)}) \qquad α ↼ X = (X, α_{(1)}) α_{(2)} . \]
The left and right actions of \( K \) are automorphisms of \( \mathcal{O}(SU_q(2)) \) and have the properties
\[ (K ⇀ α)^* = K^{-1} ⇀ α^* \qquad (α ↼ K)^* = α^* ↼ K^{-1} . \]
In terms of the Peter–Weyl basis, \( K ⇀ t^l_{i, j} = q^j t^l_{i, j} \) and \( t^l_{i, j} ↼ K = q^i t^l_{i, j} \). We also record the relationships \( S^{-1}(α) = K^2 ⇀ S(α) ↼ K^{-2} \) and \( ϕ(α β) = ϕ(β (K^2 ⇀ α ↼ K^2)) \) for the left Haar state \( ϕ \) on \( C(SU_q(2)) \). The unitary antipode \( R \) on \( C(SU_q(2)) \) is then given by \( R(α) = K ⇀ S(α) ↼ K^{-1} \); on the Peter–Weyl basis, \( R(t^l_{i j}) = K ⇀ {t^l_{j i}}^* ↼ K^{-1} = (K^{-1} ⇀ t^l_{j i} ↼ K)^* = q^{-i + j} {t^l_{j i}}^* \).

The Podleś sphere \( \bfS^2_q \) has polynomial algebra \( \mathcal{O}(\bfS^2_q) \), the subalgebra of \( \mathcal{O}(SU_q(2)) \) generated by 
\begin{align*}
	A & = -q^{-1} b c = c^* c = t^{1/2*}_{1/2,-1/2} t^{1/2}_{1/2,-1/2}=q^{-2}t^{1/2}_{-1/2,1/2}t^{1/2*}_{-1/2,1/2}=-q^{-1}[2]_q^{-1}t^1_{00}\\
	B & = a c^*=-q^{-1} a b = t^{1/2}_{-1/2,-1/2} t^{1/2*}_{1/2,-1/2}=-q^{-1/2}[2]_q^{-1/2}t^1_{-10}\\
	B^* & = c d = t^{1/2}_{1/2,-1/2} t^{1/2}_{1/2,1/2} = q^{-1/2}t^{1/2}_{1/2,-1/2}t^{1/2*}_{-1/2,-1/2}=[2]_q^{-1/2}q^{1/2}t^1_{10}.
\end{align*}
and is spanned by \( t^l_{i 0} \). The subspaces \( S_+ = \Span\{ t^l_{i, \frac{1}{2}} |\, l, i \} \) and \( S_- = \Span\{ t^l_{i, -\frac{1}{2}} |\, l, i \} \) of \( \mathcal{O}(SU_q(2)) \) are the spinor bundles of the Podleś sphere. They can be completed under the inner product on \( \mathcal{O}(SU_q(2)) \) given by the left Haar state. The natural Dirac operator defining a spectral triple \( (C(S_q^2), L^2(S_+ ⊕ S_-), D) \) is \cite[Theorem 8]{Dabrowski_2003}
\[ D = \begin{pmatrix} & ∂_E \\ ∂_F & \end{pmatrix} \]
where \( ∂_E = E ⇀ \) and \( ∂_F = F ⇀ \) or, in terms of the Peter–Weyl basis,
\[ ∂_E t^l_{i, j} = \sqrt{[l+1/2]_q^2-[j+1/2]_q^2} t^l_{i, j+1} \qquad ∂_F t^l_{i, j} = \sqrt{[l+1/2]_q^2-[j-1/2]_q^2} t^l_{i, j-1}\,. \]
(Here, we use the convention \( [n]_q = \frac{q^n - q^{-n}}{q - q^{-1}} \) for \( q \)-numbers.) We abbreviate these coefficients as \( κ^l_k = \sqrt{[l+1/2]_q^2-[k-1/2]_q^2} \). We have the twisted derivation property
\[ ∂_E(α β) = ∂_E(α) (K ⇀ β) + (K^{-1} ⇀ α) ∂_E(β) \qquad ∂_F(α β) = ∂_F(α) (K ⇀ β) + (K^{-1} ⇀ α) ∂_F(β) \]
which shows that \( D \) has bounded commutators with elements of \( \mathcal{O}(\bfS^2_q) \). The relationships
\[ ∂_E(α^*) = −q ∂_F(α)^* \qquad ∂_F(α^*) = -q^{-1} ∂_E(α)^* \]
can be used to show that \( D \) is self-adjoint \cite[Lemma A.1]{Senior_2011}.

There is an action of \( SU_q(2) \) on \( \bfS^2_q \) given by the restriction of the coaction of \( C(SU_q(2)) \) on itself to \( C(\bfS^2_q) \). The spectral triple \( (\mathcal{O}(S_q^2), L^2(S_+ ⊕ S_-), D) \) is constructed to be isometric with respect to this action, cf. \cite[§4]{Dabrowski_2003}. We can phrase this in terms of a right coaction
\begin{align*}
	δ_Δ : α & ↦ Σ (R ⊗ 1) Δ α & t^l_{i j} & ↦ \sum_k t^l_{k j} ⊗ q^{-i + k} {t^l_{k i}}^*
\end{align*}
of \( C(SU_q(2)) \) on \( C(\bfS^2_q) \), where \( R \) is the unitary antipode. We can write the admissible unitary as
\[ V_Δ (t^l_{i j} ⊗ t^{l'}_{i' j'}) = \sum_k t^l_{k j} ⊗ q^{-i + k} {t^l_{k i}}^* t^{l'}_{i' j'} . \]
We then have
\begin{multline*}
	(∂_E ⊗ 1) V_Δ (t^l_{i j} ⊗ t^{l'}_{i' j'}) = \sum_k κ^l_{j+1} t^l_{k, j+1} ⊗ q^{-i + k} {t^l_{k i}}^* t^{l'}_{i' j'} \\
	= κ^l_{j+1} V_Δ (t^l_{i, j+1} ⊗ t^{l'}_{i' j'}) = V_Δ (∂_E ⊗ 1) (t^l_{i j} ⊗ t^{l'}_{i' j'})
\end{multline*}
and, similarly, that \( ∂_F ⊗ 1 \) commutes with \( V_Δ \), which means that \( (C(S_q^2), L^2(S_+ ⊕ S_-), D) \) is isometrically equivariant for the action of \( SU_q(2) \).

In addition, there is an action of \( \widehat{SU_q(2)} \) on \( \bfS^2_q \) given by the restriction of the adjoint action of \( C(SU_q(2)) \) on itself to \( C(\bfS^2_q) \) \cite[§4]{Voigt_2011}. Together, these actions give an action of \( SL_q(2) = SU_q(2) ⋈ \widehat{SU_q(2)} \), the Drinfeld double of \( SU_q(2) \), which can be thought of as the quantisation of the classical Lorentz group \( SL(2, \bbC) \) action on the sphere \( \bfS^2 \). The left adjoint action of \( C(SU_q(2)) \) is given by
\[ \ad(α) : β \to α_{(1)} β S(α_{(2)})\,. \]
For \( z ∈ \bbC \), we define a slightly adjusted action
\[ ω_z(α) : β \to α_{(1)} β (K^{2 z} ⇀ S(α_{(2)}))\,. \]
For any \( α ∈ C(SU_q(2)) \), \( ω_z(α) \) preserves the subalgebra \( C(\bfS^2_q) \) and its spinor bundles. In terms of the Peter–Weyl basis,
\begin{align*}
	ω_z(t^l_{i, j})(β) 
	 = \sum_k q^{-2 z k} t^l_{i, k} β {t^l_{j, k}}^*\quad\mbox{and}\quad
	ω_z({t^l_{i, j}}^*)(β) 
	& = \sum_k q^{2 ((z - 1) k + j)} {t^l_{i, k}}^* β t^l_{j, k}\, .
\end{align*}
With respect to the inner product on \( C(SU_q(2)) \) given by the Haar state \( ϕ \), \( ω_1 \) is self-adjoint; in general,
\begin{align*}
	⟨ω_z(α)(β) \mid γ⟩
	& = ⟨β \mid ω_{-z+2}(α^*)(γ)⟩ .
\end{align*}
From the left action \( ω_1 \) of \( C(SU_q(2)) \) on itself, we obtain a right coaction of \( C^*(SU_q(2)) \) on \( C(SU_q(2)) \) by the formula
\[ β_{(0)} (β_{(1)}, α) = ω_1(α)(β) , \]
using the Sweedler notation \( δ_{ω_1}(β) = β_{(0)} ⊗ β_{(1)} \) for the coaction. In particular, we obtain that
\[ δ_{ω_1}(t^l_{i, j}) = \sum_{l', i', j'} ω_1(t^{l'}_{i', j'})(t^l_{i, j}) ⊗ τ^{l'}_{i', j'}\,. \]
The admissible unitary \( V_{ω_1} \) on \( L^2(S_+ ⊕ S_-) ⊗ C^*(SU_q(2)) \) is given by
\[ 
	V_{ω_1} = \sum_{l, i, j} ω_1(t^l_{i, j}) ⊗ τ^l_{i, j} = \sum_k q^{-2k} t^l_{i, k} \cdot {{t^l_{j, k}}}^* ⊗ τ^l_{i, j} = \sum_{k, k'} q^{-2k'} t^l_{i, k} \cdot {{t^l_{j, k'}}}^* ⊗ τ^l_{i, k} τ^l_{k', j}\,. 
\]
We claim that the spectral triple \( (C(S_q^2), L^2(S_+ ⊕ S_-), D) \) is conformally \( \widehat{SU_q(2)} \)-equivariant. The conformal geometry of the Podleś sphere is examined at the level of bounded KK-theory in \cite{Nest_2010, Voigt_2011}. Because \( \widehat{SU_q(2)} \) is discrete, the conformal factor \( μ \) will be the sum of components \( μ^l ∈ B(L^2(S_+ ⊕ S_-)) ⊗ M_{2l}(\bbC) \), \( l ∈ \frac{1}{2}\bbN_{≥1} \) labelling the irreducible representations of \( SU_q(2) \). Noting that \( C(\bfS^2_q) \) is unital, conformal equivariance will be a consequence of
\[ V_{ω_1}^l (D ⊗ 1) V_{ω_1}^{l *} - μ^l (D ⊗ 1) μ^{l *} \qquad [D ⊗ 1, μ^l] \]
being bounded for all \( l ∈ \frac{1}{2} \bbN_{≥1} \).

Note that \( (K ⊗ K) ⇀ (1 ⊗ S) Δ(α) = (1 ⊗ S) Δ(α) \) because
\begin{align*}
	(K ⊗ K) ⇀ (1 ⊗ S) Δ(t^l_{i, j})
	& = \sum_k K ⇀ t^l_{i, k} ⊗ K ⇀ S(t^l_{k, j}) \\
	& = \sum_k q^k t^l_{i, k} ⊗ K ⇀ {t^l_{k, j}}^* \\
	& = \sum_k q^k t^l_{i, k} ⊗ (K^{-1} ⇀ t^l_{k, j})^* \\
	& = \sum_k t^l_{i, k} ⊗ {t^l_{k, j}}^* \\
	& = (1 ⊗ S) Δ(t^l_{i, j})\,.
\end{align*}
Then
\begin{align*}
	∂_E(ω_z(α)(β))
	& = ∂_E(α_{(1)} β (K^{2 z} ⇀ S(α_{(2)}))) \\
	& = ∂_E(α_{(1)}) (K ⇀ β) (K^{2 z + 1} ⇀ S(α_{(2)})) + (K^{-1} ⇀ α_{(1)}) ∂_E(β) (K^{2 z + 1} ⇀ S(α_{(2)})) \\
	& \qquad + (K^{-1} ⇀ α_{(1)} β) ∂_E(K^{2 z} ⇀ S(α_{(2)})) \\
	& = ∂_E(α_{(1)}) (K ⇀ β) (K^{2 z + 1} ⇀ S(α_{(2)})) + ω_{z+1}(α)(∂_E(β)) \\
	& \qquad + (K^{-1} ⇀ α_{(1)} β) ∂_E(K^{2 z} ⇀ S(α_{(2)}))
\end{align*}
so that \( ∂_E ω_z(α) - ω_{z+1}(α) ∂_E \) and \( ∂_F ω_z(α) - ω_{z+1}(α) ∂_F \), similarly, are bounded on \( S_+ ⊕ S_- \). Furthermore,
\begin{align*}
	\sum_j ω_0(t^l_{i, j})(ω_1({t^l_{i', j}}^*)(β))
	& = \sum_{j, k} t^l_{i, k} ω_1({t^l_{i', j}}^*)(β) {t^l_{j, k}}^* \\
	& = \sum_{j, k, k'} q^{2 j} t^l_{i, k} {t^l_{i', k'}}^* β t^l_{j, k'} {t^l_{j, k}}^* \\
	& = \sum_{j, k, k'} q^{2 k'} t^l_{i, k} {t^l_{i', k'}}^* β (K^{-2} ⇀ t^l_{j, k'} ↼ K^2) S(t^l_{k, j}) \\
	& = \sum_{j, k, k'} q^{2 k'} t^l_{i, k} {t^l_{i', k'}}^* β (K^{-2} ⇀ (t^l_{j, k'} (K^2 ⇀ S(t^l_{k, j}) ↼ K^{-2})) ↼ K^2) \displaybreak[0] \\
	& = \sum_{j, k, k'} q^{2 k'} t^l_{i, k} {t^l_{i', k'}}^* β (K^{-2} ⇀ (t^l_{j, k'} S^{-1}(t^l_{k, j})) ↼ K^2) \\
	& = \sum_k q^{2 k} t^l_{i, k} {t^l_{i', k}}^* β (K^{-2} ⇀ 1 ↼ K^2) \\
	& = \sum_k q^{2 k} t^l_{i, k} {t^l_{i', k}}^* β \\
	& = ω_{-1}(t^l_{i, i'})(1) β\,.
\end{align*}
Let \( μ^l = \sum_{i, j} ω_{-1/2}(t^l_{i, j})(1) ⊗ τ^l_{i, j} = \sum_{i, j, k} q^k t^l_{i, k} {t^l_{j, k}}^* ⊗ τ^l_{i, j} \).  For \( l = \frac{1}{2} \),
\[ 
	μ^{\frac{1}{2}}
	= q^{\frac{1}{2}} T^{\frac{1}{2}}_{\frac{1}{2}} {T^{\frac{1}{2}}_{\frac{1}{2}}}^* + q^{-\frac{1}{2}} T^{\frac{1}{2}}_{-\frac{1}{2}} {T^{\frac{1}{2}}_{-\frac{1}{2}}}^*
	= q^{\frac{1}{2}} \begin{pmatrix} q^2 A & -B \\ -B^* & 1 - A \end{pmatrix} + q^{-\frac{1}{2}} \begin{pmatrix} 1 - q^2 A & B \\ B^* & A \end{pmatrix}. 
\]
Thus, with $P_\pm$ the projections onto the positive and negative spinors, $\mu^{1/2}=q^{1/2}P_++q^{-1/2}P_-$. If we regard \( K \) as an unbounded operator on \( C^*(SU_q(2)) \) the conformal factor is 
\[ μ = W (1 ⊗ K) W^* \]
where \( W \) is the multiplicative unitary of \( SU_q(2) \).

We remark that \( μ^l \) is positive and \( (μ^l)^z = \sum_{i, j} ω_{-z/2}(t^l_{i, j})(1) ⊗ τ^l_{i, j} \). Because \( μ^l ∈ \mathcal{O}(\bfS^2_q) ⊗ M_{2 l}(\bbC) \), it is clear that \( [D ⊗ 1, μ^l] \) is bounded. We are now in a position to see also that
\begin{align*}
	& V_{ω_1}^l (D ⊗ 1) V_{ω_1}^{l *} - μ^l (D ⊗ 1) μ^{l *} \\
	& \qquad = \sum_{l, i, j, i', j'} \left( ω_1(t^l_{i, j}) D ω_1({t^l_{i', j'}}^*) - ω_{-1/2}(t^l_{i, j})(1) D ω_{-1/2}(t^l_{i', j'})(1) \right) ⊗ τ^l_{i, j} τ^l_{j', i'} \\
	& \qquad = \sum_{l, i, j, i'} \left( ω_1(t^l_{i, j}) D ω_1({t^l_{i', j}}^*) - ω_{-1/2}(t^l_{i, j})(1) D ω_{-1/2}(t^l_{i', j})(1) \right) ⊗ τ^l_{i, i'} \\
	& \qquad = \sum_{l, i, j, i'} \left( - (D ω_0(t^l_{i, j}) - ω_1(t^l_{i, j}) D) ω_1({t^l_{i', j}}^*) + [D, ω_{-1/2}(t^l_{i, j})(1)] ω_{-1/2}(t^l_{i', j})(1) \right) ⊗ τ^l_{i, i'}
\end{align*}
is bounded. Finally, we obtain that \( (\mathcal{O}(S_q^2), L^2(S_+ ⊕ S_-), D) \) is conformally \( \widehat{SU_q(2)} \)-equivariant with conformal factor \( μ \).

The locally compact quantum group \( SL_q(2) \), the quantum deformation of \( SL(2, \bbC) \), is the Drinfeld double \( SU_q(2) ⋈ \widehat{SU_q(2)} \); see e.g. \cite[§4.4.1]{Voigt_2020}. As C*-algebras,
\[ C(SL_q(2)) = C(SU_q(2)) ⊗ C^*(SU_q(2)) . \]
The comultiplication on \( C(SL_q(2)) \) is
\[ Δ_{SL_q(2)} = (1 ⊗ Σ ⊗ 1) (\id ⊗ \ad(W) ⊗ \id) \circ (Δ ⊗ \hat{Δ}) \]
and the antipode is
\[ S_{SL_q(2)} = \ad(W^*) \circ (S ⊗ \hat{S}) = (S ⊗ \hat{S}) \circ \ad(W) . \]
By \cite[Theorem 5.3]{Baaj_2005} the unitary antipode is similarly
\[ R_{SL_q(2)} = \ad(W^*) \circ (R ⊗ \hat{R}) = (R ⊗ \hat{R}) \circ \ad(W) . \]
Our conventions differ from those of \cite{Nest_2010} in that we use right coactions rather than left ones. The translation between these is not difficult: a left coaction can be turned into a right coaction, and vice versa, by applying the unitary antipode to the C*-bialgebra leg and then flipping the legs. Taking this into account in \cite[Proposition 3.2]{Nest_2010} the action of \( SL_q(2) \) on \( S_q^2 \) is given by the coaction
\begin{align*}
	δ_{⋈}
	& = (Σ ⊗ 1) (1 ⊗ Σ) (\ad(W^*) ⊗ \id) (R ⊗ \hat{R} ⊗ \id) (1 ⊗ Σ) (\id ⊗ \id ⊗ \hat{R}) (\id ⊗ δ_{ω_1}) Σ (\id ⊗ R) δ_Δ \\
	& = (Σ ⊗ 1) (1 ⊗ Σ) (\ad(W^*) ⊗ \id) (1 ⊗ Σ) (Σ ⊗ 1) (1 ⊗ Σ) (δ_{ω_1} ⊗ \id) δ_Δ \\
	& = (\id ⊗ \ad(W^*)) (1 ⊗ Σ) (δ_{ω_1} ⊗ \id) δ_Δ
\end{align*}
of \( C(SL_q(2)) \). Using the standard leg-numbering notation the admissible unitary is
\[ V_{⋈} = (1 ⊗ W^*) V_{ω_1, 13} (V_Δ ⊗ 1) (1 ⊗ W). \]
Let \( μ_{⋈} = (1 ⊗ W^*) μ_{13} (1 ⊗ W) \). Then
\[ V_{⋈} (D ⊗ 1) V_{⋈}^* - μ_{⋈} (D ⊗ 1) μ_{⋈}^* = (1 ⊗ W^*) \left( V_{ω_1, 13} (D ⊗ 1 ⊗ 1) V_{ω_1, 13}^* - μ_{13} (D ⊗ 1 ⊗ 1) μ_{13}^* \right) (1 ⊗ W) \]
is \( C(SU_q(2)) ⊗ C^*(SU_q(2)) \)-matched because it is bounded when restricted to each of the submodules \(L^2(S_+\oplus S_-)\otimes C(SU_q(2)) ⊗ M_{2l}(\bbC) \).
In terms of the Peter–Weyl basis,
\begin{align*}
	μ_{⋈}
	& = \sum_{i, j, k, l, i', j', i'', j''} q^k t^l_{i, k} {t^l_{j, k}}^* ⊗ {t^l_{i'', j''}}^* t^l_{i', j'} ⊗  τ^l_{j'', i''} τ^l_{i, j} τ^l_{i', j'} \\
	& = \sum_{i, j, k, l, m, n} q^k t^l_{i, k} {t^l_{j, k}}^* ⊗ {t^l_{i, m}}^* t^l_{j, n} ⊗  τ^l_{m, n} .
\end{align*}
This shows that the first leg of \( μ_{⋈} \) is in \( \mathcal{O}(\bfS^2_q) \) so that \( [D ⊗ 1 ⊗ 1, μ_{⋈}] \) is similarly \( C(SL_q(2)) \)-matched. Regarding \( K \) as an unbounded operator on \( C^*(SU_q(2)) \), the conformal factor is 
\[ μ_{⋈} = W^*_{23} W_{13} (1 ⊗ 1 ⊗ K) W^*_{13} W_{23} . \]
We have now demonstrated

\begin{proposition}
	The spectral triple \( (C(S_q^2), L^2(S_+ ⊕ S_-), D) \) is conformally \( SL_q(2) \)-equivariant with conformal factor \( μ_{⋈} \).
\end{proposition}

\begin{remark}
As a consequence, applying Theorem \ref{theorem:bbd-transform-quantum-conformal-log}, the logarithmically dampened spectral triple \( (C(S_q^2), L^2(S_+ ⊕ S_-), L_D) \) is uniformly \( SL_q(2) \)-equivariant. Recalling the expressions for \( \partial_E \) and \( \partial_F \) in terms of the Peter–Weyl basis,
\[ D \begin{pmatrix} t^l_{i, \frac{1}{2}} \\ t^{l'}_{i', -\frac{1}{2}} \end{pmatrix} = [l+1/2]_q \begin{pmatrix}  t^{l'}_{i', \frac{1}{2}} \\ t^l_{i, -\frac{1}{2}} \end{pmatrix} . \]
One can check that
\[ \frac{[l+\frac{1}{2}]_q}{\sqrt{1+[l+\frac{1}{2}]_q^2}} \log \sqrt{1+[l+\tfrac{1}{2}]_q^2} - (l+\tfrac{1}{2}) \log q^{-1} \]
converges to \( \log(q^{-1} - q) \) as \( l \to \infty \). Hence, up to a bounded difference, \( L_D \) is equal to \( \log(q^{-1}) D_1 \), where \( D_1 \) is the Dirac operator on the classical 2-sphere; cf. \cite{D_Andrea_2007}.
\end{remark}

\section{Conformally generated cycles and twisted spectral triples}
\label{sec:cgc}

In this section, we present a new way of guaranteeing that unbounded cycles without bounded commutators in the conventional sense have well-defined bounded transforms. In particular, our approach covers all known examples of twisted spectral triples with well-defined bounded transforms. One of the features of our approach is that no `twist' or automorphism of the algebra is involved, which suggests that this structure is a red herring, at least as far as KK-theory is concerned.

So far, relatively few examples of twisted spectral triples have been described in the literature. One reason for this is the difficulty in guaranteeing that the bounded transform is well-defined. The \emph{Lipschitz regularity} condition \cite[Definition 3.1 (3.3)]{Connes_2008}, although natural in a relatively classical situation, where a pseudodifferential calculus is available, is not so satisfactory in general. Part of the motivation for developing the technical results in this paper was the construction of twisted spectral triples for certain badly behaved dynamical systems, for which Lipschitz regularity becomes intractable.

The framework of conformally generated cycles is applicable to all examples of twisted spectral triples with topological content in the literature, as far as we are aware. Among those examples to which it can be applied are
\begin{itemize}
	\item Conformal perturbations of spectral triples (or Kasparov modules) of the \( D \leadsto k D k \) type \cite[§2.2]{Connes_2008};
	\item Crossed products by groups of conformal diffeomorphisms \cite[§2.3]{Connes_2008} \cite[§3.1]{Moscovici_2010} (and, more generally, the dual Green–Julg map of conformally equivariant unbounded Kasparov modules);
	\item Cuntz–Krieger algebras, as in \cite[Chapter 6]{Hawkins_2013a};
	\item Unbounded modular cycles, in the sense of \cite[Definition 3.1]{Kaad_2021}; and
	\item Pseudodifferential calculus on the Podleś sphere and other examples with \emph{diagonalisable twist}, as treated in \cite{Matassa_2019}.
\end{itemize}
The multiplicative perturbation theory developed in §\ref{section:multiplicative-perturbations} was partly inspired by \cite{Matassa_2019}. In principle, the techniques here could be used to build pseudodifferential calculi, mimicking the approach in \cite{Matassa_2019}. Examples of twisted spectral triples to which our methods do not apply are
\begin{itemize}
	\item The quantum statistical mechanics constructions of \cite{Greenfield_2014} which are not Lipschitz regular and, indeed, whose bounded transform is manifestly not a Fredholm module;
	\item The Lorentzian geometry constructions of \cite{Devastato_2018}, whose twist is an involution and not relevant to the topology; and
	\item Examples without (locally) compact resolvent, such as those in \cite{Kaad_2012a} and \cite{Iochum_2016}.
\end{itemize}

To formulate a framework sufficient to describe the examples, we will again use the notions of matched operators and compactly supported states from Appendices \ref{appendix:temp-bdd-operators} and \ref{subsec:compact-states}. Recall from 
Proposition \ref{prop:star-alg} the $*$-algebra of matched operators $\Mtc^*(F, C)$ on the module $F$ with respect to the algebra $C$. 

\begin{definition}
	\label{definition:conformally-generated-cycle-general}
	A \emph{conformally generated \( A \)-\( B \)-cycle} $(A,E_B,D; C, \mu)$ is an \( A \)-\( B \)-correspondence \( E \), a regular operator \( D \) on \( E \), a C*-algebra \( C \), and a pair \( μ = (μ_L, μ_R) \) of (even) \( C \)-matched operators on \( E ⊗ C \), whose inverses are also \( C \)-matched, such that
	\begin{enumerate}
		\item \( D \) is self-adjoint;
		\item \( (1 + D^2)^{-1} a \) is compact for all \( a ∈ A \); and
		\item With \( \mathscr{L} \) the set of \( a ∈ \Mtc^*(E ⊗ C, C) \) such that
			\[ [D ⊗ 1, a] \quad [μ_L (D ⊗ 1) μ_L^*, a] \quad [D ⊗ 1, μ_L^* a] \quad [D ⊗ 1, μ_L^{-1} a] \quad [D ⊗ 1, a μ_L] \quad [D ⊗ 1, a μ_L^{-1 *}] \]
			are \( C \)-matched, with \( \mathscr{R} \) the set of \( a ∈ \Mtc^*(E ⊗ C, C) \) such that
			\[ [D ⊗ 1, a] \quad [μ_R (D ⊗ 1) μ_R^*, a] \quad [D ⊗ 1, μ_R^* a] \quad [D ⊗ 1, μ_R^{-1} a] \quad [D ⊗ 1, a μ_R] \quad [D ⊗ 1, a μ_R^{-1 *}] \]
			are \( C \)-matched, and with
			\[ \mathscr{T} = \left\{ a ∈ \Mtc^*(E ⊗ C, C) \middle| \, μ_L (D ⊗ 1) μ_L^* a - a μ_R (D ⊗ 1) μ_R^* ∈ \Mtc^*(E ⊗ C, C) \right\} , \]
			the algebra \( A \) is contained in \( C^*((1 ⊗ ψ) (\mathscr{L T R}) |\, ψ ∈ \mathcal{S}_c(C)) \), where $\mathcal{S}_c(C)$ are the compactly supported states on $C$.
	\end{enumerate}
	If \( E \) is a \( \bbZ/2\bbZ \)-graded \( A \)-\( B \)-correspondence (that is, with \( A \) acting by even operators), we require that \( D \) be an odd operator and that \( μ_L \) and \( μ_R \) be even and call \( (A, E_B, D; C, \mu) \) an \emph{even} conformally generated cycle. If \( E \) is ungraded, \( (A, E_B, D; C, \mu) \) is \emph{odd}.
\end{definition}

\begin{remarks}
	\item The spaces \( \mathscr{L} \) and \( \mathscr{R} \) are $*$-algebras. The space \( \mathscr{T} \) is a ternary ring of \( C \)-matched operators. We have \( \mathscr{L T} ⊆ \mathscr{T} \) and \( \mathscr{T R} ⊆ \mathscr{T} \), and \( \mathscr{L T R} \) is also a ternary ring of \( C \)-matched operators.
	\item Proposition \ref{proposition:compactly-supp-state-temp-bdd} shows that the application of a compactly supported state on \( C \) to a \( C \)-matched operator is well-defined. By Proposition \ref{proposition:tensor-product-compactly-supported-states-dense}, \( \mathcal{S}_c(C) \) in condition {3.} of Definition \ref{definition:conformally-generated-cycle-general} could be replaced with \( \mathcal{S}(C) \), the set of all states on \( C \), at least to those elements of \( \mathscr{L T R} \) which are adjointable.
	\item Any unbounded Kasparov module \( (A, E_B, D) \) can be regarded as a conformally generated cycle \( (A, E_B, D; \bbC, (1, 1)) \).
\end{remarks}

One should think of conformally generated cycles as having a dynamical quality, in addition to a strictly geometrical one, with the C*-algebra \( C \) as a `dynamical direction'. In examples, the elements of $\mathscr{T}$ correspond to endomorphisms with bounded `twisted' commutators with $D$, as we will see in Theorem \ref{theorem:used-to-be-twisted}. Elements of $\mathscr{L},\mathscr{R}$ encode the regularity of the `conformal factors' $\mu_L,\mu_R$.

Definition \ref{definition:conformally-generated-cycle-general} could be extended to higher order cycles but, in the interests of readability, we do not pursue this here.

\begin{remark}
	\label{remark:formerly-known-as-twitter}
	Using Proposition \ref{proposition:temp-bbd-operators-topological}, we may specialise Definition \ref{definition:conformally-generated-cycle-general} to the case when \( C = C_0(X) \) for a locally compact Hausdorff space \( X \). Consider a conformally generated \( A \)-\( B \)-cycle \( (A, E_B, D; C_0(X), \mu) \). We may interpret \( μ = (μ_L, μ_R) \) as a pair of $*$-strongly continuous families \( (μ_{L, x})_{x∈X} \) and \( (μ_{R, x})_{x∈X} \) of (even) invertible adjointable operators over \( X \). Condition {3.} of Definition \ref{definition:conformally-generated-cycle-general} becomes:
	\begin{enumerate}[start=3, label={\arabic*'.}]
		\item With \( \mathscr{L} \) the set of $*$-strongly continuous maps \( a : X \to \End^*(E) \) such that the maps
			\begin{gather*}
				x ↦ [D, a_x] \qquad x ↦ [μ_{L, x} D μ_{L, x}^*, a_x] \\
				x ↦ [D, μ_{L, x}^* a_x] \qquad x ↦ [D, μ_{L, x}^{-1} a_x] \qquad x ↦ [D, a_x μ_{L, x}] \qquad x ↦ [D, a_x μ_{L, x}^{-1 *}]
			\end{gather*}
			are $*$-strongly continuous to $\End^*(E) $, with \( \mathscr{R} \) the set of $*$-strongly continuous maps \( a : X \to \End^*(E) \) such that the maps
			\begin{gather*}
				x ↦ [D, a_x] \qquad x ↦ [μ_{R, x} D μ_{R, x}^*, a_x] \\
				x ↦ [D, μ_{R, x}^* a_x] \qquad x ↦ [D, μ_{R, x}^{-1} a_x] \qquad x ↦ [D, a_x μ_{R, x}] \qquad x ↦ [D, a_x μ_{R, x}^{-1 *}]
			\end{gather*}
			are $*$-strongly continuous to $\End^*(E) $, and with
			\[ \mathscr{T} = \{ a ∈ C(X, \End^*(E)_{*-s}) |\, x ↦ μ_{L, x} D μ_{L, x}^* a_x - a_x μ_{R, x} D μ_{R, x}^* ∈ C(X, \End^*(E)_{*-s}) \} , \]
			the algebra \( A \) is contained in \( C^*((1 \otimes m)(\mathscr{L} \mathscr{T} \mathscr{R}) |\, m ∈ \mathcal{M}_c(X)) \), where \( \mathcal{M}_c(X) \) is the set of compactly supported Radon measures on \( X \).
	\end{enumerate}
	An important special case is when $X$ is a discrete set (and, in particular, when $X$ is a point). In this case, Condition {3.} of Definition \ref{definition:conformally-generated-cycle-general} becomes:
		\begin{enumerate}[start=3, label={\arabic*".}]
		\item With \( \mathscr{L}_x \) the set of \( a ∈ \End^*(E) \) such that
			\[ [D, a] \qquad [μ_{L, x} D μ_{L, x}^*, a] \qquad [D, μ_{L, x}^* a] \qquad [D, μ_{L, x}^{-1} a] \qquad [D, a μ_{L, x}] \qquad [D, a μ_{L, x}^{-1 *}] \]
			are adjointable, with \( \mathscr{R}_x \) the set of \( a ∈ \End^*(E) \) such that
			\[ [D, a] \qquad [μ_{R, x} D μ_{R, x}^*, a] \qquad [D, μ_{R, x}^* a] \qquad [D, μ_{R, x}^{-1} a] \qquad [D, a μ_{R, x}] \qquad [D, a μ_{R, x}^{-1 *}] \]
			are adjointable, and with
			\[ \mathscr{T}_x = \{ a ∈ \End^*(E) |\, μ_{L, x} D μ_{L, x}^* a - a μ_{R, x} D μ_{R, x}^* ∈ \End^*(E) \} , \]
			the algebra \( A \) is contained in the C*-algebra \( C^*(\mathscr{L}_x \mathscr{T}_x \mathscr{R}_x |\, x ∈ X) \).
	\end{enumerate}
\end{remark}

\begin{theorem}
	\label{theorem:C-gen-kas-mod}
	Let \( (A, E_B, D; C, \mu) \) be a conformally generated $A$-$B$-cycle. Then \( (A, E_B, F_D) \) is a bounded Kasparov module of the same parity.
\end{theorem}
\begin{proof}
	The main point to check is that \( [F_D, a] \) is compact for all \( a ∈ A \). Let \( c \) be a positive element of the Pedersen ideal \( K_C \), so that, by Proposition \ref{proposition:temp-bdd-on-ideals}, the restriction of \( μ \) to the \( B ⊗ \overline{\Span}(C c C) \)-module \( E ⊗ \overline{\Span}(C c C) \) is bounded. From now on, we work on the module \( E ⊗ \overline{\Span}(C c C) \). Let \( l_1, l_2 ∈ \mathscr{L} \) and \( r_1, r_2, r_3 ∈ \mathscr{R} \).  Omitting instances of $\otimes 1$ for simplicity, Theorem \ref{theorem:conformal-result-nonunital} shows that
	\[ (F_{μ_L D μ_L^*} - F_D) l_1 l_2 ⟨μ_L D μ_L^*⟩^β \qquad (F_{μ_R D μ_R^*} - F_D) r_1 r_2 r_3 ⟨D⟩^β \]
	are bounded for \( β < 1 \). With \( l = l_1 l_2 \) and \( r=r_1 r_2 r_3 \),
	\[ (F_D l - l F_{μ_L D μ_L^*}) ⟨μ_L D μ_L^*⟩^β \qquad (F_{μ_R D μ_R^*} r - r F_D) ⟨D⟩^β \]
	are hence bounded. Let \( t ∈ \mathscr{T} \). By Proposition \ref{proposition:bounded-transform-of-commutator-additive-perturbation},
	\[ (F_{μ_L D μ_L^*} t - t F_{μ_R D μ_R^*}) ⟨μ_R D μ_R^*⟩^β \]
	is bounded and we have
	\begin{align*}
		[F_D, l t r]
		& = (F_D l - l F_{μ_L D μ_L^*}) t r + l (F_{μ_L D μ_L^*} t - t F_{μ_R D μ_R^*}) r + l t (F_{μ_R D μ_R^*} r - r F_D) .
	\end{align*}
	We see that \( [F_D ⊗ 1, l t r] ⟨D⟩^β ⊗ 1 \) is bounded on the module \( E ⊗ \overline{\Span}(C c C) \). This is the case for every positive \( c ∈ K_C \) so, by Proposition \ref{proposition:temp-bdd-on-ideals}, \( [F_D ⊗ 1, l t r] ⟨D⟩^β ⊗ 1 \) is a \( C \)-matched operator on \( E ⊗ C \). Let \( ψ \) be a compactly supported state on \( C \). By Proposition \ref{proposition:compactly-supp-state-temp-bdd}, we may apply $1\otimes\psi$ to \( [F_D ⊗ 1, l t r] ⟨D⟩^β ⊗ 1 \) to obtain the bounded operator
	\[ (1 ⊗ ψ)([F_D ⊗ 1, l t r] ⟨D⟩^β ⊗ 1) = [F_D, (1 \otimes ψ)(l t r)] ⟨D⟩^β . \]
	For \( a ∈ A \) the operator
	\[ [F_D, 1\otimes ψ(l t r)] a = [F_D, (1 \otimes ψ)(l t r)] ⟨D⟩^β ⟨D⟩^{-β} a \]
	is compact. Using the Leibniz rule and taking norm limits, \( [F_D, b] \) is compact for all \( b ∈ C^*((1 \otimes ψ)(\mathscr{L T R}) \mid ψ ∈ \mathscr{S}_c(C)) \), which includes \( A \).
\end{proof}

We now consider conformal perturbations of unbounded Kasparov modules, which include the conformal perturbations of noncommutative tori \cite[§2.2]{Connes_2008}.

\begin{theorem}
\label{theorem:used-to-be-twisted}
	Let \( (A, E_B, D) \) be an unbounded Kasparov module. Let \( k \) be an invertible normal element of \( \End^*(E) \). Suppose that \( \overline{\Span}(\mathscr{M} A \mathscr{M}) ⊇ A \) where \( \mathscr{M} \) is the set of \( a ∈ \End^*(E) \) such that
	\[ [k D k^*, a] \qquad [D, a] \qquad [D, k^*] a \qquad [D, k^* k] a \qquad a [D, k] \qquad a [D, k^* k] \]
	are bounded. Then \( (A, E_B, k D k^*; \bbC, (k^{-1}, k^{-1})) \) is a conformally generated cycle. In particular, if $k$ is normal and invertible and \( (A, E_B, D) \) is an unbounded Kasparov module with $[D,k]$ bounded then the data \( (A, E_B, k D k^*; \bbC, (k^{-1}, k^{-1})) \) define a conformally generated cycle. Hence \( (A, E_B, F_{k D k^*}) \) is a Kasparov module and $[(A, E_B, F_{k D k^*})]=[(A, E_B, F_{D})]\in KK(A,B)$.
	\end{theorem}
\begin{proof}
	It is straightforward to check that, for all \( a ∈ \mathscr{M} \),
	\[ [k D k^*, a] \qquad [D, a] \qquad [k D k^*, k^{-1 *} a] \qquad [k D k^*, k a] \qquad [k D k^*, a k^{-1}] \]
	are bounded so that \( \mathscr{M} ⊆ \mathscr{L} ∩ \mathscr{R} \) where $\mathscr{L},\mathscr{R}$ are as in Definition \ref{definition:conformally-generated-cycle-general}. As $A ⊆ \overline{\mathscr{T}} = \overline{\Lip_0^*(D)}$, we are done.	
	For the final statements, if $[D,k]$ is bounded then $\mathscr{M}$ contains scalar multiples of the identity and so $\overline{\Span}(\mathscr{M} A \mathscr{M}) ⊇ A $. An application of Theorem \ref{theorem:conformal-result-nonunital}
	gives the equality of the Kasparov classes.
\end{proof}

\begin{example}
\label{eg:torus-cgc}
	We recall the noncommutative torus \( C(\bbT^2_α) \) from Example \ref{eg:torus-ct} and the spectral triple
	\[ \left( C(\bbT^2_α), L^2(\bbT^2_α) ⊗ \bbC^2, D := \left(\begin{smallmatrix} & δ_1 + τ δ_2 \\ δ_1 + \bar{τ} δ_2 & \end{smallmatrix}\right) \right). \]
	As in Example \ref{eg:torus-ct}, choose a positive invertible element \( k ∈ C(\bbT^2_α) \) in the domains of \( δ_1 \) and \( δ_2 \). Using left multiplication by $k$ yields a conformally generated cycle
	 \[ \left( C(\bbT^2_α), L^2(\bbT^2_α) ⊗ \bbC^2, kDk; \bbC,(k^{-1},k^{-1})\right). \]
	 Thus the classes defined by $F_{D}$ and $F_{kDk}$ in $KK(C(\bbT^2_α),\bbC)$ coincide.

	 The unbounded  Kasparov module $(C(\bbT^2_α), L^2(C(\bbT^2_α),\Phi)_{C(\bbT)}, \delta_2)$   also gives rise to a conformally generated cycle 
	\[ 
	 (C(\bbT^2_α),L^2(C(\bbT^2_α),\Phi)_{C(\bbT)},k\delta_2k;\bbC,(k^{-1},k^{-1}))
	 \]
	where \( k ∈ C(\bbT^2_α) \) is now a positive invertible element  in the domain of \( δ_2 \).
	Thus the classes defined by $F_{\delta_2}$ and $F_{k\delta_2 k}$ in $KK(C(\bbT^2_α),C(\bbT))$ coincide.
\end{example}

Next we consider unbounded modular cycles in the sense of \cite[Definition 3.1]{Kaad_2021} \cite[Definition 8.1]{Kaad_2024}. Using our methods the bounded transform can be achieved in greater generality. Compare Proposition \ref{proposition:singular-conformal-transformation-delta}.

\begin{proposition}
	\label{proposition:conformally-generated-delta}
	Let \( E \) be an \( A \)-\( B \) correspondence. Let \( D \) be a self-adjoint regular operator and \( Δ_+ \) and \( Δ_- \) a pair of commuting positive adjointable operators on \( E \) such that
	\begin{itemize}
		\item For all \( a ∈ A \), \( (1 + D^2)^{-1} a \) is compact and the sequence \( (a (Δ_+ + Δ_-) (Δ_+ + Δ_- + \frac{1}{n})^{-1})_{n=1}^∞ \) converges in norm to (the representation of) \( a \);
		\item \( \{ Δ_+, Δ_- \} \dom D ⊆ \dom D \) and \( [D, Δ_+] \), \( [D, Δ_-] \) are bounded; and
		\item \( A ⊆ \overline{\mathscr{N}} \), where \( \mathscr{N} \) is the set of \( a ∈ \End^*(E) \) such that \( a \dom D \subseteq \dom D \) and \( Δ_- D a Δ_+ - Δ_+ a D Δ_- \) is bounded.
	\end{itemize}
	Let \( (h_n)_{n ∈ \bbN_{≥1}} ⊆ C_b^∞(\bbR^×_+) \) be any sequence of positive functions with bounded reciprocals which agree with the function \( x ↦ x^{-1/2} \) on the interval \( [\frac{1}{n}, n] \). Then, with $\mu_{L,n}=\mu_{R,n}=h_n(Δ_+) h_n(Δ_-)^{-1}$, the data $ (A, E_B, D; C_0(\bbN_{≥1}), \mu) $ define a conformally generated cycle.
\end{proposition}
\begin{proof}
	First, by the smooth functional calculus of Theorem \ref{theorem:smooth-fc}, \( [D, h_n(Δ_+) h_n(Δ_-)^{-1}] \) is bounded so \( 1 ∈ \mathscr{L}_n, \mathscr{R}_n \) for every \( n ∈ \bbN_{≥1} \). Second, \( \mathscr{T}_n \) consists of those \( b ∈ \End^*(E) \) such that
	\[ [h_n(Δ_+) h_n(Δ_-)^{-1} D h_n(Δ_+) h_n(Δ_-)^{-1}, b] \]
	extends to an adjointable operator. Let $f_1, f_2, f_3, f_4 \in C_c^\infty((\tfrac{1}{n}, n))$ and $a\in \mathscr{N}$ and define $b ∈ \End^*(E)$ to be the product 
	\[ f_1(Δ_+) f_2(Δ_-) a f_3(Δ_+) f_4(Δ_-) ∈ C_0((\tfrac{1}{n}, n))(Δ_+) C_0((\tfrac{1}{n}, n))(Δ_-) \mathscr{N} C_0((\tfrac{1}{n}, n))(Δ_+) C_0((\tfrac{1}{n}, n))(Δ_-) . \]
	Then $b h_n(Δ_+) h_n(Δ_-)^{-1} = b Δ_+^{-1/2} Δ_-^{1/2}$ and, again using the smooth functional calculus,
	\begin{align*}
		& [h_n(Δ_+) h_n(Δ_-)^{-1} D h_n(Δ_+) h_n(Δ_-)^{-1}, b] \\
		& \qquad = f_1(Δ_+) f_2(Δ_-) Δ_+^{-1} (Δ_- D a Δ_+ - Δ_+ a D Δ_-) Δ_+^{-1} f_3(Δ_+) f_4(Δ_-) \\
		& \qquad\qquad + h_n(Δ_+) h_n(Δ_-)^{-1} \left[ D, Δ_+^{-1/2} Δ_-^{1/2} f_1(Δ_+) f_2(Δ_-) \right] a f_3(Δ_+) f_4(Δ_-) \\
		& \qquad\qquad + f_1(Δ_+) f_2(Δ_-) a \left[ D, Δ_-^{1/2} Δ_+^{-1/2} f_3(Δ_+) f_4(Δ_-)  \right] h_n(Δ_+) h_n(Δ_-)^{-1}
	\end{align*}
	is bounded. The closure of \( C_0((\tfrac{1}{n}, n))(Δ_+) C_0((\tfrac{1}{n}, n))(Δ_-) \) is \( C^*(Δ_+, Δ_-) \). By Lemma \ref{lemma:delta-approximate-unit-convergence}, \( A ⊆ \overline{A C^*(Δ_+, Δ_-)} \) so that
	\[ \overline{\mathscr{L T R}} ⊇ \overline{C^*(Δ_+, Δ_-) \mathscr{N} C^*(Δ_+, Δ_-)} ⊇ \overline{C^*(Δ_+, Δ_-) A C^*(Δ_+, Δ_-)} ⊆ A \]
	and we are done.
\end{proof}

As a last application we consider again the relation to the logarithmic transform.

\begin{proposition}
	\label{proposition:r5ybwe56uw56ubwe}
	cf.~\cite[Corollary 1.20]{Goffeng_2019a}
	Let \( (A, E_B, D) \) consist of a C*-algebra \( A \) represented on a Hilbert \( B \)-module \( E \) and a regular operator \( D \) on \( E \), such that
	\begin{itemize}
		\item \( D \) is self-adjoint;
		\item \( (1 + D^2)^{-1/2} a \) is compact for all \( a ∈ A \); and
		\item There is a dense subset of \( a ∈ A \) such that \( a \dom D ⊆ \dom D \) and \( [F_D, a] \log⟨D⟩ \) is bounded.
	\end{itemize}
	Then, with \( L_D = F_D \log ⟨D⟩ \), the triple \( (A, E_B, L_D) \) is an unbounded Kasparov module whose bounded transform is equal to \( (A, E_B, F_D) \) up to a locally compact difference.
\end{proposition}

\begin{theorem}
	Let \( (A, E_B, D;C, (\mu_L,\mu_R)) \) be a conformally generated cycle. Then \( (A, E_B, L_D) \) is an unbounded Kasparov module of the same parity.
\end{theorem}
\begin{proof}
	By the Proof of Theorem \ref{theorem:C-gen-kas-mod}, \( [F_D, (1 \otimes ψ)(l t r)] ⟨D⟩^β \) is bounded for \( ψ ∈ \mathcal{S}_c(C) \), \( l ∈ \mathscr{L}^2 \), \( t ∈ \mathscr{T} \), \( r ∈ \mathscr{R}^3 \), and \( β < 1 \). We have
	\begin{align*} l t r \dom(D ⊗ 1) (1 ⊗ K_C) & ⊆ l t μ_R^{-1 *} \dom(D ⊗ 1) (1 ⊗ K_C) ⊆ l μ_L^{-1 *} \dom(D ⊗ 1) (1 ⊗ K_C)\\
	& ⊆ \dom(D ⊗ 1) (1 ⊗ K_C) . 
	\end{align*}
	Hence \( (⟨D⟩ ⊗ 1) l t r (⟨D⟩^{-1} ⊗ 1) \) is \( C \)-matched. Applying Proposition \ref{proposition:compactly-supp-state-temp-bdd},
	\[ (1 \otimes ψ)(⟨D⟩ ⊗ 1) l t r (⟨D⟩^{-1} ⊗ 1) = ⟨D⟩ (1 \otimes ψ)(l t r) ⟨D⟩^{-1} \]
	is an adjointable operator on \( E \), and so \( (1 \otimes ψ)(l t r) \dom D ⊆ \dom D \). By Proposition \ref{proposition:logarithmic-transform}, the commutator \( [L_D, (1 \otimes ψ)(l t r)] \) is bounded. By the Leibniz rule, \( [L_D, b] \) is bounded for all \( b \) in the $*$-algebra generated by \( \left\{ (1 \otimes ψ)(\mathscr{L T R}) \mid ψ ∈ \mathscr{S}_c(C) \right\} \). This is dense in \( C^*((1 \otimes ψ)(\mathscr{L T R}) \mid ψ ∈ \mathscr{S}_c(C)) \), which includes \( A \).
\end{proof}

In principle, the logarithmic transform, if carried out piece-by-piece, could be used to produce KK-classes from `multi-twisted' spectral triples which have appeared in the literature, such as for quantum groups \cite{Kaad_2025} and dynamical systems \cite{Kaad_2020}. (See also \cite{Dabrowski_2022}, where an approach similar to that of \cite{Sitarz_2015} is used to obtain ordinary spectral triples from partial conformal rescalings.) An example of this can be found in \cite[\S II.4.1]{Masters_2025a}. In \cite{Fries_2025a}, a framework is developed to account for a different kind of `multidirectional' behaviour, under the name of \emph{tangled spectral triple}. However, it is not clear whether conformally generated cycles and tangled spectral triples can be reconciled; we leave this for the future.

\subsection{Descent and the dual Green–Julg map for conformal equivariance}
\label{subsec:conf-descent}

In the conformally equivariant setting, the descent map and the dual Green–Julg map produce conformally generated cycles.

\begin{proposition}
	\label{proposition:conformal-descent}
	Let $G$ be a locally compact group and let \( (A, E_B, D) \) be a \( (μ_g)_{g ∈ G} \)-conformally \( G \)-equivariant unbounded Kasparov module. Then, for \( t ∈ \{ u, r \} \), 
	\[ (A ⋊_t G, (E ⋊_t G)_{B ⋊_t G}, \tilde{D}; C_0(G), (1, \tilde{μ}_g)_{g ∈ G}) \]
	is a conformally generated cycle, where \( \tilde{D} \) is the regular operator given on \( ξ ∈ C_c(G, E) ⊆ E ⋊_t G \) by \( (\tilde{D} ξ)(h) = D(ξ(h)) \) and \( (\tilde{μ}_g)_{g ∈ G} \) are given by \( (\tilde{μ}_g ξ)(h) = μ_g (ξ(h)) \).
\end{proposition}
\begin{proof}
	The local compactness of the resolvent is the same as in the uniform case, Proposition \ref{proposition:descent-ordinary-unbounded}. Recall the spaces \( \mathscr{L} \), \( \mathscr{T} \), and \( \mathscr{R} \) of Remark \ref{remark:formerly-known-as-twitter}. It is straightforward to verify that the constant families \( (\tilde{d})_{g ∈ G} ∈ \mathscr{L} \) and \( (\tilde{b}^* \tilde{c})_{g ∈ G} ∈ \mathscr{R} \) for all \( d ∈ \Lip_0^*(D) \) and \( b, c ∈ \mathscr{Q} \). Let \( (u_g)_{g ∈ G} ⊆ \End_{B ⋊_t G}^*(E ⋊_t G) \) be the canonical unitaries implementing the group action, given by
	\[ (u_h ξ)(g) = U_h ξ(h^{-1} g) \]
	on \( ξ ∈ C_c(G, E) \) (where we recall the notation of Definition \ref{definition:hilbert-module-crossed-products}). A family of operators \( t \) is in \( \mathscr{T} \) if
	$ g ↦ \tilde{D} t_g - t_g \tilde{μ}_g \tilde{D} \tilde{μ}_g^* $
	is $*$-strongly continuous into bounded operators. Using the condition for conformal equivariance that for \( a ∈ \mathscr{Q} \) the map
	\[ g ↦ U_g D U_g^* a - a μ_g D μ_g^* \]
	is $*$-strongly continuous into bounded operators, we see that \( g ↦ u_g^* \tilde{a} \) is in \( \mathscr{T} \). So, \( g ↦ \tilde{d} u_g^* \tilde{a} \tilde{b}^* \tilde{c} \) is in \( \mathscr{L} \mathscr{T} \mathscr{R} \).
	
	We now evaluate \( \mathscr{LTR} \) on compactly supported Radon measures on \( G \) and ask if this generates \( A ⋊_t G \). It will suffice to integrate the paths \( g ↦ \tilde{d} u_g^* \tilde{a} \tilde{b}^* \tilde{c} \), which are constant apart from \( u_g^* \), against compactly supported continuous functions on \( G \). Proceeding step-by-step,
	\begin{align*}
		\overline{\Span}(\Lip_0^*(D) C_c^*(G) \mathscr{Q} \mathscr{Q}^* \mathscr{Q})
		& ⊇ \overline{\Span}(A C^*_t(G) \mathscr{Q} \mathscr{Q}^* \mathscr{Q}) \\
		& = \overline{\Span}((A ⋊_t G) \mathscr{Q} \mathscr{Q}^* \mathscr{Q}) \\
		& = \overline{\Span}(C^*_t(G) A \mathscr{Q} \mathscr{Q}^* \mathscr{Q}) \\
		& ⊇ \overline{\Span}(C^*_t(G) A) \\
		& = A ⋊_t G
	\end{align*}
	as required.
\end{proof}

\begin{proposition}
	Let \( (A, E_B, D) \) be a \( (μ_g)_{g ∈ G} \)-conformally \( G \)-equivariant unbounded Kasparov module, with \( G \) acting trivially on \( B \). Then
	\[ (A ⋊_u G, E_B, D; C_0(G), (1, μ_g)_{g ∈ G}) \]
	is a conformally generated cycle, with the integrated representation of \( A ⋊_u G \).
\end{proposition}
\begin{proof}
	The local compactness of the resolvent is the same as in the uniform case, Proposition \ref{proposition:dual-green-julg-ordinary-unbounded}. Recall the spaces \( \mathscr{L} \), \( \mathscr{T} \), and \( \mathscr{R} \) of Remark \ref{remark:formerly-known-as-twitter}. It is straightforward to verify that the constant families \( (d)_{g ∈ G} ∈ \mathscr{L} \) and \( (b^* c)_{g ∈ G} ∈ \mathscr{R} \) for all \( d ∈ \Lip_0^*(D) \) and \( b, c ∈ \mathscr{Q} \). A path of operators \( t \) is in \( \mathscr{T} \) if
	\[ g ↦ D t_g - t_g μ_g D μ_g^* \]
	is $*$-strongly continuous into bounded operators. Using the condition for conformal equivariance that
	$ g ↦ U_g D U_g^* a - a μ_g D μ_g^* $
	is $*$-strongly continuous into bounded operators for \( a ∈ \mathscr{Q} \), we see that \( g ↦ U_g^* a \) is in \( \mathscr{T} \). So, \( g ↦ d U_g^* a b^* c \) is in \( \mathscr{L} \mathscr{T} \mathscr{R} \). As in the Proof of Proposition \ref{proposition:conformal-descent}, the closed span of \( \Lip_0^*(D) C_c^*(G) \mathscr{Q} \mathscr{Q}^* \mathscr{Q} \) includes \( A ⋊_u G \).
\end{proof}

\begin{remark}
	It is clear that the bounded transform \( (A ⋊_t G, (E ⋊_t G)_{B ⋊_t G}, F_{\tilde{D}} = \tilde{F_D}) \) of the descent
	\[ (A ⋊_t G, (E ⋊_t G)_{B ⋊_t G}, \tilde{D}; C_0(G), (1, \tilde{μ}_g)_{g ∈ G}) \]
	of a conformally \( G \)-equivariant cycle \( (A, E_B, D) \) is exactly the descent of the bounded transform \( (A, E_B, F_D) \). The same is true for the dual Green–Julg map.	 
\end{remark}

We recall the identity
\[ 2 A^* C B = (A + B)^* C (A + B) - i (A + i B)^* C (A + i B) + (-1+i) (B^* C B + A^* C A) \]
for elements \( A \), \( B \), and \( C \) of a $*$-algebra, which implies that
\[ \Span\{ x^* C x \mid x ∈ \Span\{ A, B \} \} = \Span\{ x^* C y \mid x, y ∈ \Span\{ A, B \} \} . \]

\begin{proposition}
	\label{proposition:descent-quantum-conformal}
	Let \( \mathbb{G} \) be a locally compact quantum group and let \( (A, E_B, D) \) be a \( μ \)-conformally \( \mathbb{G} \)-equivariant unbounded Kasparov module. For \( t ∈ \{ u, r \} \), let \( ι \) be the inclusion \( \End^0(E) \to M(\End^0(E) ⋊_t \mathbb{G}) ≅ \End_{B ⋊_t \mathbb{G}}^*(E ⋊_t \mathbb{G}) \). Then 
	\[ (A ⋊_t \mathbb{G}, (E ⋊_t \mathbb{G})_{B ⋊_t \mathbb{G}}, ι(D); C_0^r(\mathbb{G}), (1, (ι ⊗ \id)(μ))) \]
	is a conformally generated cycle.
\end{proposition}
\begin{proof}	
	The compactness of the resolvent is as in the Proof of Proposition \ref{proposition:descent-uniform-unbounded-quantum}. Recall the spaces \( \mathscr{L} \), \( \mathscr{T} \), and \( \mathscr{R} \) of Definition \ref{definition:conformally-generated-cycle-general}. It is straightforward to verify that \( ι(d) ⊗ 1 ∈ \mathscr{L} \) and \( ι(b^* c) ⊗ s_2^* s_3 ∈ \mathscr{R} \) for all \( d ∈ \Lip_0^*(D) \), \( b, c ∈ \mathscr{Q} \), and \( s_2, s_3 ∈ \mathscr{S}_b, \mathscr{S}_c \).
	
	By the universality of the crossed product, see \cite[§4.1]{Vergnioux_2002} \cite[§2.3]{Vaes_2005}, the morphism
	\[ \End^0(E) ⋊_u \mathbb{G} \to \End^0(E) ⋊_t \mathbb{G} \]
	gives rise both to the morphism 
	\[ ι : \End^0(E) \to M(\End^0(E) ⋊_t \mathbb{G}) ≅ \End^*(E ⋊_t \mathbb{G}) \]
	and a unitary \( X ∈ M((\End^0(E) ⋊_t \mathbb{G}) ⊗ C_0^r(\mathbb{G})) ≅ \End^*((\End^0(E) ⋊_t \mathbb{G}) ⊗ C_0^r(\mathbb{G})) \) such that
	\[ X (ι(T) ⊗ 1) X^* = (ι ⊗ \id) (V_E (T ⊗_{δ_B} 1) V_E^*) \]
	for \( T ∈ \End^0(E) \). Let \( a ∈ \mathscr{Q} \) and \( s_1 ∈ \mathscr{S}_a \); then \( X^* (ι(a) ⊗ s_1) ∈ \mathscr{T} \) because
	\begin{align*}
		& (ι(D) ⊗ 1) X^* (ι(a) ⊗ s_1) - X^* (ι(a) ⊗ s_1) (ι ⊗ \id)(μ) (ι(D) ⊗ 1) (ι ⊗ \id)(μ)^* \\
		& \qquad = X^* \left( X (ι(D) ⊗ 1) X^* (ι(a) ⊗ s_1) - (ι ⊗ \id) \left( (a ⊗ s_1) μ (D ⊗ 1) μ^* \right) \right) \\
		& \qquad = X^* (ι ⊗ \id) \left( V_E (D ⊗_{δ_B} 1) V_E^* (a ⊗ s_1) - (a ⊗ s_1) μ (D ⊗ 1) μ^* \right)
	\end{align*}
	is \( C_0^r(\mathbb{G}) \)-matched. So, \( (d ⊗ 1) X^* (ι(a b^* c) ⊗ s_1 s_2^* s_3) \) is in \( \mathscr{L T R} \).

	We need to show that \( A ⋊_t \mathbb{G} \) is contained in \( C^*((1 ⊗ ω) (\mathscr{L T R}) \mid ω ∈ \mathscr{S}_c(C)) \). Proceeding step-by-step,
	\begin{align*}
		& C^*((1 ⊗ ω) (\mathscr{L T R}) \mid ω ∈ \mathscr{S}_c(C)) \\
		& \quad ⊇ \overline{\Span} \left\{ (1 ⊗ ω) \left( (ι(d) ⊗ 1) X^* (ι(a b^* c) ⊗ s_1 s_2^* s_3) \right) = ι(d) (1 ⊗ ω)\left( (1 ⊗ s_3^* s_2 s_1^*) X \right)^* ι(a b^* c) \right. \\
		& \quad\qquad \left|\, a, b, c ∈ \mathscr{Q}; d ∈ \Lip_0^*(D); s_1 ∈ \mathscr{S}_a; s_2 ∈ \mathscr{S}_b; s_3 ∈ \mathscr{S}_c; ω ∈ \mathcal{S}_c(C_0^r(\mathbb{G})) \right\} \\
		& \quad = \overline{\Span} \left(\! ι(\Lip_0^*(D)) \!\left\{ (1 ⊗ ω) \left( (1 ⊗ s_3^* s_2 s_1^*) X \right) |\, s_1 ∈ \mathscr{S}_a; s_2 ∈ \mathscr{S}_b; s_3 ∈ \mathscr{S}_c; ω ∈ \mathcal{S}_c(C_0^r(\mathbb{G})) \right\}^*\! ι(\mathscr{Q}) \!\right) \\
		& \quad ⊇ \overline{\Span} \Big( ι(\Lip_0^*(D)) \big\{ (1 ⊗ η^* s_4^* s_3^* s_2 s_1^*) X (1 ⊗ s_4 η)  \\
		& \quad\qquad  \big|\, s_1 ∈ \mathscr{S}_a; s_2 ∈ \mathscr{S}_b; s_3 ∈ \mathscr{S}_c; s_4 ∈ K_{C_0^r(\mathbb{G})}; η ∈ L^2(C_0^r(\mathbb{G})) \big\}^* ι(\mathscr{Q}) \Big) \\
		& \quad = \overline{\Span} \Big( ι(\Lip_0^*(D)) \big\{ (1 ⊗ η_1^* s_4^* s_3^* s_2 s_1^*) X (1 ⊗ s_5 η_2)  \\
		& \quad\qquad  \big|\, s_1 ∈ \mathscr{S}_a; s_2 ∈ \mathscr{S}_b; s_3 ∈ \mathscr{S}_c; s_4, s_5 ∈ K_{C_0^r(\mathbb{G})}; η_1, η_2 ∈ L^2(C_0^r(\mathbb{G})) \big\}^* ι(\mathscr{Q}) \Big) \\
		& \quad = \overline{\Span} \left( ι(\Lip_0^*(D)) \{ (1 ⊗ η_1^*) X (1 ⊗ η_2) |\, η_1, η_2 ∈ L^2(C_0^r(\mathbb{G})) \}^* ι(\mathscr{Q}) \right) \\
		& \quad = \overline{\Span} \left( ι(\Lip_0^*(D)) \{ (1 ⊗ ω)(X) |\, ω ∈ L^1(\mathbb{G}) \}^* ι(\mathscr{Q}) \right) \\
		& \quad = \overline{\Span}( ι(\Lip_0^*(D)) C^*_u(\mathbb{G}) ι(\mathscr{Q})) \\
		& \quad ⊇ \overline{\Span}(ι(A) C^*_u(\mathbb{G}) ι(\mathscr{Q})) = \overline{\Span}((A ⋊_u \mathbb{G}) ι(\mathscr{Q})) = \overline{\Span}(C^*_u(\mathbb{G}) ι(A \mathscr{Q})) \\
		& \quad ⊇ \overline{\Span}(C^*_u(\mathbb{G}) ι(A)) = A ⋊_u \mathbb{G}
	\end{align*}
	by the density of \( L^2(\mathbb{G}) K_{C_0^r(\mathbb{G})} \mathscr{S}_c^* \mathscr{S}_b \mathscr{S}_a^* ⊆ L^2(\mathbb{G}) \) and the inclusion \( A ⊆ \overline{\Span}(A \mathscr{Q}) \).
\end{proof}

\begin{proposition}
	Let \( \mathbb{G} \) be a locally compact quantum group and let \( (A, E_B, D) \) be a conformally \( \mathbb{G} \)-equivariant unbounded Kasparov module, with \( \mathbb{G} \) acting trivially on \( B \). Then
	\[ (A ⋊_u \mathbb{G}, E_B, D; C_0^r(\mathbb{G}), (1, μ)) \]
	is a conformally generated cycle, with the integrated representation of \( A ⋊_u \mathbb{G} \). 
\end{proposition}
\begin{proof}
	Recall the spaces \( \mathscr{L} \), \( \mathscr{T} \), and \( \mathscr{R} \) of Definition \ref{definition:conformally-generated-cycle-general}. It is straightforward to verify that \( d ⊗ 1 ∈ \mathscr{L} \) and \( b^* c ⊗ s_2^* s_3 ∈ \mathscr{R} \) for all \( d ∈ \Lip_0^*(D) \), \( b, c ∈ \mathscr{Q} \), and \( s_2, s_3 ∈ \mathscr{S}_b, \mathscr{S}_c \). Let \( a ∈ \mathscr{Q} \) and \( s_1 ∈ \mathscr{S}_a \); then, by Definition \ref{defn:conf-ess-equi},
	\[ (D ⊗ 1) V_E^* (a ⊗ s) - V_E^* (a ⊗ s) μ (D ⊗ 1) μ^* \]
	is \( C_0^r(\mathbb{G}) \)-matched and \( V_E^* (a ⊗ s_1) ∈ \mathscr{T} \). So \( (d ⊗ 1) V_E^* (a b^* c ⊗ s_1 s_2^* s_3) \) is in \( \mathscr{L T R} \).
	
	We need to show that \( A ⋊_u \mathbb{G} \) is contained in \( C^*((1 ⊗ ω) (\mathscr{L T R}) |\, ω ∈ \mathscr{S}_c(C)) \). The same manipulations as in the proof of Proposition \ref{proposition:descent-quantum-conformal}, with \( V_E \) in place of \( X \), show that
	\[ C^*((1 ⊗ ω) (\mathscr{L T R}) \mid ω ∈ \mathscr{S}_c(C)) ⊇ \overline{\Span}(\Lip_0^*(D) C^*_u(\mathbb{G}) \mathscr{Q}) ⊇ A ⋊_u \mathbb{G} , \]
	as required.
\end{proof}

\begin{remark}
	It is again clear that the bounded transform \( (A ⋊_t \mathbb{G}, (E ⋊_t \mathbb{G})_{B ⋊_t \mathbb{G}}, F_{ι(D)} = ι(F_D)) \) of the descent
	\[ (A ⋊_t \mathbb{G}, (E ⋊_t \mathbb{G})_{B ⋊_t \mathbb{G}}, ι(D); C_0^r(\mathbb{G}), (1, (ι ⊗ \id)(μ))) \]
	of a conformally \( \mathbb{G} \)-equivariant cycle \( (A, E_B, D) \) is exactly the descent of the bounded transform \( (A, E_B, F_D) \). The same is true for the dual Green–Julg map.
\end{remark}

\section{Equivalence relations for unbounded KK-theory}
\label{sec:equiv-unbdd}

For C*-algebras \( A \) and \( B \), a \emph{homotopy} between two bounded Kasparov \(A\)-\(B\)-modules is a Kasparov \(A\)-\(C([0,1],B)\)-module whose evaluations at the endpoints \( 0, 1 \in [0, 1] \) recover them \cite[Definition 4.2.2]{Kasparov_1981}. Homotopy is an equivalence relation and compatible with direct sums. The homotopy classes of bounded Kasparov \(A\)-\(B\)-modules, together with the direct sum, form the \( \bbZ/2\bbZ \)-graded abelian group \( KK_*(A,B) = KK_0(A, B) \oplus KK_1(A, B) \) \cite[Theorem 4.1, Definition 4.4]{Kasparov_1981}.
The details of homotopy for unbounded Kasparov modules have only recently been worked out \cite{Dungen_2020a, Kaad_2020a}. It turns that out that, provided that \( A \) is separable, one can indeed obtain \( KK_*(A,B) \) from homotopy classes of unbounded Kasparov modules.

On the other hand, the strongest reasonable equivalence relation in the bounded picture of KK-theory (apart from unitary equivalence) is locally compact perturbation. If \( (A, E_B, F) \) is a bounded Kasparov module and \( T ∈ \End^*(E) \) is such that \( T a, a T ∈ \End^0(E) \) for all \( a ∈ A \), then \( (A, E_B, F+T) \) will still be a bounded Kasparov module. The only condition which is not immediate is that \( ((F+T)^2 - 1) a ∈ \End^0(E) \), demonstrated by the computation
\[ ((F + T)^2 - 1) a = (F^2 - 1) a + (F + T) T a + T F a = (F^2 - 1) a + (F + T) T a + T [F, a] + T a F . \]
It not so clear, in the unbounded picture of KK-theory, what should stand in for equivalence up to locally compact perturbation. The most immediate relation that suggests itself is equivalence up to bounded perturbation. If \( (A, E_B, D) \) is an unbounded Kasparov module and \( T = T^* ∈ \End^*(E) \), then \( (A, E_B, D+T) \) will still be an unbounded Kasparov module. The local compactness of the resolvent takes a little work, see e.g. \cite[Lemma B.6]{Carey_1998}. One can similarly consider locally bounded perturbations, at least in the presence of an adequate approximate unit \cite[\S4]{Dungen_2018}. We soon shall apply Proposition \ref{proposition:bounded-transform-of-commutator-additive-perturbation} and Corollary \ref{proposition:bounded-transform-of-additive-perturbation-local} to the study of additive perturbations although they do not give locally compact resolvent as a conclusion.

In \cite[§3]{Cuntz_1986}, \emph{cobordism} is introduced as another equivalence relation on bounded Kasparov modules; slightly weakening locally compact perturbation. (We remark that the similarly named equivalence relation of \emph{bordism} of unbounded Kasparov modules \cite{Hilsum_2010, Deeley_2018} is unrelated and will not appear in this paper.) First, we require a small Lemma.

\begin{lemma}
	\cite[§3]{Cuntz_1986}
	If \( (A, E_B, F) \) is a bounded Kasparov module and \( p ∈ \End^*(E) \) is an even projection commuting with the representation of \( A \) such that \( [F, p] a \) is compact for all \( a ∈ A \), then \( (A, p E_B, p F p) \) is a Kasparov module.
\end{lemma}

\begin{definition}
	\cite[Definition 3.1]{Cuntz_1986}
	Two bounded Kasparov modules \( (A, E'_B, F_1) \) and \( (A, E''_B, F_2) \) of the same parity are \emph{cobordant} if there exists a Kasparov module \( (A, E_B, F) \) of that parity and an partial isometry \( v ∈ \End^*(E) \) (even if the parity is even), such that
	\begin{itemize}
		\item \( v \) commutes with (the representation of) \( A \);
		\item \( [F, v] a \) is compact for all \( a ∈ A \);
		\item \( (A, (1 - v v^*) E_B, (1 - v v^*) F (1 - v v^*)) \) is unitarily equivalent to \( (A, E'_B, F_1) \); and
		\item \( (A, (1 - v^* v) E_B, (1 - v^* v) F (1 - v^* v)) \) is unitarily equivalent to \( (A, E''_B, F_2) \).
	\end{itemize}
	We call \( (A, E_B, F; v) \) a cobordism.
\end{definition}

It turns out that cobordism is an equivalence relation, and is compatible with direct sums \cite[Lemma 3.3]{Cuntz_1986}. Even though apparently much stronger than homotopy, cobordism gives rise to the same KK-groups, provided \( A \) is separable \cite[Theorem 3.7]{Cuntz_1986}. (Our definition differs slightly from that of \cite[Definition 3.1]{Cuntz_1986}, in that we deal only with trivially graded C*-algebras and work with odd as well as even Kasparov modules. By Remark \ref{remark:w45yne57u357ue56ne56uw46hv4ec}, it is straighforward to check that \cite[Lemma 3.6, Theorem 3.7]{Cuntz_1986} are still valid.)

\begin{example}
	Suppose that two bounded Kasparov modules \( (A, E'_B, F_1) \) and \( (A, E''_B, F_2) \) of the same parity are unitarily equivalent, up to a locally compact perturbation, that is, there exists a unitary \( U : E'_B \to E''_B \) (even if the parity is even), intertwining the representations of \( A \), such that \( (U^* F_2 U - F_1) a ∈ \End^0(E) \) for all \( a ∈ A \). Then
	\[ \left( A, (E' ⊕ E'')_B, \begin{pmatrix} F_1 & \\ & F_2 \end{pmatrix} \right) \qquad v = \begin{pmatrix} & 0 \\ U & \end{pmatrix} \]
	constitute a cobordism between the two modules.
\end{example}

\begin{lemma}
	\label{lemma:bounded-cobordism-commuting}
	If two bounded Kasparov modules \( (A, E_{1, B}, F_1) \) and \( (A, E_{2, B}, F_2) \) are cobordant, there exists a cobordism \( (A, E_B, F; v) \) such that \( v v^* \), \( v^* v \), and \( F \) mutually commute.
\end{lemma}
\begin{proof}
	Let \( (A, E'_B, F'; v') \) be any cobordism between \( (A, E_{1, B}, F_1) \) and \( (A, E_{2, B}, F_2) \). Let \( w_1 : E_1 \to (1 - v' v'^*) E' \) and \( w_2 : E_2 \to (1 - v'^* v') E' \) be the unitaries of the cobordism. Then
	\[ \left( A, E_1 ⊕ E' ⊕ E_2, F_1 ⊕ F' ⊕ F_2; w_1^* + v' + w_2 \right) \]
	is a cobordism between \( (A, E_{1, B}, F_1) \) and \( (A, E_{2, B}, F_2) \). We have
	\[  (w_1^* + v' + w_2)^* (w_1^* + v' + w_2) = 0 ⊕ 1 ⊕ 1 \qquad (w_1^* + v' + w_2) (w_1^* + v' + w_2)^* = 1 ⊕ 1 ⊕ 0 . \]
	We can check that
	\begin{align*}
		[F_1 ⊕ F' ⊕ F_2, w_1^* + v' + w_2] a
		& = \left( F_1 w_1^* + F' (v' + w_2) - (w_1^* + v') F' - w_2 F_2 \right) a \\
		& = \left( [F', v'] + F_1 w_1^* + F' w_2 - w_1^* F' - w_2 F_2 \right) a \\
		& = \left( [F', v'] + w_1^* F' (1 - v' v'^*) + F' w_2 - w_1^* F' - (1 - v'^* v') F' w_2 \right) a \\
		& = \left( [F', v'] - w_1^* F' v' v'^* + v'^* v' F' w_2 \right) a \\
		& = [F', v'] a - w_1^* [F', v'] a v'^* - v'^* [F', v'] a w_2
	\end{align*}
	is compact for all \( a ∈ A \), as required.
\end{proof}

\subsection{Cobordism of higher order cycles and positive degeneracy}

We shall make a natural generalisation of cobordism to unbounded cycles but, first, a Lemma.

\begin{lemma}
	\label{lemma:bounded-transform-vs-projection-commuting}
	Let \( (A, E_B, D) \) be an order-\( \frac{1}{1-\alpha} \) cycle and \( p ∈ \End^*(E) \) a projection (even if the cycle is of even parity) such that \( p \) commutes with \( A \) and \( D \). Then \( (A, p E_B, p D p) \) is an order-\( \frac{1}{1-\alpha} \) cycle and, furthermore, \( F_{p D p} = p F_D p \)  on \( p E \).
\end{lemma}

A similar result to Lemma \ref{lemma:bounded-transform-vs-projection-commuting} would follow from weaker assumptions than that \( p \) and \( D \) commute but we do without. We give the following definition in the higher-order setting as it is no more difficult to do so.

\begin{definition}
	Two order-\( \frac{1}{1-\alpha} \) cycles \( (A, E'_B, D_1) \) and \( (A, E''_B, D_2) \) of the same parity are \emph{cobordant} if there exist an order-\( \frac{1}{1-\alpha} \) cycle \( (A, E_B, D) \) of that parity and a partial isometry \( v ∈ \End^*(E) \) (even if the parity is even), such that
	\begin{itemize}
		\item \( v \) commutes with (the representation of) \( A \), and \( v v^* \) and \( v^* v \) commute with \( D \);
		\item \( v A ⊆ \overline{\Lip_{\alpha}^*(D)} \);
		\item \( (A, (1 - v v^*) E_B, (1 - v v^*) D (1 - v v^*)) \) is unitarily equivalent to \( (A, E'_B, D_1) \); and
		\item \( (A, (1 - v^* v) E_B, (1 - v^* v) D (1 - v^* v)) \) is unitarily equivalent to \( (A, E''_B, D_2) \).
	\end{itemize}
	For a dense $*$-subalgebra \( \mathscr{A} ⊆ A \), \( (\mathscr{A}, E_B, D; v) \) is a cobordism between \( (\mathscr{A}, E'_B, D_1) \) and \( (\mathscr{A}, E''_B, D_2) \) if \( v^* v \mathscr{A} ⊆ \mathscr{Q} \).
\end{definition}

At the cost of further technicalities, we could proceed with weaker assumptions than that \( D \) commute with \( v v^* \) and \( v^* v \). However, by a similar argument to Lemma \ref{lemma:bounded-cobordism-commuting}, this would not be worth the cost.

\begin{proposition}
	cf. \cite[Lemma 3.3]{Cuntz_1986}
	\label{proposition:unbounded-cobordism-is-equivalence-relation}
	Cobordism of higher order cycles is an equivalence relation and is compatible with direct sums.
\end{proposition}
\begin{proof}
	For reflexivity, we take \( v = 0 \in \End^*(E) \) to see that \( (A, E_B, D) \) is cobordant to itself.

	For symmetry, note that \( v^* A = (v A)^* ⊆ \overline{\Lip_{\alpha}^*(D)} \) so that making the substitution of \( v^* \) for \( v \) reverses the roles of \( (A, E'_B, D_1) \) and \( (A, E''_B, D_2) \).

	For transitivity, suppose that \( (A, E_B, D; v) \) is a cobordism between the cycles \( (A, E_{1, B}, D_1) \) and \( (A, E_{2, B}, D_2) \), and that \( (A, E'_B, D'; v') \) is a cobordism between \( (A, E_{2, B}, D_2) \) and \( (A, E_{3, B}, D_3) \). Let \( U : (1 - v^* v) E \to E_2 \) and \( U' : (1 - v' v'^*) E \to E_2 \) be the unitary equivalences between the cycles
	\[ (A, (1 - v^* v) E_B, (1 - v^* v) D (1 - v^* v)) \qquad (A, (1 - v' v'^*) E'_B, (1 - v' v'^*) D' (1 - v' v'^*))  \]
	and the cycle \( (A, E_{2, B}, D_2) \), respectively. Then
	\[ (A, (E ⊕ E')_B, D ⊕ D'; v + U'^* U + v') \]
	is a cobordism between \( (A, E_{1, B}, D_1) \) and \( (A, E_{3, B}, D_3) \). We have
	\[ (v + U'^* U + v') (v + U'^* U + v')^* = v v^* ⊕ 1 \qquad (v + U'^* U + v')^* (v + U'^* U + v') = 1 ⊕ v'^* v' . \]
	Furthermore,
	\[ \Lip_{\alpha}^*(D) ⊕ \Lip_{\alpha}^*(D') \subseteq \Lip_{\alpha}^*(D \oplus D') , \]
	so that \( (v + v') A ⊆ \overline{\Lip_{\alpha}^*(D \oplus D')} \). Because \( D \) commutes with \( (1 - v^* v) \) and \( D' \) commutes with \( (1 - v' v'^*) \), \( D' U'^* U = U'^* D_2 U = U'^* U D \) on \( E ⊕ E' \). Hence
	\[ U'^* \Lip_{\alpha}^*(D_2) U ⊆ \Lip_{\alpha}^*(D \oplus D') \]
	and so \( U'^* U A = U'^* A U ⊆ U'^* \overline{\Lip_{\alpha}^*(D_2)} U ⊆ \overline{\Lip_{\alpha}^*(D \oplus D')} \)
	as required.	
	
	Finally, it is straightforward to check that direct sums of cobordisms are cobordisms of direct sums in an obvious way.
\end{proof}

\begin{example}
	Let \( (A, E'_B, D_1) \) and \( (A, E''_B, D_2) \) be two order-\( \frac{1}{1-\alpha} \) cycles of the same parity.
	Suppose that there exists a unitary \( U : E'_B \to E''_B \) (even if the parity is even), intertwining the representations of \( A \), such that \( A \) is contained in the closure of the set of \( a ∈ \End^*(E') \) for which
	\( U a \dom D_1 ⊆ \dom D_2 \)
	and
	\[ (U^* D_2 U a - a D_1) \langle D_1 \rangle^{-\alpha} \qquad U \langle D_2 \rangle^{-\alpha} U (U^* D_2 U a - a D_1) \]
	extend to adjointable operators on \( E' \). Then
	\[ \left( A, (E' ⊕ E'')_B, \begin{pmatrix} D_1 & \\ & D_2 \end{pmatrix} \right) \qquad v = \begin{pmatrix} & 0 \\ U & \end{pmatrix} \]
	constitute a cobordism between the two cycles.
\end{example}

\begin{proposition}
\label{prop:cobord-classes}
	Given two cobordant order-\( \frac{1}{1-\alpha} \) cycles \( (A, E'_B, D_1) \) and \( (A, E''_B, D_2) \), their bounded transforms \( (A, E'_B, F_{D_1}) \) and \( (A, E''_B, F_{D_2}) \) are cobordant and so they define the same element in \( KK_*(A, B) \).
\end{proposition}
\begin{proof}
	Let \( (A, E_B, D; v) \) be a cobordism between \( (A, E'_B, D_1) \) and \( (A, E''_B, D_2) \). By Lemma \ref{lemma:bounded-transform-vs-projection-commuting}, \( (A, E_B, F_D; v) \) is a bounded cobordism between \( (A, E'_B, F_{D_1}) \) and \( (A, E''_B, F_{D_2}) \). 
\end{proof}

A natural question to ask is whether one can identify unbounded cycles cobordant to the zero module. In \cite[§3–4]{Dungen_2020a}, several notions of degenerate module are surveyed and shown to be homotopic to zero. Instead of making a similar survey, we shall make the following definition, in the safety of the knowledge that it contains as special cases the \emph{spectrally degenerate} cycles of \cite[Definition 3.5]{Dungen_2020a}, the \emph{spectrally symmetric} cycles of \cite[Definition 4.6]{Dungen_2020a} (which, in turn, include the \emph{spectrally decomposable} cycles of \cite[Definition 4.1]{Kaad_2020a}), the \emph{Clifford symmetric} cycles of \cite[Definition 4.13]{Dungen_2020a}, and the \emph{weakly degenerate} cycles of \cite[Definition 3.1]{Deeley_2018}.

\begin{definition}
	\label{definition:positively-degenerate}
	An order-\( \frac{1}{1-\alpha} \) cycle \( (A, E_B, D) \) is \emph{positively degenerate} if there exists a self-adjoint unitary \( s ∈ \End^*(E) \) (odd if the cycle is of even parity), preserving the domain of \( D \), such that
	\begin{itemize}
		\item As operators on \( \dom D \), \( D s + s D \geq -c \langle D \rangle^{\alpha} \) for some constant \( c \geq 0 \) and
		\item \( A ⊆ \overline{\mathscr{P}} \), where \( \mathscr{P} \) is the set of \( a ∈ \Lip^*_{\alpha}(D) \) such that \( [s, a] = 0 \).
	\end{itemize}
\end{definition}

\begin{proposition}
\label{prop:degen}
	A positively degenerate order-\( \frac{1}{1-\alpha} \) cycle \( (A, E_B, D) \) is cobordant to \( (A, 0_B, 0) \).
\end{proposition}
\begin{proof}
	Let \( s ∈ \End^*(E) \) be a symmetry implementing the degeneracy. Let \( N \) be the number operator and \( S \) the unilateral shift on \( \ell^2(\bbN_{≥0}) \). Then \( (A, E_B ⊗ \ell^2(\bbN_{≥0}), D ⊗ 1 + s ⊗ N) \) is an order-\( \frac{1}{1-\alpha} \) cycle. The main point to check is the local compactness of the resolvent, for which we compute
	\[ (D ⊗ 1 + s ⊗ N)^2 = D^2 ⊗ 1 + 1 ⊗ N^2 + (D s + s D) ⊗ N ≥ D^2 ⊗ 1 + 1 ⊗ N^2 - c \langle D \rangle^{\alpha} \otimes N . \]
	Fix \( \varepsilon \in (0, 1) \). The function \( f : \bbR^2 \to \bbR \) given by
	\[ f : (x, y) \mapsto \varepsilon (x^2 + y^2) - c (1 + x^2)^{\alpha/2} y \]
	has a global minimum. Hence, for large enough \( \lambda > 0 \), 
	\[ \lambda + (D ⊗ 1 + s ⊗ N)^2 ≥ (1 - \epsilon) (D^2 ⊗ 1 + 1 ⊗ N^2) \]
	and so
	\[ a (\lambda + 1 + (D ⊗ 1 + s ⊗ \kappa N)^2)^{-1} \]
	is compact. The constructed order-\( \frac{1}{1-\alpha} \) cycle, together with the isometry \( 1 ⊗ S \), implements the required cobordism. Using the relation \( N S = S (N+1) \), we check that
	\[ [D ⊗ 1 + s ⊗ N, (1 ⊗ S) a] = s ⊗ [N, S] a = s ⊗ S a \]
	is bounded for \( a \in \mathscr{P} \).
\end{proof}

We can now show that higher order cycles, subject to the equivalence relation of cobordism, form a group under direct sum.

\begin{corollary}
\label{cor:inverses}
	Given an order-\( \frac{1}{1-\alpha} \) cycle \( (A, E_B, D) \),
	\[ (A, E_B, D) ⊕ (A, E_B^{(\mathrm{op})}, -D) = \left( A, (E ⊕ E)^{(\mathrm{op})})_B, \begin{pmatrix} D & \\ & -D \end{pmatrix} \right) , \]
	where \( E^{(\mathrm{op})} \) is \( E \) with the opposite grading if \( E \) is graded,
	is cobordant to \( (A, 0_B, 0) \).
\end{corollary}
\begin{proof}
	The symmetry \( s = \left( \begin{smallmatrix} & 1 \\ 1 & \end{smallmatrix} \right) \) makes the direct sum cycle positively degenerate.
\end{proof}

Combining Propositions \ref{proposition:unbounded-cobordism-is-equivalence-relation}, \ref{prop:cobord-classes} and Corollary \ref{cor:inverses} proves

\begin{theorem}
\label{theorem:cobordism-gp}
Let \( 0 \leq \alpha < 1 \). Cobordism classes of order-\( \frac{1}{1-\alpha} \) $A$-$B$-cycles form a \( \bbZ/2\bbZ \)-graded abelian group which surjects onto $KK_*(A,B)$. Further, cobordism classes of higher order $A$-$B$-cycles (without a constraint on their order) form a \( \bbZ/2\bbZ \)-graded abelian group which surjects onto $KK_*(A,B)$.
\end{theorem}
 
For the final statement, we note than any order-\( \frac{1}{1-\alpha} \) cycle can be considered to be an order-\( \frac{1}{1-\beta} \) cycle for \( \alpha \leq \beta < 1 \). It is presumably the case that cobordism of higher order cycles is strictly stronger than homotopy. It is possible that the addition of the functional dampening of \cite{Dungen_2020a} could make cobordism equivalent to homotopy. This remains a matter for future investigation.

\subsection{Cobordism of conformally generated cycles and conformism}
\label{section:equivalence-conf-gen}

In this section, we consider an equivalence relation on conformally generated cycles making the equivalence classes an abelian group.

\begin{remark}
\label{rmk:cgc-sums}
The direct sum of two conformally generated cycles \( (A, E_{1, B}, D_1; C_1, μ_1) \) and \linebreak \( (A, E_{2, B}, D_2; C_2, μ_2) \) is
\[ (A, E_{1, B} ⊕ E_{2, B}, D_1 ⊕ D_2; C_1 ⊕ C_2, μ_1 ⊕ 1 ⊕ 1 ⊕ μ_2) \]
where \( μ_1 ⊕ 1 ⊕ 1 ⊕ μ_2 ∈ (E_1 ⊗ C_1 ⊕ E_2 ⊗ C_1 ⊕ E_1 ⊗ C_2 ⊕ E_2 ⊕ C_2)^2 \). If \( C_1 = C_2 \) or, more generally, if \( C_1 \) and \( C_2 \) have a common ideal \( J \), one could write the direct sum in a smaller way. In practice, also, it is often possible to change \( C \) and \( μ \) without affecting the validity of a cycle \( (A, E_B, D; C, μ) \). 
One should therefore think of conformally generated cycles \( (A, E_B, D; C_1, μ_1) \) and \( (A, E_B, D; C_2, μ_2) \) as equivalent.
\end{remark}

Note that the external product of conformally generated cycles is not constructive.

\begin{definition}
\label{definition:cobordism-cgc}
	Two conformally generated cycles \( (A, E_{1, B}, D_1; C_1, μ_1) \) and \( (A, E_{2, B}, D_2; C_2, μ_2) \) are \emph{cobordant} if there exists a conformally generated cycle \( (A, E_B, D; C, μ) \) and an even partial isometry \( v ∈ \End^*(E) \) such that
	\begin{enumerate}
		\item \( v \) commutes with (the representation of) \( A \), and \( v v^* \) and \( v^* v \) commute with \( D \);
		\item \( v A ⊆ C^*((1 ⊗ ψ) (\mathscr{L T R}) \mid ψ ∈ \mathscr{S}_c(C)) \);
		\item \( (A, (1 - v v^*) E_B, (1 - v v^*) D (1 - v v^*)) \) is unitarily equivalent to \( (A, E_{1, B}, D_1) \); and
		\item \( (A, (1 - v^* v) E_B, (1 - v^* v) D (1 - v^* v)) \) is unitarily equivalent to \( (A, E_{2, B}, D_2) \).
	\end{enumerate}
\end{definition}

\begin{example}
	Let \( (A, E_B, D; v) \) be a cobordism between unbounded Kasparov modules \( (A, E'_B, D_1) \) and \( (A, E''_B, D_2) \). Then
	\[ \left( A, E_B, D; \bbC, (1, 1); v \right) \]
	is a cobordism between \( (A, E'_B, D_1; \bbC, (1, 1)) \) and \( (A, E''_B, D_2; \bbC, (1, 1)) \).
\end{example}

When applied to ordinary Kasparov modules, Definition \ref{definition:cobordism-cgc} also encompasses conformal transformations and singular conformal transformations.

\begin{example}
	Suppose that \( (U, μ) \) is a conformal transformation from the unbounded Kasparov module \( (A, E_B, D_1) \) to \( (A, E'_B, D_2) \). Then
	\[
	\left( A, (E ⊕ E')_B, \left(\begin{smallmatrix} D_1 & \\ & D_2 \end{smallmatrix}\right);  \bbC^2, \left( \left(\begin{smallmatrix} 1 & \\ & 1 \end{smallmatrix}\right) ⊕ \left(\begin{smallmatrix} 1 & \\ & 1 \end{smallmatrix}\right), \left(\begin{smallmatrix} 1 & \\ & 1 \end{smallmatrix}\right) ⊕ \left(\begin{smallmatrix} μ & \\ & 1 \end{smallmatrix}\right) \right) ; \left(\begin{smallmatrix} & 0 \\ U & \end{smallmatrix}\right) \right) 
	\]
	is a cobordism between \( (A, E_B, D_1; \bbC, (1, 1)) \) and \( (A, E'_B, D_2; \bbC, (1, 1)) \). We leave the demonstration of this as a special case of Example \ref{example:sing-conf-cgc}.
\end{example}

More generally,

\begin{example}
	\label{example:sing-conf-cgc}
	Let \( (U, (μ_i)_{i ∈ I}) \) be a singular conformal transformation from one unbounded Kasparov module, \( (A, E_B, D_1) \), to another, \( (A, E'_B, D_2) \), as in Definition \ref{definition:sing-conf-transformation}. We will show that
	\[
	\left( A, (E ⊕ E')_B, \left(\begin{smallmatrix} D_1 & \\ & D_2 \end{smallmatrix}\right);  C_0(\{\pt\} ⊔ I), \left( \left(\begin{smallmatrix} 1 & \\ & 1 \end{smallmatrix}\right) ⊕ \left(\begin{smallmatrix} 1 & \\ & 1 \end{smallmatrix}\right)_{i ∈ I}, \left(\begin{smallmatrix} 1 & \\ & 1 \end{smallmatrix}\right) ⊕ \left(\begin{smallmatrix} μ_i & \\ & 1 \end{smallmatrix}\right)_{i ∈ I} \right) ; \left(\begin{smallmatrix} & 0 \\ U & \end{smallmatrix}\right) \right) 
	\]
	is a cobordism between \( (A, E_B, D_1; \bbC, (1, 1)) \) and \( (A, E'_B, D_2; \bbC, (1, 1)) \). Here, \( I \) is treated as a discrete set. For \( a ∈ \mathscr{M}_i \), we can check that
	\[ \begin{pmatrix} D_1 & \\ & D_2 \end{pmatrix}\!\! \begin{pmatrix} & 0 \\ U a & \end{pmatrix} - \begin{pmatrix} & 0 \\ U a & \end{pmatrix}\!\! \begin{pmatrix} μ_i & \\ & 1 \end{pmatrix}\!\! \begin{pmatrix} D_1 & \\ & D_2 \end{pmatrix}\!\! \begin{pmatrix} μ_i & \\ & 1 \end{pmatrix}^* \!\!=\! \begin{pmatrix} & 0 \\ U (U^* D_2 U a - a μ_i D_1 μ_i^*) & \end{pmatrix} \]
	is bounded, so that \( \left(\begin{smallmatrix} & 0 \\ U & \end{smallmatrix}\right) \left(\begin{smallmatrix} \mathscr{M} & \\ & 0 \end{smallmatrix}\right) ∈ \mathscr{T}_i \). One can check that \( \left(\begin{smallmatrix} 1 & \\ & 1 \end{smallmatrix}\right) ∈ \mathscr{L}_i \) and that \( \mathscr{R}_i \) contains \( \left(\begin{smallmatrix} \mathscr{M}_i^* \mathscr{M}_i & \\ & 0 \end{smallmatrix}\right) \). Furthermore, \( \mathscr{L}_{\pt} \), \( \mathscr{T}_{\pt} \), and \( \mathscr{R}_{\pt} \) all contain \( \left(\begin{smallmatrix} \Lip_0^*(D_1) & \\ & \Lip_0^*(D_2) \end{smallmatrix}\right) \). Hence
	\begin{multline*}
		\left(\begin{smallmatrix} & 0 \\ U & \end{smallmatrix}\right) A
		⊆ \overline{\Span}_{i ∈ I} \left( \left(\begin{smallmatrix} 1 & \\ & 1 \end{smallmatrix}\right) \left(\begin{smallmatrix} & 0 \\ U & \end{smallmatrix}\right) \left(\begin{smallmatrix} \mathscr{M} & \\ & 0 \end{smallmatrix}\right) \left(\begin{smallmatrix} \mathscr{M}_i^* \mathscr{M}_i & \\ & 0 \end{smallmatrix}\right) \left(\begin{smallmatrix} \Lip_0^*(D_1) & \\ & 0 \end{smallmatrix}\right) \right) \\
		⊆ \overline{\Span}_{i ∈ I} \left( \mathscr{L}_i \mathscr{T}_i \mathscr{R}_i \mathscr{L}_{\pt} \mathscr{T}_{\pt} \mathscr{R}_{\pt} \right)
		⊆ C^*(\mathscr{L}_x \mathscr{T}_x \mathscr{R}_x \mid x ∈ \{\pt\} ⊔ I)
	\end{multline*}
	and we are done.
\end{example}

We are led to the following definition.

\begin{definition}
	\label{definition:conformism}
	Two unbounded Kasparov modules \( (A, E_B, D_1) \) and \( (A, E'_B, D_2) \) are \emph{conformant} if there exists if there exists a conformally generated cycle \( (A, E_B, D; C, μ) \) and an even partial isometry \( v ∈ \End^*(E) \) making the conformally generated cycles \( (A, E_B, D_1; \bbC, (1, 1)) \) and \( (A, E'_B, D_2; \bbC, (1, 1)) \) cobordant. We call the data \( (A, E_B, D; C, μ; v) \) a \emph{conformism} between \( (A, E_B, D_1) \) and \( (A, E'_B, D_2) \).
\end{definition}

\begin{example}
	\label{example:used-to-be-twisted-equivalent}
	We pick up from the setting of Theorem \ref{theorem:used-to-be-twisted}, adopting the notation there. We will show that the conformally generated cycles
	\[ (A, E_B, D; \bbC, (1, 1)) \qquad (A, E_B, k D k^*; \bbC, (k^{-1}, k^{-1})) \]
	are cobordant. A suitable cobordism is
	\[ \left( A, (E ⊕ E)_B, \begin{pmatrix} D & \\ & k D k^* \end{pmatrix}; \bbC, \left( \begin{pmatrix} 1 & \\ & k^{-1} \end{pmatrix}, \begin{pmatrix} 1 & \\ & k^{-1} \end{pmatrix} \right); \begin{pmatrix} & 0 \\ 1 & \end{pmatrix} \right)\,. \]
	We check that
	\[ \begin{pmatrix} 1 & \\ & k^{-1} \end{pmatrix} \begin{pmatrix} D & \\ & k D k^* \end{pmatrix} \begin{pmatrix} 1 & \\ & k^{-1} \end{pmatrix} \begin{pmatrix} & 0 \\ 1 & \end{pmatrix} - \begin{pmatrix} & 0 \\ 1 & \end{pmatrix} \begin{pmatrix} 1 & \\ & k^{-1} \end{pmatrix} \begin{pmatrix} D & \\ & k D k^* \end{pmatrix} \begin{pmatrix} 1 & \\ & k^{-1} \end{pmatrix}^* = 0 \]
	so that \( \left(\begin{smallmatrix} & 0 \\ 1 & \end{smallmatrix}\right) ∈ \mathscr{T} \). Both \( \mathscr{L} \) and \( \mathscr{R} \) contain \( \bbC 1 ⊕ \mathscr{M} \). We remark that \( \mathscr{M} \) is a $*$-algebra of operators, so \( \overline{\Span}(\mathscr{M}^2) = \overline{\mathscr{M}} \). We have
	\[ \left(\begin{smallmatrix} & 0 \\ 1 & \end{smallmatrix}\right) A ⊆ \overline{\Span} \left( \left(\begin{smallmatrix} 1 & \\ & 0 \end{smallmatrix}\right) \left(\begin{smallmatrix} & 0 \\ 1 & \end{smallmatrix}\right) \left(\begin{smallmatrix} 0 & \\ & \mathscr{M}^2 \end{smallmatrix}\right) A \left(\begin{smallmatrix} 0 & \\ & \mathscr{M} \end{smallmatrix}\right) \right) ⊆ \overline{\Span}(\mathscr{L T R L T R}) \]
	and we are done.
\end{example}

\begin{proposition}
	\label{proposition:conformism-conf-gen-equivalence}
	Cobordism of conformally generated cycles is an equivalence relation and is compatible with direct sums.
\end{proposition}
\begin{proof}
	For reflexivity, we take \( v = 0 \in \End^*(E) \) to see that \( (A, E_B, D; C, \mu) \) is cobordant to itself.

	For symmetry, note that \( v^* A = (v A)^* ⊆ C^*((1 ⊗ ψ) (\mathscr{L T R}) \mid ψ ∈ \mathscr{S}_c(C)) \) so that making the substitution of \( v^* \) for \( v \) reverses the roles of \( (A, E_{1, B}, D_1; C_1, μ_1) \) and \( (A, E_{2, B}, D_2; C_2, μ_2) \).

	For transitivity, suppose that \( (A, E_B, D; C, μ; v) \) is a cobordism between \( (A, E_{1, B}, D_1; C_1, μ_1) \) and \( (A, E_{2, B}, D_2; C_2, μ_2) \), and \( (A, E'_B, D'; C', μ'; v') \) is a cobordism between \( (A, E_{2, B}, D_2; C_2, μ_2) \) and \( (A, E_{3, B}, D_3; C_3, μ_3) \). Let \( U : (1 - v^* v) E \to E_2 \) and \( U' : (1 - v' v'^*) E \to E_2 \) be the unitary equivalences between the cycles
	\[
	 (A, (1 - v^* v) E_B, (1 - v^* v) D (1 - v^* v)) \quad\mbox{and}\quad  (A, (1 - v' v'^*) E'_B, (1 - v' v'^*) D' (1 - v' v'^*))  \]
	 and the cycle \( (A, E_{2, B}, D_2)\).
	 Then
	\[ (A, (E ⊕ E')_B, D ⊕ D'; C ⊕ C_2 ⊕ C', μ ⊕ 1 ⊕ v^* v + U^* μ_2 U ⊕ v' v'^* + U'^* μ_2 U' ⊕ 1 ⊕ μ'; v + U'^* U + v') , \]
	is a cobordism between \( (A, E_{1, B}, D_1; C_1, μ_1) \) and \( (A, E_{3, B}, D_3; C_3, μ_3) \), where
	\begin{align*} μ ⊕ 1 ⊕ v^* v + &U^* μ_2 U ⊕ v' v'^* + U'^* μ_2 U' ⊕ 1 ⊕ μ' \\
	&∈ (E ⊗ C) ⊕ (E' ⊗ C) ⊕ (E ⊗ C_2) ⊕ (E' ⊗ C_2) ⊕ (E ⊗ C') ⊕ (E' ⊗ C'). 
	\end{align*}
	  We have
	\[ (v + U'^* U + v') (v + U'^* U + v')^* = v v^* ⊕ 1 \qquad (v + U'^* U + v')^* (v + U'^* U + v') = 1 ⊕ v'^* v'. \]
	Let \( \mathscr{L}'' \), \( \mathscr{T}'' \), and \( \mathscr{R}'' \) be the spaces of Definition \ref{definition:conformally-generated-cycle-general}, corresponding to this cycle. We have
	\[ \mathscr{L} ⊕ \mathscr{L}' ⊆ \mathscr{L}'' \qquad \mathscr{T} ⊕ \mathscr{T}' ⊆ \mathscr{T}'' \qquad \mathscr{R} ⊕ \mathscr{R}' ⊆ \mathscr{R}'' , \]
	so that \( (v + v') A ⊆ C^*((1 ⊗ ψ) (\mathscr{L'' T'' R''}) \mid ψ ∈ \mathscr{S}_c(C ⊕ C_2 ⊕ C')) \). Because \( D \) commutes with \( (1 - v^* v) \) and \( D' \) commutes with \( (1 - v' v'^*) \), \( D' U'^* U = U'^* D_2 U = U'^* U D \) on \( E ⊕ E' \). Hence
	\[ U'^* \mathscr{L}_2 \mathscr{T}_2 \mathscr{R}_2 U ⊆ \mathscr{L'' T'' R''} \]
	and
	\begin{align*}
		U'^* U A = U'^* A U
		& ⊆ U'^* C^*((1 ⊗ ψ) (\mathscr{L}_2 \mathscr{T}_2 \mathscr{R}_2) \mid ψ ∈ \mathscr{S}_c(C_2)) U \\
		& ⊆ C^*((1 ⊗ ψ) (\mathscr{L'' T'' R''}) \mid ψ ∈ \mathscr{S}_c(C ⊕ C_2 ⊕ C'))
	\end{align*}
	as required.
\end{proof}

Unlike additive perturbations of unbounded Kasparov modules, conformal transformations are not necessarily reversible nor composable. The extra room in the definition of conformism circumvents this issue.
As a special case of Proposition \ref{proposition:conformism-conf-gen-equivalence}, we have

\begin{corollary}
	\label{corollary:conformism-is-equivalence-relation}
	Conformism of unbounded Kasparov modules is an equivalence relation and is compatible with direct sums.
\end{corollary}

\begin{proposition}
	Given two cobordant conformally generated cycles \( (A, E_{1, B}, D_1; C_1, μ_1) \) and \( (A, E_{2, B}, D_2; C_2, μ_2) \), their bounded transforms \( (A, E_{1, B}, F_{D_1}) \) and \( (A, E_{2, B}, F_{D_2}) \) are cobordant and so define the same element in \( KK(A, B) \).
\end{proposition}
\begin{proof}
	Let \( (A, E_B, D; C, μ; v) \) be a cobordism between
	\( (A, E_{1, B}, D_1; C_1, μ_1) \) and \( (A, E_{2, B}, D_2; C_2, μ_2) \). By Theorem \ref{theorem:C-gen-kas-mod}, \( (A, E_B, F_D) \) is a bounded Kasparov module and \( [F_D, v A] ⊆ \End^0(E) \). By Lemma \ref{lemma:bounded-transform-vs-projection-commuting}, \( F_{(1 - v v^*) D (1 - v v^*)} = (1 - v v^*) F_D (1 - v v^*) \) on the module \( (1 - v v^*) E \) and \( F_{(1 - v^* v) D (1 - v^* v)} = (1 - v^* v) F_D (1 - v^* v) \) on the module \( (1 - v^* v) E \). Hence \( (A, E_B, F_D; v) \) is a bounded cobordism between \( (A, E'_B, F_{D_1}) \) and \( (A, E''_B, F_{D_2}) \).
\end{proof}

In the following, we use the notation \( Z_{\mathscr{X}}(T) = \{ x ∈ \mathscr{X} |\, [T, x] = 0 \} \) for the centraliser in a subspace \( \mathscr{X} ⊆ \Mtc^*(E ⊗ C, C) \) of an adjointable operator \( T \) on \( E ⊗ C \).

\begin{definition}
	A conformally generated cycle \( (A, E_B, D; C, μ) \) is \emph{positively degenerate} if there exists a self-adjoint unitary \( s ∈ \End^*(E) \) (odd if the cycle is of even parity), preserving the domain of \( D \), such that
	\begin{itemize}
		\item The anticommutator \( D s + s D \) is semibounded below, i.e. \( D s + s D \geq -c \) for some \( c > 0 \);
		\item \( [μ, s ⊗ 1] = 0 \); and
		\item \( A ⊆ C^*((1 ⊗ ψ) (Z_{\mathscr{L}}(s ⊗ 1) Z_{\mathscr{T}}(s ⊗ 1) Z_{\mathscr{R}}(s ⊗ 1)) \mid ψ ∈ \mathscr{S}_c(C)) \).
	\end{itemize}
\end{definition}

\begin{proposition}
	A positively degenerate conformally generated cycle \( (A, E_B, D; C, μ) \) is cobordant to the zero cycle \( (A, 0_B, 0; 0, 0) \).
\end{proposition}
\begin{proof}
	Let \( s ∈ \End^*(E) \) be a symmetry implementing the degeneracy. Let \( N \) be the number operator and \( S \) the unilateral shift on \( \ell^2(\bbN_{≥0}) \). Then
	\begin{equation} 
	(A, E_B ⊗ \ell^2(\bbN_{≥0}), D ⊗ 1 + s ⊗ N; C ⊕ \bbC, (μ_L ⊗ 1 ⊕ 1 ⊗ 1, μ_R ⊗ 1 ⊕ 1 ⊗ 1); 1 ⊗ S) 
	\label{eq:this-cycle}
	\end{equation}
	is a cobordism from \( (A, E_B, D; C, μ) \) to \( (A, 0_B, 0; 0, 0) \). The compactness of the resolvent is as in Proposition \ref{prop:degen}.
	
	Let \( \mathscr{L}' \), \( \mathscr{T}' \), and \( \mathscr{R}' \) be the spaces of Definition \ref{definition:conformally-generated-cycle-general}, corresponding to the cycle \eqref{eq:this-cycle}. Using the relation \( N S = S (N+1) \), we check that
	\[ (D ⊗ 1 + s ⊗ N) (1 ⊗ S) - (1 ⊗ S) (D ⊗ 1 + s ⊗ N) = s ⊗ [N, S] = s ⊗ S \]
	is bounded. Hence, noting that \( [μ, s ⊗ 1] = 0 \),
	\begin{align*}
		\mathscr{L}' & ⊇ Z_{\mathscr{L}}(s ⊗ 1) ⊕ \bbC 1 ⊗ \Span\{ 1, S \} \\
		\mathscr{R}' & ⊇ Z_{\mathscr{R}}(s ⊗ 1) ⊕ \bbC 1 ⊗ \Span\{ 1, S \} \\
		\mathscr{T}' & ⊇ Z_{\mathscr{T}}(s ⊗ 1) ⊕ \bbC 1 ⊗ \Span\{ 1, S \}
	\end{align*}
	and \( (1 ⊗ S) A ⊆ C^*((1 ⊗ ψ) (\mathscr{L' T' R'}) \mid ψ ∈ \mathscr{S}_c(C ⊕ \bbC)) \), as required.
\end{proof}

\begin{corollary}
\label{corollary:inverses-conformism-cgc}
	Given a conformally generated cycle \( (A, E_B, D; C, μ) \),
	\[ (A, E_B, D; C, μ) ⊕ (A, E_B^{(\mathrm{op})}, -D; C, μ) = \left( A, (E ⊕ E^{(\mathrm{op})})_B, \left(\begin{smallmatrix} D & \\ & -D \end{smallmatrix}\right); C ⊕ C, μ ⊕ 1 ⊕ 1 ⊕ μ \right) , \]
	where \( E^{(\mathrm{op})} \) is \( E \) with the opposite grading if \( E \) is graded,
	is cobordant to \( (A, 0_B, 0; 0, 0) \).
\end{corollary}
\begin{proof}
	Using the observations of Remark \ref{rmk:cgc-sums}, we may replace the direct sum cycle with
	\[ \left( A, (E ⊕ E^{(\mathrm{op})})_B, \left(\begin{smallmatrix} D & \\ & -D \end{smallmatrix}\right); C, μ \right) \]
	and the symmetry \( s = \left( \begin{smallmatrix} & 1 \\ 1 & \end{smallmatrix} \right) \) makes this positively degenerate.
\end{proof}

We thus obtain

\begin{theorem}
\label{theorem:conform-gp}
Cobordism classes of conformally generated $A$-$B$-cycles form a \( \bbZ/2\bbZ \)-graded abelian group which surjects onto $KK(A,B)$.
Similarly, conformism classes of unbounded Kasparov $A$-$B$-modules form a \( \bbZ/2\bbZ \)-graded abelian group which surjects onto $KK(A,B)$.
\end{theorem}

\appendix
\section{Appendix}

\subsection{Fractional powers of positive operators on Hilbert C*-modules}
\label{section:fractional-powers}

The proof of Theorem \ref{theorem:fractional-power-bounds} below can be found for the Hilbert space case in \cite[Theorem 12.5]{Krasnoselskii_1976}. We include a proof in the generality of Hilbert modules, beginning with a few basic Lemmas.

\begin{lemma}
	\label{lemma:product-of-closed-operators}
	Let \( A \) and \( B \) be closed densely defined operators on a Banach space \( X \). If the product \( A B \) with domain \( \dom(A B) = \{ ξ ∈ \dom B \mid B ξ ∈ \dom A \} \) is densely defined then $AB$ is closed if either
	\begin{itemize}
		\item \( A \) has everywhere defined and bounded inverse, or
		\item \( B \) is everywhere defined and bounded.
	\end{itemize}
\end{lemma}
\begin{proof}
	Take the case that \( A \) is invertible, so that \( \dom A = A^{-1} X \). Suppose that \( (ξ_n)_{n∈\bbN} ⊆ \dom(A B) = \{ x ∈ \dom B \mid B x ∈ A^{-1} X \} \) such that \( ξ_n \to ξ \) and \( A B ξ_n \to η \) as \( n \to ∞ \). Because \( A^{-1} \) is bounded, \( B ξ_n = A^{-1} A B ξ_n \to A^{-1} η \). As \( B \) is closed, \( ξ ∈ \dom B \) and \( B ξ = A^{-1} η \). So \( ξ ∈ \dom(A B) \) and \( A B ξ = A A^{-1} η = η \) and we conclude that \( A B \) is closed.

	Take the case that \( B \) is bounded. Suppose that \( (ξ_n)_{n∈\bbN} ⊆ \dom(A B) = \{ x ∈ X \mid B x ∈ \dom A \} \) such that \( ξ_n \to ξ \) and \( A B ξ_n \to η \) as \( n \to ∞ \). Because \( B \) is bounded, \( B ξ_n \to B ξ \). As \( A \) is closed, \( B ξ ∈ \dom A \) (meaning that \( ξ ∈ \dom(A B) \)) and \( A B ξ = η \). Hence, \( A B \) is closed.
\end{proof}

\begin{lemma}
	\label{lemma:domain-preservation-on-banach-space}
	Let \( A \) and \( B \) be closed densely defined operators on Banach spaces \( X_1 \) and \( X_2 \). Let \( T \) be a bounded operator from \( X_2 \) to \( X_1 \) with \( T \dom B ⊆ \dom A \). Suppose that \( B \) is invertible (so \( B^{-1} \) is everywhere-defined and bounded). Then \( A T B^{-1} \) is everywhere-defined and bounded.
\end{lemma}
\begin{proof}
	By construction, \( A T B^{-1} \) is defined everywhere. By the closed graph theorem, it is bounded if and only if it is closed, which it is by Lemma \ref{lemma:product-of-closed-operators}.
\end{proof}

We also recall a basic fact about the norm on a Hilbert module. 

\begin{lemma}
	\label{lemma:rt7ub46n78i5n6b}
	Let \( B \) be a C*-algebra and \( E \) a Hilbert \( B \)-module. For \( \xi \in E \),
	\[ \| \xi \|_E = \sup_{[\pi] \in \hat{B}} \sup_{\eta \in H_{\pi}} \frac{\| ξ \otimes \eta \|_{E \otimes_{\pi} H_{\pi}}}{\| \eta \|_{H_{\pi}}} \]
	where \( \hat{B} \) is the set of equivalence classes of unitary representations of \( B \). For \( T \in \End_B^*(E) \),
	\[ \| T \|_{\End^*(E)} = \sup_{[\pi] \in \hat{B}} \| T \otimes 1 \|_{B(E \otimes_{\pi} H_{\pi})} . \]
\end{lemma}
\begin{proof}
By e.g. \cite[Theorem A.14]{Raeburn_1998},
\begin{multline*}
	\| \xi \| = \big\| \langle ξ \mid ξ \rangle \big\|^{1/2} = \big\| \langle ξ \mid ξ \rangle^{1/2} \big\| = \sup_{\pi} \| \pi(\langle ξ \mid ξ \rangle^{1/2}) \| \\
	= \sup_{\pi} \sup_{\eta \in H_{\pi}} \frac{\| \langle ξ \mid ξ \rangle^{1/2} \eta \|}{\| \eta \|} = \sup_{\pi} \sup_{\eta \in H_{\pi}} \frac{\langle ξ \otimes \eta \mid ξ \otimes \eta \rangle^{1/2}}{\| \eta \|} = \sup_{\pi} \sup_{\eta \in H_{\pi}} \frac{\| ξ \otimes \eta \|}{\| \eta \|} .
\end{multline*}
Next, note that \( \| T \otimes 1 \|_{B(E \otimes_{\pi} H_{\pi})} \leq \| T \|_{\End^*(E)} \).
On the other hand,
\begin{multline*}
	\| T \|_{\End^*(E)} = \sup_{\xi \in E} \frac{\| T \xi \|_E}{\| \xi \|_E} = \sup_{\xi \in E} \frac{\sup_{\pi} \sup_{\eta \in H_{\pi}} \frac{\| T ξ \otimes \eta \|}{\| \eta \|}}{\sup_{\pi} \sup_{\eta \in H_{\pi}} \frac{\| ξ \otimes \eta \|}{\| \eta \|}} \\
	\leq \sup_{\xi \in E} \sup_{\pi} \sup_{\eta \in H_{\pi}} \frac{\| T ξ \otimes \eta \|}{\| ξ \otimes \eta \|} \leq \sup_{\pi} \| T \otimes 1 \|_{B(E \otimes_{\pi} H_{\pi})}
\end{multline*}
and we obtain the required equality.
\end{proof}

\begin{theorem}
	\label{theorem:fractional-power-bounds}
	cf. \cite[Theorem 12.5]{Krasnoselskii_1976}
	Let \( A \) and \( B \) be positive regular operators on Hilbert \( B \)-modules \( E_1 \) and \( E_2 \) respectively. Let \( T \) be an adjointable operator from \( E_2 \) to \( E_1 \). If \( T \dom(B) ⊆ \dom(A) \), then \( T \dom(B^α) ⊆ \dom(A^α) \) for any \( 0 < α ≤ 1 \). If, in addition, there exists an \( M ≥ 0 \) such that, for all $\xi\in \dom(B)$,
	\begin{equation}
		\label{equation:fractional-power-bound}
		\| A T ξ \| ≤ M \| B ξ \| ,
	\end{equation}
	then
	\[ \| A^α T ξ \| ≤ M^α \| T \|^{1-α} \| B^α ξ \| .\]
	In particular, if \( B \) is invertible,
	\[ \| A^α T B^{-α} \| ≤ \| A T B^{-1} \|^α \| T \|^{1-α} . \]
\end{theorem}
\begin{proof}
	By considering the direct sum \( E_1 \oplus E_2 \), if necessary, we can without loss of generality assume that \( E_1 = E_2 =: E \).

	We will begin with the case of \( A \) bounded and adjointable and \( B \) invertible. In this case, a bound of the form \eqref{equation:fractional-power-bound} always holds, the best available bound being given by \( M = \| A T B^{-1} \| \). For any \( 0 < α ≤ 1 \), \( A^α \) is adjointable and \( B^α \) is invertible. Let \( \pi : B \to B(H_{\pi}) \) be an irreducible representation of \( B \) and let \( \xi \in E_2 \otimes_{\pi} H_{\pi} \).
	Define the holomorphic function
	\[ f : z ↦ \langle \xi \mid (B \otimes 1)^{-z} (T \otimes 1)^* (A \otimes 1)^{2 z} (T \otimes 1) (B \otimes 1)^{-z} \xi \rangle \| \xi \|^{-2} \]
	on the strip where \( 0 \leq \Re(z) ≤ 1 \). We have
	\[ | f(z) | \leq \| (A \otimes 1)^{\bar{z}} (T \otimes 1) (B \otimes 1)^{-\bar{z}} \| \| (A \otimes 1)^z (T \otimes 1) (B \otimes 1)^{-z} \| . \]
	For \( β ∈ \bbR \),
	\[ | f(1 + β i) | \leq \| (A \otimes 1) (T \otimes 1) (B \otimes 1) \|^2 ≤ \| A T B^{-1} \|^2 \]
	and
	\[ | f(β i) | ≤ \| T \otimes 1 \|^2 ≤ \| T \|^2 . \]
	By Hadamard's three-line theorem, we obtain that
	\[ \| (A \otimes 1)^{\alpha} (T \otimes 1) (B \otimes 1)^{-\alpha} \xi \|^2 \| \xi \|^{-2} = | f(α) | ≤ \| A T B^{-1} \|^{2 α} \| T \|^{2- 2α} \]
	for \( 0 ≤ α ≤ 1 \). Hence \( \| (A \otimes 1)^{\alpha} (T \otimes 1) (B \otimes 1)^{-\alpha} \| \leq \| A T B^{-1} \|^{α} \| T \|^{1-α} \). Assuming further that \( α \neq 0 \), so that \( A^{\alpha} \) and \( B^{-\alpha} \) are well-defined as adjointable operators on \( E \),
	\[ \| A^{\alpha} T B^{-\alpha} \|_{\End^*(E)} = \sup_{[\pi] \in \hat{B}} \| A^{\alpha} T B^{-\alpha} \otimes 1 \|_{B(E \otimes_{\pi} H_{\pi})} ≤ \| A T B^{-1} \|^{α} \| T \|^{1-α} \]
	by Lemma \ref{lemma:rt7ub46n78i5n6b}. For \( ξ ∈ \dom(B^α) \),
	\[ \| A^α T ξ \| ≤ \| A^α T B^{-α} \| \| B^α ξ \| ≤ \| A T B^{-1} \|^α \| T \|^{1-α} \| B^α ξ \| \]
	as required.
	
	Now consider the case of general \( A \) and \( B \) when the bound \eqref{equation:fractional-power-bound} applies. As in the previous section, let \( (φ_n)_{n ∈ \bbN} ⊂ C_c(\bbR) \) be a sequence of positive functions, bounded by 1 and converging uniformly on compact subsets to the constant function 1. Let
	\[ A_n = A φ_n(A) \qquad B_n = B + \frac{1}{n} \qquad (n > 0). \]
	The operators \( A_n \) are bounded and adjointable and \( B_n \) are invertible. For \( η ∈ \dom A \) and \( ξ ∈ \dom B \),
	\[ \| A_n η \| ≤ \| A η \| \qquad \| B ξ \| ≤ \| B_n ξ \| \]
	and so
	\[ \| A_n T ξ \| ≤ M \| B_n ξ \|. \]
	As we have seen, for \( 0 < α ≤ 1 \),
	\[ \| A_n^α T ξ \| ≤ M^α \| T \|^{1-α} \| B_n^α ξ \| \qquad (ξ ∈ \dom(B_n^α) = \dom(B^α)). \]
	The sequence \( φ_n(A)^α T ξ \to T ξ \) as \( n \to ∞ \) by Theorem \ref{theorem:strong-convergence-from-compact-convergence}. The bounded functions
	\[ x ↦ (x + 1/n)^α - x^α \]
	converge uniformly to zero as \( n \to ∞ \), hence \( B_n^α ξ \to B^α ξ \), again by Theorem \ref{theorem:strong-convergence-from-compact-convergence}. Then
	\[ \sup_n \| A^α φ_n(A)^α T ξ \| = \sup_n \| A_n^α T ξ \| ≤ \sup_n M^α \| T \|^{1-α} \| B_n^α ξ \| < ∞. \]
	Because \( A^α \) is a closed operator, \( T ξ ∈ \dom(A^α) \) and \( A_n^α T ξ = A^α φ_n(A)^α T ξ \to A^α T ξ \) as \( n \to ∞ \). Taking the limit as \( n \to ∞ \), we find that for $ξ ∈ \dom(B^α)$
	\[ \| A^α T ξ \| ≤ M^α \| T \|^{1-α} \| B^α ξ \|.\]
For the case of general \( A \) and \( B \) with \( T \dom(B) ⊆ \dom(A) \) but without the bound \eqref{equation:fractional-power-bound}, we let \( B_1 = B + 1 \). As \( B_1 \) is invertible, for $ξ ∈ \dom(B)$
	\[ \| A T ξ \| ≤ \| A T B_1^{-1} \| \| B_1 ξ \|. \]
	We have shown that \( T \dom(B_1^α) ⊆ \dom(A^α) \) and, as \( \dom(B_1) = \dom(B) \), we are done.
\end{proof}

\subsection{Hilbert C*-modules over topological spaces}
\label{appendix:hilbert-modules-top}

We review and extend some known facts about Hilbert modules built from functions $X\to E_B$ for a fixed Hilbert module $E_B$ and a locally compact Hausdorff space $X$.

\begin{definition}
	e.g.~\cite[\S B.2]{Raeburn_1998}
	Let \( A \) be a C*-algebra and \( X \) a locally compact Hausdorff space. Define \( C_0(X, A) \) to be the C*-algebra of norm-continuous functions \( f : X \to A \) such that \( x \mapsto \| f(x) \|_A \) vanishes at infinity, equipped with the supremum norm. Let \( E \) be a right Hilbert \( A \)-module. Define \( C_0(X, E) \) to be the set of continuous functions \( f : X \to E \) such that \( x \mapsto \| f(x) \|_E \) vanishes at infinity.
\end{definition}

\begin{lemma}
	cf.~\cite[Example 2.13]{Raeburn_1998}
	Let \( E \) be a right Hilbert \( A \)-module and \( X \) a locally compact Hausdorff space. Then \( C_0(X, E) \) is a right Hilbert \( C_0(X, A) \)-module with inner product and right action defined pointwise in \( X \).
\end{lemma}
\begin{proof}
	The algebraic conditions on a Hilbert module are satisfied for \( C_0(X, E) \) since they are satisfied pointwise for \( E \). The norm on an element \( f ∈ C_0(X, E) \) arising from the inner product is
	\[ \left\| \langle f \mid f \rangle_{C_0(X, A)} \right\|_{C_0(X, A)}^{1/2} = \sup_{x ∈ X} \left\| \langle f \mid f \rangle_{C_0(X, A)}(x) \right\|_A^{1/2} = \sup_{x ∈ X} \left\| \langle f(x) \mid f(x) \rangle_A \right\|_A^{1/2} = \sup_{x ∈ X} \| f(x) \|_E \]
	which is the supremum norm. Hence, \( C_0(X, E) \) is complete as Hilbert module.
\end{proof}

\begin{lemma}
	\label{lemma:inner_product_ideal_of_functions_into_hilbert_module}
	Let \( E \) be a right Hilbert \( B \)-module and \( X \) a locally compact Hausdorff space. Let \( J = \overline{\Span} \langle E \mid E \rangle_B \) be the ideal of \( A \) generated by inner products on \( E \). There is an equality
	\[ \overline{\Span} \langle C_0(X, E) \mid C_0(X, E) \rangle_{C_0(X, B)} = C_0(X, J) \]
	of ideals of \( C_0(X, B) \).
\end{lemma}
\begin{proof}
	Consider \( f_1, f_2 ∈ C_0(X, E) \). Their inner product is given at \( x ∈ X \) by
	\[ \langle f_1 \mid f_2 \rangle_{C_0(X, B)}(x) = \langle f_1(x) \mid f_2(x) \rangle_B ∈ J . \]
	Noting that
	\[ \left\| \langle f_1(x) \mid f_2(x) \rangle_B \right\|_B \leq \left\| \langle f_1(x) \mid f_1(x) \rangle_B \right\|_B^{1/2} \left\| \langle f_2(x) \mid f_2(x) \rangle_B \right\|_B^{1/2} = \| f_1(x) \|_E \| f_2(x) \|_E , \]
	we see that \( \langle f_1 \mid f_2 \langle_{C_0(X, B)} ∈ C_0(X, J) \). Hence
	\[ \langle C_0(X, E) \mid C_0(X, E) \rangle_{C_0(X, B)} ⊆ C_0(X, J) . \]
	Label the ideal \( I = \overline{\Span} \langle C_0(X, E) \mid C_0(X, E) \rangle_{C_0(X, B)} \) of \( C_0(X, B) \). By e.g.~\cite[\S 1.2]{Fell_1961}, \( I \) must have the form
	\[ \{ s ∈ C_0(X, B) \mid \forall x ∈ X, s(x) ∈ I_x \} \]
	where each \( I_x = \{ s(x) \mid s ∈ I \} \) is an ideal of \( B \). We must have \( I_x ⊆ J \) for every \( x ∈ X \). Suppose that \( I_{x_0} \neq J \) for some \( x_0 ∈ X \). Since \( \langle E \mid E \rangle_B \) is linearly dense in \( J \), it is not contained in \( I_{x_0} \), and there must be a pair \( e_1, e_2 ∈ E \) such that \( \langle e_1 \mid e_2 \rangle_B ∈ J \setminus I_{x_0} \). Choose a function \( h ∈ C_0(X) \) for which \( h(x_0) = 1 \) and define \( f_1, f_2 ∈ C_0(X, E) \) on \( x ∈ X \) by
	\( f_i(x) = e_i h(x) \).
	Then
	\[ \langle f_1 \mid f_2 \rangle_{C_0(X, B)}(x_0) = \langle f_1(x_0) \mid f_2(x_0) \rangle_B = \langle e_1 \mid e_2 \rangle_B \]
	is not in \( I_{x_0} \), so \( \langle f_1 \mid f_2 \rangle_{C_0(X, B)} \) is not in \( I \), which is a contradiction. In other words, \( I_x = J \) for every \( x ∈ X \) and \( I = C_0(X, J) \).
\end{proof}

\begin{lemma}
	\label{lemma:functions_into_meb_gives_meb}
	Let \( E \) be a Morita equivalence \( A \)-\( B \)-bimodule and \( X \) a locally compact Hausdorff space. Then \( C_0(X, E) \) is a Morita equivalence \( C_0(X, A) \)-\( C_0(X, B) \)-bimodule.
\end{lemma}
\begin{proof}
	The left and right norms on \( E \) agree by \cite[Lemma 2.30]{Raeburn_1998}, so there is no ambiguity in the continuity used to define \( C_0(X, E) \). The algebraic properties of a Morita equivalence bimodule are satisfied for \( C_0(X, E) \) because they are satisfied pointwise for \( E \). The fullness of \( C_0(X, E) \) as a right and left Hilbert module follows from Lemma \ref{lemma:inner_product_ideal_of_functions_into_hilbert_module} and the fullness of \( E \).
\end{proof}

\begin{lemma}
	\label{lemma:module_over_space}
	Let \( E \) be a right Hilbert \( B \)-module and \( X \) a locally compact Hausdorff space. Then
	\[ \End^*(C_0(X, E)) = C_b(X, \End^*(E)_{*-s}) \]
	the C*-algebra of $*$-strong-continuous functions \( f : X \to \End^*(E) \) such that \( \sup_{x ∈ X} \| f(x) \|_{\End^*(E)} < \infty \). Furthermore, \( \End^0(C_0(X, E)) = C_0(X, \End^0(E)) \).
\end{lemma}
\begin{proof}
	Let \( A = \End^0(E) \), so that \( E \) is a Morita equivalence \( A \)-\( B \)-bimodule. By \cite[Corollary 2.54]{Raeburn_1998}, \( \End^*(E) = M(A) \), the multiplier algebra of \( A \). The equality
	\[ \End^0(C_0(X, E)) = C_0(X, \End^0(E)) = C_0(X, A) \]
	is a consequence of Lemma \ref{lemma:functions_into_meb_gives_meb}. Again by \cite[Corollary 2.54]{Raeburn_1998},
	\[ \End^*(C_0(X, E)) = M(\End^0(C_0(X, E))) = M(C_0(X, A)) . \]
	Let \( M(A)_β \) be \( M(A) \) equipped with the strict topology. By \cite[Corollary 3.4]{Akemann_1973},
	\[ M(C_0(X, A)) = C_b(X, M(A)_β) , \]
	the C*-algebra of strictly continuous and norm-bounded functions. By \cite[Proposition C.7]{Raeburn_1998}, the strict topology on \( M(A) = \End^*(E) \) agrees with the $*$-strong topology on norm-bounded subsets. Hence
	\[ C_b(X, M(A)_β) = C_b(X, \End^*(E)_{*-s}) , \]
	where the norm on both algebras is given by the operator norm on \( E \) composed with the supremum norm over \( X \). Finally, we obtain
	\[ \End^*(C_0(X, E)) = C_b(X, \End^*(E)_{*-s}) , \]
	as required.
\end{proof}

\begin{definition}
	\label{definition:k-space}
	e.g.~\cite[Definition 43.8]{Willard_1970}
	A topological space \( X \) is a \emph{k-space} if a subset \( Y \) of \( X \) is open if, and only if, for every compact subset \( K \) of \( X \), \( Y \cap K \) is open in \( K \). Conditions on \( X \) which imply that it is a k-space include local compactness and first-countability \cite[Theorem 43.9]{Willard_1970}.
\end{definition}

\begin{lemma}
	\label{lemma:compact_to_whole}
	e.g.~\cite[Lemma 43.10]{Willard_1970}
	Let \( f : X \to Y \) be a map between topological spaces with \( X \) a k-space. Then the continuity of \( f \) is equivalent to the continuity of \( f \) restricted to \( K \) for all compact subsets \( K \subseteq X \).
\end{lemma}

\begin{lemma}
	\label{lemma:norm-continuous_on_compact_subsets_to_compact_endomorphism}
	Let \( E \) be a right Hilbert \( A \)-module and \( X \) a locally compact Hausdorff space. The norm-continuity of a function \( f : X \to \End^0(E) \) is equivalent to the condition that \( f|_K ∈ \End^0(C(K, E)) \) for all compact subsets \( K \subseteq X \).
\end{lemma}
\begin{proof}
	By Lemma \ref{lemma:compact_to_whole}, the norm-continuity of a function \( f : X \to \End^0(E) \) is equivalent to the norm-continuity of \( f|_K \) for every compact subset \( K \subseteq X \). By Lemma \ref{lemma:module_over_space}, the norm-continuous functions from a given \( K \) to \( \End^0(E) \) can be identified with the elements of \( \End^0(C(K, E)) \).
\end{proof}

\begin{theorem}
	\label{theorem:banach-steinhaus}
	(Banach--Steinhaus or uniform boundedness principle)
	e.g.~\cite[Theorem III.9]{Reed_1980}
	Let \( V \) be a Banach space and \( W \) a normed linear space. Let \( \mathscr{F} \subset B(V, W) \) be a family of bounded operators from \( V \) to \( W \) with \( \sup_{T ∈ \mathscr{F}} \| T v \|_W < \infty \) for each \( v ∈ V \). Then \( \sup_{T ∈ \mathscr{F}} \| T \|_{B(V, W)} < \infty \).
\end{theorem}

\begin{corollary}
	\label{corollary:strongly_continuous_to_norm_bounded}
	Let \( V \) be a Banach space and \( X \) be a compact space. Let \( f : X \to B(V) \) be a strongly continuous map. Then \( f \) is bounded in operator norm; in other words, \( \sup_{x ∈ X} \| f(x) \|_{B(V)} < \infty \).
\end{corollary}
\begin{proof}
	We have a family \( \mathscr{F} = ( f(x) )_{x ∈ X} \subset B(V) \) of bounded operators. The strong continuity of \( f \) implies that
	\( x \mapsto f(x) v \)
	is continuous for every \( v ∈ V \). Since \( X \) is compact, its image \( f(X) v ⊆ V \) is compact and thus bounded. Hence, for a fixed \( v ∈ V \),
	\[ \sup_{T ∈ \mathscr{F}} \| T v \|_V = \sup_{x ∈ X} \| f(x) v \|_V < \infty . \]
	Applying Theorem \ref{theorem:banach-steinhaus}, we obtain that
	\[ \sup_{x ∈ X} \| f(x) \|_{B(V)} = \sup_{T ∈ \mathscr{F}} \| T \|_{B(V)} < \infty , \]
	as required.
\end{proof}

\begin{lemma}
	\label{lemma:star-strongly_continuous_on_compact_to_bounded_endomorphism}
	Let \( E \) be a right Hilbert \( A \)-module and \( X \) a compact Hausdorff space. The $*$-strong continuity of a function \( f : X \to \End^*(E) \) is equivalent to the condition that \( f ∈ \End^*(C(X, E)) \).
\end{lemma}
\begin{proof}
	By Lemma \ref{lemma:module_over_space}, \( \End^*(C(X, E)) = C_b(X, \End^*(E)_{*-s}) \), the C*-algebra of $*$-strongly continuous functions \( f : X \to \End^*(E) \) such that \( \sup_{x ∈ X} \| f(x) \|_{\End^*(E)} < \infty \).
	If \( f ∈ \End^*(C(X, E)) \), then it is $*$-strongly continuous as a function \( f : X \to \End^*(E) \).
	On the other hand, if we assume \( f : X \to \End^*(E) \) is $*$-strongly continuous, we may apply Corollary \ref{corollary:strongly_continuous_to_norm_bounded}. Thereby, \( \sup_{x ∈ X} \| f(x) \|_{\End^*(E)} < \infty \) and so \( f ∈ \End^*(C(X, E)) \).
\end{proof}

\begin{lemma}
	\label{lemma:strongly-continuous_on_compact_subsets_to_bounded_endomorphism}
	Let \( E \) be a right Hilbert \( A \)-module and \( X \) a locally compact Hausdorff space. The $*$-strong continuity of a function \( f : X \to \End^*(E) \) is equivalent to the condition that \( f|_K ∈ \End^*(C(K, E)) \) for all compact subsets \( K \subseteq X \).
\end{lemma}
\begin{proof}
	By Lemma \ref{lemma:compact_to_whole}, the $*$-strong continuity of a function \( f : X \to \End^*(E) \) is equivalent to the $*$-strong continuity of \( f|_K \) for every compact subset \( K \subseteq X \). By Lemma \ref{lemma:star-strongly_continuous_on_compact_to_bounded_endomorphism}, the $*$-strong continuity of \( f|_K : K \to \End^*(E) \) for a given \( K \) is equivalent to the condition that \( f|_K ∈ \End^*(C(K, E)) \).
\end{proof}

\subsection{Matched operators}
\label{appendix:temp-bdd-operators}

\begin{definition}
	Let \( E \) be a Hilbert \( B \)-module and \( C \) a C*-algebra represented on the right of \( E \) by a nondegenerate C*-homomorphism \( ρ : C \to M(B) \). A regular operator \( T \) on \( E \) is \emph{\( C \)-matched} if those \( c ∈ C \) for which
	\[ E ρ(c) ⊆ \dom(T) \]
	are dense in \( C \).
\end{definition}

\begin{remark}
	The condition that \( E ρ(c) ⊆ \dom(T) \) combined with Lemma \ref{lemma:domain-preservation-on-banach-space} implies that the \( \bbC \)-linear map 
	\[ E \to E \qquad ξ ↦ T ξ c \]
	is bounded.
\end{remark}

\begin{lemma}
	Let \( E \) be a Hilbert \( B \)-module and \( C \) a C*-algebra represented on the right of \( E \) by a C*-homomorphism \( ρ : C \to M(B) \). Let \( T \) be a regular operator on \( E \). The set of \( c ∈ C \) for which
	\[ E ρ(c) ⊆ \dom(T) \]
	form a (not necessarily closed) two-sided ideal in \( C \).
\end{lemma}
\begin{proof}
	This follows from a general statement about rings and modules. Suppose that we have 
	$ E ρ(c) ⊆ \dom(T) $
	for some \( c ∈ C \). If \( c_1, c_2 ∈ C \), then
	\[ E ρ(c_1 c c_2) = E ρ(c_1) ρ(c) ρ(c_2) ⊆ E ρ(c) ρ(c_2) ⊆ \dom(T) ρ(c_2) ⊆ \dom(T) \]
	and we are done.
\end{proof}

Recall that the Pedersen ideal \( K_C \) of a C*-algebra \( C \) is the minimal dense two-sided ideal of \( C \); see e.g. \cite[§II.5.2]{Blackadar_2006}.

\begin{proposition}
	Let \( T \) be a regular operator on \( E_B \) which is \( C \)-matched. Then
	\[ E ρ(c) ⊆ \dom(T) \]
	for all \( c ∈ K_C \), the Pedersen ideal of \( C \). Furthermore, \( E ρ(K_C) B \) is a core for \( T \).
\end{proposition}
\begin{proof}
	As those \( c ∈ C \) for which
	$ E ρ(c) ⊆ \dom(T) $
	form a dense two-sided ideal, they must include the Pedersen ideal. For an element \( c ∈ K_C \), there exists an element \( d ∈ K_C \) such that \( d c = c \). Hence
	\[ E ρ(c) = E ρ(d) ρ(c) ⊆ \dom(T) ρ(c) ⊆ E ρ(c) \]
	and \( E ρ(K_C) = \dom(T) ρ(K_C) = (1 + T^* T)^{-1/2} E ρ(K_C) \). Next, note that \( ρ(K_C) \) is dense in \( ρ(C) \). By the continuity of multiplication, \( E ρ(K_C) B \) is dense in \( E ρ(C) B \). By nondegeneracy of \( ρ \), \( B ρ(C) \) is dense in \( B \) and, again, by the continuity of multiplication, \( E B ρ(C) B = E ρ(C) B \) is dense in \( E B = E \). Hence \( E ρ(K_C) B \) is dense in \( E \) and \( E ρ(K_C) B = (1 + T^* T)^{-1/2} E ρ(K_C) B \) is consequently a core for \( T \).
\end{proof}	
\begin{remark}
	In \cite{Webster_2004}, the multiplier algebra \( Γ(K_B) \) of the Pedersen ideal of \( B \) is shown to consist of exactly those unbounded operators affiliated with \( B \), in the sense of \cite{Woronowicz_1991}, whose domains include \( K_B \). A similar characterisation is given in \cite[Théorème 1.30]{Pierrot_2006}. The previous Proposition can be used to show that, if \( ρ(C) = B \), the \( C \)-matched operators on \( E_B \) are exactly the multipliers \( Γ(K_{\End^0(E)}) \) of the Pedersen ideal of \( \End^0(E) \). See \cite[Proposition 1.7]{Ara_2001} for the details of passing through the Morita equivalence bimodule \( {}_{\End^0(E)}E_B \).
\end{remark}

\begin{lemma}
	\label{lemma:temp-bdd-on-hereditary}
	Let \( E \) be a Hilbert \( B \)-module and \( C \) a C*-algebra represented on the right of \( E \) by a C*-homomorphism \( ρ : C \to M(B) \). A regular operator \( T \) on \( E \) is \( C \)-matched if and only if, for all \( c ∈ K_C \), the restriction \( T|_{\overline{E ρ(c)}} \) of \( T \) to the Hilbert submodule \( \overline{E ρ(c)} \) over the hereditary C*-subalgebra \( \overline{ρ(c)^* B ρ(c)} \) of \( B \) is bounded.
\end{lemma}
\begin{proof}
	Assume that \( E ρ(c) ⊆ \dom(T) \) for \( c ∈ K_C \). Choose \( d ∈ K_C \) such that \( d c = c \). As \( E ρ(d) ⊆ \dom(T) \), the \( \bbC \)-linear map \( ξ ↦ T ξ ρ(d) \) on \( E \) is bounded by Lemma \ref{lemma:domain-preservation-on-banach-space}. On \( \overline{E ρ(c)} \), \( ρ(d) \) acts as the identity, meaning \( T \) restricts to a bounded operator on \(\overline{E ρ(c)} \).
		
	On the other hand, assume that \( T|_{\overline{E ρ(c)}} \) is bounded for \( c ∈ K_C \). Then \( \dom(T) ⊇ \overline{E ρ(c)} ⊇ E ρ(c) \), as required.
\end{proof}

The following is well-known.

\begin{lemma}
	Let \( a \) be an element of the multiplier algebra of a C*-algebra \( A \). Then the closed right ideal \( \overline{a A} \) is a Morita equivalence bimodule between the hereditary C*-subalgebra \( \overline{a A a^*} \) of \( A \) and the (closed two-sided) ideal \( \overline{\Span}(A a^* a A) ⊴ A \).
\end{lemma}

\begin{proposition}
	\label{proposition:temp-bdd-on-ideals}
	Let \( E \) be a Hilbert \( B \)-module and \( C \) a C*-algebra represented on the right of \( E \) by a C*-homomorphism \( ρ : C \to M(B) \). A regular operator \( T \) on \( E \) is \( C \)-matched if and only if, for all positive \( c ∈ K_C \), the restriction \( T|_{\overline{\Span}(E ρ(c) B)} \) of \( T \) to the Hilbert submodule \( \overline{\Span}(E ρ(c) B) \) over the ideal \( \overline{\Span}(B ρ(c) B) ⊴ B \) is bounded.
\end{proposition}
\begin{proof}
	Assume that \( E ρ(c) ⊆ \dom(T) \) for \( c ∈ K_C \). Then the restriction of \( T \) to \( \overline{E ρ(c)}_{\overline{ρ(c)^* B ρ(c)}} \) is bounded. The closed right ideal \( \overline{ρ(c)^* B} \) of \( B \) is a Morita equivalence \( \overline{ρ(c)^* B ρ(c)} \)-\( \overline{\Span}(B ρ(c c^*) B) \)-bimodule. We have a natural isomorphism
	\[ \overline{\Span}(E ρ(c c^*) B)_{\overline{\Span}(B ρ(c c^*) B)} ≅ \overline{E ρ(c)}_{\overline{ρ(c)^* B ρ(c)}} ⊗_{\overline{ρ(c)^* B ρ(c)}} \overline{ρ(c)^* B}_{\overline{\Span}(B ρ(c c^*) B)} \]
	of Hilbert \( \overline{\Span}(B ρ(c c^*) B) \)-modules, under which \( T|_{\overline{E ρ(c c^*) B}} ≅ T|_{\overline{E ρ(c)}} ⊗_{\overline{ρ(c)^* B ρ(c)}} 1 \). Hence the restriction \( T|_{\overline{\Span}(E ρ(c c^*) B)} \) is bounded. Since every positive element of \( K_C \) is of the form \( c c^* \), we conclude this direction of the argument.
			
	On the other hand, assume that \( T|_{\overline{\Span}(E ρ(c) B)} \) is bounded for \( c ∈ K_C \). Recall that the product of (two-sided) closed ideals in a C*-algebra is again a closed ideal, so that \( \overline{B ρ(c) B} = B \overline{M(B) ρ(c) M(B)} \). Then
	\[ \dom(T) ⊇ E \: \overline{\Span}(B ρ(c) B) = E \: \overline{\Span}(M(B) ρ(c c^*) M(B)) ⊇ E ρ(c) , \]
	as required.
\end{proof}

\begin{lemma}
	\label{lemma:irreducible-rep-complete-image}
	cf. \cite[Proof of Proposition 4.5]{Lazar_1976}
	Let \( π \) be an irreducible representation of a C*-algebra \( A \) on a Hilbert space \( H \). Then \( K_A H = H \).
\end{lemma}
\begin{proof}
	Let \( ξ ∈ H \) be a cyclic vector and choose \( a ∈ K_A \) such that \( \| π(a) ξ \| = 1 \). (Such an \( a ∈ K_A \) can always be found; otherwise the density of \( K_A \) in \( A \) would imply that \( ξ = 0 \).) Let \( η ∈ H \) be any non-zero vector. The finite rank operator \( |η⟩⟨π(a) ξ| \) takes \( a ξ \) to \( η \). By \cite[Theorem 2.8.3(i)]{Dixmier_1977}, there exists an element \( b ∈ A \) such that
	\[ η = |η⟩⟨π(a) ξ| π(a) ξ = π(b) π(a) ξ ∈ K_A H \]
	as required.
\end{proof}

\begin{proposition}
\label{prop:star-alg}
	The \( C \)-matched operators on \( E_B \) form a $*$-algebra \( \Mtc^*_B(E, C) \).
\end{proposition}
\begin{proof}
	Let \( T \) be a regular operator on \( E_B \) which is \( C \)-matched. By Lemma \ref{lemma:temp-bdd-on-hereditary}, \( T \) restricts to a bounded operator on \( \overline{E ρ(c)}|_{\overline{ρ(c) B ρ(c)}} \) for all \( c ∈ K_C \). The restrictions $(T|_{E\rho(c)})^*=T^*|_{E\rho(c)}$ of the adjoint \( T^* \) of \( T \) are consequently bounded, and so \( T^* \) is also \( C \)-matched, again by Lemma \ref{lemma:temp-bdd-on-hereditary}.
		
	Let \( T_1 \) and \( T_2 \) be \( C \)-matched operators. For an element \( c ∈ K_C \), we have
	\[ T_2 E ρ(c) = T_2 \dom(T_2) ρ(c) ⊆ E ρ(c) ⊆ \dom(T_1) \]
	so that \( T_1 T_2 \) is well-defined on \( E ρ(K_C) B \). Similarly, \( T_2^* T_1^* \) is also well-defined on \( E ρ(K_C) B \) so that \( T_1 T_2 \) is semiregular. The localisation of \( E K_B ⊆ E ρ(K_C) B \) to any irreducible \( π ∈ \hat{B} \) is equal to
	\[ E_B K_B ⊗_π H_π = E_B ⊗_π π(K_B) H_π = E_B ⊗_π H_π \]
	by Lemma \ref{lemma:irreducible-rep-complete-image}. Hence, \( \dom((T_1 T_2)^π) = E_B ⊗_π H_π \) and \( (T_1 T_2)^π \) is bounded. As the same is true for \( (T_2^* T_1^*)^π \), we may apply the local-global principle \cite[Théorème 1.18(2)]{Pierrot_2006} to obtain that the closure of \( T_1 T_2 \) is a regular operator on \( E \).  By similar reasoning, we conclude that the closure of the sum \( T_1 + T_2 \), defined on the common core \( E ρ(K_C) B \), is a regular operator on \( E \).
\end{proof}

\begin{remark}
	Combined with Proposition \ref{proposition:temp-bdd-on-ideals}, Proposition \ref{prop:star-alg} could be used to show that \( \Mtc^*_B(E, C) \) is a pro-C*-algebra (or locally C*-algebra) \cite{Phillips_1988}, \cite[Chapter II]{Fragoulopoulou_2005}.
\end{remark}

\begin{proposition}
	\label{proposition:temp-bbd-operators-topological}
	Let \( X \) be a locally compact Hausdorff space and \( E \) a Hilbert \( B \)-module. Then the \( C_0(X) \)-matched operators on \( C_0(X, E) \) are exactly the elements of \( C(X, \End^*(E)_{*-s}) \), the (not necessarily bounded) $*$-strongly continuous functions from \( X \) to \( \End^*(E) \).
\end{proposition}
\begin{proof}
	Suppose that \( T \) is a \( C_0(X) \)-matched operator on \( C_0(X, E) \). Because \( T (1 + T^* T)^{-1/2} ∈ \End^*(C_0(X, E)) = C_b(X, \End^*(E)_{*-s}) \) uniquely determines \( T \), we may conclude that \( T \) is given by a function from \( X \) to regular operators on \( E \). Let \( K \) be a compact subset of \( X \). The Pedersen ideal of \( C_0(X) \) is \( C_c(X) \), the compactly supported functions on \( X \). Let \( f \) be a positive element of \( C_c(X) \) which is nonzero on \( K \). We have
	\[ \dom(T) ⊇ C_0(X, E) f = C_0(\supp f, E) \]
	so that \( T \) restricts to a bounded operator on \( C_0(\supp f, E)_{C_0(\supp f, B)} \). By Lemma \ref{lemma:module_over_space},
	\[ \End^*(C_0(\supp f, E)) = C_b(\supp f, \End^*(E)_{*-s}) . \]
	Furthermore, the localisation of \( T \) to \( C(K, E)_{C(K, B)} \) must also be bounded and so an element of \( C_b(K, \End^*(E)_{*-s}) \). Given that \( T \) is a $*$-strongly continuous function on every compact subset \( K \) of the k-space \( X \), by Lemma \ref{lemma:compact_to_whole}, \( T \) is a $*$-strongly continuous function on \( X \).
		
	Let \( T ∈ C(X, \End^*(E)_{*-s}) \). Then \( T (1 + T^* T)^{-1/2} ∈ C_b(X, \End^*(E)_{*-s}) \) and
	\[ (1 + T^* T)^{-1/2} C_0(X, E) ⊇ C_c(X, E) \]
	so that \( T \) is a regular operator on \( C_0(X, E) \). (For a more detailed argument, cf. \cite[§4]{Pal_1999}.) Furthermore, for an element \( f ∈ K_{C_0(X)} = C_c(X) \), \( C_0(X, E) f ⊆ C_c(X, E) ⊆ \dom(T) \) and \( T \) is \( C_0(X) \)-matched.
\end{proof}

\subsection{Compactly supported states}
\label{subsec:compact-states}

\begin{definition}
	\cite[Definition 6.11]{Harris_2023}
	A state \( ψ \) on a C*-algebra \( A \) is \emph{compactly supported} if there exists an \( a ∈ A \) such that \( ψ(a) = \| a \| \). We denote the set of compactly supported states on \( A \) by \( \mathcal{S}_c(A) \).
\end{definition}

\begin{proposition}
	\label{proposition:compactly-supp-states}
	For a state \( ψ \) of a C*-algebra \( A \), the following are equivalent:
	\begin{enumerate}[(1)]
		\item \( ψ \) is compactly supported, i.e. there exists an \( a ∈ A \) such that \( ψ(a) = \| a \| \).
		\item There exists an \( a ∈ K_A \) such that \( ψ(a) = \| a \| \).
		\item There exists a positive \( a ∈ K_A \) such that \( ψ(a) = 1 = \| a \| \) and \( ψ(a b) = ψ(b) \) for all \( b ∈ A \).
		\item \( ψ \) is given by \( b ↦ \frac{ϕ(a^* b a)}{ϕ(a^* a)} \) for a state \( ϕ \) of A and an \( a ∈ K_A \).
	\end{enumerate}
\end{proposition}
\begin{proof}
	(2) clearly implies (1). (4) implies (2) almost by definition of the Pedersen ideal. If \( ψ : b ↦ \frac{ϕ(a^* b a)}{ϕ(a^* a)} \) for \( a ∈ K_A \), there exists positive \( c ∈ A \) such that \( c a = a \). Let \( f ∈ C_c(\bbR^×_+) \) be a compactly supported continuous function which is equal to 1 on the spectrum of \( c \). By the continuous functional calculus, we obtain \( f(c) ∈ K_A \) such that \( f(c) a = a \) and \( \| f(c) \| = 1 \), and therefore
	\[ ψ(f(c)) = \frac{ϕ(a^* f(c) a)}{ϕ(a^* a)} = 1 = \| f(c) \| . \]
	
	To see that (1) implies (3), let \( a ∈ A \) be such that \( ψ(a) = 1 = \| a \| \). By the Kadison inequality, \( ψ(a^* a) ≥ |ψ(a)|^2 = 1 \) and since \( \| a^* a \| = \| a \|^2 = 1 \), we must have \( ψ(a^* a) = 1 \). We may assume, without loss of generality, that \( a \) is positive. Let \( \tilde{A} \) be the minimal unitisation of \( A \) and \( \tilde{ψ} \) the unique extension of \( ψ \). Let \( H_{\tilde{ψ}} \) be the Hilbert space of the corresponding GNS representation and \( ξ_{\tilde{ψ}} \) the cyclic vector. Then
	\[ \| ξ_{\tilde{ψ}} - a ξ_{\tilde{ψ}} \| = ⟨ (1 - a)^2 ξ_ψ \mid ξ_ψ ⟩ = ψ(1 - 2 a + a^2) = 0 \]
	and so \( a ξ_{\tilde{ψ}} = ξ_{\tilde{ψ}} \). Let \( f ∈ C_c(\bbR^×_+) \) be a compactly supported continuous function such that \( f(1) = 1 \) and \( \| f \|_∞ = 1 \). By the continuous functional calculus, \( f(a) \) is an element of the Pedersen ideal of \( A \) such that \( f(a) ξ_{\tilde{ψ}} = ξ_{\tilde{ψ}} \) and \( ψ(f(a)) = ⟨ f(a) ξ_ψ \mid ξ_ψ ⟩ = 1 = \| f(a) \| \). Hence \( ψ \) satisfies
	\begin{equation}
		\label{eq:states}
		ψ(f(a) b) = ⟨ f(a) b ξ_ψ \mid ξ_ψ ⟩ = ⟨ b ξ_ψ \mid f(a) ξ_ψ ⟩ = ⟨ b ξ_ψ \mid ξ_ψ ⟩ = ψ(b)
	\end{equation}
	for all \( b ∈ B \).	

	To see that (3) implies (4), let positive \( a ∈ A \) be such that \( ψ(a) = 1 = \| a \| \). As before, we must have \( ψ(a^2) = 1 \). For all \( b ∈ A \), as in \eqref{eq:states} we have
	\[ \frac{ψ(a b a)}{ψ(a^2)} =ψ(a b a) =\langle aba\xi_\psi,\xi_\psi\rangle=\langle b\xi_\psi,\xi_\psi\rangle= ψ(b) \]
	so we may simply choose \( ϕ = ψ \).
\end{proof}

\begin{remarks}
	\item In \cite[Chapter 3]{Lazar_1976}, a topology \( κ \) on \( Γ(K_A) \), the multipliers of the Pedersen ideal of \( A \), is introduced. In \cite[Proposition 6.5]{Lazar_1976}, condition (4) of Proposition \ref{proposition:compactly-supp-states} is shown to be equivalent to \( ψ \) being a norm-1 positive \( κ \)-continuous functional on \( Γ(K_A) \).
	\item For a locally compact Hausdorff space \( X \), recall that the states on \( C_0(X) \) are exactly given by the Radon probability measures on \( X \) \cite[II.6.2.3(ii)]{Blackadar_1998}. The compactly supported states on \( C_0(X) \) are then exactly given by the compactly supported Radon probability measures on \( X \).
\end{remarks}

\begin{proposition}
	cf. \cite[Lemma 6.12]{Harris_2023}
	The compactly supported states \( \mathcal{S}_c(A) \) on a C*-algebra \( A \) are weak-$*$-dense in \( \mathcal{S}(A) \).
\end{proposition}
\begin{proof}
	Let \( ψ \) be a state on \( A \). Using \cite[II.4.1.4]{Blackadar_1998}, let \( (h_λ)_{λ ∈ Λ} \) be an approximate unit for \( A \) contained in the Pedersen ideal \( K_A \). Consider the net of states \( (ψ_λ)_{λ ∈ Λ} \) given by
	\[ ψ_λ : a ↦ \frac{ψ(h_λ a h_λ)}{ψ(h_λ^2)} . \]
	Each of these is compactly supported by Proposition \ref{proposition:compactly-supp-states}(4). The net \( (ψ(h_λ^2))_{λ ∈ Λ} \) converges to \( 1 \) by \cite[II.6.2.5(i)]{Blackadar_1998}. To see that the net \( (ψ(h_λ a h_λ))_{λ ∈ Λ} \) converges to \( ψ(a) \), observe that
	\begin{align*}
		\| ψ(a) - ψ(h_λ a h_λ) \|
		& = \| ψ((1 - h_λ) a) + ψ(h_λ a (1 - h_λ)) \| \\
		& ≤ \left( \| (1 - h_λ) a \| + \| a (1 - h_λ) \| \right) \\
		& \to 0 ,
	\end{align*}
	where we have used the bounds \( \| ψ \| = 1 \) and \( \| h_λ \| ≤ 1 \).
\end{proof}

\begin{proposition}
	\label{proposition:compactly-supp-state-temp-bdd}
	Let \( E \) be a Hilbert \( B \)-module and \( C \) a C*-algebra. Let \( T \) be a regular operator on \( (E ⊗ C)_{B ⊗ C} \) which is \( C \)-matched. Then, for any compactly supported state \( ψ \) on \( C \), \( (1 ⊗ ψ)(T) \) is well-defined and a bounded operator on \( E \).
\end{proposition}
\begin{proof}
	The state \( ψ \) extends to a completely positive map \( 1 ⊗ ψ \) from \( \End^0(E ⊗ C) = \End^0(E) ⊗ C \) to \( \End^0(E) \). Being nondegenerate, this completely positive map further extends to a map from \( M(\End^0(E) ⊗ C) = \End^*(E ⊗ C) \) to \( M(\End^0(E)) = \End^*(E) \) \cite[Corollary 5.7]{Lance_1995}.
	
	Let \( a \) be a positive element of \( K_C \) such that \( ψ(a) = 1 = \| a \| \) and \( ψ(c) = ψ(a c) = ψ(c a) \) for all \( c ∈ C \). As \( (E ⊗ C) K_C ⊆ \dom T \), \( 1\otimes a (E ⊗ C) ⊆ \dom T \). By Lemma \ref{lemma:domain-preservation-on-banach-space}, \( T (1 ⊗ a) \) is a bounded operator on \( E ⊗ C \). Hence we may apply \( 1 ⊗ ψ \) to \( T (1 ⊗ a) \) to obtain an element of \( \End^*(E) \). To see that the choice of \( a \) does not affect the value of \( (1 ⊗ ψ)(T (1 ⊗ a)) \), let \( b ∈ K_C \) be another positive element such that \( ψ(b) = 1 = \| b \| \) and \( ψ(c) = ψ(b c) = ψ(c b) \) for all $c\in C$. We note that, because \( T^* \) is \( C \)-matched, \( T^* (1 ⊗ a) \) is also a bounded operator. We have a series of equalities
	\begin{align*}
		(1 ⊗ ψ)(T (1 ⊗ b))
		& = (1 ⊗ ψ)((1 ⊗ a) T (1 ⊗ b)) \\
		& = (1 ⊗ ψ)((1 ⊗ b) T^* (1 ⊗ a))^* \\
		& = (1 ⊗ ψ)(T^* (1 ⊗ a))^* \\
		& = (1 ⊗ ψ)((1 ⊗ a) T^* (1 ⊗ a))^* \\
		& = (1 ⊗ ψ)((1 ⊗ a) T (1 ⊗ a)) \\
		& = (1 ⊗ ψ)(T (1 ⊗ a))
	\end{align*}
	so that \( (1 ⊗ ψ)(T) \) has a unique meaning.
\end{proof}

\begin{proposition}
	\label{proposition:tensor-product-compactly-supported-states-dense}
	Let \( E \) be a Hilbert \( B \)-module and \( C \) a C*-algebra. Then \( 1 ⊗ \mathcal{S}_c(C) \) is dense in \( 1 ⊗ \mathcal{S}(C) \) in the pointwise-norm topology on completely positive maps from \( \End^0(E) ⊗ C \) to \( \End^0(E) \). That is, for \( ψ ∈ \mathcal{S}(C) \), there exists a net \( (ψ_λ)_{λ ∈ Λ} ⊆ \mathcal{S}_c(C) \) such that, for all \( y ∈ \End^0(E) ⊗ C \), \( (1 ⊗ ψ)(y) ∈ \End^0(E) \) is the norm limit of \( (1 ⊗ ψ_λ)(y) \). As a consequence, \( 1 ⊗ \mathcal{S}_c(C) \) is dense in \( 1 ⊗ \mathcal{S}(C) \) in the pointwise-norm topology on completely positive maps from \( \End^*(E ⊗ C) \) to \( \End^*(E) \).
\end{proposition}
\begin{proof}
	Let \( (h_λ)_{λ ∈ Λ} \) be an approximate unit for \( C \) contained in the Pedersen ideal \( K_C \). Let
	\[ ψ_λ : a ↦ \frac{ψ(h_λ a h_λ)}{ψ(h_λ^2)} . \]
	By \cite[Lemma 29.8]{Fragoulopoulou_2005}, \( (1 ⊗ h_λ)_{λ ∈ Λ} \) is an approximate unit for \( \End^*(E) ⊗ C \). For \( y ∈ \End^0(E) ⊗ C \),
	\begin{align*}
		\| (1 ⊗ ψ)(y) - (1 ⊗ ψ_λ)(y) \|
		& = \| (1 ⊗ ψ)((1 ⊗ (1 - h_λ)) y) + (1 ⊗ ψ)((1 ⊗ h_λ) y (1 ⊗ (1 - h_λ))) \| \\
		& ≤ \| 1 ⊗ ψ \| \left( \| (1 ⊗ (1 - h_λ)) y \| + \| y (1 ⊗ (1 - h_λ)) \| \right) \\
		& \to 0 ,
	\end{align*}
	as required.
	
	For the second statement, let \( H_ψ \) be the Hilbert space of the GNS representation of \( C \) corresponding to \( ψ \). One can check that the KSGNS construction \cite[Chapter 5]{Lance_1995} gives
	\[ (\End^0(E) ⊗ C) ⊗_{1 ⊗ ψ} E ≅ H_ψ ⊗ E . \]
	Let \( ξ_ψ \) be the cyclic vector of the GNS construction. Then, by \cite[Theorem 5.6]{Lance_1995},
	\[ (1 ⊗ ψ)(y) = (1 ⊗ ξ_ψ)^* y (1 ⊗ ξ_ψ) \]
	for \( y ∈ \End^0(E) ⊗ C \). By 	\cite[Corollary 5.7]{Lance_1995}, \( 1 ⊗ ψ \) is extended to a completely positive map from \( \End^*(E ⊗ C) \) to \( \End^*(E) \) by the same formula, viz.
	\[ (1 ⊗ ψ)(y) = (1 ⊗ ξ_ψ^*) y (1 ⊗ ξ_ψ) \]
	for \( y ∈ \End^*(E ⊗ C) \). We have
	\begin{align*}
		\| (1 ⊗ ψ)(y) - (1 ⊗ ψ_λ)(y) \|
		& = \| (1 ⊗ ψ)((1 ⊗ (1 - h_λ)) y) + (1 ⊗ ψ)((1 ⊗ h_λ) y (1 ⊗ (1 - h_λ))) \| \\
		& = \| (1 ⊗ ξ_ψ^*) (1 ⊗ (1 - h_λ)) y (1 ⊗ ξ_ψ) \\
		& \qquad + (1 ⊗ ξ_ψ^*) (1 ⊗ h_λ) y (1 ⊗ (1 - h_λ)) (1 ⊗ ξ_ψ) \| \\
		& ≤ 2 \| y \| \| (1 - h_λ) ξ_ψ \| \\
		& \to 0 ,
	\end{align*}
	as required.
\end{proof}

\let\oldbibliography\thebibliography
\renewcommand{\thebibliography}[1]{
  \oldbibliography{#1}
  \setlength{\itemsep}{1.7pt}
}

{
\small
\bibliography{../../Dynamics} 
\bibliographystyle{../../amsalphaprime}
}

\end{document}